\documentclass[10pt]{article}
\usepackage[utf8]{inputenc}
\usepackage{hyperref}
\usepackage{amsmath}
\usepackage{amssymb}
\usepackage{amsthm}
\usepackage[shortlabels]{enumitem}
\usepackage{verbatim}
\usepackage{microtype}
\usepackage{fullpage}
\usepackage{mathrsfs}
\usepackage{mleftright}
\usepackage{xcolor}

\usepackage{hyperref}

\begin{document}

\newcommand{\N}{\mathbb{N}}
\newcommand{\R}{\mathbb{R}}
\newcommand{\C}{\mathbb{C}}
\newcommand{\ZygSymb}{\mathscr{C}}
\newcommand{\BesovSymb}{\mathscr{B}}
\newcommand{\Manifold}{\mathfrak{M}}
\newcommand{\NManifold}{\mathfrak{M}}
\newcommand{\Ball}{\mathbb{B}}
\newcommand{\Dist}{\mathscr{D}'}
\newcommand{\loc}{\mathrm{loc}}
\newcommand{\IntProd}[1]{\iota_{#1}}
\newcommand{\Lie}[1]{\mathrm{Lie}_{#1}}
\newcommand{\mywedge}{{\textstyle \bigwedge }}
\newcommand{\Lap}{\bigtriangleup}
\newcommand{\codiff}{\vartheta}
\newcommand{\SchwartzSymb}{\mathscr{S}}
\newcommand{\psih}{\hat{\psi}}
\newcommand{\phih}{\hat{\phi}}
\newcommand{\phit}{\tilde{\phi}}
\newcommand{\rhot}{\tilde{\rho}}
\newcommand{\omegat}{\tilde{\omega}}
\newcommand{\tmu}{\tilde{\mu}}

\newcommand{\As}{\mathscr{A}}
\newcommand{\Bs}{\mathscr{B}}
\newcommand{\cfr}{\mathfrak{c}}
\newcommand{\Df}{\mathfrak{D}}
\newcommand{\Dc}{\mathcal{D}}
\newcommand{\Ec}{\mathcal{E}}
\newcommand{\Lc}{\mathcal{L}}
\newcommand{\Tc}{\mathcal{T}}
\newcommand{\Rc}{\mathcal{R}}
\newcommand{\Sp}{\mathbb{S}}
\newcommand{\Nanifold}{\mathfrak{N}}
\newcommand{\Mbb}{\mathbb{M}}
\newcommand{\FrP}{\mathfrak{P}}
\newcommand{\FrR}{\mathfrak{R}}
\newcommand{\eps}{\varepsilon}
\newcommand{\Co}{\mathscr{C}}
\newcommand{\Xs}{\mathscr{X}}
\newcommand{\Z}{\mathbb{Z}}
\newcommand{\supp}{\operatorname{supp}}
\newcommand{\dist}{\operatorname{dist}}
\newcommand{\id}{\mathrm{id}}
\newcommand{\Green}{\mathcal{G}}
\newcommand{\Coorvec}[1]{\frac\partial{\partial#1}}
\newcommand{\ZygVF}[2]{\widetilde{\Co^{#1}_{#2,\loc}}}

\newtheorem{question}{Question}

\newtheorem{thm}{Theorem}[section]
\newtheorem{cor}[thm]{Corollary}
\newtheorem{prop}[thm]{Proposition}
\newtheorem{lemma}[thm]{Lemma}
\newtheorem{conj}[thm]{Conjecture}
\newtheorem{prob}[thm]{Problem}
\newtheorem{sublem}{Sublemma}[thm]

\theoremstyle{remark}
\newtheorem{rmk}[thm]{Remark}

\theoremstyle{definition}
\newtheorem{defn}[thm]{Definition}
\newtheorem{conv}[thm]{Convention}

\newtheorem{note}[thm]{Notation}

\theoremstyle{definition}
\newtheorem{assumption}[thm]{Assumption}

\theoremstyle{remark}
\newtheorem{example}[thm]{Example}

\numberwithin{equation}{section}

\title{Improving the Regularity of Vector Fields}

\author{Brian Street\footnote{The first author was partially supported by National Science Foundation Grant No.\ 1764265.}\:  and Liding Yao\footnote{Corresponding Author. \color{blue}\url{lyao26@wisc.edu}}}
\date{}

\maketitle

\begin{abstract}
    Let $\alpha>0$, $\beta>\alpha$, and let $X_1,\ldots, X_q$ be $\mathscr{C}^{\alpha}_{\mathrm{loc}}$ vector fields on a $\mathscr{C}^{\alpha+1}$ manifold which span the tangent space at every point, where $\mathscr{C}^{s}$ denotes the Zygmund-H\"older space of order $s$.  We give necessary and sufficient conditions for when there is a $\mathscr{C}^{\beta+1}$ structure on the manifold, compatible with its $\mathscr{C}^{\alpha+1}$ structure, with respect to which $X_1,\ldots, X_q$ are $\mathscr{C}^{\beta}_{\mathrm{loc}}$.  This strengthens previous results of the first author which dealt with the setting $\alpha>1$, $\beta>\max\{ \alpha, 2\}$.
\end{abstract}

\section{Introduction}
Fix $\alpha>0$ and let $X_1,\ldots, X_q$ be $\ZygSymb^{\alpha}_{\loc}$ vector fields on a $\ZygSymb^{\alpha+1}$ manifold $\Manifold$ of dimension $n$, which span the tangent space at every point, where
$\ZygSymb^{\alpha}$ denotes the Zygmund-H\"older space of order $\alpha$.\footnote{For $m\in \N$ and $\alpha\in (0,1)$, $\ZygSymb^{m+\alpha}$ coincides with the usual H\"older space $C^{m,\alpha}$.  For $\alpha\in \{0,1\}$, these spaces differ:  $C^{m+1,0}\subsetneq C^{m,1}\subsetneq \ZygSymb^{m+1}$.}
In this paper, we investigate the following:
\begin{question}\label{Question::Intro::Global}
Fix $\beta\in [\alpha,\infty)$.  When is there a $\ZygSymb^{\beta+1}$ manifold structure on $\Manifold$,
compatible with its $\ZygSymb^{\alpha+1}$ structure, respect to which $X_1,\ldots, X_q$ are $\ZygSymb^{\beta}_{\loc}$?
\end{question}
Question \ref{Question::Intro::Global} is local in nature, so we focus instead on the following local version:
\begin{question}\label{Question::Intro::Local}
Fix $\beta\in [\alpha,\infty)$ and $x\in \Manifold$.  When is there a neighborhood $U\subseteq \Manifold$
of $x$ and a $\ZygSymb_{\loc}^{\alpha+1}$ diffeomorphism $\Phi:\Ball^n\xrightarrow{\sim} U$, such that
$\Phi^{*}X_1,\ldots, \Phi^{*}X_q$ are $\ZygSymb^{\beta}$ vector fields on $\Ball^n$?
Here, $\Ball^n$ denotes the open unit ball in $\R^n$.
\end{question}
We give necessary and sufficient conditions on $X_1,\ldots, X_q$ for when Question \ref{Question::Intro::Local} has an affirmative answer;
and therefore give necessary and sufficient conditions for when
Question \ref{Question::Intro::Global} has an affirmative answer.

When $\alpha>1$ and $\beta>2$, Questions \ref{Question::Intro::Global} and \ref{Question::Intro::Local}
were completely answered in work of the first author and Stovall \cite{StovallStreetI,StovallStreetII,StovallStreetIII}; which also proved stronger
quantitative results (see Section \ref{Section::Intro::QualVsQuant} for our distinction between quantitative and qualitative results).  
In this paper, by focusing only on the qualitative Questions \ref{Question::Intro::Global} and \ref{Question::Intro::Local} we are able to prove results for all $\alpha>0$, $\beta\geq \alpha$,
and the proof is simpler.  Our methods can also be used
to improve the quantitative results of \cite{StovallStreetI,StovallStreetII,StovallStreetIII}:
see Section \ref{Section::Quant}.

    \subsection{Informal Statement of Results}\label{Section::Intro::ResultsInTheSmoothSetting}
    Much of the difficulty in this paper comes from working with $\alpha$ and $\beta$ small.
In this section, we informally describe the results without worrying about such difficulties.

We begin with the case when $\beta=\infty$ and $X_1,\ldots, X_q$
are $C^1_{\loc}$ vector fields on a
$C^2$ manifold $\Manifold$ of dimension $n$, which span the tangent space at every point.  
This is a special cases of results in \cite{StovallStreetII}.


Because $X_1,\ldots, X_q$ are $C^1$ vector fields which span the tangent space at every point,
we may write
\begin{equation*}
    [X_i, X_j]=\sum_{k=1}^q c_{i,j}^k X_k, \quad c_{i,j}^k:\Manifold\rightarrow \R\text{ is continuous}.
\end{equation*}

\begin{thm}[\cite{StovallStreetII}]\label{Thm::IntroSmooth::SmoothThm}
Fix $x\in \Manifold$.  The following are equivalent:
\begin{enumerate}[(i)]
    \item There is a neighborhood $U\subseteq \Manifold$ of $x$ and a $C^2_{\loc}$ diffeomoprhism
    $\Phi:\Ball^n\xrightarrow{\sim} U$, such that $\Phi^{*}X_1,\ldots, \Phi^{*}X_q$
    are $C^\infty_{\loc}$ vector fields on $\Ball^n$.
    \item\label{Item::IntroSmooth::CinftyX} 
    The functions $c_{i,j}^k$ can be chosen so that the following holds.
    There is a neighborhood $V\subseteq \Manifold$ of $x$ such that 
    for every $L\in \N$ and every list $l_1,\ldots, l_L\in \{1,\ldots, q\}$, the functions
    $X_{l_1}X_{l_2}\cdots X_{l_L} c_{i,j}^k\big|_{V}:V\rightarrow \R$ are continuous.
    We will write this condition as $c_{i,j}^k\big|_{V}\in C^\infty_X(V)$.
\end{enumerate}
\end{thm}

In this paper, $X_1,\ldots, X_q$ are $\ZygSymb^{\alpha}_{\loc}$ vector fields
on a $\ZygSymb^{\alpha+1}$ manifold, $\Manifold$.  Informally, we wish our main result to
say that the following are equivalent:
\begin{enumerate}[(a)]
    \item There is a neighborhood $U\subseteq \Manifold$ of $x$ and a $\ZygSymb^{\alpha+1}_{\loc}$ diffeomoprhism
    $\Phi:\Ball^n\xrightarrow{\sim} U$, such that $\Phi^{*}X_1,\ldots, \Phi^{*}X_q$
    are $\ZygSymb^{\beta}_{\loc}$ vector fields on $\Ball^n$.
    \item\label{Item::IntroSmooth::MainCondition} The $c_{i,j}^k$ can be chosen such that $c_{j,k}^l\big|_{V}\in \ZygSymb^{\beta-1}_{X}(V)$ for some neighborhood $V\subseteq \Manifold$ of $x$, where $\ZygSymb^{\beta-1}_X(V)$
    is an appropriate Zygmund-H\"older space with respect to the vector fields $X_1,\ldots, X_q$.
\end{enumerate}

When $\alpha>1$ and $\beta>2$, this equivalence was proved in \cite{StovallStreetII}; the main result of this paper gives an extension of this result to all $\beta\geq \alpha>0$.  Unfortunately,
for $\alpha>0$ small, the commutator of $[X_j,X_k]$ does not even immediately make sense
and many of the usual operations on $\Manifold$ do not make immediate sense.  Thus, much
of this paper is devoted to making sense of conditions similar to \ref{Item::IntroSmooth::MainCondition}
in such low regularity.  As we will see, this is a bit easier when $\alpha>1/2$ and so our results
take a different form depending on whether $\alpha\in (0,1/2]$ or $\alpha>1/2$.

    \subsection{Relation to Results of DeTurck and Kazdan}
    The results in this paper may be reminiscent of the celebrated results of DeTurck and Kazdan
\cite{DeTurckKazdan} regarding a coordinate system in which a Riemannian metric tensor has optimal regularity.
It seems that there are no direct implications between our results and their
results; however there are many similarities.
We present this in more detail in Section \ref{Section::DeTurck};
there the following ideas are discussed.

DeTurck and Kazdan showed that a Riemannian metric tensor has optimal
regularity in harmonic coordinates \cite[Lemma 1.2]{DeTurckKazdan}.
Analogously, we present a natural Riemannian metric associated to
vector fields $X_1,\ldots, X_n$ (which form a basis of the tangent
space at every point)
such that $X_1,\ldots, X_n$ have optimal regularity in harmonic
coordinates with respect to this metric.

DeTurck and Kazdan also showed that a Riemannian metric tensor
may not have optimal regularity in geodesic normal
coordinates \cite[Example 2.3]{DeTurckKazdan}.
Analogously, we show that vector fields may not have optimal
regularity in canonical coordinates of the first kind.

However, the heart of our main result is not just to provide a coordinate
system in which vector fields have optimal regularity.  Instead,
we provide a test to determine what that optimal regularity is.
This test can be carried out in any coordinate system and does
not require solving any differential equations.

Both this work and \cite{DeTurckKazdan} 
use methods introduced by Malgrange \cite{MalgrangeSurLIntegrabiliteDesStructures}. 

\begin{rmk}
It may be somewhat unexpected that vector fields may not have optimal
regularity in canonical coordinates of the first kind.  Indeed,
there is a long history of writing vector fields in these coordinates
because they provide a coordinate system in which the vector fields
are often particularly easy to study.  In the theory of Lie groups
this is classical (see, for example, \cite[page 115]{ChevalleyTheoryOfLieGroups}).
Outside of the setting of Lie groups, canonical coordinates have been
used in the quantitative study of sub-Riemannian geometry,
beginning with the work of Nagel, Stein, and Wainger \cite{NagelSteinWaingerBallsAndMetricsDefinedByVectorFields},
and later used by Tao and Wright \cite{TaoAndWrightLpImprovingBoundsForAveragesAlongCurves},
the first author \cite{StreetMultiparameterCarnotCaratheodoryBalls},
Montanari and Morbidelli \cite{MontanariMorbidellliNonsmoothHormanderVectorFields},
and the first author and Stovall \cite{StovallStreetI}, among others.
In \cite{StovallStreetII}, the first author moved beyond canonical coordinates
to strengthen these theories.  We see now that this is necessary:
sharp results like the ones in this paper and in \cite{StovallStreetII}
cannot be obtained using canonical coordinates.
\end{rmk}

\begin{rmk}
If $X$ is a nonzero $C^1$ vector field on a one dimensional manifold $\Manifold$,
then the canonical coordinate system with respect to $X$, near the point $x_0$,
is the map $\Phi_{x_0}(t)=e^{tX}x_0$.  Since $\Phi_{x_0}^{*}X=\partial_t$,
we see that canonical coordinate system do provide optimal regularity in this simple
setting.  However, once we move to two dimensions, with two vector fields,
Lemma \ref{Lemma::CanonicalCoords} shows that canonical coordinates may not
give the optimal regularity.
\end{rmk}
    
    \subsection{Qualitative versus Quantitative}\label{Section::Intro::QualVsQuant}
    Most of the results in this paper are \textit{qualitative} in the following sense.
We give necessary and sufficient conditions so that a map $\Phi$ as in Question \ref{Question::Intro::Local}
exists.  If one traces through the proof, the $\ZygSymb^{\beta}$ norms of the coefficients of
$\Phi^{*}X_1,\ldots, \Phi^{*}X_q$ depend on   (among other things) quantities like:
\begin{itemize}[parsep=-0.3ex]
    \item Upper bounds for $\ZygSymb^{\alpha}$ norms of the coefficients of $X_1,\ldots, X_q$ in some fixed coordinate system
    near $x$.
    \item A lower bound, $>0$, for the quantity:
    \begin{equation}\label{Eqn::Intro::QualVsQuant::LowerBoundDet}
        \max_{j_1,\ldots, j_n\in \{1,\ldots q\}} \mleft| \det \mleft( X_{j_1}(x)|\ldots| X_{j_n}(x)\mright) \mright|,
    \end{equation}
    where in the above expression $X_1,\ldots, X_q$ are written as column vectors in the same fixed coordinate system near $x$.
\end{itemize}
Unfortunately, both above quantities depend on the choice of the original coordinate system.
Thus, if the vector fields are given in a coordinate system where the above upper and lower bounds
are bad, the estimates our proof gives are bad; even if there exists a different (unknown) choice of coordinate
system where the above estimates are better.  Thus, while our results are qualitatively optimal
(we give necessary and sufficient conditions for each $\beta$),
the estimates which follow from our proofs may be far from optimal unless one happens to know a good coordinate
system in which to write the vector fields in the first place.

In the papers \cite{StovallStreetI,StovallStreetII,StovallStreetIII}, the first author and Stovall
give such estimates on $\Phi^{*}X_1,\ldots, \Phi^{*} X_q$ in terms of quantities which are invariant under arbitrary $C^2$ diffeomorphisms (we call such estimates quantitative).
Thus, they do not depend on any choice of coordinate system.
This is useful for questions from partial differential equations and harmonic analysis, where
$\Phi$ can be used as a scaling map.  Such scaling maps originated in the smooth setting
in the foundational work of Nagel, Stein, and Wainger \cite{NagelSteinWaingerBallsAndMetricsDefinedByVectorFields} and were later worked on
by Tao and Wright \cite{TaoAndWrightLpImprovingBoundsForAveragesAlongCurves},
the first author \cite{StreetMultiparameterCarnotCaratheodoryBalls}, and in the 
above-mentioned series of papers by the first author and Stovall \cite{StovallStreetI,StovallStreetII,StovallStreetIII}.  Similar scaling in a non-smooth
setting was studied by Montanari and Morbidelli \cite{MontanariMorbidellliNonsmoothHormanderVectorFields},
though they do not address questions like the ones in this paper.

In Section \ref{Section::Quant}, we use the main methods of this paper,
combined with the methods of \cite{StovallStreetI,StovallStreetII}
to improve the main quantitative result of \cite{StovallStreetII} (and also the main
quantitative result of \cite{StreetSubHermitian}).

\section{Function Spaces}\label{Section::Func}

To state our main result, we need to introduce several function spaces related to the classical
Zygmund-H\"older spaces.  Because we are working in low regularity, some care is needed in the
definitions.

    
    \subsection{Classical Zygmund-H\"older spaces}
    In this section, we describe the classical Zygmund-H\"older spaces, and the corresponding
spaces on a manifold; see Section \ref{Section::FuncRevisited} for proofs of the results stated here.  

In what follows, $U$ will either be equal to $\R^n$ or equal to {an bounded open set with smooth boundary in $\R^n$}--we will usually be interested in the case when $U$ is either $\R^n$ or
an open ball in $\R^n$. 
We define the Zygmund-H\"older space
$\ZygSymb^s(U):=\BesovSymb^s_{\infty,\infty}(U)$, where $\BesovSymb^s_{\infty,\infty}$
denotes the classical Besov space (see, \cite[Section 2.3]{TriebelTheoryOfFunctionSpacesI}
for $\BesovSymb^s_{\infty,\infty}(U)$ 
when $U=\R^n$, and \cite[Chapter 5]{TriebelTheoryOfFunctionSpacesII} or {\cite[Chapter 1.11]{TriebelTheoryOfFunctionSpacesIII}} when 
$U$ is an bounded smooth domain).\footnote{Many results concerning $\ZygSymb^s(U)$, where $U$ is a bounded smooth domain, follow from the corresponding
results concerning $\ZygSymb^s(\R^n)$ via the theory described in \cite[Chapter 5]{TriebelTheoryOfFunctionSpacesII}.}
  We similarly define the vector valued space $\ZygSymb^s(U;\R^m)$.
The space $\ZygSymb^s(U)$ has some particularly concrete characterizations:

\begin{rmk}\label{Rmk::FuncSpace::CharofZyg}
For $U=\R^n$ or $U$ an bounded open set with smooth boundary in $\R^n$, we have
\begin{enumerate}[parsep=-0.3ex,label=(\roman*)]
    \item\label{Item::FuncSpace::CharofZyg::02} $s\in (0,2)$: $\ZygSymb^s(U)$ consists of those continuous functions $f:U\rightarrow \R$ such that the following norm is finite
    \begin{equation*}
         \sup_{x\in U} |f(x)| + \sup_{\substack{x\in U \\ h\in \R^n, h\ne 0 \\ x+h, x+2h\in U}} |h|^{-s} |f(x+2h)-2f(x+h)+f(x)|.
    \end{equation*}
    Moreover, the above expression gives a norm equivalent to $\| f\|_{\ZygSymb^s(U)}$.
    See \cite[Theorem 2.5.7 (ii)]{TriebelTheoryOfFunctionSpacesI} and \cite[(3.4.2/6)]{TriebelTheoryOfFunctionSpacesI}.

    \item\label{Item::FuncSpace::CharofZyg::01} $s\in (0,1)$: $\ZygSymb^s(U)$ consists of those continuous functions $f:U\rightarrow \R$ such that the following norm is finite
    \begin{equation*}
         \sup_{x\in U} |f(x)| + \sup_{\substack{x,y\in U \\ x\neq y}} |x-y|^{-s} |f(x)-f(y)|.
    \end{equation*}
    Moreover, the above expression gives a norm equivalent to $\| f\|_{\ZygSymb^s(U)}$.
    See \cite[Remark 2.2.2/3]{TriebelTheoryOfFunctionSpacesI} and \cite[(3.4.2/6)]{TriebelTheoryOfFunctionSpacesI}.
    
    \item\label{Item::FuncSpace::CharofZyg::>1} $s\in (1,\infty]$: $\ZygSymb^s(U)$ consists of those continuous $f:U\rightarrow \R^n$, such that
    $f, \partial_{x_j} f\in \ZygSymb^{s-1}(U)$, $1\leq j\leq n$.  We have
    $\| f\|_{\ZygSymb^s(U)}\approx \| f\|_{\ZygSymb^{s-1}(U)} + \sum_{j=1}^n \| \partial_{x_j} f\|_{\ZygSymb^{s-1}(U)}$.
    See \cite[Theorem 2.5.7 (ii)]{TriebelTheoryOfFunctionSpacesI} and \cite[Theorem 3.3.5(i)]{TriebelTheoryOfFunctionSpacesI}.
    \item\label{Item::FuncSpace::CharofZyg::<0} $s\in (-\infty,0]$:  $\ZygSymb^s(U)$ consists of those distributions $f\in \Dist(U)$
    such that $f=g_0+\sum_{j=1}^n \partial_{x_j} g_j$ for some $g_0,\ldots, g_n\in \ZygSymb^{s+1}(U)$.
    We have $\|f\|_{\ZygSymb^{s}(U)}\approx \inf \sum_{j=0}^n \|g_j\|_{\ZygSymb^{s+1}(U)}$, where the
    infimum is taken over all such choices of $g_0,\ldots, g_n$.
    When $U=\R^n$, this can be seen by letting $g_0=(I+\Lap)^{-1}f$ and $g_j=-\partial_{x_j} (I+\Lap)^{-1}f$, for $j=1,\ldots n$. See \cite[Theorem 2.3.8]{TriebelTheoryOfFunctionSpacesI}.
    \item When $m\in \N$ and $r\in (0,1)$, then
    $\ZygSymb^{m+r}(U)=C^{m,r}(U)$, with equivalence of norms.
    See \cite[Theorem 1.118 (i)]{TriebelTheoryOfFunctionSpacesIII}.
    However, when $r\in \{0,1\}$ these spaces differ.
\end{enumerate}
\end{rmk}

\begin{lemma}\label{Lemma::FuncSpace::Product}
Let $r,s\in \R$ with $r+s>0$, $r\geq s$.  
The product map $(f,g)\mapsto fg$ can be defined as a continuous map $\ZygSymb^{r}(U)\times \ZygSymb^{s}(U)\rightarrow \ZygSymb^{s}(U)$.
\end{lemma}
\begin{proof}
This is a special case of \cite[Theorem 2.8.2(i)]{TriebelTheoryOfFunctionSpacesI} {when $U=\R^n$} and \cite[Theorem 3.3.2(ii)]{TriebelTheoryOfFunctionSpacesI} {when $U$ is bounded open set with smooth boundary}.
\end{proof}

\begin{defn}
Let $U\subseteq \R^n$ be an open set. For $s\in \R$, we define $\ZygSymb^s(U; TU)$, to be the space of vector fields (with distribution
coefficients)
$Y=\sum_{j=1}^n a_j \partial_{x_j}$, where $a_j\in \ZygSymb^s(U)$.  We identify
$Y$ with the distribution $(a_1,\ldots, a_n)\in \ZygSymb^s(U; \R^n)$ and define
\begin{equation*}
    \| Y\|_{\ZygSymb^s(U)}:= \| (a_1,\ldots, a_n)\|_{\ZygSymb^s(U;\R^n)}.
\end{equation*}
\end{defn}

\begin{defn}
Let $U\subseteq \R^n$ be an open set. For $s\in \R$ and $k\in \N$, we define $\ZygSymb^s\mleft(U; \mywedge^k T^{*}U\mright)$,
to be the space of $k$-forms (with distribution coefficients)
$\omega=\sum_{1\leq i_1<i_2<\cdots<i_k\leq n} \omega_{i_1,\ldots, i_k} dx_1\wedge \cdots\wedge dx_k$,
where $\omega_{i_1,\ldots, i_k}\in \ZygSymb^s(U)$.  We identify $\omega$ with the
distribution $(\omega_{i_1,\ldots, i_k})_{1\leq i_1<i_2<\ldots<i_k\leq n}$, and define
\begin{equation*}
    \| \omega\|_{\ZygSymb^s(U)} := \| (\omega_{i_1,\ldots, i_k})\|_{\ZygSymb^s(U; \R^{Q_{n,k}})},
\end{equation*}
where $Q_{n,k}=\dim \mywedge^k \R^n$.
\end{defn}

\begin{defn}\label{Defn::FuncRn:DefnofZygloc}
Let $U\subseteq \R^n$ be an open set. For $s\in\R$, we define $\ZygSymb^s_{\loc}(U)$ to be the space of distributions $f\in \Dist(U)$, such that
for every $x\in U$, there exists an open ball $U'\subseteq U$ containing $x$ with 
$f\big|_{U'}\in \ZygSymb^s(U')$.  We similarly define $\ZygSymb^s_{\loc}(U; TU)$ and
$\ZygSymb^s_{\loc}\mleft(U; \mywedge^k T^{*}U\mright)$.
\end{defn}

For $\alpha>0$, we define $\ZygSymb^{\alpha+1}$ manifolds in the usual way:  the transition
functions are assumed to be $\ZygSymb^{\alpha+1}_{\loc}$ (see \cite[Section 5.4]{StovallStreetII}
for some comments on this).  Such manifold are, in particular, $C^1$ manifolds, and so it
makes sense to talk about, for example, vector fields on such manifolds.

\begin{rmk}\label{Rmk::FuncRn::DefinedOnMfld}
On a $\ZygSymb^{\alpha+1}$ manifold $\Manifold$, it makes sense to talk about functions
in $\ZygSymb^s_{\loc}(\Manifold)$ for $s\in (-\alpha,\alpha+1]$, vector fields
in $\ZygSymb^s_{\loc}(\Manifold; T\Manifold)$ for $s\in (-\alpha,\alpha]$,
and $k$-forms in $\ZygSymb^s_{\loc}\mleft(\Manifold;\mywedge^k T^{*}\Manifold\mright)$ for $s\in (-\alpha,\alpha]$.
See Lemma \ref{Lemma::FuncRevis::PushForwardFuncSpaces} {and Definition \ref{Defn::FuncRevis::ObjonManifolds}}.
\end{rmk}

Let $\IntProd{Y}$ denote the interior product with respect to the vector field $Y$
and let $\Lie{Y}$ denote the Lie derivative with respect to $Y$.  

\begin{prop}\label{Prop::FuncRn::LieOnManiofold}
Let $\Manifold$ be a $\ZygSymb^{\alpha+1}$ manifold for some
$\alpha>0$.
\begin{enumerate}[parsep=-0.3ex,label=(\roman*)]
\item\label{Item::FuncRn::IntProdCont} For $\beta\in (-\alpha,\alpha]$, the map
$(Y,\omega)\mapsto \IntProd{Y}\omega$ is a continuous map
\begin{equation*}
 \ZygSymb^\alpha_{\loc}(\Manifold;T\Manifold)\times \ZygSymb^\beta_{\loc}\mleft(\Manifold; \mywedge^k T^*\Manifold\mright)\rightarrow \ZygSymb_{\loc}^{\beta} \mleft(\Manifold; \mywedge^{k-1}T^{*}\Manifold\mright).
\end{equation*}

\item\label{Item::FuncRn::WedgeProdCont} For $\beta\in (-\alpha,\alpha]$, the map
$(\eta,\omega)\mapsto \eta\wedge\omega$ is a continuous map
\begin{equation*}
 \ZygSymb^\alpha_{\loc}\mleft(\Manifold;\mywedge^lT^*\Manifold\mright)\times \ZygSymb^\beta_{\loc}\mleft(\Manifold; \mywedge^k T^*\Manifold\mright)\rightarrow \ZygSymb_{\loc}^{\beta} \mleft(\Manifold; \mywedge^{k+l}T^{*}\Manifold\mright).
\end{equation*}

\item\label{Prop::FuncRn::dCont} If $\alpha>1/2$ and $\beta\in (-\alpha+1,\alpha]$, the map $\omega\mapsto d\omega$ is continuous 
\begin{equation*}
    \ZygSymb^\beta_{\loc}\mleft(\Manifold; \mywedge^k T^*\Manifold\mright)\rightarrow 
    \ZygSymb^{\beta-1}_{\loc}\mleft(\Manifold; \mywedge^{k+1} T^*\Manifold\mright).
\end{equation*}

\item\label{Item::FuncRn::VectActionCont} For $\beta\in (-\alpha+1,\alpha+1]$, the map
\begin{equation*}
(Y,f)\mapsto Yf =:\Lie{Y}f
\end{equation*}
is continuous
\begin{equation*}
\ZygSymb^\alpha_{\loc}(\Manifold;T\Manifold)\times \ZygSymb^\beta_{\loc}(\Manifold)\rightarrow \ZygSymb^{\beta-1}_{\loc}(\Manifold).
\end{equation*}


\item\label{Item::FuncRn::CommutatorCont} If $\alpha>1/2$, then for $\beta\in (-\alpha+1, \alpha]$, the map
\begin{equation*}
    (Y,Z)\mapsto [Y,Z]=:\Lie{Y}Z
\end{equation*}
is continuous
\begin{equation*}
    \ZygSymb^\alpha_{\loc}(\Manifold;T\Manifold)\times \ZygSymb^\beta_{\loc}(\Manifold;T\Manifold)\rightarrow \ZygSymb^{\beta-1}_{\loc}(\Manifold;T\Manifold).
\end{equation*}
\item\label{Item::FuncRn::LieFormsCont} If $\alpha>1/2$, then for $\beta\in (-\alpha+1,\alpha]$, the map
\begin{equation*}
    (Y,\omega)\mapsto d\IntProd{Y}\omega+\IntProd{Y}d\omega=:\Lie{Y}\omega
\end{equation*}
is continuous
\begin{equation*}
    \ZygSymb^\alpha_{\loc}(\Manifold;T\Manifold)\times \ZygSymb^\beta_{\loc}(\Manifold;T\Manifold)\rightarrow \ZygSymb^{\beta-1}_{\loc}(\Manifold;T\Manifold).
\end{equation*}
\end{enumerate}
\end{prop}

As can be seen in Section \ref{Section::Intro::ResultsInTheSmoothSetting}, our main results
are in terms of the commutators of vector fields:  i.e., the Lie derivative of one
vector field with respect to another.  Proposition \ref{Prop::FuncRn::LieOnManiofold} \ref{Item::FuncRn::CommutatorCont} shows
that  such Lie derivatives, $\Lie{Y}$, only make sense
when $Y\in \ZygSymb^{\alpha}_{\loc}(\Manifold;T\Manifold)$ for $\alpha>1/2$.  Because
of this, when $\alpha>1/2$, the characterizations in our main result can be made somewhat simpler.
However, it is still possible to make sense of some of these ideas when $\alpha\in (0,1/2]$,
as we now make precise.

\begin{prop}\label{Prop::FuncRn::dOfLowRegularityForm}
Let $\alpha>0$ and $\beta, \gamma\in (-\alpha,\alpha+1]$. 
Let $U, V\subseteq \R^n$ be open and let $F:U\xrightarrow{\sim} V$ be a $\ZygSymb^{\alpha+1}_{\loc}$ diffeomorphism.  Fix a $k$-form $\theta\in \ZygSymb^{\gamma}_{\loc}\mleft(U; \mywedge^{k} T^{*}U\mright)$.
Then, the following are equivalent:
\begin{enumerate}[parsep=-0.3ex,label=(\roman*)]
    \item\label{Item::FuncRn::dOfLowRegularityForm::dthetasmooth} $d\theta\in \ZygSymb^{\beta-1}_{\loc}\mleft(U; \mywedge^{k+1} T^{*}U\mright)$.
    \item\label{Item::FuncRn::dOfLowRegularityForm::dFthetasmooth} $d(F_{*}\theta)\in \ZygSymb^{\beta-1}_{\loc}\mleft(V; \mywedge^{k+1} T^{*}V\mright)$.
\end{enumerate}
Moreover, in this case, for all $p\in U$, there is a neighborhood $V'\subseteq V$ of $F(p)$
and $\tau\in \ZygSymb^{\beta}_{\loc}\mleft(V'; \mywedge^{k} T^{*}V'\mright)$ such that
$d(F_{*} \theta)\big|_{V'}=d\tau$.
\end{prop}

For the remainder of this section, let $\Manifold$ be a $\ZygSymb^{\alpha+1}$ manifold for some $\alpha>0$.

\begin{defn}\label{Defn::FuncRn::dRegularityLow}
Let $\gamma\in (-\alpha,\alpha]$,  $\beta\in [\gamma,\alpha+1]$, and $\theta\in \ZygSymb^{\gamma}_{\loc}\mleft(\Manifold; \mywedge^{k} T^{*}\Manifold\mright)$.  We say $d\theta$ has regularity $\ZygSymb^{\beta-1}_{\loc}(\Manifold)$,
if for any $p\in \Manifold$, there is a $\ZygSymb^{\alpha+1}_{\loc}$ coordinate system $F:V\xrightarrow{\sim} U$,
where $V$ is a neighborhood of $p$ and $U\subseteq \R^n$ is open, such that $dF_{*}\theta\in \ZygSymb^{\beta-1}_{\loc}\mleft(U; \mywedge^{k+1} T^{*}U\mright)$.
\end{defn}

Proposition  \ref{Prop::FuncRn::dOfLowRegularityForm} shows that 
Definition \ref{Defn::FuncRn::dRegularityLow} is well-defined:  it does not depend on 
the choice of the coordinate system $F$.  However, we do not define the form $d\theta$ itself:
we only define its regularity.  Indeed, if $\beta-1\leq -\alpha$, the space
$\ZygSymb^{\beta-1}_{\loc}\mleft(\Manifold; \mywedge^{k+1} T^{*}\Manifold\mright)$
is not well-defined. 
However, when $\beta-1>-\alpha$, the form $d\theta$ is well-defined, as the next result shows.

\begin{lemma}\label{Lemma::FuncRn::dwelldefined}
Let $\gamma\in (-\alpha,\alpha]$ and $\beta\in (-\alpha+1,\alpha+1]$.
Suppose $\theta\in \ZygSymb^\gamma_{\loc}\mleft(\Manifold;\mywedge^k T^*\Manifold\mright)$
is such that $d\theta$ has regularity $\ZygSymb^{\beta-1}_{\loc}(\Manifold)$.
Then, $d\theta$ is given by a well-defined form in $\ZygSymb^{\beta-1}_{\loc}\mleft(\Manifold;\mywedge^{k+1} T^*\Manifold\mright)$.
I.e., there is a unique form $\tau\in \ZygSymb^{\beta-1}_{\loc}\mleft(\Manifold;\mywedge^{k+1} T^*\Manifold\mright)$ such that in every coordinate system $F:V\xrightarrow{\sim}U$, where $V\subseteq \Manifold$ and $U\subseteq \R^n$ are open, we have $F_{*} \tau = d(F_* \theta)$.  Furthermore,
this form $\tau$ is closed in the sense that in every such coordinate system, we have $d F_{*}\tau =0$.
\end{lemma}

\begin{defn}\label{Defn::FuncRn::dthetaisZyg}
For $\gamma\in (-\alpha,\alpha]$, $\beta\in (-\alpha+1,\alpha+1]$, 
and $\theta \in \ZygSymb^\gamma_{\loc}\mleft(\Manifold;\mywedge^k T^*\Manifold\mright)$,
we write $d\theta\in \ZygSymb^{\beta-1}_{\loc}\mleft(\Manifold; \mywedge^{k+1} T^*\Manifold\mright)$ to mean $d\theta$ has regularity $\ZygSymb^{\beta-1}_{\loc}(\Manifold)$,
and we identify $d\theta$ with the unique closed form 
in $\ZygSymb^{\beta-1}_{\loc}\mleft(\Manifold; \mywedge^{k+1} T^*\Manifold\mright)$
given in Lemma \ref{Lemma::FuncRn::dwelldefined}.
\end{defn}



\begin{conv}\label{Conv::C^+Space}
For $\beta\in\R$, we say $\theta\in\Co^{\beta^+}_\loc\mleft(\Manifold;\mywedge^kT^*\Manifold\mright)$, if $\theta\in\Co^{\beta+\eps}_\loc\mleft(\Manifold;\mywedge^kT^*\Manifold\mright)$ for some $\eps>0$.
\end{conv}


    
    \subsection{Zygmund-H\"older spaces with respect to vector fields}
    Let $\Manifold$ be a $\ZygSymb^{\alpha+1}$ manifold for some $\alpha>0$, and let $X_1,\ldots, X_q\in \ZygSymb^{\alpha}_{\loc}(\Manifold; T\Manifold)$ be $\ZygSymb^{\alpha}_{\loc}$ vector fields which span the tangent
space to $\Manifold$ at every point.
Since $\Manifold$ is only a $\ZygSymb^{\alpha+1}$ manifold, it does not make sense to talk about whether a function on $\Manifold$ has regularity higher than $\ZygSymb^{\alpha+1}_{\loc}$.
However, we can make sense of higher regularity with respect to the vector fields $X_1,\ldots, X_q$.

\begin{defn}[$\Co^\beta_{X,\loc}$-functions]\label{Defn::FuncVF::Funregularity}
For $\beta>-\alpha$, we let $\ZygSymb^{\beta}_{X,\loc}(\Manifold)$ be the space of those functions $\Manifold \rightarrow \R$ 
defined recursively by:
\begin{itemize}[parsep=-0.3ex]
    \item If $\beta\in (-\alpha,1]$, $\ZygSymb^{\beta}_{X,\loc}(\Manifold):=\ZygSymb^{\beta}_{\loc}(\Manifold)$.
    \item If $\beta>1$,  $\ZygSymb^{\beta}_{X,\loc}(\Manifold)$ consists of those $f\in \ZygSymb^{\beta-1}_{X,\loc}(\Manifold)\bigcap C^1_{\loc}(\Manifold)$ such that 
    $\Lie{X_j} f = X_j f \in \ZygSymb^{\beta-1}_{X,\loc}(\Manifold)$, for $1\leq j\leq q$.
\end{itemize}
\end{defn}

We can make a similar definition for vector fields and forms, so long as $\alpha>1/2$:
\begin{defn}[$\Co^\beta_{X,\loc}$-vector fields]\label{Defn::FuncVF::VFregularity}
Suppose $\alpha>1/2$.  For $\beta>-\alpha$, we let $\ZygSymb^{\beta}_{X,\loc}(\Manifold; T\Manifold)$ be
be the space of those vector fields on $\Manifold$ defined recursively by:
\begin{itemize}[parsep=-0.3ex]
    \item If $\beta\in (-\alpha,\frac12]$, $\ZygSymb^{\beta}_{X,\loc}(\Manifold;T\Manifold)=\ZygSymb^{\beta}_{\loc}(\Manifold;T\Manifold)$.
    \item If $\beta>\frac12$,  $\ZygSymb^{\beta}_{X,\loc}(\Manifold; T\Manifold)$ consists
    of those $Y\in \ZygSymb^{\beta-1}_{X,\loc}(\Manifold;T\Manifold)\bigcap \ZygSymb^\frac12_{\loc}(\Manifold;T\Manifold)$ such that
    $\Lie{X_j}Y = [X_j,Y]\in \ZygSymb^{\beta-1}_{X,\loc}(\Manifold; T\Manifold)$, for $1\leq j\leq q$.
\end{itemize}
\end{defn}

\begin{defn}[$\Co^\beta_{X,\loc}$-forms]\label{Defn::FuncVf::formregularity}
Suppose $\alpha>1/2$ and $k\geq 1$.  For $\beta>-\alpha$, we let $\ZygSymb^{\beta}_{X,\loc}\mleft(\Manifold; \mywedge^k T^{*}\Manifold\mright)$
be the space of those $k$-forms on $\Manifold$ defined recursively by:
\begin{itemize}[parsep=-0.3ex]
    \item If $\beta\in (-\alpha,\frac12]$,
    $\ZygSymb^{\beta}_{X,\loc}\mleft(\Manifold; \mywedge^kT^*\Manifold\mright)=\ZygSymb^{\beta}_{\loc}\mleft(\Manifold; \mywedge^kT^*\Manifold\mright)$.
    
    \item If $\beta>\frac12$,  $\ZygSymb^{\beta}_{X,\loc}\mleft(\Manifold; \mywedge^kT^*\Manifold\mright)$ consists of those
    $\theta\in \ZygSymb^{\beta-1}_{X,\loc}\mleft(\Manifold; \mywedge^kT^*\Manifold\mright)\bigcap \ZygSymb^\frac12_{\loc}\mleft(\Manifold; \mywedge^kT^*\Manifold\mright)$ such that
    $\Lie{X_j} \theta= (d\IntProd{X_j} + \IntProd{X_j} d) \theta\in \ZygSymb^{\beta-1}_{X,\loc}\mleft(\Manifold; \mywedge^kT^*\Manifold\mright)$,
    for $1\leq j\leq q$.
\end{itemize}
\end{defn}

When $\alpha\in (0,1/2]$, we cannot use Definition \ref{Defn::FuncVf::formregularity}.  However, we can make an appropriate analog of Definition \ref{Defn::FuncRn::dRegularityLow}
with respect to the vector fields $X_1,\ldots, X_q$:
\begin{defn}[$\Co^\beta_{X,\loc}$ for differentials]\label{Defn::FuncVf::RegularityofForms}
Suppose $\alpha>0$.
Let $\beta>-\alpha$, $1\le k\le n-1$ and let $\theta\in \ZygSymb^{(-\alpha)^+}_{\loc}\mleft(\Manifold; \mywedge^kT^*\Manifold\mright)$.
We say $d\theta$ has regularity $\ZygSymb^{\beta-1}_{X,\loc}(\Manifold)$ if:
\begin{itemize}[parsep=-0.3ex]
\item If $\beta\in (-\alpha,1]$, we assume $d\theta$ has regularity
    $\ZygSymb^{\beta-1}_{\loc}(\Manifold)$. 
    \item\label{Item::FuncVf::Recursion} If $\beta\in(1,2]$, we assume $d\theta\in \Co^{0^+}_{\loc}\mleft(\Manifold; \mywedge^{k+1}T^*\Manifold\mright)$ and $\Lie{X_j} d\theta = d\IntProd{X_j} d\theta$ has regularity $\ZygSymb^{\beta-2}_{\loc}(\Manifold)$,
    for $1\leq j\leq q$.
    \item If $\beta>2$,  we assume  $d\theta$ has regularity $\ZygSymb^{\beta-2}_{X,\loc}(\Manifold)$,
    and $\Lie{X_j} d\theta = d\IntProd{X_j} d\theta$ has regularity $\ZygSymb^{\beta-2}_{X,\loc}(\Manifold)$,
    for $1\leq j\leq q$.
\end{itemize}
\end{defn}


\begin{rmk}\label{Rmk::FuncVF::Independofalpha}
Note that if $\beta\ge0$, Definitions \ref{Defn::FuncVF::Funregularity} and \ref{Defn::FuncVf::RegularityofForms} do not depend on $\alpha$. 
Similarly, when $\beta\ge-\frac12$, Definitions \ref{Defn::FuncVF::VFregularity} and {\ref{Defn::FuncVf::formregularity}} do not depend on $\alpha$. 

\end{rmk}




{
}
    
\section{The Main Result}
\begin{thm}\label{Thm::TheMainResult}
Let $\alpha,\beta>0$, and let $X_1,\ldots, X_q$ be $\ZygSymb^{\alpha}_{\loc}$
vector fields on a $\ZygSymb^{\alpha+1}$ manifold $\Manifold$ of dimension $n$, which span the tangent space
at every point.  Fix a point $p\in \Manifold$, and re-order $X_1,\ldots, X_q$
so that $X_1(p),\ldots, X_n(p)$ form a basis for $T_p\Manifold$.
Let $\lambda^1,\ldots, \lambda^n$ be the dual basis for $X_1,\ldots, X_n$, defined
on a neighborhood of $p$.  Then, the following are equivalent:
\begin{enumerate}[label=(\alph*),series=maintheoremenumeration]
    \item\label{Item::TheMainResult::ExistChart} There is a neighborhood $U\subseteq \Manifold$ of $p$ and a $\ZygSymb^{\alpha+1}_{\loc}$ diffeomorphism $\Phi:\Ball^n\xrightarrow{\sim} U$,
    such that $\Phi(0)=p$ and $\Phi^{*}X_1,\ldots, \Phi^{*} X_q\in \ZygSymb^{\beta}(\Ball^n;T\Ball^n)$.
    \item\label{Item::TheMainResult::dForm} There is a neighborhood $U$ of $p$ such that for $1\leq j\leq n$,
    $d\lambda^j$ has regularity $\ZygSymb^{\beta-1}_{X,\loc}(U)$, and 
    for $1\leq j\leq n$, $n+1\leq k\leq q$, $\langle \lambda^j, X_k\rangle \in \ZygSymb^{\beta}_{X,\loc}(U)$.
\end{enumerate}
If $\alpha>1/2$, then in addition we have the following equivalent conditions:
\begin{enumerate}[resume*=maintheoremenumeration]
    \item\label{Item::TheMainResult::1Form} There is a neighborhood $U$ of $p$ such that
    for $1\leq j\leq n$, $\lambda^j\in \ZygSymb^{\beta}_{X,\loc}\mleft(U, T^{*}U\mright)$, and for $1\leq j\leq n$, $n+1\leq k\leq q$, $\langle \lambda^j, X_k\rangle \in \ZygSymb^{\beta}_{X,\loc}(U)$.
    
    \item\label{Item::TheMainResult::VF} There is a neighborhood $U$ of $p$ such that for $1\leq j\leq q$,
    $X_j\in \ZygSymb^{\beta}_{X,\loc}\mleft(U; TU\mright)$.
\end{enumerate}
\end{thm}

\begin{rmk}
    \ref{Item::TheMainResult::ExistChart} $\Leftrightarrow$ \ref{Item::TheMainResult::VF} is the conclusion alluded to in Section \ref{Section::Intro::ResultsInTheSmoothSetting}.
\end{rmk}


\section{Classical Function Spaces, Revisited}\label{Section::FuncRevisited}

In this section, we prove the basic results we require about Zygmund-H\"older spaces.  In particular, we prove the results from Section \ref{Section::Func}.  We begin by discussing the main
result we need to help understand 
the various objects under consideration on $\ZygSymb^{\alpha+1}$-manifolds.

\begin{lemma}\label{Lemma::FuncRevis::PushForwardFuncSpaces}
    Fix $\alpha>0$,
    let $U,V\subseteq \R^n$ be open sets, and let 
    $\Phi:U\xrightarrow{\sim} V$ be a $\ZygSymb^{\alpha+1}_{\loc}$ diffeomorphism.
    \begin{enumerate}[parsep=-0.3ex,label=(\roman*)]
        \item\label{Item::FuncRevis::PullBackFunc} For $\beta\in (-\alpha,\alpha+1]$ and
        $f\in \ZygSymb^{\beta}_{\loc}(V)$, 
        $f\circ \Phi$
        is defined as a distribution and we have $f\circ \Phi\in \ZygSymb^{\beta}_{\loc}(U)$.
        
        \item\label{Item::FuncRevis::PullBackVect}  For $\beta\in (-\alpha,\alpha]$ and $X\in \ZygSymb^{\beta}_{\loc}(V; TV)$, $\Phi^{*} X$
        is defined and 
        $\Phi^{*}X\in \ZygSymb^{\beta}_{\loc}(U; TU)$.
        
        \item\label{Item::FuncRevis::PullBackForm} For $\beta\in (-\alpha,\alpha]$ and
        $\omega\in \ZygSymb^{\beta}_{\loc}\mleft(V; \mywedge^{k} T^{*}V\mright)$, then $\Phi^{*}\omega$ is defined
        and $\Phi^{*}\omega\in \ZygSymb^{\beta}_{\loc}\mleft(U; \mywedge^{k} T^{*}U\mright)$
    \end{enumerate}
\end{lemma}

\begin{proof}
    We give $U$ coordinates $x^1,\ldots, x^n$ and $V$ coordinates
    $y^1,\ldots, y^n$.
    We begin with \ref{Item::FuncRevis::PullBackFunc}.  For $\beta>0$,
    $f$ is a continuous function and the regularity
    of $f\circ \Phi$ is classical.  See \cite[Lemma 5.15]{StovallStreetII}
    for a discussion of this classical case.
    For nonpositive $\beta$, we proceed by induction.
    We prove the result for $\beta\in (-l+1,-l]\cap (-\alpha,\alpha+1]$,
    for $l= -1,0,1,2,\ldots$.  The base case, $l=-1$, follows
    from the above discussion for $\beta>0$.  For $l\in \N$,
    we assume the result for $l-1$ and prove it for $l$.
    
    Fix a point $y_0\in V$, and let $B_{y_0}\subseteq V$ be an open ball
    centered at $y_0$ with $\overline{B_{y_0}}\subseteq V$.
    By Remark \ref{Rmk::FuncSpace::CharofZyg} \ref{Item::FuncSpace::CharofZyg::<0}, 
     we may write
    $f=g_0+\sum_{j=1}^n \partial_{y^j} g_j$, where
    $g_0,\ldots, g_n\in \ZygSymb^{\beta+1}(B_{y_0})$.
    Letting $\Psi=(\psi_1,\ldots \psi_n):=\Phi^{-1}$, we have
    \begin{equation*}
        (\partial_{y^j} g_j)\circ \Phi = \sum_{k=1}^n \mleft( \partial_{x^j} (g_j\circ \Phi)\mright) \mleft(\mleft(\partial_{y^j} \psi_k\mright)\circ \Phi\mright).
    \end{equation*}
    By the already proved case $\beta=\alpha$, we have
    $\mleft(\partial_{y^k} \psi_k\mright)\circ \Phi\in \ZygSymb^{\alpha}_{\loc}(\Psi(B_{y_0}))$ and by the inductive
    hypothesis $\partial_{x^k} (g_j\circ \Phi)\in \ZygSymb^{\beta}_{\loc}( \Psi(B_{y_0}))$.
    Also, by the inductive hypothesis $g_0\circ \Phi \in \ZygSymb^{\beta+1}_{\loc}(\Psi(B_{y_0}))\subseteq \ZygSymb^{\beta}_{\loc}(\Psi(B_{y_0}))$.  Using Lemma \ref{Lemma::FuncSpace::Product}, we conclude
    $f\circ \Phi= g_0\circ\Phi + \sum (\partial_{y^j} g_j) \circ \Phi\in \ZygSymb^{\beta}_{\loc}(\Psi(B_{y_0}))$.  It is easy to see
    that the distribution obtained in this way does not depend on
    the choice of $g_0,\ldots, g_n$ with $f=g_0+\sum_{j=1}^n \partial_{y^j} g_j$.
    Since $y_0\in V$ was arbitrary, this completes the proof of \ref{Item::FuncRevis::PullBackFunc}.
    
    {For \ref{Item::FuncRevis::PullBackVect}, write
    $X=\sum_{j=1}^n a_j \partial_{y^j}$, where $a_j\in \ZygSymb^{\beta}_{\loc}(V)$.  Then, if $\Psi=(\psi_1,\ldots, \psi_n)=\Phi^{-1}$,
    $$\Phi^{*} X = \sum_{k=1}^n \sum_{j=1}^n \mleft( a_j \circ \Phi\mright)\mleft(( \partial_{y^j} \psi_k)\circ \Phi\mright) \partial_{x^k}.$$
    By  \ref{Item::FuncRevis::PullBackFunc}, 
    $a_j\circ \Phi\in \ZygSymb^{\beta}_{\loc}(U)$,
    and $(\partial_{x^j}\phi_k)\circ \Phi\in \ZygSymb^{\alpha}_{\loc}(U)$.  Since $\beta\in  (-\alpha,\alpha]$, by hypothesis,
    Lemma \ref{Lemma::FuncSpace::Product} implies
    $\Phi^{*}X\in \ZygSymb^{\beta}_{\loc}(U; TU)$.}
    
    The proof of \ref{Item::FuncRevis::PullBackForm} is very
    similar to the proof of \ref{Item::FuncRevis::PullBackVect},
    and follows easily by combining \ref{Item::FuncRevis::PullBackFunc}
    with Lemma \ref{Lemma::FuncSpace::Product}.  We leave the details to the reader.
\end{proof}

Lemma \ref{Lemma::FuncRevis::PushForwardFuncSpaces} 
establishes Remark \ref{Rmk::FuncRn::DefinedOnMfld}:
On a $\ZygSymb^{\alpha+1}$ manifold $\Manifold$, it makes sense to talk about functions
in $\ZygSymb^s_{\loc}(\Manifold)$ for $s\in (-\alpha,\alpha+1]$, vector fields in $\ZygSymb^s_{\loc}(\Manifold; T\Manifold)$ for $s\in (-\alpha,\alpha]$,
and $k$-forms in $\ZygSymb^s_{\loc}\mleft(\Manifold;\mywedge^k T^{*}\Manifold\mright)$ for $s\in (-\alpha,\alpha]$.
This is because these properties are invariant under $\ZygSymb^{\alpha+1}_{\loc}$ diffeomorphisms.  By a similar proof,
one can show that the more general concept of a $\ZygSymb^{\beta}_{\loc}$ tensor
is well-defined for $\beta\in (-\alpha,\alpha]$, though the only
tensors we use in this paper are vector fields and forms.

For completeness we put the definition of functions, vector fields and differential forms on manifold as below,
which is the obvious analog of 
the standard definitions (see, for example, \cite[Definition 6.3.3]{Hormander}):

\begin{defn}\label{Defn::FuncRevis::ObjonManifolds}
Let $\alpha>0$, $\beta\in(-\alpha,\alpha]$ and $\gamma\in(-\alpha,\alpha+1]$. Let  $\Manifold$ be a $n$-dimensional $\Co^{\alpha+1}$-manifold equipped with
the maximal $\Co^{\alpha+1}$-atlas $\As=\{\phi:U_\phi\subseteq\Manifold\to\R^n\}$. {Namely, each $\phi\in\As$ is a homeomorphism $\phi:U_\phi\xrightarrow{\sim}\phi(U_\phi)\subseteq\R^n$; $\phi\circ\psi^{-1}\in\Co^{\alpha+1}_\loc(\psi(U_\psi\cap U_\phi);\R^n)$ whenever $\phi,\psi\in\As$ satisfy $U_\psi\cap U_\phi\neq \varnothing$; and $\As$ is maximal with these properties.}
\begin{itemize}[parsep=-0.3ex]
    \item A $\Co^\gamma_\loc$-function is a collection $f=\{f_\phi\in\Co^\gamma_\loc(\phi(U_\phi))\}_{\phi\in\As}$, such that 
    \begin{equation*}\label{Eqn::FuncRevis::TransitionFun}
        f_\phi=f_\psi\circ(\psi\circ\phi^{-1}),\quad\text{on }\phi(U_\phi\cap U_\psi)\text{ whenever }U_\phi\cap U_\psi\neq\varnothing.
    \end{equation*}
    \item A $\Co^\beta_\loc$-vector field is a collection $X=\{X_\phi\in\Co^\beta_\loc(\phi(U_\phi);T\R^n)\}_{\phi\in\As}$, such that
    \begin{equation}\label{Eqn::FuncRevis::TransitiofVF}
        X_\phi=(\psi\circ\phi^{-1})^*X_\psi,\quad\text{on }\phi(U_\phi\cap U_\psi)\text{ whenever }U_\phi\cap U_\psi\neq\varnothing.
    \end{equation}
    \item Let $1\le k\le n$. A $\Co^\beta_\loc$ $k$-form is a collection $\omega=\mleft\{\omega_\phi\in\Co^\beta_\loc\mleft(\phi(U_\phi);\mywedge^kT^*\R^n\mright)\mright\}_{\phi\in\As}$, such that
    \begin{equation}\label{Eqn::FuncRevis::TransitionForms}
        \omega_\psi=(\psi\circ\phi^{-1})^*\omega_\phi,\quad\text{on }\phi(U_\phi\cap U_\psi)\text{ whenever }U_\phi\cap U_\psi\neq\varnothing.
    \end{equation}

\end{itemize}
\end{defn}

\begin{rmk}\label{Rmk::FuncRevis::ObjonManifolds}
    By Lemma \ref{Lemma::FuncRevis::PushForwardFuncSpaces} we are able to pullback functions, vector fields, and differential forms using $\Co^{\alpha+1}$-transition maps. Thus, the above objects are well-defined.
    
    
\end{rmk}

To prove Proposition \ref{Prop::FuncRn::LieOnManiofold} we use the following:

\begin{lemma}\label{Lemma::FuncRn::PullbackComm}
Let $\alpha>0$, let $U,V\subseteq\R^n$ be two open sets and let $\Phi:U\xrightarrow{\sim}V$ be a $\Co^{\alpha+1}$-diffeomorphism. 
\begin{enumerate}[parsep=-0.3ex,label=(\roman*)]
    \item\label{Item::FuncRn::IntProdComm} For $\beta\in(-\alpha,\alpha]$, if $Y\in\Co^\alpha_\loc(V;TV)$ and $\omega\in\Co^\beta_\loc\mleft(V;\mywedge^kT^*V\mright)$ then $\Phi^*(\IntProd{Y}\omega)=\IntProd{\Phi^*Y}\Phi^*\omega$, and their common value is in $\Co^\beta_\loc\mleft(U;\mywedge^{k-1}T^*U\mright)$.
    \item\label{Item::FuncRn::WedgeProdComm} For $\beta\in(-\alpha,\alpha]$, if $\sigma\in\Co^\alpha_\loc\mleft(V;\mywedge^lT^*V\mright)$ and $\omega\in\Co^\beta_\loc\mleft(V;\mywedge^kT^*V\mright)$ then $\Phi^*(\sigma\wedge\omega)=\Phi^*\sigma\wedge\Phi^*\omega$, and their common value is in  $\Co^\beta_\loc\mleft(U;\mywedge^{k+l}T^*U\mright)$.
    \item\label{Item::FuncRn::dComm} For $\alpha>\frac12$ and $\beta\in(1-\alpha,\alpha]$, if $\omega\in\Co^\beta_\loc\mleft(V;\mywedge^kT^*V\mright)$ then $\Phi^*d\omega=d\Phi^*\omega$, and their common value is in  $\Co^{\beta-1}_\loc\mleft(U;\mywedge^{k+1}T^*U\mright)$.
    \item\label{Item::FuncRn::VectActionComm} For $\beta\in(1-\alpha,\alpha+1]$, if $Y\in\Co^\alpha_\loc(V;TV)$ and $f\in\Co^\beta_\loc(V)$ then $\Phi^*(Yf)=(\Phi^*Y)(\Phi^*f)$,  and their common value is in  $\Co^{\beta-1}_\loc(U)$.
    \item If $\alpha>\frac12$ and $\beta\in(1-\alpha,\alpha]$, if $Y\in\Co^\alpha_\loc(V;TV)$ and $Z\in\Co^\beta_\loc(V;TV)$ then $\Phi^*[Y,Z]=[\Phi^*Y,\Phi^*Z]$ , and their common value is in  $\Co^{\beta-1}_\loc(U;TU)$.
    \item If $\alpha>\frac12$ and $\beta\in(1-\alpha,\alpha]$, if $Y\in\Co^\alpha_\loc(V;TV)$ and $\omega\in\Co^\beta_\loc\mleft(V;\mywedge^kT^*V\mright)$ then $\Phi^*\Lie Y\omega=\Lie{\Phi^*Y}\Phi^*\omega$,  and their common value is in $\Co^{\beta-1}_\loc\mleft(U;\mywedge^kT^*U\mright)$.
\end{enumerate}
\end{lemma}
\begin{proof}The formal computations are standard in differential geometry. What we need to be careful is that the products and compositions are defined, due to the low regularity of the objects involved.

We only prove \ref{Item::FuncRn::IntProdComm} and \ref{Item::FuncRn::dComm}, since the arguments for the others are similar.
We endow $V\subseteq\R^n$ with the standard coordinate system $(x^1,\dots,x^n)$, and write $\Phi=:(\phi^1,\dots,\phi^n)$ where $\phi^j\in\Co^{\alpha+1}_\loc(U)$. 

\ref{Item::FuncRn::IntProdComm}: We  write $Y=\sum_{i=1}^na^i\Coorvec{x^i}$ and $\omega_{j_1\dots j_k}:=\omega(\Coorvec{x^{j_1}},\dots,\Coorvec{x^{j_k}})$ for $1\le j_1,\dots,j_k\le n$. By Lemma \ref{Lemma::FuncSpace::Product}, $a^i\omega_{ij_1\dots j_{k-1}}\in\Co^{\beta}_\loc(V)$ and therefore
\begin{equation}\label{Eqn::FuncRevis::IntProdCoordinate}
    \IntProd{Y}\omega=\frac1{(k-1)!}\sum_{i=1}^n\sum_{j_1,\dots,j_{k-1}=1}^na^i\omega_{ij_1\dots j_{k-1}}dx^{j_1}\wedge\dots\wedge dx^{j_{k-1}}\in\Co^{\beta}_\loc\mleft(U;\mywedge^{k-1}T^*U\mright).
\end{equation}

Note that $\Coorvec{\phi^i}=\Phi^*\Coorvec{x^i}$, $i=1,\dots,n$, are $\Co^\alpha$-vector fields on $U$ that satisfy\footnote{We write $\delta_i^j$ for the Kronecker delta functions (see \eqref{Eqn::Key::Krnoecker}).} $d\phi^j(\Coorvec{\phi^i})=\frac{\partial\phi^j}{\partial\phi^i}=\delta_i^j$ for $1\le i,j\le n$. By Lemma \ref{Lemma::FuncRevis::PushForwardFuncSpaces} \ref{Item::FuncRevis::PullBackFunc}, $a^i\circ\Phi,\omega_{ij_1\dots j_{k-1}}\circ\Phi,(a^i\omega_{ij_1\dots j_{k-1}})\circ\Phi\in\Co^{\beta}_\loc(U)$ are all defined. Therefore we have the following, where all the products and compositions are defined.
\begin{align*}
    &\IntProd{\Phi^*Y}(\Phi^*\omega)=\frac1{k!}\sum_{j_0,\dots,j_{k-1}=1}^n\IntProd{\sum_{i=1}^n(a^i\circ\Phi)\Coorvec{\phi^i}}\left((\omega_{j_0\dots j_{k-1}}\circ\Phi)\cdot d\phi^{j_0}\wedge\dots\wedge d\phi^{j_{k-1}}\right)
    \\&=\frac1{k!}\sum_{i=1}^n\sum_{ j_0,j_1,\dots,j_{k-1}=1}^n\sum_{\rho=0}^{k-1}(a^i\circ\Phi)\cdot(\omega_{j_0\dots\widehat{j_\rho}\dots j_k}\circ\Phi)\cdot(-1)^\rho \frac{\partial\phi^{j_\rho}}{\partial\phi^i}d\phi^{j_0}\wedge\dots\wedge d\phi^{j_{\rho-1}}\wedge d\phi^{j_{\rho+1}}\wedge\dots\wedge d\phi^{j_{k-1}}
    \\&=\frac1{(k-1)!}\sum_{i=1}^n\sum_{j_1,\dots,j_{k-1}=1}^n((a^i\omega_{j_1\dots j_{k-1}})\circ\Phi)\cdot d\phi^{j_1}\wedge\dots\wedge d\phi^{j_{k-1}}=\Phi^*(\IntProd{Y}\omega).
\end{align*}
The equality holds in $\Co^{\beta}_\loc\mleft(V;\mywedge^{k-1}T^*V\mright)$, completing the proof.

\medskip
\ref{Item::FuncRn::dComm}: By passing to linear combinations it suffices to consider the form $\omega=fdx^{i_1}\wedge\dots\wedge dx^{i_k}$ where $f\in\Co^\beta_\loc(V)$. 
   
    By Lemma \ref{Lemma::FuncRevis::PushForwardFuncSpaces}, $f\circ\Phi\in\Co^{\beta}_\loc(U)$ and $\frac{\partial f}{\partial x^j}\circ\Phi\in\Co^{\beta-1}_\loc(U)$.
    Since we also have $d\phi^{i_1},\dots, d\phi^{i_k}\in\Co^\alpha_\loc(U)$, 
     Lemma \ref{Lemma::FuncSpace::Product} shows that all below products are defined.  
    We have,
    \begin{equation}\label{Eqn::FuncRevis::PullbackDComm::Tmp1}
        \Phi^*d(fdx^{i_1}\wedge\dots\wedge dx^{i_k}  )=\Phi^*\sum_{j=1}^n\frac{\partial f}{\partial x^j}dx^j\wedge dx^{i_1}\wedge\dots\wedge dx^{i_k}=\sum_{j=1}^n\Big(\frac{\partial f}{\partial x^j}\circ\Phi\Big)d\phi^j\wedge d\phi^{i_1}\wedge\dots\wedge d\phi^{i_k},
    \end{equation}
    \begin{equation}\label{Eqn::FuncRevis::PullbackDComm::Tmp2}
        \begin{split}
            &d\Phi^*(f dx^{i_1}\wedge\dots\wedge dx^{i_k})=d(f\circ\Phi)\wedge d\phi^{i_1}\wedge\dots\wedge d\phi^{i_k}=\sum_{l=1}^n\frac{\partial(f\circ\Phi)}{\partial x^l}\wedge d\phi^{i_1}\wedge\dots\wedge d\phi^{i_k}
        \\&=\sum_{l,j=1}^n\Big(\frac{\partial f}{\partial x^j}\circ\Phi\Big)\frac{\partial\phi^j}{\partial x^l}dx^l\wedge d\phi^{i_1}\wedge\dots\wedge d\phi^{i_k}.
        \end{split}
    \end{equation}
    Since $\sum_{l=1}^n\frac{\partial\phi^j}{\partial x^l}dx^l=d\phi^j$, we have that
    \eqref{Eqn::FuncRevis::PullbackDComm::Tmp1} and
    \eqref{Eqn::FuncRevis::PullbackDComm::Tmp2} are equal,
    completing the proof.
\end{proof}

\begin{proof}[Proof of Proposition \ref{Prop::FuncRn::LieOnManiofold}]
We only prove \ref{Item::FuncRn::IntProdCont}, since the arguments for the other parts are similar.

By Definition \ref{Defn::FuncRevis::ObjonManifolds}, we can write $Y=\{Y_\phi\in\Co^\alpha_\loc(\phi(U_\phi);T\R^n)\}_{\phi\in\As}$ and $\omega=\mleft\{\omega_\phi\in\Co^\beta_\loc\mleft(\phi(U_\phi);\mywedge^kT^*\R^n\mright)\mright\}_{\phi\in\As}$, where $\As$ is the maximal $\Co^{\alpha+1}$-atlas for $\Manifold$. 

For each $\phi\in\As$, by applying Lemma \ref{Lemma::FuncSpace::Product} on the coordinate components of $Y_\phi$ and $\omega_\phi$, we see that the map $(Y_\phi,\omega_\phi)\mapsto \IntProd{Y_\phi}\omega_\phi$ is a continuous map $\Co^\alpha_\loc(\phi(U_\phi);T\R^n)\times \Co^\beta_\loc\mleft(\phi(U_\phi);\mywedge^kT^*\R^n\mright)\to \Co^\beta_\loc\mleft(\phi(U_\phi);\mywedge^{k-1}T^*\R^n\mright)$.

By \eqref{Eqn::FuncRevis::TransitiofVF} and \eqref{Eqn::FuncRevis::TransitionForms} we have $(\psi\circ\phi^{-1})^*Y_\psi=Y_\phi$ and $(\psi\circ\phi^{-1})^*\omega_\psi=\omega_\phi$ on $\phi(U_\phi\cap U_\psi)$.
By Lemma \ref{Lemma::FuncRn::PullbackComm} \ref{Item::FuncRn::IntProdComm} we see that $$(\psi\circ\phi^{-1})^*(\IntProd{Y_\psi}\omega_\psi)=\IntProd{(\psi\circ\phi^{-1})^*Y_\psi}\left((\psi\circ\phi^{-1})^*\omega_\psi\right)=\IntProd{Y_\phi}\omega_\phi,\quad\text{on }\phi(U_\phi\cap U_\psi),\text{ whenever }U_\phi\cap U_\psi\neq\varnothing.$$

Therefore $\{\IntProd{Y_\phi}\omega_\phi\}_{\phi\in\As}$ is a collection of $(k-1)$-forms satisfying \eqref{Eqn::FuncRevis::TransitionForms}, and therefore defines a $(k-1)$-form on $\Manifold$, which is denoted by $\IntProd{Y}\omega$.

Finally, the continuity of $(Y,\omega)\mapsto \IntProd{Y}\omega$ comes from the fact that $(Y_\phi,\omega_\phi)\mapsto \IntProd{Y_\phi}\omega_\phi$ is continuous for each $\phi\in\As$.
\end{proof}





We now turn to the proofs of Proposition \ref{Prop::FuncRn::dOfLowRegularityForm}  and Lemma \ref{Lemma::FuncRn::dwelldefined}.
For these, we require several standard objects and results.

\begin{note}\label{Note::FuncRevis::Codiff}
We use the co-differential, $\codiff=\codiff_{\R^n}$, which is a linear operator
taking $k$ forms to $k-1$ forms, satisfying for $1\leq i_1<i_2<\cdots<i_k\leq n$,
\begin{equation*}
    \codiff \mleft(f dx^{i_1}\wedge \cdots \wedge dx^{i_k}\mright)
    =\sum_{l=1}^k \frac{\partial f}{\partial x^{i_l}} (-1)^{l} dx^{i_1}\wedge \cdots \wedge dx^{i_{l-1}}\wedge dx^{i_{l+1}}\wedge \cdots \wedge dx^{i_k}.
\end{equation*}

\end{note}
In particular on 1-forms, $-\codiff$ is the divergence operator, namely
\begin{equation*}
    \text{For }\theta=\sum_{i=1}^n\theta_idx^i,\quad\codiff\theta=-\sum_{i=1}^n \frac{\partial\theta_i}{\partial x^i}.
\end{equation*}

For any form $\omega$, we have $d(\codiff\omega)+\codiff(d\omega)=\Lap \omega$, where
$\Lap=-\sum_{i=1}^n\partial_{x^i}^2$ is the positive Laplacian  acting on the components of $\omega$; in this
setting $\Lap$ is called the Hodge Laplacian.

We will often convolve functions with $k$-forms.
Formally, if $\omega=\sum_{1\leq i_1<\cdots<i_k\leq n} \omega_{i_1,\ldots, i_k}dx^{i_k}\wedge \cdots \wedge dx^{i_k}$ is a $k$-form, and $\phi$ is a function, we set $\phi*\omega=\sum_{1\leq i_1<\cdots<i_k\leq n} \mleft(\phi*\omega_{i_1,\ldots, i_k}\mright)dx^{i_k}\wedge \cdots \wedge dx^{i_k}$.

We will make use of the classical Newtonian potential.
Let
\begin{equation}\label{Eqn::FuncRevis::GreensFunction}
    \Green(x):=
    \begin{cases}-\frac{|x|}2,&n=1,\\-\frac1{2\pi}\log|x|,&n=2,\\|\mathbb S^{n-1}|^{-1}|x|^{2-n},&n\ge3.\end{cases}
\end{equation}

\begin{lemma}\label{Lemma::FuncRevis::GreensOp}
    Let $\omega=\sum_{1\leq i_1<\cdots<i_k\leq n} \omega_{i_1,\ldots, i_k}dx^{i_k}\wedge \cdots \wedge dx^{i_k}$ be a $k$-form
    where each $\omega_{i_1,\ldots, i_k}$ is a compactly supported distribution on $\R^n$.
    Then, $\sigma:=\Green*\omega$ is a distribution on $\R^n$ satisfying $\Lap \sigma=\omega$.
    Moreover, if for some open set $U\subseteq \R^n$ and $\beta\in \R$ we have
    $\omega\big|_U\in \ZygSymb_{\loc}^{\beta}\mleft(U;\mywedge^k T^{*}U\mright)$, then $\sigma\big|_U \in \ZygSymb_{\loc}^{\beta+2}\mleft(U;\mywedge^k T^*U\mright)$.
\end{lemma}
\begin{proof}
    The convolution $\Green*\omega$ makes sense because $\omega_{i_1,\ldots, i_k}$ is a compactly
    supported distribution and $\Green$ is a distribution.
    Since $\Green$ is well-known to be a fundamental solution for
    the Laplacian $\Lap$, we have $\Lap(\Green\ast \omega) = (\Lap \Green)*\omega = \omega$.
    Since $\Lap \sigma=\omega$, the classical interior regularity for elliptic
    equations shows that if $\omega\big|_U\in \ZygSymb_{\loc}^{\beta}\mleft(U;\mywedge^k T^{*}U\mright)$, then $\sigma\big|_U \in \ZygSymb_{\loc}^{\beta+2}(U;\mywedge^k T^{*} U)$; see, for example, \cite[Proposition 4.1]{TaylorPDE3}.
\end{proof}

\begin{lemma}\label{Lemma::FuncRevis::NewtonianBoundedness}
    Let $0\le k\le n$, $\gamma\in\R$ and let $U\subset\R^n$ be a bounded open set, then there is a $C_{U,\gamma}>0$ such that $\|\Green\ast \omega\|_{\Co^{\gamma+2}(U;\wedge^kT^*U)}\le C_{U,\gamma}\|\omega\|_{\Co^\gamma(U;\wedge^kT^*U)}$ for all compactly supported $k$-forms $\omega\in\Co_c^\gamma\mleft(U;\mywedge^kT^*U\mright)$. 
\end{lemma}
\begin{proof}
    Set $\mathscr X$ to be the completion of $\Co_c^\gamma\mleft(U;\mywedge^kT^*U\mright)$ under the $\Co^\gamma$-norm. Thus, $\mathscr X$ is a closed subspace of $\Co^\gamma\mleft(\R^n;\mywedge^kT^*\R^n\mright)$ and $\|\omega\|_{\mathscr X}=\|\omega\|_{\Co^\gamma}$ for all $\omega\in\mathscr X$.
    
    When $\omega\in\mathscr X$, we have that $\supp\omega\subseteq\overline{U}$ so $\omega\in\Co^\gamma_c\mleft(\R^n;\mywedge^kT^*\R^n\mright)$. By Lemma \ref{Lemma::FuncRevis::GreensOp},  $\Green\ast \omega\in\Co^{\gamma+2}_\loc\mleft(\R^n;\mywedge^kT^*\R^n\mright)$ is well-defined. By restricting it to $U$ we get $(\Green\ast\omega)\big|_U\in\Co^{\gamma+2}\mleft(U;\mywedge^kT^*U\mright)$.
    
    By the closed graph theorem we have $\|\Green\ast\omega\|_{\Co^{\gamma+2}(U;\wedge^kT^*U)}\le C\|\omega\|_{\mathscr X}=C\|\omega\|_{\Co^\gamma(U; \wedge^k T^{*}U)}$ for some $C$ that does not depend on $\omega$.
    
    Therefore, for the same constant $C$ we have $\|\Green\ast \omega\|_{\Co^{\gamma+2}(U;\wedge^kT^*U)}\le C\|\omega\|_{\Co^\gamma(U;\wedge^kT^*U)}$ for all $\omega\in\Co_c^\gamma\mleft(U;\mywedge^kT^*U\mright)$.
\end{proof}

\begin{lemma}\label{Lemma::FuncRevis::DecomposeForms}
    Let $0\le k\le n$, $\gamma\in\R$, $\beta>\gamma-1$ and let $U'\Subset U\subseteq\R^n$ be two open sets.\footnote{Here, and in the rest of the paper, $A\Subset B$ denotes that $A$ is a relatively compact subset of $B$.}
    Suppose $\theta\in\Co^\gamma_\loc\mleft(U,\mywedge^kT^*U\mright)$ satisfies $d\theta\in \Co^{\beta-1}_\loc\mleft(U,\mywedge^{k+1}T^*U\mright)$.
    Then, there exist $\rho\in \ZygSymb^{\beta}_{\loc}\mleft(U'; \mywedge^{k}T^{*}U\mright)$ and $\xi\in \ZygSymb^{\gamma+1}_{\loc}\mleft(U'; \mywedge^{k-1}T^{*}U'\mright)$ such that 
\begin{equation*}
    \theta\big|_{U'}=\rho+d\xi.
\end{equation*}
\end{lemma}
Note that the case $\beta\le\gamma-1$ holds automatically if we pick $\rho:=\theta$ and $\xi:=0$.
\begin{proof}
Let $\chi\in C^\infty_c(U)$ satisfy $\chi\equiv 1$ on a neighborhood
over $\overline{U'}$.
Define 
\begin{equation*}
    \widehat\rho:=\Green*\codiff d (\chi\theta), \quad \widehat\xi:=\Green*\codiff(\chi\theta), \qquad \rho:=\widehat\rho\big|_{U'},\quad\xi:=\widehat\xi\big|_{U'}.
\end{equation*}
Since $\chi \theta\in \ZygSymb^{\gamma}$, Lemma \ref{Lemma::FuncRevis::GreensOp}
shows $\widehat\rho\in \ZygSymb^{\gamma}_{\loc}\mleft(\R^n; \mywedge^{k}T^{*}\R^n\mright)$ and $\widehat\xi\in \ZygSymb^{\gamma+1}_{\loc}\mleft(\R^n; \mywedge^{k-1}T^{*}\R^n\mright)$. Thus, $\xi\in\ZygSymb^{\gamma+1}_\loc\mleft(U'; \mywedge^{k-1}T^{*}U'\mright)$.
Also, $\Lap \widehat\rho\big|_{U'} = \codiff d(\chi\theta)\big|_{U'} = \codiff d\theta\big|_{U'}\in \ZygSymb_{\loc}^{\beta-2}$, by hypothesis.
Thus, by the interior regularity of elliptic PDEs (see  \cite[Proposition 4.1]{TaylorPDE3}), we have $ \rho=\widehat{\rho}\big|_{U'}\in \ZygSymb^{\beta}_{\loc}$. We also have,
\begin{equation*}
    \theta\big|_{U'} = \Lap (\Green*(\chi\theta))\big|_{U'}
    =(\codiff d + d\codiff) (\Green*(\chi\theta))\big|_{U'}
    = \Green*\codiff d(\chi\theta)\big|_{U'} + d \Green*\codiff(\chi\theta)\big|_{U'}
    =\widehat\rho\big|_{U'}+ d\widehat\xi\big|_{U'}=\rho+d\xi,
\end{equation*}
as desired.
\end{proof}

Finally, we require paraproduct decompositions.
Let $\psi_0\in \SchwartzSymb(\R^n)$ be a Schwartz function whose
Fourier transform, $\psih_0(\xi)=\int \psi_0(x) e^{2\pi i x\cdot \xi}\: dx$,
satisfies $\supp\psih_0\subseteq \{ \xi : |\xi|<8/3\}$ and
$\psih_0(\xi)\equiv 1$ for $|\xi|\leq 3/2$.  Set
\begin{equation*}
    \psi_j(x):=
    \begin{cases}
        2^{nj}\psi_0(2^j x) & j>0,\\
        0 & j\leq -1.
    \end{cases}
\end{equation*}
Associated to $\psi_0$, we define two bilinear operators each taking a
$k$ form $\sigma$ and an $l$-form $\omega$ 
and outputting a $(k+l)$-form,
\begin{equation*}
    \FrP(\sigma,\omega):=\sum_{j=0}^\infty \mleft((\psi_j-\psi_{j-1})*\sigma\mright)\wedge (\psi_{j-2}*\omega),
    \quad\FrR(\sigma,\omega):=\sum_{|j-k|\leq 1}\mleft( (\psi_j-\psi_{j-1})*\sigma \mright)\wedge \mleft((\psi_k-\psi_{k-1})*\omega\mright).
\end{equation*}

\begin{lemma}\label{Lemma::FuncRevis::ParaProdResults}
    We have the following properties of $\FrP$ and $\FrR$.  Fix $k,l\in \{0,\ldots n\}$.
    \begin{enumerate}[parsep=-0.3ex,label=(\roman*)]
        \item\label{Item::FuncRevis::FrPCont1} For $\alpha\in \R$, $\FrP$ defines a continuous
        bilinear map $\FrP:\ZygSymb^{\alpha}\mleft(\R^n; \mywedge^k T^{*}\R^n\mright)\times L^\infty\mleft(\R^n; \mywedge^l T^{*}\R^n\mright)\rightarrow \ZygSymb^{\alpha}\mleft(\R^n; \mywedge^{k+l} T^{*}\R^n\mright)$.
        
        \item\label{Item::FuncRevis::FrPCont2} For $\alpha\in \R$ and $\beta<0$, $\FrP$ defines a continuous
        bilinear map $\FrP:\ZygSymb^{\alpha}\mleft(\R^n; \mywedge^k T^{*}\R^n\mright)\times \ZygSymb^{\beta}\mleft(\R^n; \mywedge^l T^{*}\R^n\mright)\rightarrow \ZygSymb^{\alpha+\beta}\mleft(\R^n; \mywedge^{k+l} T^{*}\R^n\mright)$.
        
        \item\label{Item::FuncRevis::FrRCont} For $\alpha,\beta\in \R$ with $\alpha+\beta>0$, 
        $\FrR$ defines a continuous bilinear map
        $\FrR:\ZygSymb^{\alpha}\mleft(\R^n; \mywedge^k T^{*}\R^n\mright)\times \ZygSymb^{\beta}\mleft(\R^n; \mywedge^l T^{*}\R^n\mright)\rightarrow \ZygSymb^{\alpha+\beta}\mleft(\R^n; \mywedge^{k+l} T^{*}\R^n\mright)$.
        
        \item\label{Item::FuncRevis::ParaProdDecomp} $\sigma \wedge \omega = \FrP(\sigma, \omega) +(-1)^{kl}\FrP(\omega,\sigma)+\FrR(\omega,\sigma)$ holds
        for $\sigma\in \ZygSymb^{\alpha}\mleft(\R^n; \mywedge^k T^{*}\R^n\mright)$
        and $\omega\in \ZygSymb^{\beta}\mleft(\R^n; \mywedge^{l} T^{*}\R^n\mright)$, where $\alpha+\beta>0$.
        
        \item\label{Item::FuncRevis::ProductRule} $\FrP$ satisfies $d\FrP(\sigma, \omega) =\FrP(d\sigma,\omega)+ (-1)^k \FrP(\sigma, d\omega)$,  for $k$-forms $\sigma$ and $l$-forms $\omega$.
    \end{enumerate}
\end{lemma}
\begin{proof}
    For $0$ forms, \ref{Item::FuncRevis::FrPCont1} and \ref{Item::FuncRevis::FrPCont2} can be found in \cite[Theorem 2.82]{BahouriCheminDanchin}
    and \ref{Item::FuncRevis::FrRCont} can be found in
      \cite[Theorem 2.85]{BahouriCheminDanchin}. 
      By passing to their coordinate components we obtain the results for arbitrary forms.
      \ref{Item::FuncRevis::ParaProdDecomp} follows easily from the fact that
      $\sum_{j=0}^{N} (\psi_j-\psi_{j-1}) *\omega=\psi_{N}*\omega \xrightarrow{N\rightarrow \infty} \omega$. \ref{Item::FuncRevis::ProductRule} follows directly from the definitions. 
\end{proof}

\begin{proof}[Proof of Proposition \ref{Prop::FuncRn::dOfLowRegularityForm}]
Let $\theta\in \ZygSymb^\gamma_{\loc}\mleft(U; \mywedge^k T^{*}U\mright)$
be such that $d\theta\in \ZygSymb^{\beta-1}_{\loc}\mleft(U; \mywedge^k T^{*}U\mright)$.  We will show, for any point $p\in U$, there is a neighborhood 
$V'\subseteq V$ of $F(p)$ and $\tau\in \ZygSymb^{\beta}_{\loc}\mleft(V'; \mywedge^k T^{*} V'\mright)$, with $d(F_{*}\theta)\big|_{V'}=d\tau$.  This will
prove \ref{Item::FuncRn::dOfLowRegularityForm::dthetasmooth}$\Rightarrow$\ref{Item::FuncRn::dOfLowRegularityForm::dFthetasmooth} and the existence of $\tau$ as claimed in the proposition.  The reverse implication follows by reversing the roles
of $F_{*}\theta$ and $\theta$.

Let $U'\Subset U$ be an open neighborhood of $p$. 
By Lemma \ref{Lemma::FuncRevis::DecomposeForms}, there are $\rho\in\Co^\beta_\loc\mleft(U',\mywedge^kT^*U'\mright)$ and $\xi\in\Co^{\gamma+1}_\loc\mleft(U',\mywedge^{k-1}T^*U'\mright)$ such that $\theta\big|_{U'}=\rho+d\xi$. By Lemma \ref{Lemma::FuncRn::PullbackComm} \ref{Item::FuncRn::dComm}, $dF_*d\xi=d^2F_*\xi=0$, so
\begin{equation}\label{Eqn::FuncRevis::DecomposeFstarTheta}
    dF_{*} \theta\big|_{F(U')}= d F_{*} \rho\big|_{F(U')}+ dF_{*}d\xi\big|_{F(U')} =d F_{*} \rho\big|_{F(U')}.
\end{equation}

Since $ \rho\in \ZygSymb^{\beta}_{\loc}\mleft(U',\mywedge^kT^*U'\mright)$ and $F$ is a $\ZygSymb^{\alpha+1}_{\loc}$ diffeomorphism, Lemma \ref{Lemma::FuncRevis::PushForwardFuncSpaces} shows $F_{*} \rho\in \ZygSymb^{\min\{\alpha, \beta\}}_{\loc}\mleft(F(U'),\mywedge^k T^*\R^n\mright)$.  If $\beta\leq \alpha$,
then we have $dF_{*} \rho\in \ZygSymb^{ \beta-1}_{\loc}$,
completing the proof with $V':=F(U')$ and $\tau:=F_{*}\rho$.

However, if $\beta>\alpha$, this does not imply the desired result.
To show $d F_{*} \rho\in \ZygSymb^{\beta-1}_{\loc}$ near $F(p)$, we construct a new $k$-form $\tau\in\Co^\beta$ such that $dF_*\rho=d\tau$ near $F(p)$. The construction requires paraproducts.


Let $U''\Subset U'$ be a smaller open neighborhood of $p$. 
We claim that there exist $\rhot_{i_1\cdots i_k}\in \ZygSymb^{\beta}(\R^n)$
and $\tmu^{i_1\cdots i_k}\in \ZygSymb^{\alpha+1}\mleft(\R^n; \mywedge^{k-1} T^{*}\R^n\mright)$, with compact supports and such that
\begin{equation}\label{Eqn::FuncRevis::BetterDReg::1}
    F_{*}\rho\big|_{F(U'')} =\sum_{1\leq i_1<\cdots<i_k\leq n} \rhot_{i_1\ldots i_k} d \tmu^{i_1\ldots i_k}\big|_{F(U'')}.
\end{equation}
Write $\rho=\sum_{1\leq i_1<\cdots<i_k\leq n} \rho_{i_1\ldots i_k} dx^{i_1}\wedge \cdots \wedge dx^{i_k}$ and denote $\Phi:=F^{-1}:V\xrightarrow{\sim}U$. Take $\chi_1\in C_c^\infty(F(U'))$  such that $\chi_1\big|_{F(U'')}\equiv1$.
Define
$(\phit^{1},\ldots, \phit^{n}):=\chi_1\Phi$, so that each $\phit^j$
is compactly supported and $\phit^j\in \ZygSymb^{\alpha+1}(\R^n)$.
Let $\rhot_{i_1\ldots i_k}:=\chi_1(\rho_{i_1\ldots_k}\circ \Phi)\in \ZygSymb^{\beta}(\R^n)$.
We have $\rho_{i_1\dots i_k}\circ\Phi\big|_{F(U'')}=\rhot_{i_1\dots i_k}\big|_{F(U'')}$ and $\phi^j\big|_{F(U'')}=\phit^j\big|_{F(U'')}$, so
\begin{equation}\label{Eqn::FuncRevis::BetterDReg::2}
    F_{*}\rho\big|_{F(U'')} = \sum_{1\leq i_1<\cdots <i_k\leq n} \rhot_{i_1\ldots i_k} d\phit^{i_1}\wedge \cdots \wedge d\phit^{i_k}\big|_{F(U'')}.
\end{equation}
Using Lemma \ref{Lemma::FuncRevis::GreensOp}, we set $\mu^{i_1\cdots i_k}:=\Green*\codiff (d\phit^{i_1}\wedge \cdots \wedge d\phit^{i_k})\in \ZygSymb_{\loc}^{\alpha+1}\mleft(\R^n; \mywedge^{k-1} T^{*}\R^n\mright)$.  Since $d\phit^{i_1}\wedge \cdots \wedge d\phit^{i_k}$ is closed, we have
\begin{equation}\label{Eqn::FuncRevis::BetterDReg::3}
\begin{split}
    &d\mu^{i_1\cdots i_k} = d\codiff \Green*(d\phit^{i_1}\wedge \cdots \wedge d\phit^{i_k}) + \codiff \Green * d(d\phit^{i_1}\wedge \cdots \wedge d\phit^{i_k})
    \\&=\Lap \Green *(d\phit^{i_1}\wedge \cdots \wedge d\phit^{i_k})
    =d\phit^{i_1}\wedge \cdots \wedge d\phit^{i_k}.
\end{split}
\end{equation}
Setting $\tmu^{i_1\cdots i_k}:= \chi_1\mu^{i_1\cdots i_k}\in \ZygSymb^{\alpha+1}\mleft(\R^n; \mywedge^{k-1} T^{*}\R^n\mright)$, 
\eqref{Eqn::FuncRevis::BetterDReg::1} follows by combining
\eqref{Eqn::FuncRevis::BetterDReg::2} and \eqref{Eqn::FuncRevis::BetterDReg::3}.

Define
\begin{equation}\label{Eqn::FuncRevis::BetterDReg::4}
    \tau:=\sum_{1\leq i_1<\cdots<i_k\leq n} \FrP( \rhot_{i_1\ldots i_k}, d\tmu^{i_1\ldots i_k}) + (-1)^k \FrP(\tmu^{i_1\ldots, i_k}, d\rhot_{i_1\ldots i_k}) +\FrR(\rhot_{i_1\ldots i_k}, d\tmu^{i_1\ldots i_k}).
\end{equation}
We will show $\tau\in \ZygSymb^{\beta}\mleft(\R^n;\mywedge^kT^*\R^n\mright)$
and $d\tau\big|_{F(U'')} = d F_{*} \rho\big|_{F(U'')}$; this will complete
the proof with $V':=F(U'')$ and  we have used \eqref{Eqn::FuncRevis::DecomposeFstarTheta}.

We turn to showing $d\tau\big|_{F(U'')} = d F_{*} \rho\big|_{F(U'')}$.
Since $\tmu^{i_1\ldots i_k}$ is a $(k-1)$-form, 
Lemma \ref{Lemma::FuncRevis::ParaProdResults} \ref{Item::FuncRevis::ProductRule}
shows $d\FrP(\tmu^{i_1\ldots i_k}, \rhot_{i_1\ldots i_k}) = \FrP( d\tmu^{i_1\ldots i_k}, \rhot_{i_1\ldots i_k}) + (-1)^{k-1} \FrP(\tmu^{i_1\ldots i_k}, d\rhot_{i_1\ldots i_k})$.  Applying differential to both sides of this equation 
and using $d^2=0$, we obtain
\begin{equation}\label{Eqn::FuncRevis::dImprov::MovedToOtherSide}
    d\FrP(d\tmu^{i_1\ldots i_k}, \rhot_{i_1\ldots i_k}) = (-1)^k d\FrP(\tmu^{i_1\ldots i_k}, d\rhot_{i_1\ldots i_k}).
\end{equation}
Using Lemma \ref{Lemma::FuncRevis::ParaProdResults} \ref{Item::FuncRevis::ParaProdDecomp} in the case $l=0$, \eqref{Eqn::FuncRevis::dImprov::MovedToOtherSide}, and \eqref{Eqn::FuncRevis::BetterDReg::1}, we have
\begin{equation*}
    \begin{split}
        &d\tau\big|_{F(U'')}
        = \sum_{1\leq i_1<\cdots<i_k\leq n} d\FrP( \rhot_{i_1\ldots i_k}, d\tmu^{i_1\ldots i_k}) + (-1)^k d\FrP(\tmu^{i_1\ldots i_k}, d\rhot_{i_1\ldots i_k}) + d\FrR(\rhot_{i_1\ldots i_k}, d\tmu^{i_1\ldots i_k})\big|_{F(U'')}
        \\&=\sum_{1\leq i_1<\cdots<i_k\leq n} d\FrP( \rhot_{i_1\ldots i_k}, d\tmu^{i_1\ldots i_k}) + d\FrP(d\tmu^{i_1\ldots i_k}, \rhot_{i_1\ldots i_k}) + d\FrR(\rhot_{i_1\ldots i_k}, d\tmu^{i_1\ldots i_k})\big|_{F(U'')}
        \\&= \sum_{1\leq i_1<\cdots<i_k\leq n} d \mleft(\rhot_{i_1\ldots i_k} d\tmu^{i_1\ldots i_k}\mright)\big|_{F(U'')}
        \\&=d F_{*}\rho\big|_{F(U'')},
    \end{split}
\end{equation*}
as desired.

Finally, we show $\tau\in \ZygSymb^{\beta}\mleft(\R^n;\mywedge^kT^*\R^n\mright)$, which will complete the proof.
Using that $\rhot_{i_1\ldots i_k}\in \ZygSymb^{\beta}(\R^n;T^*\R^n)$ and $d\tmu^{i_1\ldots i_k}\in \ZygSymb^{\alpha}\mleft(\R^n;\mywedge^kT^*\R^n\mright)\subset L^\infty$, 
Lemma \ref{Lemma::FuncRevis::ParaProdResults} \ref{Item::FuncRevis::FrPCont1}
shows $\FrP(\rhot_{i_1\ldots i_k}, d\tmu^{i_1\ldots i_k})\in \ZygSymb^{\beta}\mleft(\R^n;\mywedge^kT^*\R^n\mright)$
and Lemma \ref{Lemma::FuncRevis::ParaProdResults} \ref{Item::FuncRevis::FrRCont}
shows $\FrR(\rhot_{i_1\ldots i_k}, d\tmu^{i_1\ldots i_k})\in \ZygSymb^{\alpha+\beta}\mleft(\R^n;\mywedge^kT^*\R^n\mright)\subsetneq \ZygSymb^{\beta}$.

Thus, the proof will be complete once we show $\FrP(\tmu^{i_1\ldots i_k}, d\rhot_{i_1\ldots i_k})\in \ZygSymb^\beta\mleft(\R^n;\mywedge^kT^*\R^n\mright)$.
If $\beta>1$, then $d\rhot_{i_1\ldots i_k}\in L^\infty$.
Using that $\tmu^{i_1\ldots i_k}\in \ZygSymb^{\alpha+1}$,
Lemma \ref{Lemma::FuncRevis::ParaProdResults} \ref{Item::FuncRevis::FrPCont1} then
implies $\FrP(\tmu^{i_1\ldots i_k}, d\rhot_{i_1\ldots i_k})\in \ZygSymb^{\alpha+1}\subseteq \ZygSymb^{\beta}$.
If $\beta\leq 1$, then $d\rhot_{i_1\ldots i_k}\in \ZygSymb^{\beta}\subsetneq \ZygSymb^{\beta-1-\alpha}$.  Since $\beta-1-\alpha<0$
and $\tmu^{i_1\ldots i_k}\in \ZygSymb^{\alpha+1}$,
Lemma \ref{Lemma::FuncRevis::ParaProdResults} \ref{Item::FuncRevis::FrPCont2}
shows $\FrP(\tmu^{i_1\ldots i_k}, d\rhot_{i_1\ldots i_k}) \in \ZygSymb^{\alpha+1+\beta-1-\alpha}= \ZygSymb^{\beta}$.
This completes the proof.
\end{proof}


\begin{proof}[Proof of Lemma \ref{Lemma::FuncRn::dwelldefined}]
Let $\As=\{\phi:U_\phi\subseteq\Manifold\to\R^n\}$ be the $\Co^{\alpha+1}$-atlas of $\Manifold$. By Definition \ref{Defn::FuncRevis::ObjonManifolds}, $\theta\in\Co^\gamma_\loc\mleft(\Manifold;\mywedge^kT^*\Manifold\mright)$ is the collection $\Big\{\theta_\phi\in\Co^\gamma_\loc\mleft(U_\phi;\mywedge^kT^*U_\phi\mright)\Big\}_{\phi\in\As}$  which satisfies $(\phi\circ\psi)^*\theta_\psi\big|_{\phi(U_\phi\cap U_\psi)}=\theta_\phi\big|_{\phi(U_\phi\cap U_\psi)}$ whenever $U_\phi\cap U_\psi\neq\varnothing$.

We claim that $\big\{d\theta_\phi\big\}_{\phi\in\As}$ defines a $\Co^{\beta-1}_\loc$ $(k+1)$-form on $\Manifold$ (see Definition \ref{Defn::FuncRevis::ObjonManifolds}). Namely, we claim
\begin{equation}\label{Eqn::FuncRevis::ToShowdWellDefined}
(\phi\circ\psi^{-1})^*d\theta_\psi\big|_{\phi(U_\phi\cap U_\psi)}=d\theta_\phi\big|_{\phi(U_\phi\cap U_\psi)},\quad\text{ whenever }\phi,\psi\in\As\text{ satisfy }U_\phi\cap U_\psi\neq\varnothing,
\end{equation}
and their common value on $\phi(U_\phi\cap U_\psi)$ is $\Co^{\beta-1}_\loc$.

Indeed, once \eqref{Eqn::FuncRevis::ToShowdWellDefined} is shown, then $\tau=\{d\theta_\phi\}_{\phi\in\As}$ is the  desired $(k+1)$-form.  To see that this $\tau$ is closed in the sense
of the statement of the lemma, note that if $F\in\As$ is a $\Co^{\alpha+1}$-coordinate chart on $\Manifold$, then $F_*\tau=d\theta_F$ and therefore
$dF_{*}\tau = d^2 \theta_F=0$.

First, we claim that $d\theta_\phi$ is $\Co^{\beta}_{\loc}$ for every $\phi\in \As$.
The assumption that  $d\theta$ has regularity $\Co^{\beta-1}_\loc(\Manifold)$
(see Definition \ref{Defn::FuncRn::dRegularityLow}) says that we can find a covering of coordinate charts $\{\phi_j:U_{\phi_j}\subseteq\Manifold\to\R^n\}_{j\in I}\subseteq\As$ (that is $\bigcup_jU_{\phi_j}=\Manifold$) such that $d(\phi_j)_*\theta\in \Co^{\beta-1}_\loc\mleft(U_{\phi_j};\mywedge^{k+1}T^*(U_{\phi_j})\mright)$, i.e. $d\theta_{\phi_j}\in \Co^{\beta-1}_\loc\mleft(U_{\phi_j};\mywedge^{k+1}T^*(U_{\phi_j})\mright)$ for each $j\in I$.

Let $\psi\in\As$. For each $p\in U_\psi\subseteq\Manifold$, we can find a $j_0\in I$ such that $p\in U_{\phi_{j_0}}$. By Proposition \ref{Prop::FuncRn::dOfLowRegularityForm} we see that $d(\psi\circ \phi_{j_0})^*\theta_{\phi_{j_0}}\big|_{\psi(U_\psi\cap U_{\phi_{j_0}})}\in \Co^{\beta-1}_\loc\mleft(\psi(U_\psi\cap U_{\phi_{j_0}});\mywedge^{k+1}T^*\R^n\mright)$. By \eqref{Eqn::FuncRevis::TransitionForms}, $d\theta_\psi\big|_{\psi(U_\psi\cap U_{\phi_{j_0}})}=d(\psi\circ \phi_{j_0})^*\theta_{\phi_{j_0}}\big|_{\psi(U_\psi\cap U_{\phi_{j_0}})}$ so $d\theta_\psi $ is $\Co^{\beta-1}_\loc$ near $p\in U_\psi$. Since $p$ is arbitrary, we know $d\theta_{\psi}\in \Co^{\beta-1}_\loc\mleft(\psi(U_\psi);\mywedge^{k+1}T^*\R^n\mright)$. Therefore, $\{d\theta_\phi\}$ is a collection of $\Co^{\beta-1}_\loc$-forms.

We turn to proving \eqref{Eqn::FuncRevis::ToShowdWellDefined}.
Let $p\in U_\phi$, 
and let $V\Subset U_\phi$ be a neighborhood of $p$.
We will show:
\begin{equation}\label{Eqn::FuncRevis::ToShowChangeCoords}
    (\phi\circ\psi^{-1})^*d\theta_\psi\big|_{\phi(V\cap U_\psi)}=d\theta_\phi\big|_{\phi(V\cap U_\psi)}\text{ whenever }\psi\in\As \text{ satisfies }V\cap U_\psi\neq\varnothing.
\end{equation}


Since $\phi(V)\Subset \phi(U_\phi)$, by Lemma \ref{Lemma::FuncRevis::DecomposeForms} there exists a $\rho\in\Co^\beta_\loc\mleft(\phi(V);\mywedge^kT^*\R^n\mright)$ and $\xi\in\Co^{\gamma+1}_\loc\mleft(\phi(V);\mywedge^{k-1}T^*\R^n\mright)$ such that $\theta_\phi\big|_{\phi(V)}=\rho+d\xi$, and
therefore
$d\theta_\phi\big|_{\phi(V)}=d\rho$.

By Lemma \ref{Lemma::FuncRn::PullbackComm} \ref{Item::FuncRn::dComm}, since $\beta>1-\alpha$ and $\gamma+1>1-\alpha$,  $$d(\psi\circ\phi^{-1})^*\rho=(\psi\circ\phi^{-1})^*d\rho,\quad d(\psi\circ\phi^{-1})^*\xi=(\psi\circ\phi^{-1})^*d\xi,\qquad\text{on }\psi(V\cap U_\psi). $$

Therefore on $\psi(V\cap U_\psi)$,
$$(\psi\circ\phi^{-1})^*d\theta_\phi=(\psi\circ\phi^{-1})^*d\rho=d(\psi\circ\phi^{-1})^*\rho=d(\psi\circ\phi^{-1})^*\theta_\phi-d(\psi\circ\phi^{-1})^*d\xi=d\theta_\psi-d^2(\psi\circ\phi^{-1})^*\xi=d\theta_\psi.$$
This proves \eqref{Eqn::FuncRevis::ToShowChangeCoords}.
Since $p\in U_\phi$ was arbitrary,
\eqref{Eqn::FuncRevis::ToShowdWellDefined} follows, completing the proof.
\end{proof}

In the proof of Theorem 3.1, we need a  version of Proposition \ref{Prop::FuncRn::dOfLowRegularityForm} on 1-forms where we keep track of various estimates. 
We are concerned with the case when
$F$ is a $\Co^{\alpha+1}$-diffeomorphism on $\Ball^n$ and 
is close to the identity map, and our 1-form $\theta$ defined on $\Ball^n$ is such that $\|\theta\|_{\Co^\alpha}+\|d\theta\|_{\Co^{\beta-1}}$ is small.

\begin{prop}\label{Prop::FuncRevis::QuantdOfLowRegularityForm}
    Let $\alpha>0$ and $\beta\in[\alpha,\alpha+1]$ be two real numbers, then there is a constant $C_0=C_0(n,\alpha,\beta)>1$ satisfying the following:
    
    Suppose $R\in\Co^{\alpha+1}(\Ball^n;\R^n)$ satisfies $R\big|_{\partial\Ball^n}=0$ and $\|R\|_{\Co^{\alpha+1}}\le C_0^{-1}$, then the map $F:=\id+R:\Ball^n\to\R^n$ is a $\Co^{\alpha+1}$-diffeomorphism of $\Ball^n$. Moreover,
    \begin{enumerate}[parsep=-0.3ex,label=(\roman*)]
        \item\label{Item::FuncRevis::QuantdOfLowRegularityForm::1} Let $\Phi=(\phi^1,\dots,\phi^n):\Ball^n\xrightarrow{\sim}\Ball^n$ be the inverse map of $F$.  Then, \begin{equation}\label{Eqn::FuncRevis::QuantdOfLowRegularityForm::1}
            \|g\circ\Phi\|_{\Co^\beta(\Ball^n)}\le C_0\|g\|_{\Co^\beta(\Ball^n)},\quad\forall g\in\Co^\beta(\Ball^n).
        \end{equation}
        In particular $\|\Phi\|_{\Co^{\alpha+1}(\Ball^n;\R^n)}\le C_0$.
        \item\label{Item::FuncRevis::QuantdOfLowRegularityForm::2}  $\|\nabla\Phi-I_n\|_{\Co^{\alpha}(\Ball^n;\Mbb^{n\times n})}\le C_0\|R\|_{\Co^{\alpha+1}(\Ball^n;\R^n)}$.
        \item\label{Item::FuncRevis::QuantdOfLowRegularityForm::3} If $\theta\in\Co^\alpha_c\mleft(\Ball^n;T^*\Ball^n\mright)$ satisfies $\supp\theta\subsetneq\frac12\Ball^n$ and $d\theta\in\Co^{\beta-1}\mleft(\Ball^n;\mywedge^2T^*\Ball^n\mright)$, then $\supp F_*\theta\subsetneq\frac34\Ball^n$ and 
        \begin{equation}\label{Eqn::FuncRevis::QuantdOfLowRegularityForm::2}
            \|d(F_*\theta)\|_{\Co^{\beta-1}\mleft(\Ball^n;\wedge^2T^*\Ball^n\mright)}\le C_0\|d\theta\|_{\Co^{\beta-1}\mleft(\Ball^n;\wedge^2T^*\Ball^n\mright)}.
        \end{equation}
    \end{enumerate}
\end{prop}

In the proof of Proposition \ref{Prop::FuncRevis::QuantdOfLowRegularityForm} we need to follow convention for matrix-valued functions.

\begin{conv}\label{Conv::FuncRevis::MatrixNorms}
For matrix-valued map $A=(a_i^j):\Ball^n\to\Mbb^{n\times n}$, we use the matrix norm
\begin{equation}\label{Eqn::FuncRevis::MatrixNorms}
    |A(x)|_{\Mbb^{n\times n}}:=\sup\limits_{v\in\R^n\backslash\{0\}}\frac{|A(x)\cdot v|}{|v|},\quad\|A\|_{C^0(\Ball^n;\Mbb^{n\times n})}=\sup\limits_{x\in\Ball^n}|A(x)|_{\Mbb^{n\times n}}.
\end{equation}

For Zygmund-H\"older norms of $A$ we  use the component-wise norm; namely, $\|A\|_{\Co^\alpha(\Ball^n;\Mbb^{n\times n})}:=\|(a_1^1,\dots,a_n^n)\|_{\Co^\alpha(\Ball^n;\R^{n^2})}$.
\end{conv}
\begin{rmk}
Let $\alpha>0$, and let $U\subseteq\R^n$ be an open set. It follows
from Lemma \ref{Lemma::FuncSpace::Product} that there is a $\tilde C_{U,\alpha}>0$, such that  
\begin{equation}\label{Eqn::FuncRevis::ProdConst}
    \|AB\|_{\Co^\alpha(U;\Mbb^{n\times n})}\le \tilde C_{U,\alpha}\|A\|_{\Co^\alpha(U;\Mbb^{n\times n})}\|B\|_{\Co^\alpha(U;\Mbb^{n\times n})},\quad\forall A,B\in\Co^\alpha(U;\Mbb^{n\times n}).
\end{equation}
\end{rmk}

\begin{lemma}\label{Lemma::FuncRevis::InvertingMatrix}
    Let $\alpha>0$, and let $B\subset\R^n$ be an open ball. There is a $\tilde c_{B,\alpha}>0$, such that if $\|A\|_{\Co^\alpha(B;\Mbb^{n\times n})}<\tilde c_{B,\alpha}$, then $I+A(x)$ is an invertible matrix for every $x\in B$. In addition, the map
    \begin{equation*}
        A\mapsto (I+A)^{-1}:\{M\in\Co^\alpha(B;\Mbb^{n\times n}):\|M\|_{\Co^\alpha}<\tilde c_{B,\alpha}\}\to\Co^\alpha(B;\Mbb^{n\times n}),
    \end{equation*}
    is continuous and satisfies $\|(I+A)^{-1}-I\|_{\Co^\alpha(B;\Mbb^{n\times n})}\le 2\|A\|_{\Co^\alpha(B;\Mbb^{n\times n})}$.
\end{lemma}
\begin{proof}
We take $\tilde c_{B,\alpha}=\frac12\min\{\tilde C_{B,\alpha}^{-1},1/2\}<\frac12$ where $\tilde C_{B,\alpha}$ is in \eqref{Eqn::FuncRevis::ProdConst}. 

When $\|A\|_{\Co^\alpha}<\tilde c_{B,\alpha}$,  we have $\|A^k\|_{\Co^\alpha}\le \tilde C_{B,\alpha}\|A\|_{\Co^\alpha}\|A^{k-1}\|_{\Co^\alpha}\le\frac12\|A^{k-1}\|_{\Co^\alpha}$ for all $k\in\Z_+$. Therefore, for such $A$ we have $\|A^k\|_{\Co^\alpha}\le 2^{1-k}\|A\|_{\Co^\alpha}\le2^{-k}$.

We know $(I+A)^{-1}=\sum_{k=0}^\infty(-1)^kA^k$ whenever the right hand side absolutely converges. This  power series has $\Co^\alpha$-norm convergent radius larger than $\tilde c_{B,\alpha}$ and is  continuous in the domain $\{\|A\|_{\Co^\alpha}<\tilde c_{B,\alpha}\}$. 

Finally we have $$\|(I+A)^{-1}-I\|_{\Co^\alpha}=\Big\|\sum_{k=1}^\infty(-1)^k A^k\Big\|_{\Co^\alpha}\le\sum_{k=1}^\infty\|A^k\|_{\Co^\alpha}\le\sum_{k=1}^\infty2^{1-k}\|A\|_{\Co^\alpha}\le2\|A\|_{\Co^\alpha}.$$
\end{proof}

\begin{proof}[Proof of Proposition \ref{Prop::FuncRevis::QuantdOfLowRegularityForm}]
We let $C_0$ be a large constant which may change from line to line.
In particular, we will choose $C_0$ large enough such that $\|R\|_{\Co^{\alpha+1}}\le C_0^{-1}$ implies $\| R\|_{C^0}+\|\nabla R\|_{C^0}\le\frac14$ and $\|\nabla R\|_{\Co^\alpha}\le \tilde c_{\Ball^n,\alpha}$, where $\tilde c_{\Ball^n,\alpha}$ is in Lemma \ref{Lemma::FuncRevis::InvertingMatrix}.

By Lemma \ref{Lemma::FuncRevis::InvertingMatrix}, $\nabla F(x)=I+\nabla R(x)$ is an invertible matrix for every $x\in\Ball^n$, and we have 
\begin{equation}\label{Eqn::FuncRevis::QuantdOfLowRegularityForm::Proof1}
\|(\nabla F)^{-1}-I\|_{\Co^\alpha}=\|(I+\nabla R)^{-1}-I\|_{\Co^\alpha}\le 2\|\nabla R\|_{\Co^\alpha}.
 \end{equation}



Since $\|\nabla R\|_{C^0}\le\frac14$, we have $|R(x_1)-R(x_2)|\le\|\nabla R\|_{C^0}|x_1-x_2|\le\frac14|x_1-x_2|$, which implies 
\begin{equation}\label{Eqn::FuncRevis::QuantdOfLowRegularityForm::Proof1.5}
    |F(x_1)-F(x_2)|\ge|x_1-x_2|-|R(x_1)-R(x_2)|\ge\textstyle\frac34|x_1-x_2|.
\end{equation}

This implies $F$ is injective. By the Inverse Function Theorem, we know $F:\Ball^n\xrightarrow{\sim} F(\Ball^n)$ is a $\Co^{\alpha+1}$-diffeomorphism.

The assumption $R\big|_{\partial\Ball^n}=0$ gives $F(\partial\Ball^n)=\partial\Ball^n$. Since $F(\Ball^n)$ is contractible and $\overline{F(\Ball^n)}\supset\partial\Ball^n$, we get that $F(\Ball^n)=\Ball^n$. We conclude $F$ is a $\Co^{\alpha+1}$-diffeomorphism on $\Ball^n$.

\medskip
\noindent\ref{Item::FuncRevis::QuantdOfLowRegularityForm::1}: 
First, we claim that there is a $C_1(n,\alpha,\beta)>0$, which does not 
depend on $R$, such that whenever $R$ satisfies the assumptions
of the proposition, we have 
\begin{equation}\label{Eqn::FuncRevis::QuantdOfLowRegularityForm::1::Tmp1}
\|\Phi \|_{\ZygSymb^{\alpha+1}(\Ball^n;\R^n)}\leq C_1.
\end{equation}

Since $\Phi$ is the inverse map of $F$, by \cite[Lemma 5.9]{StovallStreetII}, we know $\|\Phi\|_{\Co^{\alpha+1}(\Ball^n;\R^n)}$ only depends on {$n$, $\alpha$,} $\|F\|_{\Co^{\alpha+1}(\Ball^n;\R^n)}$, and $\|(\nabla F)^{-1}\|_{C^0(\Ball^n;\Mbb^{n\times n})}$. 
We will show that $\|F\|_{\Co^{\alpha+1}(\Ball^n;\R^n)}$ and $\|(\nabla F)^{-1}\|_{C^0(\Ball^n;\Mbb^{n\times n})}$ have bounds that do not depend on $R$.

We have $\|F\|_{\Co^{\alpha+1}(\Ball^n;\R^n)}\le\|\id\|_{\Co^{\alpha+1}(\Ball^n)}+\|R\|_{\Co^{\alpha+1}(\Ball^n;\R^n)}\le \|\id\|_{\Co^{\alpha+1}(\Ball^n)}+\frac14$. The right hand side of this inequality does not depend on $R$.

{By \eqref{Eqn::FuncRevis::QuantdOfLowRegularityForm::Proof1.5},} $|F(x_1)-F(x_2)|\ge\frac34|x_1-x_2|$ implies $\sup\limits_{x\in\Ball^n}|(\nabla F(x))^{-1}|\le\frac43$, so $\|(\nabla F)^{-1}\|_{C^0(\Ball^n;\Mbb^{n\times n})}\le\frac43$, which does not depend on $R$ as well.  This establishes \eqref{Eqn::FuncRevis::QuantdOfLowRegularityForm::1::Tmp1}.


By \cite[Lemma 5.8]{StovallStreetII}, we know for every $\tilde C_1>0$ there is a $C_2=C_2(\alpha,\beta,\tilde C_1)>0$ such that $\|g\circ\Phi\|_{\Co^{\beta}(\Ball^n)}\le C_2\|g\|_{\Co^\beta(\Ball^n)}$ holds when $\|\Phi\|_{\Co^{\alpha+1}(\Ball^n;\R^n)}\le \tilde C_1$. By possibly increasing $C_0$ so that $C_0\ge C_2(\alpha,\beta,C_1)$, we obtain $\|g\circ\Phi\|_{\Co^{\beta}(\Ball^n)}\le C_0\|g\|_{\Co^\beta(\Ball^n)}$, which is  \eqref{Eqn::FuncRevis::QuantdOfLowRegularityForm::1}.

\medskip
\noindent\ref{Item::FuncRevis::QuantdOfLowRegularityForm::2}: Note that the identity matrix $I$ can be viewed as a constant function defined on the unit ball. Since $\Phi$ is a $\Co^{\alpha+1}$-diffeomorphism on $\Ball^n$, we get the equality $I=I\circ\Phi$ as a matrix function on $\Ball^n$.

By the chain rule $I=\nabla(F\circ\Phi)=((\nabla F)\circ\Phi)\cdot\nabla\Phi$, so $\nabla\Phi-I=(\nabla F)^{-1}\circ\Phi-I=((\nabla F)^{-1}-I)\circ\Phi$. By \eqref{Eqn::FuncRevis::QuantdOfLowRegularityForm::Proof1}
and \ref{Item::FuncRevis::QuantdOfLowRegularityForm::1}, we have
\begin{equation*}
    \| \nabla \Phi-I \|_{\Co^\alpha(\Ball^n;\Mbb^{n\times n})} = \|((\nabla F)^{-1}-I)\circ \Phi\|_{\Co^\alpha} \leq C_0 \|(\nabla F)^{-1}-I\|_{\Co^\alpha}
    \leq 2C_0 \| \nabla R\|_{\Co^{\alpha}(\Ball^n;\Mbb^{n\times n})},
\end{equation*}
and we obtain \ref{Item::FuncRevis::QuantdOfLowRegularityForm::2} by
replacing $C_0$ with $2C_0$.


\medskip
\noindent\ref{Item::FuncRevis::QuantdOfLowRegularityForm::3}: Let $\theta\in\Co_c^\alpha(\Ball^n;T^*\Ball^n)$ be as in the assumption of \ref{Item::FuncRevis::QuantdOfLowRegularityForm::3}. In particular, $\supp\theta\subsetneq\frac12\Ball^n$.

By the assumption $\|R\|_{C^0}\le\frac14$, we have $F(\frac12\Ball^n)\subseteq\frac12\Ball^n+\frac14\Ball^n=\frac34\Ball^n$, so $\supp F_*\theta=F(\supp\theta)\subseteq\frac34\Ball^n$.

Similar to the proof of Lemma \ref{Lemma::FuncRevis::DecomposeForms}, we define a 1-form $\rho$ and a function $\xi$ by (see \eqref{Eqn::FuncRevis::GreensFunction})
$$\rho=\sum_{i=1}^n\rho_idx^i:=\Green\ast \codiff d\theta,\quad \xi:=\Green\ast \codiff\theta.$$
$\rho$ and $\xi$ are globally defined in $\R^n$ because $\theta$ is compactly supported in $\frac12\Ball^n$.

By Lemma \ref{Lemma::FuncRevis::GreensOp}, we have $\rho\in\Co^\beta_\loc\mleft(\R^n;T^*\R^n\mright)$, $\xi\in\Co^{\alpha+1}_\loc(\R^n)$, and $\rho+d\xi=\Green\ast (\codiff d+d\codiff)\theta=\theta$. 
{Moreover, by Lemma \ref{Lemma::FuncRevis::NewtonianBoundedness} with $\gamma=\beta-2$ and the assumption $\supp\codiff d\theta\subseteq\supp\theta\subseteq\frac34\Ball^n\Subset\Ball^n$, we have 
\begin{equation}\label{Eqn::FuncRevis::QuantdOfLowRegularityForm::RhoisBoundedbyDTheta}
    \|\rho\|_{\Co^\beta(\Ball^n;T^*\Ball^n)}\lesssim_{\beta}\|\codiff d\theta\|_{\Co^{\beta-2}}\lesssim_\beta\|d\theta\|_{\Co^{\beta-1}}.
\end{equation}}

By Lemma \ref{Lemma::FuncRn::PullbackComm} \ref{Item::FuncRn::dComm}, $d(F_*d\xi)=d^2F_*\xi=0$ on $\Ball^n$, so we have that 
\begin{equation}\label{Eqn::FuncRevis::QuantdOfLowRegularityForm::RestrictionOn1Ball}
    d(F_*\theta)=d\mleft(F_*\rho\big|_{\Ball^n}\mright)+d\mleft(F_*d\xi\big|_{\Ball^n}\mright)=d\mleft(F_*\rho\big|_{\Ball^n}\mright)=d\left(\sum_{i=1}^n(\rho_i\circ\Phi)d\phi^i\right),\quad\text{on }\Ball^n.
\end{equation} 
Thus, to prove \eqref{Eqn::FuncRevis::QuantdOfLowRegularityForm::2}, {by \eqref{Eqn::FuncRevis::QuantdOfLowRegularityForm::RhoisBoundedbyDTheta} }it suffices to show 
\begin{equation}\label{Eqn::FuncRevis::QuantdOfLowRegularityForm::Proof2}
    \|d(F_*\theta)\|_{\Co^{\beta-1}(\frac34\Ball^n;\wedge^2T^*\R^n)}\lesssim\|\rho\|_{\Co^\beta\mleft(\Ball^n;T^*\Ball^n\mright)}.
\end{equation}

Fix $\chi\in C_c^\infty(\Ball^n)$ such that $\chi|_{\frac34\Ball^n}\equiv1$. For each $1\le i\le n$, set $$\tilde\rho_i:=\chi(\rho_i\circ\Phi),\quad\tilde\phi^i:=\chi\phi^i.$$
So $\tilde\rho_i\in\Co^\beta_c(\Ball^n)$ and $d\tilde\phi^i\in\Co^\alpha_c\mleft(\Ball^n;T^*\Ball^n\mright)$ are globally defined 1-forms for each $i$, such that 
\begin{equation}\label{Eqn::FuncRevis::QuantdOfLowRegularityForm::RestrictionOn3/4Ball}
    \sum_{i=1}^n(\tilde\rho_id\tilde\phi^i)\big|_{\frac34\Ball^n}=\sum_{i=1}^n\left((\rho_i\circ\Phi)d\phi^i\right)\big|_{\frac34\Ball^n}=(F_*\rho)\big|_{\frac34\Ball^n}.
\end{equation}

By \eqref{Eqn::FuncRevis::QuantdOfLowRegularityForm::RestrictionOn3/4Ball} and \eqref{Eqn::FuncRevis::QuantdOfLowRegularityForm::RestrictionOn1Ball} we have $$\|d(F_*\theta)\|_{\Co^{\beta-1}(\frac34\Ball^n;\wedge^2T^*\R^n)}=\Big\|\sum_{i=1}^nd(\tilde\rho_id\tilde\phi^i)\Big\|_{\Co^{\beta-1}(\frac34\Ball^n;\wedge^2T^*\R^n)}\le\sum_{i=1}^n\|d(\tilde\rho_id\tilde\phi^i)\|_{\Co^{\beta-1}\mleft(\R^n;\wedge^2T^*\R^n\mright)}.$$
By Lemma \ref{Lemma::FuncSpace::Product} we have $\|\tilde\rho_i\|_{\Co^{\beta}(\R^n)}\lesssim_\beta\|\chi\|_{\Co^\beta(\Ball^n)}\|\rho_i\circ\Phi\|_{\Co^\beta(\Ball^n)}$ and $\|\tilde\phi^i\|_{\Co^{\alpha+1}(\R^n)}\lesssim_\alpha\|\chi\|_{\Co^{\alpha+1}(\Ball^n)}\|\Phi\|_{\Co^{\alpha+1}(\Ball^n)}$, by \eqref{Eqn::FuncRevis::QuantdOfLowRegularityForm::1} we have $\|\rho_i\circ\Phi\|_{\Co^\beta(\Ball^n)}\lesssim\|\rho_i\|_{\Co^\beta(\Ball^n)}$ and $\|\Phi\|_{\Co^{\alpha+1}(\Ball^n;\R^n)}\lesssim1$. Combining them we get
\begin{equation}\label{Eqn::FuncRevis::QuantdOfLowRegularityForm::BddTrhoTphi}
    \|\tilde\rho_i\|_{\Co^\beta(\R^n)}\lesssim_{\alpha,\beta}\|\rho_i\|_{\Co^\beta(\Ball^n)},\quad\|\tilde\phi^i\|_{\Co^{\alpha+1}(\R^n)}\lesssim_\alpha1,\quad1\le i\le n.
\end{equation}

Thus, to obtain \eqref{Eqn::FuncRevis::QuantdOfLowRegularityForm::Proof2} and complete the proof, it 
suffices to show
\begin{equation}\label{Eqn::FuncRevis::QuantdOfLowRegularityForm::Proof2::Tmp1}
\|d(\tilde\rho_id\tilde\phi^i)\|_{\Co^{\beta-1}\mleft(\R^n;\wedge^2T^*\R^n\mright)}\lesssim_{\alpha,\beta}\|\tilde\rho_i\|_{\Co^\beta(\Ball^n)},\quad 1\le i\le n. 
\end{equation}


Similar to \eqref{Eqn::FuncRevis::BetterDReg::4}, we define 1-form $\tau_i$ on $\R^n$ by $$\tau_i:=\sum_{i=1}^n\FrP(\tilde\rho_i,d\tilde\phi^i)-\FrP(\tilde\phi^i,d\tilde\rho_i)+\FrR(\tilde\rho_i,d\tilde\phi^i),\quad1\le i\le n.$$

By Lemma \ref{Lemma::FuncRevis::ParaProdResults} \ref{Item::FuncRevis::ParaProdDecomp} and \ref{Item::FuncRevis::ProductRule}, we have $\tilde\rho_id\tilde\phi^i=\FrP(\tilde\rho_i,d\tilde\phi^i)+\FrP(d\tilde\phi^i,\tilde\rho_i)+\FrR(\tilde\rho_i,d\tilde\phi^i)$ and $d\FrP(\tilde\phi^i,\tilde\rho_i)=\FrP(d\tilde\phi^i,\tilde\rho_i)+\FrP(\tilde\phi^i,d\tilde\rho_i)$. Therefore for $1\le i\le n$,
\begin{equation}\label{Eqn::FuncRevis::QuantdOfLowRegularityForm::Proof2::Tmp2}
d\tau_i=d\FrP(\tilde\rho_i,d\tilde\phi^i)-d\FrP(\tilde\phi^i,d\tilde\rho_i)+d\FrR(\tilde\rho_i,d\tilde\phi^i)=d\FrP(\tilde\rho_i,d\tilde\phi^i)+d\FrP(d\tilde\phi^i,\tilde\rho_i)+d\FrR(\tilde\rho_i,d\tilde\phi^i)=d(\tilde\rho_id\tilde\phi^i),\quad\text{on }\R^n.
\end{equation}

We claim
\begin{equation}\label{Eqn::FuncRevis::QuantdOfLowRegularityForm::Proof2::Tmp3}
    \| \tau_i\|_{\ZygSymb^{\beta}(\R^n;T^*\R^n)}\lesssim
    \|\tilde\rho_i\|_{\Co^\beta(\R^n;T^*\R^n)},\quad 1\le i\le n.
\end{equation}
By Lemma \ref{Lemma::FuncRevis::ParaProdResults} \ref{Item::FuncRevis::FrPCont1} and \ref{Item::FuncRevis::FrRCont}, along with the fact that $\|\tilde\phi^i\|_{\Co^{\alpha+1}}\lesssim1$, we get $\|\FrP(\tilde\rho_i,d\tilde\phi^i)\|_{\Co^\beta}\lesssim\|\tilde\rho_i\|_{\Co^\beta}$ and $\|\FrR(\tilde\rho_i,d\tilde\phi^i)\|_{\Co^\beta}\lesssim\|\FrR(\tilde\rho_i,d\tilde\phi^i)\|_{\Co^{\alpha+\beta}}\lesssim\|\tilde\rho_i\|_{\Co^\beta}$. 

\medskip

To complete the proof of \eqref{Eqn::FuncRevis::QuantdOfLowRegularityForm::Proof2::Tmp3}, we
need to show 
\begin{equation}\label{Eqn::FuncRevis::QuantdOfLowRegularityForm::Proof2::Tmp4}
    \|\FrP(\tilde\phi^i,d\tilde\rho_i)\|_{\Co^\beta(\R^n;T^*\R^n)}\lesssim_{\alpha,\beta}\|\tilde\rho_i\|_{\Co^\beta(\R^n;T^*\R^n)},\quad1\le i\le n.
\end{equation}
We separate the proof of \eqref{Eqn::FuncRevis::QuantdOfLowRegularityForm::Proof2::Tmp4} into two cases:  $\beta>1$ and $0<\beta\leq 1$.

For the case $\beta>1$, \eqref{Eqn::FuncRevis::QuantdOfLowRegularityForm::Proof2::Tmp4} follows from Lemma \ref{Lemma::FuncRevis::ParaProdResults} \ref{Item::FuncRevis::FrPCont1} with \eqref{Eqn::FuncRevis::QuantdOfLowRegularityForm::BddTrhoTphi} that $\tilde\phi^i\in\Co^{\alpha+1}\subseteq\Co^\beta$ and $d\tilde\rho_i\in\Co^{\beta-1}\subsetneq L^\infty$. 
For the case $\beta\le1$, our assumption $0<\alpha\le\beta$ implies that $\alpha\in(0,1]$. So $\beta-\alpha-1<\beta-1\leq 0$ and we have $d\tilde\rho_i\in\Co^{\beta-1}\subsetneq \Co^{\beta-\alpha-1}$. By Lemma \ref{Lemma::FuncRevis::ParaProdResults} \ref{Item::FuncRevis::FrPCont2} along with $\tilde\phi^i\in\Co^{\alpha+1}$ from \eqref{Eqn::FuncRevis::QuantdOfLowRegularityForm::BddTrhoTphi}, we get $\FrP(\tilde\phi^i,d\tilde\rho_i)\in\Co^{(\alpha+1)+(\beta-\alpha-1)}=\Co^\beta$, which gives \eqref{Eqn::FuncRevis::QuantdOfLowRegularityForm::Proof2::Tmp4} and establishes \eqref{Eqn::FuncRevis::QuantdOfLowRegularityForm::Proof2::Tmp3}.

Using \eqref{Eqn::FuncRevis::QuantdOfLowRegularityForm::Proof2::Tmp2} and \eqref{Eqn::FuncRevis::QuantdOfLowRegularityForm::Proof2::Tmp3},
we see 
$\|d(\tilde\rho_id\tilde\phi^i)\|_{\Co^{\beta-1}(\R^n;\wedge^2T^*\R^n)}=\|d\tau_i\|_{\Co^{\beta-1}}\lesssim\|\tau_i\|_{\Co^\beta}\lesssim\|\tilde\rho_i\|_{\Co^\beta}$, establishing \eqref{Eqn::FuncRevis::QuantdOfLowRegularityForm::Proof2::Tmp1}
and completing the proof (by possibly increasing $C_0$).
\end{proof}

\section{The Key Estimate}\label{Section::KeyThm}



Let $\alpha>0$ and $\beta\in [\alpha,\alpha+1]$.  Suppose $\lambda^1,\ldots, \lambda^{n}$
are $\ZygSymb^{\alpha}$ $1$-forms on an open set $U\subseteq \R^n$ which span
the cotangent space at every point of $U$.  If we know that $d\lambda^j\in \ZygSymb^{\beta-1}_{\loc}$ for each $j$, it is not necessarily true that $\lambda^j\in \ZygSymb^{\beta}_{\loc}$.  
However, it is a consequence of our main result (Theorem \ref{Thm::TheMainResult}) that near each point, one can always
change coordinates so that the forms are in $\ZygSymb^{\beta}_{\loc}$
(see also Corollary \ref{Cor::Key::BasicCaseforMainThm}
and Remark \ref{Rmk::Key::SpecialCase}).

The next result is a special case of this idea,
where we present an intital 
setting where we may find
a $\ZygSymb^{\alpha+1}$-diffemorphism such that $F_{*}\lambda^j\in \ZygSymb^{\beta}_{\loc}$.

\begin{thm}\label{Thm::Keythm}
Let $\alpha>0$ and $\beta\in [\alpha,\alpha+1]$. Let $x=(x^1,\dots,x^n)$ and $y=(y^1,\dots,y^n)$ be two coordinate systems for $\R^n$.
There exists $c=c(n,\alpha,\beta)>0$ such that the following holds.

Suppose $\lambda^i$, $i=1,\dots,n\in \ZygSymb^{\alpha}(\Ball^n; T^{*}\Ball^n)$ are 1-forms on $\Ball^n$ such that $\supp (\lambda^i-dx^i)\subsetneq\frac{1}{2}\Ball^n$ for each $i$, and
\begin{equation}\label{Eqn::Keythm::Assumption}
    \sum_{i=1}^n(\|\lambda^i-dx^i\|_{\ZygSymb^{\alpha}(\Ball^n;T^{*}\Ball^n)}+\|d\lambda^i\|_{\ZygSymb^{\beta-1}(\Ball^n;\wedge^2 T^{*}\Ball^n)})\le c.
\end{equation}
Then, there exists a $\ZygSymb^{\alpha+1}$-diffeomorphism $F:\Ball^n_x\xrightarrow{\sim}\Ball^n_y$, such that $B^n(F(0),\frac16)\subseteq F(\frac13\Ball^n)\cap\frac34\Ball^n$, $F_*\lambda^i\in\Co^\beta(\Ball^n;T^*\Ball^n)$ for $i=1,\dots,n$, and moreover
\begin{equation}\label{Eqn::Keythm::Conclusion}
\begin{aligned}
    &\|F-\id\|_{\Co^{\alpha+1}(\Ball^n;\R^n)}+\sum_{i=1}^n\|F_*\lambda^i-dy^i\|_{\Co^\beta\mleft(\Ball^n;T^*\Ball^n\mright)}
    \\&\le c^{-1}\sum_{i=1}^n\left(\|\lambda^i-dx^i\|_{\Co^\alpha\mleft(\Ball^n;T^*\Ball^n\mright)}+\|d\lambda^i\|_{\Co^{\beta-1}\mleft(\Ball^n;\wedge^2T^*\Ball^n\mright)}\right).
\end{aligned}
\end{equation}
\end{thm}


\subsection{Outline of the proof: the dual Malgrange method}
The proof of Theorem \ref{Thm::Keythm}  is inspired by Malgrange's proof
of the Newlander-Nirenberg Theorem \cite{MalgrangeSurLIntegrabiliteDesStructures}.

Let $x=(x^1,\dots,x^n)$ and $y=(y^1,\dots,y^n)$ be two coordinate systems on the unit ball $\Ball^n\subset\R^n$. In this section, we write 1-forms $\lambda^1,\dots,\lambda^n$, $\eta^1,\dots,\eta^n$ as
\begin{equation}\label{Eqn::Key::LambdaEta}
    \lambda^i=dx^i+\sum_{j=1}^na_j^idx^j,\quad\eta^i=dy^i+\sum_{j=1}^nb_j^idy^j,\quad i=1,\dots,n,
\end{equation}
and define coefficient matrices
\begin{equation}\label{Eqn::Key::LambdaEtaMatrix}
    A:=(a_j^i)_{n\times n}=\begin{pmatrix}a_1^1&\cdots&a_n^1\\\vdots&\ddots&\vdots\\a_1^n&\cdots&a_n^n\end{pmatrix},\quad B:=(b_j^i)_{n\times n}=\begin{pmatrix}b_1^1&\cdots&b_n^1\\\vdots&\ddots&\vdots\\b_1^n&\cdots&b_n^n\end{pmatrix}.
\end{equation}

In this section, $\lambda^1,\ldots, \lambda^n$ are given $\ZygSymb^{\alpha}$
$1$-forms on $\Ball^n\subset \R^n$ which span the cotangent space at every point.
And $\eta^i:=F_{*}\lambda^i$ are the push-forward $1$-forms by the unknown
$\ZygSymb^{\alpha+1}$-diffeomorphism $F:\Ball^n\xrightarrow{\sim}\Ball^n$,
which we are solving for.  Thus, $\eta^1,\ldots, \eta^n$ are also $\ZygSymb^{\alpha}$ $1$-forms defined on $\Ball^n$ which span the cotangent
space at every point.


As in Malgrange's work \cite{MalgrangeSurLIntegrabiliteDesStructures},
the main idea is to choose $F$ so that the matrix $B$ satisfies a nonlinear
elliptic PDE.  That $\eta=F_{*}\lambda\in \ZygSymb^{\beta}$ will follow
from the classical interior regularity of elliptic PDEs.  We will show
such an $F$ exists by showing that it suffices for $F$ to satisfy
a different elliptic PDE, whose solution is guaranteed by
classical elliptic theory.

    

Given collections $(\lambda^1,\dots,\lambda^n)$ and $(\eta^1,\dots,\eta^n)$ of 1-forms on $\Ball^1$, as above, that both span their co-tangent spaces at every point, we define Riemannian metrics $g$ and $h$ by
\begin{equation}\label{Eqn::Key::Riemannianmetric1}
\begin{split}
    &g=\sum_{i,j=1}^ng_{ij}dx^idx^j:=\sum_{i,j,k=1}^n(\delta_i^k+a_i^k)(\delta_j^k+a_j^k)dx^idx^j,\\ &h=\sum_{i,j=1}^nh_{ij}dy^idy^j:=\sum_{i,j,k=1}^n(\delta_i^k+b_i^k)(\delta_j^k+b_j^k)dy^idy^j.
\end{split}
\end{equation}
Here, $\delta_i^j$, $\delta_{ij}$, $\delta^{ij}$ are the Kronecker delta functions:
\begin{equation}\label{Eqn::Key::Krnoecker}
    \delta_i^j=\delta_{ij}=\delta^{ij}=\begin{cases}1,&i=j,\\0,&i\neq j.\end{cases}
\end{equation}

We use the following notations from classical Riemannian differential geometry:
\begin{equation}\label{Eqn::Key::Riemannianmetric2}
\begin{split}
    &g^{ij}:=g(dx^i,dx^j),\quad \sqrt{\det g}:=\left|\frac{\lambda^1\wedge\dots\wedge\lambda^n}{dx^1\wedge\dots\wedge dx^n}\right|,
    \\&h^{ij}:=h(dy^i,dy^j),\quad \sqrt{\det h}:=\left|\frac{\eta^1\wedge\dots\wedge\eta^n}{dy^1\wedge\dots\wedge dy^n}\right|.
\end{split}
\end{equation}

\begin{rmk}\label{Rmk::Key::RmkRiemMetric}
\begin{enumerate}[label=(\alph*)]
    \item We can write $h=\sum_{i=1}^n\eta^i\cdot\eta^i$. It is non-degenerate since $\eta^1,\dots,\eta^n$ span the cotangent space at every point. Moreover $\eta^1,\dots,\eta^n$ form an orthogonal basis with respect to this metric $h$. Similar remarks hold for $g=\sum_{i=1}^n\lambda^i\cdot\lambda^i$.
    
    \item\label{Item::Key::RmkRiemMetric::RationalFunction} Using matrix notations in \eqref{Eqn::Key::LambdaEta} we have $(h_{ij})_{n\times n}=(I+B)^\top(I+B)$ and so we know $(h^{ij})_{n\times n}=(h_{ij})^{-1}=(I+B)^{-1}\left((I+B)^{-1}\right)^\top$ and $\sqrt{\det h}=\det(I+B)$. Similarly, we  have $(g_{ij})=(I+A)^\top(I+A)=(g^{ij})^{-1}$ and $\sqrt{\det g}=\det(I+A)$. More importantly, we have the following lemma.
\end{enumerate}
\end{rmk}
\begin{lemma}\label{Lemma::Key::TaylorExpansionofRiemMetric}
Let $B$, $h^{ij}$ and $\sqrt{\det h}$ be as above. Then $h^{ij}$ and $\sqrt{\det h}$ are rational functions of the components of $B$. Moreover for every $\gamma>0$, there is a $c_{n,\gamma}>0$ such that
\begin{equation*}
    h^{ij},\sqrt{\det h}:\mleft\{B\in\Co^\gamma(\Ball^n;\Mbb^{n\times n}):\|B\|_{\Co^\gamma}<c_{n,\gamma}\mright\}\to \Co^\gamma(\Ball^n),\quad 1\le i,j\le n,
\end{equation*}
are norm continuous maps, with 
$$\sum_{i,j=1}^n\|h^{ij}-\delta^{ij}\|_{\Co^\gamma(\Ball^n)}\leq c_{n,\gamma}^{-1}\|B\|_{\Co^\gamma(\Ball^n;\Mbb^{n\times n})},\quad\|\sqrt{\det h}-1\|_{\Co^\gamma(\Ball^n)}\leq c_{n,\gamma}^{-1}\|B\|_{\Co^\gamma(\Ball^n;\Mbb^{n\times n})}.$$
\end{lemma}
\begin{rmk}\label{Rmk::Key::TaylorExpansionofRiemMetric}
The same results hold for $g^{ij}$ and $\sqrt{\det g}$. Namely,  $\|g^{ij}-\delta^{ij}\|_{\Co^\alpha}+\|\sqrt{\det g}-1\|_{\Co^\alpha}\leq c_{n,\alpha}^{-1} \|A\|_{\Co^\alpha}$ holds with the same constant $c_{n,\alpha}>0$.
\end{rmk}
\begin{proof}
By Lemma \ref{Lemma::FuncSpace::Product}, the space $\Co^\gamma(\Ball^n;\Mbb^{n\times n})$ is closed under matrix multiplication. By Remark \ref{Rmk::Key::RmkRiemMetric} \ref{Item::Key::RmkRiemMetric::RationalFunction} $\sqrt{\det h}=\det(I+B)$ is a polynomial in the components of $B$, so in particular is a norm continuous function on $\Co^\gamma(\Ball^n;\Mbb^{n\times n})$. Note that $\det(I+0)=1$ so we have 
$\|\sqrt{\det h}-1\|_{\Co^\gamma(\Ball^n)}\lesssim_\gamma\|B\|_{\Co^\gamma(\Ball^n;\Mbb^{n\times n})}$ when the right hand side is small.

By Lemma \ref{Lemma::FuncRevis::InvertingMatrix}, by choosing $c_{n,\gamma}<\tilde c_{\Ball^n,\gamma}$ where $\tilde c_{\Ball^n,\gamma}$ is the constant in Lemma \ref{Lemma::FuncRevis::InvertingMatrix}, we see that the map $B\mapsto(I+B)^{-1}$ is $\Co^\gamma$-norm continuous with $\|(I+B)^{-1}-I\|_{\Co^\gamma(\Ball^n;\Mbb^{n\times n})}\le2\|B\|_{\Co^\gamma(\Ball^n;\Mbb^{n\times n})}$. 

Thus, in the domain $\|B\|_{\Co^\gamma}<c_{n,\gamma}$, the map $B\mapsto(I+B)^{-1}\left((I+B)^{-1}\right)^\top$ is also continuous and satisfies $\|(I+B)^{-1}\left((I+B)^{-1}\right)^\top-I\|_{\Co^\gamma}\lesssim\|(I+B)^\top(I+B)-I\|_{\Co^\gamma}\lesssim\|B\|_{\Co^\gamma}$. It follows with $(h^{ij})_{n\times n}=(h_{ij})^{-1}_{n\times n}=(I+B)^{-1}\left((I+B)^{-1}\right)^\top$, we have that $h^{ij}$ are norm continuous and $\|h^{ij}-\delta^{ij}\|_{\Co^\gamma}\lesssim\|B\|_{\Co^\gamma}$.

By possibly shrinking $c_{n,\gamma}$ we get $\|\sqrt{\det h}-1\|_{\Co^\gamma}\le c_{n,\gamma}^{-1}\|B\|_{\Co^\gamma}$ and $\sum_{i,j=1}^n\|h^{ij}-\delta^{ij}\|_{\Co^\gamma}\le c_{n,\gamma}^{-1}\|B\|_{\Co^\gamma}$.
\end{proof}




\begin{conv}
Given a Riemannian metric $h$, we use the co-differential $\codiff_h$ as the adjoint of differential with respect to $h$. That is, for any $k$-form $\phi$ and any compactly supported $(k-1)$-form $\psi$, 
\begin{equation*}
    (\codiff_h\phi,\psi)_h:=(\phi,d\psi)_h=\int h(\phi,d\psi)\ d\operatorname{vol_h},
\end{equation*}
where $d\operatorname{vol_h}$ is the Riemannian volume density induced by $h$.  In local coordinates, $d\operatorname{vol_h}=\sqrt{\det h}\ d\operatorname{vol_{\R^n}}$, where $d\operatorname{vol_{\R^n}}$ is the usual Lebesgue density on $\R^n$.
We write $\codiff_{\R^n}$ for the usual co-differential with respect to the flat metric $\sum_{i,j=1}^n\delta_{ij}dy^idy^j$ on $\R^n$.
\end{conv}


\begin{lemma}\label{Lemma::Key::TransitionBetweenPDEs}
Let $\alpha>0$, and let $\lambda^1,\dots,\lambda^n$ be $\Co^\alpha$ 1-forms defined on $\Ball^n$ that span the cotangent space at every point, with $(a_i^j)$, $g$, $(g^{ij})$, and $\sqrt{\det g}$ given in \eqref{Eqn::Key::LambdaEtaMatrix}, \eqref{Eqn::Key::Riemannianmetric1}, and \eqref{Eqn::Key::Riemannianmetric2}. 

Suppose $F=\id+R:\Ball^n_x\to\Ball^n_y$ is a $\Co^{\alpha+1}$-diffeomorphism that satisfies
\begin{equation}\label{Eqn::Key::PDEforR}
        \sum_{i,j=1}^n\Coorvec{x^j}\Big(\sqrt{\det g}\cdot g^{ij}\frac{\partial R^k}{\partial x^i}\Big)=\sum_{i,j=1}^n\Coorvec{x^j}\big(\sqrt{\det g}\cdot g^{ij}a_i^k\big),\quad\text{in }\Ball^n_x,\quad k=1,\dots,n.
\end{equation}
Then for the pushforward 1-forms $\eta^k=F_*\lambda^k$, $k=1,\dots,n$, the coefficients $(b_i^j)$ defined in \eqref{Eqn::Key::LambdaEtaMatrix} satisfy
\begin{equation}\label{Eqn::Key::PDEforB}
    \sum_{i,j=1}^n\Coorvec{y^i}\left(\sqrt{\det h}h^{ij}b_j^k\right)=0,\quad\text{in }\Ball^n_y,\quad k=1,\dots,n.
\end{equation}
\end{lemma}
\begin{proof}
    Note that the composition of a $C^1$-function and a $\Co^{\alpha+1}$-diffeomorphism is still $C^1$, and being compactly supported is preserved under homeomorphism, so we have
    \begin{equation}\label{Eqn::Key::TransitionBetweenPDEs::Proof}
        C_c^1(\Ball_x^n)=\{v\circ F:v\in C_c^1(\Ball_y^n)\}.
    \end{equation}
    
    By assumption $R\in\Co^{\alpha+1}(\Ball^n;\R^n)$, and $a_i^k,g^{ij},\sqrt{\det g}\in\Co^\alpha(\Ball^n)$, so \eqref{Eqn::Key::PDEforR} makes sense in $\Co^{\alpha-1}_\loc(\Ball^n_x)\subsetneq C_c^1(\Ball^n_x)'$ and the equality
    can be viewed as elements of the dual of $C_c^1(\Ball^n)$. 
    
    For any $u\in C_c^1(\Ball^n)$, integrating by parts, we obtain
    \begin{align*}
        0&=\left\langle\sum_{i,j=1}^n\Coorvec{x^j}\left(\sqrt{\det g}\cdot g^{ij}\Big(\frac{\partial R^k}{\partial x^i}-a_i^k\Big)\right),u\right\rangle_{\Ball^n_x}=-\int_{\Ball^n}\sum_{i,j=1}^ng^{ij}\left(\frac{\partial R^k}{\partial x^i}-a_i^k\right)\frac{\partial u}{\partial x^j}\sqrt{\det g}dx
        \\
        &=-\langle dR^k-(\lambda^k-dx^k),du\rangle_{\Ball^n_x;g}=-\langle dF^k-\lambda^k,du\rangle_{\Ball_x^n;g}
    \end{align*}
    Here $\langle\cdot,\cdot\rangle_{\Ball_x^n;g}$ are the dual pairs for linear functionals and test functions induced by $g$. Namely, for $u,v\in C^0(\Ball^n)$ and $\phi,\psi\in C^0(\Ball^n;T^*\Ball^n)$, $\langle u,v\rangle_{\Ball^n;g}=\int_{\Ball^n}uv\sqrt{\det g}dx$ and $\langle \phi,\psi\rangle_{\Ball^n;g}=\int_{\Ball^n}g(\phi,\psi)\sqrt{\det g}dx$.

    Using \eqref{Eqn::Key::PDEforR}, we get $\langle dF^k-\lambda^k,du\rangle_{\Ball^n_x;g}=0$ for all $u\in C_c^1(\Ball^n)$. By \eqref{Eqn::Key::TransitionBetweenPDEs::Proof} we have $\langle dF^k-\lambda^k,d(v\circ F)\rangle_{\Ball^n_x;g}=0$ for all $v\in C_c^1(\Ball^n_y)$.
    
    Note that $F_*(g(\phi,\psi))=(F_*g)(F_*\phi,F_*\psi)=h(F_*\phi,F_*\psi)$ for all $\phi,\psi\in C^0_\loc(\Ball^n;T^*\Ball^n)$, and $F_*\sqrt{\det g}=\sqrt{\det F_*g}=\sqrt{\det h}$, so 
    we have, for every $v\in C_c^1(\Ball^n_y)$,
    \begin{equation*}
        0=\langle dF^k-\lambda^k,d(v\circ F)\rangle_{\Ball^n_x;g}=\langle F_*(dF^k-\lambda^k),F_*F^*dv\rangle_{\Ball^n_y;h}=\langle dy^k-\eta^k,dv\rangle_{\Ball^n_y;h}=\int_{\Ball^n}\sum_{i,j=1}^nh^{ij}b_j^k\frac{\partial v}{\partial y^i}\sqrt{\det h}dy.
    \end{equation*}
    Integrating by parts, we obtain \eqref{Eqn::Key::PDEforB}.
\end{proof}

We will choose a coordinate chart $F$ so that \eqref{Eqn::Key::PDEforR}
is satisfied, and therefore \eqref{Eqn::Key::PDEforB}
will be satisfied as well.

To  prove Theorem \ref{Thm::Keythm}  we will
prove the following:
\begin{itemize}
    \item There exists a $R\in\Co^{\alpha+1}(\Ball^n;\R^n)$ satisfying \eqref{Eqn::Key::PDEforR} with boundary condition $R\big|_{\partial\Ball^n}=0$. Moreover,
    we can choose $R$ with $\|R\|_{\Co^{\alpha+1}}\lesssim_{\alpha,\beta}\|A\|_{\Co^\alpha}$.
    Thus, by taking $c>0$ small, we may take $\|R\|_{\Co^{\alpha+1}}$ small.
    
    \item When $\|R\|_{\Co^{\alpha+1}}$ is small, $F=\id+R$ is a $\Co^{\alpha+1}$-diffeomorphism of $\Ball^n$.
    And under the assumption $\supp A\subsetneq\frac12\Ball^n$, we have $\|B\|_{\Co^\beta(\partial\Ball^n)}\lesssim_{\alpha,\beta}\|A\|_{\Co^\alpha}$ and $\|d\eta\|_{\Co^{\beta-1}}\lesssim_{\alpha,\beta}\|d\lambda\|_{\Co^{\beta-1}}$.
    In particular, by taking $c>0$ small, we may take 
    $\|B\|_{\Co^\beta(\partial\Ball^n)}+\|d\eta\|_{\Co^{\beta-1}} $ small.
    
    \item Using that $B\in\Co^\alpha(\partial\Ball^n;\Mbb^{n\times n})$ satisfies \eqref{Eqn::Key::PDEforB}, if $\|B\|_{\Co^\beta(\partial\Ball^n)}+\|d\eta\|_{\Co^{\beta-1}}$ is small, we will show $B\in\Co^\beta(\Ball^n;\Mbb^{n\times n})$ and  $\|B\|_{\Co^\beta}\lesssim_{\alpha,\beta}\|A\|_{\Co^\alpha}$.
\end{itemize}

The last step above  requires the Zygmund-H\"older well-posedness for the Dirichlet problem.

\begin{lemma}[The Dirichlet Problem]\label{Lemma::Key::DiriSol}
    Let $\gamma>0$, and let $U$ be a bounded domain with smooth boundary. Then for $f\in\Co^{\gamma-2}(U)$ and $g\in\Co^\gamma(\partial U)$ there is a unique $u\in\Co^{\gamma+2}(U)$ such that $\Lap u=f$ and $u\big|_{\partial U}=g$.
    
    Moreover, the solution map $(f,g)\mapsto u$ is continuous bilinear map $\Co^{\gamma-2}(U)\times\Co^\gamma(\partial U)\rightarrow \Co^{\gamma}(U)$.
\end{lemma}
See \cite[Theorem 15]{DirichletBoundedness} for a proof of Lemma \ref{Lemma::Key::DiriSol}. 

\begin{defn}[Dirichlet solution on ball]\label{Defn::Key::DiriSol}
Let $\gamma>0$ and $f\in\Co^{\gamma-2}(\Ball^n)$, we write $\Df(f)$ for the unique solution $u\in\Co^{\gamma}(\Ball^n)$ such that $\Lap u=f$ and $u\big|_{\partial\Ball^n}=0$.
For a $\Co^{\gamma-2}$-vector valued function $g=(g_1,\dots,g_m)$ on $\Ball^n$, we let $\Df(g):=(\Df(g_1),\dots,\Df(g_m))$.
\end{defn}

\begin{rmk}\label{Rmk::Key::ZygmundonSphere}
In Lemma \ref{Lemma::Key::DiriSol}, for $\gamma>0$, $\Co^\gamma(\partial\Ball^n)$ is the Zymgund-H\"older space on the sphere $\Sp^{n-1}=\partial\Ball^n$. One can use Definition \ref{Defn::FuncRevis::ObjonManifolds} to define $\Co^\gamma$-functions on it.
Note that the sphere is a compact manifold, and therefore $\Co^\gamma_\loc(\Sp^{n-1})=\Co^\gamma(\Sp^{n-1})$.
The norm can be defined using any finite atlas,
and the equivalence class of the norm does not depend on 
the choice of atlas. Moreover we have
\begin{equation}\label{Eqn::Key::HZNormforSphere}
    \|f\|_{\Co^\gamma(\partial\Ball^n)}\approx_\gamma\inf\{\|\tilde f\|_{\Co^\gamma(\Ball^n)}:\tilde f\big|_{\partial\Ball^n}=f\},
\end{equation}
since the trace operator $(\cdot)|_{\partial\Ball^n}:\Co^\gamma(\Ball^n)\to\Co^\gamma(\partial\Ball^n)$ is continuous and surjective; see \cite[Theorem 2.7.2]{TriebelTheoryOfFunctionSpacesI}. 
In fact given a function $f\in\Co^\gamma(\Sp^{n-1})$, we can take $\tilde f(x)=f(\frac{x}{|x|})\chi(|x|)$ where $\chi\in C_c^\infty(-\frac12,2)$ such that $\chi(1)=1$, then we have $\tilde f\big|_{\Sp^{n-1}}=f$ and $\|\tilde f\|_{\Co^\gamma(\Ball^n)}\approx_\gamma\|f\|_{\Co^\gamma(\Sp^{n-1})}$.
\end{rmk}

\label{Section::KeyThm::Outline}

    \subsection{The existence proposition}
    In this section, we show that there exists a $\Co^{\alpha+1}$-diffeomorphism $F=\id+R$ solving \eqref{Eqn::Key::PDEforR} and which satisfies
good estimates.

\begin{prop}\label{Prop::Key::ExistPDE}
Let $\alpha>0$ and let $\beta\in[\alpha,\alpha+1]$. There is a $c_1=c_1(n,\alpha,\beta)\in(0,1)$ such that, if $A=(a_j^k)_{n\times n}:\Ball^n\to\Mbb^{n\times n}$ satisfies
\begin{itemize}[parsep=-0.3ex]
    \item[-] $A\in\Co^\alpha_c(\frac12\Ball^n;\Mbb^{n\times n})$ and  $\|A\|_{\Co^\alpha}<c_1$,
\end{itemize}
then the matrix $(I+A)(x)$ is invertible for every $x\in\Ball^n$, and there is a  $\Co^{\alpha+1}$-map $ F=\id+R$ on $\Ball^n$ such that
\begin{enumerate}[parsep=-0.3ex,label=(\roman*)]
    \item\label{Item::Key::ExistPDE::1} $R$ solves the equation
    \begin{equation}\tag{\ref{Eqn::Key::PDEforR}}
        \sum_{i,j=1}^n\Coorvec{x^j}\Big(\sqrt{\det g}\cdot g^{ij}\frac{\partial R^k}{\partial x^i}\Big)=\sum_{i,j=1}^n\Coorvec{x^j}\big(\sqrt{\det g}\cdot g^{ij}a_i^k\big),\quad\text{in }\Ball^n_x,\quad k=1,\dots,n.
    \end{equation}
     with boundary condition $R\big|_{\partial\Ball^n}=0$, and we have
    \begin{equation}\label{Eqn::Key::ExistPDEQuantControl1}
        \textstyle\|R\|_{\Co^{\alpha+1}(\Ball^n; \R^n)}+\|\nabla R\|_{\Co^{\beta}(\Ball^n\backslash\frac34\Ball^n;\R^n)}\le c_1^{-1}\|A\|_{\Co^\alpha}.
    \end{equation}
    \item\label{Item::Key::ExistPDE::1.5} $F:\Ball^n_x\to\Ball^n_y$ is a $\Co^{\alpha+1}$-diffeomorphism such that $B^n(F(0),\frac16)\subseteq F(\frac13\Ball^n)\cap\frac34\Ball^n $.  
    \item\label{Item::Key::ExistPDE::2}Let $\Phi=F^{-1}:\Ball^n_y\to\Ball^n_x$ be its
    inverse map,  then 
    \begin{equation}\label{Eqn::Key::ExistPDEQuantControl2}
        \|\nabla \Phi-I\|_{\Co^{\alpha}(\Ball^n;\Mbb^{n\times n})}+\|\nabla \Phi-I\|_{\Co^{\beta}(\partial\Ball^n;\Mbb^{n\times n})}\le c_1^{-1}\|A\|_{\Co^\alpha}.
    \end{equation}
    
    In particular $\|\Phi\|_{\Co^{\alpha+1}(\Ball^n;\R^n)}\le c_1^{-1}$.
\end{enumerate}
\end{prop}

\begin{rmk}
The map $F$ in Proposition \ref{Prop::Key::ExistPDE} is uniquely determined by $A$. This is due to the well-posedness of the Dirichlet problem for the second order elliptic equations, since $R$ satisfies \eqref{Eqn::Key::PDEforR} with $R\big|_{\partial\Ball^n}=0$.
\end{rmk}

\begin{rmk}\label{Rmk::Key::ExistPDE:FisNeeded}
As we will see in the proof of Theorem \ref{Thm::Keythm}, the map $F$
from Proposition \ref{Prop::Key::ExistPDE}
is the map of the same name in Theorem \ref{Thm::Keythm}.
\end{rmk}

\begin{proof}
We let $c_1$ be a small constant which may change from line to line.
     Note that if \eqref{Eqn::Key::ExistPDEQuantControl1} and \eqref{Eqn::Key::ExistPDEQuantControl2} are valid for some $\tilde c_1$, then they are also valid for any $0<c_1\le\tilde c_1$. 
    
    First pick $c_1<\tilde c_{\Ball^n,\alpha}$ where $\tilde c_{\Ball^n,\alpha}$ is the constant in Lemma \ref{Lemma::FuncRevis::InvertingMatrix}. By Lemma \ref{Lemma::FuncRevis::InvertingMatrix}, the assumption $\|A\|_{\Co^\alpha}<c_1(<\tilde c_{\Ball^n,\alpha})$ implies that $I+A$ is invertible at every point and
    $(I+A)^{-1}\in \Co^{\alpha}(\Ball^n; \Mbb^{n\times n})$. Therefore,  $g$ given in \eqref{Eqn::Key::Riemannianmetric1} is indeed a $\Co^\alpha$-Riemannian metric.

    By the assumption $\supp A\subsetneq\frac12\Ball^n$, we have $\sqrt{\det g}\cdot g^{ij}\big|_{\Ball^n\backslash\frac12\Ball^n}=\delta^{ij}$. {Since $(\sqrt{\det g}g^{ij}(x))_{n\times n}$ is  an invertible matrix for $x\in\frac12\Ball^n$}, the second order operator $\sum_{i,j=1}^n\partial_{x^j}(\sqrt{\det g}\cdot g^{ij}\partial_{x^i})$ is uniformly elliptic on $\Ball^n$. 
    Classical existence theorems (for example, \cite[Theorem 8.3]{GilbargTrudinger}) show that for each $k=1,\dots,n$ there exists\footnote{Here $H^1(\Ball^n)$ stands for the classical $L^2$-Sobolev space of order $1$, and $H^{-1}(\Ball^n)=H_0^1(\Ball^n)^*$ is the $L^2$-Sobolev space of order $-1$.} a $R^k\in H^1(\Ball^n)$ that satisfies \eqref{Eqn::Key::PDEforR} with Dirichlet boundary condition $R^k\big|_{\partial\Ball^n}=0$, since $\sum_{i,j=1}^n\Coorvec{x^j}\left(\sqrt{\det g}g^{ij}a_i^k\right)\in\Co^{\alpha-1}(\Ball^n)\subset H^{-1}(\Ball^n)$. 
    By a classical regularity estimate (see \cite[Theorem 8.34]{GilbargTrudinger} or  \cite[Theorem 15]{DirichletBoundedness}), we know $R^k\in\Co^{\alpha+1}(\Ball^n; \R^n)$.
    
    To show $\|R\|_{\Co^{\alpha+1}}\lesssim\|A\|_{\Co^\alpha}$, we  write \eqref{Eqn::Key::PDEforR} as
    \begin{equation}\label{Eqn::Key::ExistPDE::Proof1}
        \Lap R^k=\sum_{i,j=1}^n\Coorvec{x^j}\left(\left(\delta^{ij}-\sqrt{\det g}g^{ij}\right)\frac{\partial R^k}{\partial x^i}\right)+\sum_{i,j=1}^n\Coorvec{x^j}\left(\sqrt{\det g}g^{ij}a_i^k\right),\quad\text{in }\Ball^n_x,\quad k=1,\dots,n.
    \end{equation}
    
    By  Remark \ref{Rmk::Key::TaylorExpansionofRiemMetric} (see also Lemma \ref{Lemma::Key::TaylorExpansionofRiemMetric}), we see that $\|\delta^{ij}-\sqrt{\det g}g^{ij}\|_{\Co^\alpha(\Ball^n)}\lesssim_\alpha\|A\|_{\Co^\alpha}$.
    Therefore we have
    \begin{equation}\label{Eqn::Key::EstimateLapR}
        \|\Lap R\|_{\Co^{\alpha-1}}\lesssim_\alpha\sum_{i,j,k=1}^n\left(\|\delta^{ij}-\sqrt{\det g}g^{ij}\|_{\Co^\alpha}\|\partial_{x^i}R^k\|_{\Co^\alpha}+\|\sqrt{\det g}g^{ij}a_i^k\|_{\Co^\alpha}\right)\lesssim
        \|A\|_{\Co^\alpha}\|R\|_{\Co^{\alpha+1}}+\|A\|_{\Co^\alpha},
    \end{equation}
    where the implicit constants depend only on $n$ and $\alpha$ but not $A$ or $R$.
    
    The assumption $R\big|_{\partial\Ball^n}=0$ implies that $R=\Df(\Lap R)$, where $\Df$ is the zero Dirichlet boundary solution operator given in Definition \ref{Defn::Key::DiriSol}. Since, by Lemma \ref{Lemma::Key::DiriSol}, $\Df:\Co^{\alpha-1}(\Ball^n)\to\Co^{\alpha+1}(\Ball^n)$ is bounded, \eqref{Eqn::Key::EstimateLapR} implies
    \begin{equation}\label{Eqn::Key::ExistPDE::Tmp1}
        \| R\|_{\Co^{\alpha+1}(\Ball^n;\R^n)}\le\tilde C_1\|A\|_{\Co^\alpha(\Ball^n;\Mbb^{n\times n})}\|R\|_{\Co^{\alpha+1}(\Ball^n;\R^n)}+\tilde C_1\|A\|_{\Co^\alpha(\Ball^n;\Mbb^{n\times n})},
    \end{equation}
    where $\tilde C_1=\tilde C_1(n,\alpha)>1$ is a constant  depending only on $n$ and $\alpha$ but not $A$ or $R$.
    
    Choosing $c_1$ small enough so that $c_1\tilde C_1\le\frac13$, then we get $\|R\|_{\Co^{\alpha+1}}\le\frac13\|R\|_{\Co^{\alpha+1}}+\tilde C_1\|A\|_{\Co^\alpha}$ when $A$ satisfies the assumption $\|A\|_{\Co^\alpha}<c_1$. Therefore
    \begin{equation}\label{Eqn::Key::ExistPDE::Tmp2}
        \textstyle\|R\|_{\Co^{\alpha+1}(\Ball^n;\R^n)}\le \frac32\tilde C_1\|A\|_{\Co^\alpha}\le\frac12 c_1^{-1}\|A\|_{\Co^\alpha},\quad\text{when }\|A\|_{\Co^\alpha}<c_1\le(3\tilde C_1)^{-1}.
    \end{equation}
     This is part of the estimate in \eqref{Eqn::Key::ExistPDEQuantControl1}.

    \medskip
    Next we show $\|R\|_{\Co^{\beta+1}(\Ball^n\backslash\frac34\Ball^n)}\lesssim\|A\|_{\Co^\alpha}$. Note that by the support assumption {$\supp A\subsetneq\frac12\Ball^n$} we have $\sqrt{\det g}g^{ij}\big|_{\Ball^n\backslash\frac12\Ball^n}=\delta^{ij}$ and $a_i^k\big|_{\Ball^n\backslash\frac12\Ball^n}=0$, so the right hand side of \eqref{Eqn::Key::ExistPDE::Proof1} is zero in $\Ball^n\backslash\frac12\Ball^n$. Therefore each $R^k$ are harmonic functions in the domain $\Ball^n\backslash\frac12\Ball^n$.
    
    The estimate $\|R\|_{\Co^{\alpha+1}(\Ball^n)}\lesssim_\alpha\|A\|_{\Co^\alpha}$ implies $\|R\|_{\Co^{\alpha+1}(\partial(\frac12\Ball^n))}\lesssim_\alpha\|A\|_{\Co^\alpha}$ since the trace map $\Co^{\alpha+1}(\Ball^n)\to\Co^{\alpha+1}(\partial(\frac12\Ball^n))$ is bounded (see Remark \ref{Rmk::Key::ZygmundonSphere}). By classical interior estimates of harmonic functions (for example, \cite[Theorem 2.10]{GilbargTrudinger}) since $\Lap R\big|_{\Ball^n\backslash\frac12\Ball^n}=0$, we have \begin{equation}\label{Eqn::Key::ExistPDE::TmpBoundEstR}
        \|R\|_{\Co^{\beta+1}(\partial\frac34\Ball^n)}\lesssim\|R\|_{C^{\lceil\beta\rceil+1}(\partial\frac34\Ball^n)}\lesssim\|R\|_{C^0(\Ball^n\backslash\frac23\Ball^n)}\lesssim\|A\|_{\Co^\alpha}.
    \end{equation}
    
    Therefore, along with the fact that $R\big|_{\partial \Ball^n}=0$, 
    $R\big|_{\partial(\Ball^n\backslash\frac34\Ball^n)}=R\big|_{\partial\frac34\Ball^n}\cup R\big|_{\partial\Ball^n}$ has $\Co^{\beta+1}$ norm bounded by a constant times $\|A\|_{\Co^\alpha}$. By classical regularity estimates of harmonic functions (also see Lemma \ref{Lemma::Key::DiriSol}) on $\Ball^n\backslash\frac34\Ball^n$ we know
    $$\|R\|_{\Co^{\beta+1}(\Ball^n\backslash\frac34\Ball^n)}\lesssim_\beta\|R\|_{\Co^{\beta+1}\big(\partial(\Ball^n\backslash\frac34\Ball^n)\big)}\lesssim_{\alpha,\beta}\|A\|_{\Co^\alpha}.$$
    In particular, there is a $\tilde C_2=\tilde C_2(n,\alpha,\beta)>0$ that depends on neither $A$ nor $R$ such that
    \begin{equation}\label{Eqn::Key::ExistPDE::Tmp3}
        \|\nabla R\|_{\Co^{\beta}(\Ball^n\backslash\frac34\Ball^n)}\le\tilde C_2\|A\|_{\Co^\alpha}.
    \end{equation}
    
    Taking $c_1<\frac 12\tilde C_2^{-1}$ we have $\|\nabla R\|_{\Co^{\beta}(\Ball^n\backslash\frac34\Ball^n)}\le \frac12c_1^{-1}\|A\|_{\Co^\alpha}$. Combining this with \eqref{Eqn::Key::ExistPDE::Tmp2}, completes the proof of \eqref{Eqn::Key::ExistPDEQuantControl1}.

    \medskip
     We can take $c_1>0$ possibly smaller so that $c_1<\frac13(C_0\tilde C_1)^{-1}$, where $C_0$ is the constant in Proposition \ref{Prop::FuncRevis::QuantdOfLowRegularityForm} and $\tilde C_1$ is the constant in \eqref{Eqn::Key::ExistPDE::Tmp1}. By \eqref{Eqn::Key::ExistPDE::Tmp2} we know $\|R\|_{\Co^{\alpha+1}}\le C_0^{-1}$. So by Proposition \ref{Prop::FuncRevis::QuantdOfLowRegularityForm}, the map $F=\id+R$ has $\Co^{\alpha+1}$-inverse.  We conclude $\Phi=F^{-1}\in\Co^{\alpha+1}(\Ball^n;\R^n)$.
    
    Since $\|R\|_{C^1(\Ball^n;\Mbb^{n\times n})}\lesssim_\alpha\|R\|_{\Co^{\alpha+1}(\Ball^n;\R^n)}$, by possibly shrinking $c_1$ we can ensure $\|R\|_{C^0}+\|\nabla R\|_{C^0}\le\frac12$. So $F(0)=R(0)\in B^n(0,\frac12)$, which implies $B^n(F(0),\frac16)\subseteq B^n(0,\frac12+\frac16)\subset\frac34\Ball^n$, and $|F(x_1)-F(x_2)|\ge|x_1-x_2|-|R(x_1)-R(x_2)|\ge\frac12|x_1-x_2|$ for $x_1,x_2\in\Ball^n$. 
    Thus, if  $|F(x)-F(0)|<\frac16$ then $|x-0|<\frac13$; i.e., $B^n(F(0),\frac16)\subseteq F(\frac13\Ball^n)$. So $B^n(F(0),\frac16)\subseteq F(\frac13\Ball^n)\cap\frac34\Ball^n$ finishing the proof of \ref{Item::Key::ExistPDE::1.5}.
    
    
    \medskip
    Finally we prove \ref{Item::Key::ExistPDE::2}. Note that by Proposition \ref{Prop::FuncRevis::QuantdOfLowRegularityForm} \ref{Item::FuncRevis::QuantdOfLowRegularityForm::1}, \eqref{Eqn::Key::ExistPDEQuantControl2} gives $\|\Phi\|_{\Co^{\alpha+1}(\Ball^n;\R^n)}\lesssim1$, which is $\|\Phi\|_{\Co^{\alpha+1}}\le c_1^{-1}$ by choosing $c_1$ small.
    
    By Proposition \ref{Prop::FuncRevis::QuantdOfLowRegularityForm} \ref{Item::FuncRevis::QuantdOfLowRegularityForm::2} and using that $c_1<\frac13(C_0\tilde C_1)^{-1}$, we get $\|\nabla\Phi-I\|_{\Co^\alpha}\le C_0\|R\|_{\Co^{\alpha+1}}\le C_0\tilde C_1\|A\|_{\Co^\alpha}\le\frac12c_1^{-1}\|A\|_{\Co^\alpha}$, which proves half of \eqref{Eqn::Key::ExistPDEQuantControl2}.


    To show the second half  of \eqref{Eqn::Key::ExistPDEQuantControl2}, we need to show $\|\nabla\Phi-I\|_{\Co^\beta(\partial\Ball^n)}\lesssim\|A\|_{\Co^\alpha}$. 
    
    The assumption $R\big|_{\partial\Ball^n}=0$ implies $F\big|_{\partial\Ball^n}=\id\big|_{\partial\Ball^n}=\Phi\big|_{\partial\Ball^n}$ and therefore 
    \begin{equation*}
        (\nabla\Phi-I)\big|_{\partial\Ball^n}=((\nabla\Phi)\circ\Phi^{-1})\big|_{\partial\Ball^n}-I=(\nabla F)^{-1}\big|_{\partial\Ball^n}-I.
    \end{equation*}

    Fix $\chi\in C_c^\infty(2\Ball^n\backslash\frac34\Ball^n)$ such that $\chi\equiv1$ in a neighborhood of $\partial\Ball^n$, so $\nabla R(x)=\chi(x)\nabla R(x)$ for $x\in\Ball^n$ near $\partial \Ball^n$.
    
    We shrink $c_1>0$ so that $c_1<\tilde c_{\Ball^n,\beta}\cdot(\tilde C_{\Ball^n,\beta}\cdot\tilde C_2\|\chi\|_{\Co^\beta})^{-1}$, where $\tilde C_{\Ball^n,\beta}$ is in \eqref{Eqn::FuncRevis::ProdConst}, $\tilde  C_2$ is in \eqref{Eqn::Key::ExistPDE::Tmp3}, 
    and $\tilde c_{\Ball^n,\beta}$ is in Lemma \ref{Lemma::FuncRevis::InvertingMatrix}. Then the assumption $\|A\|_{\Co^\alpha}<c_1$ implies
    \begin{equation*}
        \|\chi\nabla R\|_{\Co^\beta(\Ball^n)}\le \tilde C_{\Ball^n,\beta}\|\chi\|_{\Co^\beta_c(\Ball^n\backslash\frac34\Ball^n)}\|\nabla R\|_{\Co^\beta(\Ball^n\backslash\frac34\Ball^n)}\le\tilde C_{\Ball^n,\beta}\cdot\tilde C_2\|A\|_{\Co^\alpha_c(\frac12\Ball^n)}<\tilde c_{\Ball^n,\beta}.
    \end{equation*}
    Therefore we can apply Lemma \ref{Lemma::FuncRevis::InvertingMatrix} to $\chi\nabla R\in\Co^\beta(\Ball^n;\Mbb^{n\times n})$ to obtain $\|(I+\chi\nabla R)^{-1}-I\|_{\Co^\beta}\le 2\|\chi\nabla R\|_{\Co^\beta}$. Hence, along with \eqref{Eqn::Key::HZNormforSphere},
    $$\|(I+\nabla R)^{-1}-I\|_{\Co^\beta(\partial\Ball^n)}\lesssim\|(I+\chi\nabla R)^{-1}-I\|_{\Co^\beta(\Ball^n)}\lesssim\|\chi\nabla R\|_{\Co^\beta(\Ball^n)}\lesssim_\chi\|R\|_{\Co^{\beta+1}(\Ball^n\backslash\frac34\Ball^n)}\lesssim\|A\|_{\Co^\alpha}.$$
    So by possibly shrinking $c_1>0$, we get $\|\nabla \Phi-I\|_{\Co^\beta(\partial\Ball^n)}\le \frac12 c_1^{-1}\|A\|_{\Co^\alpha}$, which completes the second half of \eqref{Eqn::Key::ExistPDEQuantControl2}.
    %
\end{proof}

We now have pushforward 1-forms $\eta^1=F_*\lambda^1,\dots,\eta^n=F_*\lambda^n$. Their norms admit some control, as the next lemma shows. 

\begin{lemma}\label{Lemma::Key::AtoB}
Let $\alpha>0$ and let $\beta\in[\alpha,\alpha+1]$. There is a $c_2=c_2(n,\alpha,\beta)>0$ such that the following holds.
Let  $A\in\Co^\alpha_c(\frac12\Ball^n,\Mbb^{n\times n})$ be the coefficient matrix for $\lambda^1,\dots,\lambda^n$ (see \eqref{Eqn::Key::LambdaEta}) satisfying the assumptions of Proposition \ref{Prop::Key::ExistPDE} and also satisfying
\begin{enumerate}[parsep=-0.3ex,label=(\alph*)]
    \item $\|A\|_{\Co^\alpha(\Ball^n;\Mbb^{n\times n})}<c_2$.
    \item For $k=1,\dots,n$, $d\lambda^k\in\Co^{\beta-1}\mleft(\Ball^n;\mywedge^2 T^*\Ball^n\mright)$  with $\sum_{k=1}^n\|d\lambda^k\|_{\Co^{\beta-1}(\Ball^n;\wedge^2T^*\Ball^n)}<c_2$.
\end{enumerate}

Suppose $\Phi =F^{-1}:\Ball^n_y\to\Ball^n_x$ satisfies the conclusions of Proposition \ref{Prop::Key::ExistPDE}.  Then, for the 1-forms $\eta^k=\Phi^*\lambda^k$ ($k=1,\dots,n$) with coefficient matrix $B=(b^i_j)_{n\times n}$ (see \eqref{Eqn::Key::LambdaEta}), we have:
\begin{enumerate}[(i)]
    \item\label{Item::Key::AtoB::1} $B$ satisfies the PDE system \eqref{Eqn::Key::PDEforB}.
    \item\label{Item::Key::AtoB::2} $\|B\|_{\Co^\alpha(\Ball^n;\Mbb^{n\times n})}+\|B\|_{\Co^\beta(\partial\Ball^n;\Mbb^{n\times n})}<c_2^{-1}\|A\|_{\Co^\alpha}$.
    \item\label{Item::Key::AtoB::3} $d\eta^k\in\Co^{\beta-1}\mleft(\Ball^n;\mywedge^2T^*\Ball^n\mright)$ for $k=1,\dots,n$, with $\|d\eta^k\|_{\Co^{\beta-1}(\Ball^n;\wedge^2T^*\Ball^n)}<c_2^{-1}\|d\lambda^k\|_{\Co^{\beta-1}}$.
\end{enumerate}
\end{lemma}
\begin{proof}Part \ref{Item::Key::AtoB::1} is obtained in Lemma \ref{Lemma::Key::TransitionBetweenPDEs}.

For part \ref{Item::Key::AtoB::2}, write $\Phi=(\phi^1,\dots,\phi^n)$, where $\phi^k\in\Co^{\alpha+1}(\Ball^n)$, $k=1,\dots,n$. Therefore 
\begin{equation}\label{Eqn::Key::AtoB::Proof1}
    \eta^k=\Phi^*\Big(dx^k+\sum_{i=1}^na_i^kdx^i\Big)=d\phi^k+\sum_{i=1}^n(a_i^k\circ\Phi)d\phi^i,\quad b_j^k=\frac{\partial(\phi^k-y^k)}{\partial y^j}+\sum_{i=1}^n(a_i^k\circ\Phi)\frac{\partial\phi^i}{\partial y^j},\quad 1\le j,k\le n.
\end{equation}

From \eqref{Eqn::Key::AtoB::Proof1} we know that $\|B\|_{\Co^\alpha}\lesssim\|\nabla\Phi-I\|_{\Co^\alpha}+\|A\circ\Phi\|_{\Co^\alpha}\|\nabla\Phi\|_{\Co^\alpha}$. By Proposition \ref{Prop::Key::ExistPDE} \ref{Item::Key::ExistPDE::2} we know $\|\Phi\|_{\Co^{\alpha+1}}\lesssim1$ and $\|\nabla\Phi-I\|_{\Co^\alpha}\lesssim\|A\|_{\Co^\alpha}$. By Proposition \ref{Prop::FuncRevis::QuantdOfLowRegularityForm} \ref{Item::FuncRevis::QuantdOfLowRegularityForm::1} we get $\|A\circ\Phi\|_{\Co^\alpha}\lesssim\|A\|_{\Co^\alpha}$. Combining these we get 
\begin{equation}\label{Eqn::Key::AtoB::Tmp1}
    \|B\|_{\Co^\alpha}\lesssim\|\nabla\Phi-I\|_{\Co^\alpha}+\|A\circ\Phi\|_{\Co^\alpha}\|\nabla\Phi\|_{\Co^\alpha}\lesssim\|A\|_{\Co^\alpha}.
\end{equation}

Since we have $A\equiv 0$ outside $\frac{1}{2}\Ball^n$ {in particular $A\big|_{\partial\Ball^n}=0$}, it follows that $\eta^k=d\phi^k$ on $\partial\Ball^n$. 
Therefore $\|B\|_{\Co^\beta(\partial\Ball^n)}=\|\nabla\Phi-I\|_{\Co^\beta(\partial\Ball^n)}$.
So by Proposition \ref{Prop::Key::ExistPDE} \ref{Item::Key::ExistPDE::2}
\begin{equation}\label{Eqn::Key::AtoB::Tmp2}
    \|B\|_{\Co^\beta(\partial\Ball^n)}=\|\nabla\Phi-I\|_{\Co^\beta(\partial\Ball^n)}\lesssim\|A\|_{\Co^\alpha}
\end{equation}

By choosing $c_2>0$ small, \eqref{Eqn::Key::AtoB::Tmp1} and \eqref{Eqn::Key::AtoB::Tmp2} complete the proof of \ref{Item::Key::AtoB::2}.

\medskip
Finally, for \ref{Item::Key::AtoB::3}, we apply 
Proposition \ref{Prop::FuncRevis::QuantdOfLowRegularityForm} \ref{Item::FuncRevis::QuantdOfLowRegularityForm::3}
with $\theta=\lambda^k$,
for each $k=1,\dots,n$. Since $d(F_*\theta)=d(F_*\lambda^k)=d\eta^k$, by \eqref{Eqn::FuncRevis::QuantdOfLowRegularityForm::2} we get $\|d\eta^k\|_{\Co^{\beta-1}(\Ball^n;\wedge^2T^*\Ball^n)}\lesssim\|d\lambda^k\|_{\Co^{\beta-1}}$. Taking $c_2$ smaller, we complete the proof.
\end{proof}

    \subsection{The regularity proposition}
    
In this part, we show that the  1-forms $\eta^1,\dots,\eta^n$ are indeed $\Co^\beta$, by using the interior regularity theory for elliptic PDEs.

\begin{prop}\label{Prop::Key::RegularityPDE}
{Let $\alpha>0$ and $\beta\in[\alpha,\alpha+1]$.
There is a $c_3=c_3(n,\alpha,\beta)>0$}, such that if $\eta^1,\dots,\eta^n\in\Co^\alpha(\Ball^n;T^*\Ball^n)$ with coefficient matrix $B\in\Co^\alpha(\Ball^n;\Mbb^{n\times n})$ (see \eqref{Eqn::Key::LambdaEta}) such that $B$ solves the PDE \eqref{Eqn::Key::PDEforB}, $B\big|_{\partial \Ball^n}\in \Co^{\beta}(\partial\Ball^n;\Mbb^{n\times n})$ with
\begin{equation}\label{Eqn::Key::RegularityPDE::Assumption}
    \|B\|_{\Co^\alpha(\Ball^n;\Mbb^{n\times n})}+\|B\|_{\Co^\beta(\partial\Ball^n;\Mbb^{n\times n})}+\sum_{l=1}^n\|d\eta^l\|_{\Co^{\beta-1}\mleft(\Ball^n;\wedge^2T^*\Ball^n\mright)}<c_3,
\end{equation} then $B\in\Co^\beta(\Ball^n;\Mbb^{n\times n})$. Moreover
\begin{equation}\label{Eqn::Key::RegularityPDE::Conclusion}
    \|B\|_{\Co^\beta(\Ball^n;\Mbb^{n\times n})}\le c_3^{-1}\Big(\|B\|_{\Co^\alpha(\Ball^n;\Mbb^{n\times n})}+\|B\|_{\Co^\beta(\partial\Ball^n;\Mbb^{n\times n})}+\sum_{l=1}^n\|d\eta^l\|_{\Co^{\beta-1}(\Ball^n;\wedge^2T^*\Ball^n)}\Big).
\end{equation}
\end{prop}
\begin{proof}
We can write $B=\Df(\Lap B)+(B-\Df(\Lap B))$, where $\Df$ is defined in Definition \ref{Defn::Key::DiriSol} and is the zero Dirichlet boundary solution operator to the Laplacian equation on the unit ball.

Note that $B-\Df(\Lap B)$ is the harmonic function whose boundary value equals to $B\big|_{\partial\Ball^n}$ (which might not be zero). By Lemma \ref{Lemma::Key::DiriSol} using the assumption $B\big|_{\partial\Ball^n}\in\Co^\beta(\partial\Ball^n;\Mbb^{n\times n})$, we get $B-\Df(\Lap B)\in\Co^\beta(\Ball^n;\Mbb^{n\times n})$ and
\begin{equation}\label{Eqn::Key::RegularityPDE::Proof0}
    \|B-\Df(\Lap B)\|_{\Co^\beta(\Ball^n)}\lesssim\|B\|_{\Co^\beta(\partial\Ball^n)}.
\end{equation}

We can rewrite \eqref{Eqn::Key::PDEforB} as
\begin{equation}\label{Eqn::Key::RegularityPDE::Proof1}
    -\sum_{i=1}^n\Coorvec{y^i}b_i^k=\sum_{i,j=1}^n\Coorvec{y^i}\left(\big(\sqrt{\det h}h^{ij}-\delta^{ij}\big)b_j^k\right),\quad\text{in }\Ball^n_y,\quad k=1,\dots,n.
\end{equation}
The left hand side of \eqref{Eqn::Key::RegularityPDE::Proof1} is $\codiff_{\R^n}\eta^k$. By Lemma \ref{Lemma::Key::TaylorExpansionofRiemMetric}, the right hand side of \eqref{Eqn::Key::RegularityPDE::Proof1} is the derivatives of rational functions of the components of $B$, which vanish to second order at $B=0$. More precisely, using Lemma \ref{Lemma::Key::TaylorExpansionofRiemMetric}, we can rewrite \eqref{Eqn::Key::RegularityPDE::Proof1} as
\begin{equation}\label{Eqn::Key::RegularityPDE::EqnasCalR}
    \codiff_{\R^n}\eta^k=-\sum_{i=1}^n\Coorvec{y^i}b_i^k=\sum_{i=1}^n\Coorvec{y^i}\Rc_i^k(B),\quad\text{in }\Ball^n_y,\quad k=1,\dots,n.
\end{equation}
Here $\Rc_i^k$ are rational functions (see Lemma \ref{Lemma::Key::TaylorExpansionofRiemMetric}) defined in a neighborhood of origin in $\Mbb^{n\times n}$ with $|\Rc_i^k(u)|\lesssim|u|_{\Mbb^{n\times n}}^2$ for suitably small matrices $u\in\Mbb^{n\times n}$, and we have 
\begin{equation}\label{Eqn::Key::RegularityPDE::ProofTmp1}
    |\Rc_i^k(u_1)-\Rc_i^k(u_2)|\lesssim(|u_1|_{\Mbb^{n\times n}}+|u_2|_{\Mbb^{n\times n}})|u_1-u_2|_{\Mbb^{n\times n}},\quad\text{when }u_1,u_2\in\Mbb^{n\times n}\text{ small}.
\end{equation}

We can pass this fact from matrices to matrix-valued functions. Indeed, $\Rc_i^k$ has convergent power expansion in a neighborhood of $0$ as
\begin{equation}\label{Eqn::Key::RegularityPDE::PowerSeriesforR}
    \Rc_i^k(u)=\sum_{r=2}^\infty\sum_{j_1,\dots,j_r,l_1,\dots,l_r=1}^n a_{i,r;l_1\dots l_r}^{k;j_1\dots j_r}u_{j_1}^{l_1}\dots u_{j_r}^{l_r},\quad\text{converging when }|u|_{\Mbb^{n\times n}}\text{ is small.}
\end{equation}
Here $a_{i,r;l_1\dots l_r}^{k;j_1\dots j_r}\in\R$. The power expansion starts at $r=2$ since the zero and the first order terms  all vanish.

By Lemma  \ref{Lemma::FuncSpace::Product}, we can replace $u\in\Mbb^{n\times n}$ in \eqref{Eqn::Key::RegularityPDE::PowerSeriesforR} by $f\in\Co^\gamma(\Ball^n;\Mbb^{n\times n})$,  
and as in \eqref{Eqn::Key::RegularityPDE::ProofTmp1}, for $\gamma>0$ there is a $\tilde C_{\mathcal R,\gamma}>0$, such that when $\|f_1\|_{\Co^\gamma(\Ball^n;\Mbb^{n\times n})}+\|f_2\|_{\Co^\gamma(\Ball^n;\Mbb^{n\times n})}\le\tilde C_{\mathcal R,\gamma}^{-1}$,
\begin{equation}\label{Eqn::Key::RegularityPDE::Proof1.5}
    \|\Rc_i^k(f_1)-\Rc_i^k(f_2)\|_{\Co^\gamma(\Ball^n)}\le \tilde C_{\mathcal R,\gamma}(\|f_1\|_{\Co^\gamma(\Ball^n;\Mbb^{n\times n})}+\|f_2\|_{\Co^\gamma(\Ball^n;\Mbb^{n\times n})})\|f_1-f_2\|_{\Co^\gamma(\Ball^n;\Mbb^{n\times n})},\quad1\le i,k\le n.
\end{equation}

Using the fact that $\Lap\eta^k=d\codiff_{\R^n}\eta^k+\codiff_{\R^n}d\eta^k$, we further have
$$\Lap\eta^k=d\sum_{i=1}^n\Coorvec{y^i}\Rc_i^k(B)+\codiff_{\R^n} d\eta^k=\sum_{i,j=1}^n\frac{\partial^2}{\partial y^j\partial y^i}\Rc_i^k(B)dy^j+\sum_{j=1}^n\mleft\langle \codiff_{\R^n} d\eta^k,\Coorvec{y^j}\mright\rangle dy^j.$$
Here $\langle \cdot, \cdot \rangle$ denotes the pairing between
$1$ forms and vector fields.

On the other hand, $\Lap\eta^k=\sum_{j=1}^n\Lap(\delta_j^k+b_j^k)dy^j=\sum_{j=1}^n\Lap b_j^kdy^j$, therefore
\begin{equation}\label{Eqn::Key::RegularityPDE::Proof2}
    \Lap b_j^k=\mleft\langle\Lap\eta^k,\Coorvec{y^j}\mright\rangle=\sum_{i,j=1}^n\frac{\partial^2}{\partial y^j\partial y^i}\Rc_i^k(B)+\mleft\langle \codiff_{\R^n} d\eta^k,\Coorvec{y^j}\mright\rangle,\quad\text{in }\Ball^n_y,\quad k=1,\dots,n.
\end{equation}


Let $\tilde\xi_0:=\min(\tilde C_{\Rc,\alpha}^{-1},\tilde C_{\Rc,\beta}^{-1})$ where $\tilde C_{\Rc,\gamma}$ is the constant in \eqref{Eqn::Key::RegularityPDE::Proof1.5}. Let $\xi=\xi_B\in(0,\tilde\xi_0]$ to be determined. We define metric spaces $\Xs_{\gamma,\xi}$ and an operator $\Tc_B:\Xs_{\alpha,\tilde\xi_0}\to \Co^\alpha(\Ball^n;\Mbb^{n\times n})$ by
\begin{equation}\label{Eqn::Key::RegularityPDE::DefofContMap}
\begin{gathered}
    \Xs_{\gamma,\xi}:=\{f\in\Co^\gamma(\Ball^n;\Mbb^{n\times n}):\|f\|_{\Co^\gamma}\le \xi\}\subset\Co^\gamma(\Ball^n;\Mbb^{n\times n}),\quad\text{for }\gamma\in\{\alpha,\beta\}\text{ and }\xi\in(0,\tilde\xi_0].
    \\
    \Tc_B[f]^k_j:=b_j^k-\Df(\Lap b_j^k)+\Big\langle\Df(\codiff_{\R^n} d\eta^k),\Coorvec{y^j}\Big\rangle+\sum_{i=1}^n\Df\Big(\frac{\partial^2\Rc_i^k(f)}{\partial y^j\partial y^i}\Big),\quad 1\le j,k\le n.
\end{gathered}
\end{equation}
We endow $\Xs_{\gamma,\xi}$ with the metric induced by the norm $\|\cdot\|_{\Co^\gamma}$, which makes $\Xs_{\gamma,\xi}$ a complete metric space.

Note that from \eqref{Eqn::Key::RegularityPDE::DefofContMap} and \eqref{Eqn::Key::RegularityPDE::Proof2} we have $B=\Tc_B[B]$. Our goal is to show that when $c_3$ and $\xi$ are both suitably small, we have $B\in\Xs_{\alpha,\xi}$ and that  $\Tc_B$ is a contraction mapping on both $\Xs_{\alpha,\xi}$ and $\Xs_{\beta,\xi}$, thus by uniqueness of the fixed point we conclude that $B$ is a $\Co^\beta$-matrix and $\|B\|_{\Co^\beta}<\xi$.
\medskip


{By Lemma \ref{Lemma::Key::DiriSol}, $\Df:\Co^{\gamma-2}(\Ball^n;\Mbb^{n\times n})\to\Co^\gamma(\Ball^n;\Mbb^{n\times n})$ is bounded for $\gamma\in\{\alpha,\beta\}$}. By \eqref{Eqn::Key::RegularityPDE::Proof1.5}, we know for every $f_1,f_2\in\Xs_{\gamma,\xi}$,
\begin{equation}\label{Eqn::KeyRegPDE::Tmp1}
    \|\Tc_B[f_1]-\Tc_B[f_2]\|_{\Co^\gamma}\le \|\Df\|_{\Co^{\gamma-2}\to\Co^\gamma}\|\nabla^2\|_{\Co^\gamma\to\Co^{\gamma-2}}\sum_{k=1}^n\|\Rc_k^l(f_1)-\Rc_k^l(f_2)\|_{\Co^\gamma}\le \xi C'_{\Rc,\gamma}\|f_1-f_2\|_{\Co^\gamma}.
\end{equation}
where $C'_{\Rc,\gamma}>1$ is a constant that only depends on $n,(\Rc_j^k),\gamma$ but not on $B,\xi,f_1,f_2$.

On the other hand $\Tc_B[0]_j^k=b_j^k-\Df(\Lap b_j^k)+\mleft\langle\Df(\codiff_{\R^n} d\eta^k),\Coorvec{y^j}\mright\rangle$. By \eqref{Eqn::Key::RegularityPDE::Proof0}, for $\gamma\in \{\alpha,\beta\}$,
\begin{equation}\label{Eqn::KeyRegPDE::Tmp2}
\|\Tc_B[0]\|_{\Co^\gamma}\le\|\Tc_B[0]\|_{\Co^\beta}\le\|B-\Df(\Lap B)\|_{\Co^\beta(\Ball^n)}+\sum_{k=1}^n\|\Df(\codiff_{\R^n} d\eta^k)\|_{\Co^\beta}\lesssim\|B\|_{\Co^\beta(\partial\Ball^n)}+\sum_{k=1}^n\|d\eta^k\|_{\Co^{\beta-1}}. 
\end{equation}
So by possibly increasing $C'_{\Rc,\gamma}$, we have, for {$f_1\in\Xs_{\gamma,\xi}$},
using \eqref{Eqn::KeyRegPDE::Tmp1} and \eqref{Eqn::KeyRegPDE::Tmp2},
\begin{equation}\label{Eqn::KeyRegPDE::Tmp3}
\|\Tc_B[f_1]\|_{\Co^\gamma}
\leq \|\Tc_B[f_1]-\Tc_B[0]\|_{\Co^\gamma}+\|\Tc_B[0]\|_{\Co^\gamma}
\le C'_{\Rc,\gamma}\Big(\|B\|_{\Co^\beta(\partial\Ball^n)}+\sum_{l=1}^n\|d\eta^l\|_{\Co^{\beta-1}}+\xi\|f_1\|_{\Co^\gamma}\Big).
\end{equation}

Take $c_3>0$ satisfying $c_3<\frac14\max(1,C'_{\Rc,\alpha},C'_{\Rc,\beta})^{-2}$, and take 
\begin{equation}\label{Eqn::Key::RegularityPDE::Proof3}
    \xi=\xi_B:=2\max(C'_{\Rc,\alpha},C'_{\Rc,\beta})\left(\|B\|_{\Co^\alpha(\Ball^n;\Mbb^{n\times n})}+\|B\|_{\Co^\beta(\partial\Ball^n)}+\sum_{k=1}^n\|d\eta^k\|_{\Co^{\beta-1}}\right).
\end{equation}

By the assumption \eqref{Eqn::Key::RegularityPDE::Assumption}, $\xi_B\le\frac12\max(C'_{\Rc,\alpha},C'_{\Rc,\beta})^{-1}<\tilde\xi_0$, so $\Tc_B$ is defined on $\Xs_{\alpha,\tilde\xi_0}$ and by \eqref{Eqn::KeyRegPDE::Tmp3} $\Tc_B$ maps $\Xs_{\gamma,\xi_B}$ into $\Xs_{\gamma,\xi_B}$ for $\gamma\in\{\alpha,\beta\}$. 

Since $\xi_B C'_{\Rc,\gamma}<\frac12$ for $\gamma\in\{\alpha,\beta\}$, using \eqref{Eqn::KeyRegPDE::Tmp1},
$\Tc_B$ is a contraction mapping on the domain $\Xs_{\gamma,\xi_B}$, for $\gamma\in\{\alpha,\beta\}$.

Note that $\xi_B\ge\|B\|_{\Co^\alpha}$, and so $B\in \{f\in\Co^\alpha(\Ball^n;\Mbb^{n\times n}):\|f\|_{\Co^\alpha}\le \xi_B\}=\Xs_{\alpha,\xi_B}$. Therefore, $B$ is a fixed point for $\Tc_B$ in $\Xs_{\alpha,\xi_B}$, which is unique since $\Tc_B$ is a contraction mapping on $\Xs_{\alpha,\xi_B}$.

On the other hand $\Tc_B$ also has a unique fixed point in $\Xs_{\beta,\xi_B}\subsetneq\Xs_{\alpha,\xi_B}$. Therefore, by uniqueness, $B\in\Xs_{\beta,\xi_B}=\{f\in\Co^\beta(\Ball^n;\Mbb^{n\times n}):\|f\|_{\Co^\beta}\le \xi_B\}$.
In particular, $\|B\|_{\Co^\beta}\le\xi_B$. Thus by \eqref{Eqn::Key::RegularityPDE::Proof3}, 
$$\|B\|_{\Co^\beta(\Ball^n)}\le\xi_B\lesssim_{\alpha,\beta} \|B\|_{\Co^\alpha(\Ball^n)}+\|B\|_{\Co^\beta(\partial\Ball^n)}+\sum_{l=1}^n\|d\eta^l\|_{\Co^{\beta-1}}.$$

Thus, we have established \eqref{Eqn::Key::RegularityPDE::Conclusion}
which completes the proof.
\end{proof}

    \subsection{The proof of Theorem \ref{Thm::Keythm} and an improvement}\label{Section::ProofKey}
    Using Propositions \ref{Prop::Key::ExistPDE} and \ref{Prop::Key::RegularityPDE} we can prove Theorem \ref{Thm::Keythm}.

\begin{proof}[Proof of Theorem \ref{Thm::Keythm}]
Let $c_1,c_2,c_3>0$ be the small constants in Proposition \ref{Prop::Key::ExistPDE}, Lemma \ref{Lemma::Key::AtoB}, and Proposition \ref{Prop::Key::RegularityPDE}. We take $c=\frac1{2n^2}\min(c_1,c_2c_3)$ in the assumption of Theorem \ref{Thm::Keythm}.

Let $F$, $R=F-\id$, $\Phi=F^{-1}$, $A$, $B$ and $\eta^i=F_*\lambda^i$ be as in Proposition \ref{Prop::Key::ExistPDE}. {Recall $\eta^i$ and $B=(b_i^j)$ are given in \eqref{Eqn::Key::LambdaEta} and \eqref{Eqn::Key::LambdaEtaMatrix}.} 

When the assumption \eqref{Eqn::Keythm::Assumption} is satisfied, by Proposition \ref{Prop::Key::ExistPDE} \ref{Item::Key::ExistPDE::1.5} we have that $F$ is $\Co^{\alpha+1}$-diffeomorphism and satisfies $B^n(F(0),\frac16)\subseteq F(\frac13\Ball^n)\cap\frac34\Ball^n$. And by \eqref{Eqn::Key::ExistPDEQuantControl1},  we have $\|F-\id\|_{\Co^{\alpha+1}}=\|R\|_{\Co^{\alpha+1}}\le c_1^{-1}\|A\|_{\Co^\alpha}\le \frac1{2n} c^{-1}\sum_{i=1}^n\|\lambda^i-dx^i\|_{\Co^\alpha}$. This implies half of the estimate \eqref{Eqn::Keythm::Conclusion}.

By Lemma \ref{Lemma::Key::AtoB} \ref{Item::Key::AtoB::2} and \ref{Item::Key::AtoB::3}, we have $\|B\|_{\Co^\alpha(\Ball^n)}+\|B\|_{\Co^\beta(\partial\Ball^n)}<c_2^{-1}\|A\|_{\Co^\alpha}$ and $\|d\eta^k\|_{\Co^{\beta-1}}<c_2^{-1}\|d\lambda^k\|_{\Co^{\beta-1}}$, $k=1,\dots,n$. 

Thus, $\|B\|_{\Co^\alpha(\Ball^n)}+\|B\|_{\Co^\beta(\partial\Ball^n)}+\sum_{k=1}^n\|d\eta^k\|_{\Co^{\beta-1}}<c_2^{-1}\|A\|_{\Co^\alpha}+c_2^{-1}\sum_{k=1}^n\|d\lambda^k\|_{\Co^{\beta-1}}<2c_2^{-1}c<c_3$.
By \eqref{Eqn::Key::RegularityPDE::Conclusion} in Proposition \ref{Prop::Key::RegularityPDE}, we get 
\begin{align*}
    &\sum_{k=1}^n\|\eta^k-dy^k\|_{\Co^\beta}\le n\|B\|_{\Co^\beta(\Ball^n)}\le n c_3^{-1}\mleft(\|B\|_{\Co^\alpha}+\|B\|_{\Co^\beta(\partial\Ball^n)}+\sum_{k=1}^n\|d\eta^k\|_{\Co^{\beta-1}}\mright)\\
    &\le n c_3^{-1}\cdot c_2^{-1}\mleft(\|A\|_{\Co^\alpha}+\sum_{k=1}^n\|d\lambda^k\|_{\Co^{\beta-1}}\mright)\le n^2(c_2c_3)^{-1}\sum_{k=1}^n\mleft(\|\lambda^k-dx^k\|_{\Co^\alpha}+\|d\lambda^k\|_{\Co^{\beta-1}}\mright).
\end{align*}
This gives the second half of the estimate \eqref{Eqn::Keythm::Conclusion} since $n^2(c_2c_3)^{-1}\le\frac12 c^{-1}$.
\end{proof}


In Theorem \ref{Thm::Keythm}, we assumed \eqref{Eqn::Keythm::Assumption} which is a smallness assumption.  When \eqref{Eqn::Keythm::Assumption}
is not satisfied, we may use a scaling argument to transfer to a setting where it is satisfied, as the next result shows.

\begin{prop}[The scaling argument]\label{Prop::Key::Scaling}Let $\alpha>0$, $\beta\in[\alpha,\alpha+1]$
and let $\mu_0,\tilde c,M>0$. There exists a $\kappa_0=\kappa_0(\alpha,\beta,\mu_0,\tilde c,M)\in(0,\mu_0]$ that satisfies the following:

Suppose $\theta^1,\dots,\theta^n\in\Co^\alpha(\mu_0\Ball^n;T^*\R^n)$ such that $\theta^i\big|_0=dx^i\big|_0$ for $i=1,\dots,n$ and $d\theta^1,\dots,d\theta^n\in\Co^{\beta-1}\mleft(\mu_0\Ball^n;\mywedge^2T^*\R^n\mright)$ with estimate 
\begin{equation}\label{Eqn::Key::Scaling::AssumptionOnTheta}
    \sum_{i=1}^n\|\theta^i\|_{\Co^\alpha(\mu_0\Ball^n;T^*\R^n)}+\|d\theta^i\|_{\Co^{\beta-1}(\mu_0\Ball^n;\wedge^2T^*\R^n)}<M.
\end{equation}

Then there are 1-forms $\lambda^1,\dots,\lambda^n\in\Co^\alpha(\Ball^n;T^*\Ball^n)$ such that
\begin{enumerate}[parsep=-0.3ex,label=(\roman*)]
    \item\label{Item::Key::Scaling::FormEqual} $\lambda^i\big|_{\frac13\Ball^n}=\frac1{\kappa_0}\cdot(\phi_{\kappa_0}^*\theta^i)\big|_{\frac13\Ball^n}$, $i=1,\dots,n$, where $\phi_{\kappa_0}(x):=\kappa_0\cdot x$ is the scaling map.
    \item\label{Item::Key::Scaling::Est} $\lambda^1,\dots,\lambda^n$ satisfy the assumptions of Theorem \ref{Thm::Keythm} with the constant $c=\tilde c$. That is,
    \begin{itemize}
        \item $\lambda^1,\dots,\lambda^n$ span the cotangent space at every point in $\Ball^n$.
        \item $\supp(\lambda^i-dx^i)\subsetneq\frac12\Ball^n $.
        \item $\sum_{i=1}^n(\|\lambda^i-dx^i\|_{\Co^\alpha}+\|d\lambda^i\|_{\Co^{\beta-1}})\le \tilde c$.
    \end{itemize}
\end{enumerate}
\end{prop}



{
The key to Proposition \ref{Prop::Key::Scaling} is the next lemma.
\begin{lemma}\label{Lemma::Key::ScalingLemma}
Let $\gamma>0$, then for any $\mu_0>0$ there is a $C_{\gamma,\mu_0}>0$ such that,
\begin{equation}\label{Eqn::Key::ScalingLemma}
    \|f(\kappa\cdot)\|_{\Co^\gamma(\Ball^n)}\le C_{\gamma,\mu_0} \kappa^{\min(\gamma,\frac12)}\|f\|_{\Co^\gamma(\mu_0\Ball^n)},\quad\forall\kappa\in(0,\mu_0], f\in\Co^\gamma(\mu_0\Ball^n)\text{ such that }f(0)=0.
\end{equation}
\end{lemma}
\begin{proof}
By taking a scaling $x\mapsto\mu_0 x$, we can assume $\mu_0=1$ without loss of generality. Thus, $f$ is defined on the unit ball.
To prove the result, we use the characterizations of Zygmund-H\"older norms in Remark \ref{Rmk::FuncSpace::CharofZyg}.

For $\kappa\in(0,1]$ set $f_\kappa(x):=f(\kappa x)$. 
For $x\in\Ball^n$ and $\kappa\in(0,1]$, by Remark \ref{Rmk::FuncSpace::CharofZyg} \ref{Item::FuncSpace::CharofZyg::01}, for $x\in\Ball^n$,
\begin{equation}\label{Eqn::Key::ScalingLemma::ProofSup}
    |f_\kappa(x)|=|f(\kappa x)-f(0)|\lesssim\|f\|_{\Co^{\min(\gamma,\frac12)}(\Ball^n)}|\kappa x-0|^{\min(\gamma,\frac12)}\lesssim\|f\|_{\Co^\gamma(\Ball^n)}\kappa^{\min(\gamma,\frac12)}. 
\end{equation}
When $\gamma\in(0,2)$, using Remark \ref{Rmk::FuncSpace::CharofZyg} \ref{Item::FuncSpace::CharofZyg::02}, for $x_1,x_2\in\Ball^n$,
\begin{equation}\label{Eqn::Key::ScalingLemma::Proof02}
    \textstyle|\frac {f_\kappa(x_1)+f_\kappa(x_2)}2-f_\kappa(\frac{x_1+x_2}2)|=\left|\frac {f(\kappa x_1)+f(\kappa x_2)}2-f(\kappa\frac{x_1+x_2}2)\right|\lesssim_\gamma\|f\|_{\Co^{\min(\gamma,\frac12)}}|\kappa(x_1-x_2)|^\gamma\le\kappa^\gamma\|f\|_{\Co^\gamma}|x_1-x_2|^\gamma.
\end{equation}
    Combining \eqref{Eqn::Key::ScalingLemma::ProofSup} and \eqref{Eqn::Key::ScalingLemma::Proof02}, we get \eqref{Eqn::Key::ScalingLemma} for the case $0<\gamma<2$, since
    $$\textstyle\|f_\kappa\|_{\Co^\gamma(\Ball^n)}\approx\sup\limits_{x\in\Ball^n}|f_\kappa(x)|+\sup\limits_{x_1,x_2\in\Ball^n}|x_1-x_2|^{-\gamma}\left|\frac {f_\kappa(x_1)+f_\kappa(x_2)}2-f_\kappa(\frac{x_1+x_2}2)\right|\lesssim_\gamma\kappa^{\min(\gamma,\frac12)}\|f\|_{\Co^\gamma(\Ball^n)}.$$

    For $\gamma\ge2$, we proceed by induction. We prove the result for $\gamma\in [l,l+1)$, for $l\in \{1,2,\ldots\}$.  The base case, $l=1$ was shown above. We assume the result for $l-1$ and prove it for $l$.
    
    Assume $\gamma\in[l,l+1)$ where $l\ge2$. Note that $\nabla f_\kappa(x)=\kappa(\nabla f)(\kappa x)$, so $\|\partial_{x^j}(f_\kappa)\|_{\Co^{\gamma-1}(\Ball^n)}=\kappa\|(\partial_{x^j}f)_\kappa\|_{\Co^{\gamma-1}(\Ball^n)}\le \|(\partial_{x^j}f)_\kappa\|_{\Co^{\gamma-1}(\Ball^n)}$ for $j=1,\dots,n$. Here $(\partial_{x^j}f)_\kappa(x)=(\partial_{x^j}f)(\kappa x)$. 
    
    By the inductive hypothesis $\|f_\kappa\|_{\Co^{\gamma-1}(\Ball^n)}\le C_{\gamma-1}\kappa^{\frac12}\|f\|_{\Co^{\gamma-1}(\Ball^n)}$ and $\|(\partial_{x^j}f)_\kappa\|_{\Co^{\gamma-1}(\Ball^n)}\le C_{\gamma-1}\kappa^{\frac12}\|\partial_{x^j}f\|_{\Co^{\gamma-1}(\Ball^n)}$ for $j=1,\dots,n$. So by Remark \ref{Rmk::FuncSpace::CharofZyg} \ref{Item::FuncSpace::CharofZyg::>1} we get 
    \begin{align*}
        &\|f_\kappa\|_{\Co^\gamma(\Ball^n)}\approx\|f_\kappa\|_{\Co^{\gamma-1}(\Ball^n)}+\sum_{j=1}^n\|\partial_{x^j}(f_\kappa)\|_{\Co^{\gamma-1}(\Ball^n)}
        \\&\lesssim\kappa^\frac12\Big(\|f\|_{\Co^{\gamma-1}(\Ball^n)}+\sum_{j=1}^n\|\partial_{x^j}f\|_{\Co^{\gamma-1}(\Ball^n)}\Big)\approx \kappa^\frac12\|f\|_{\Co^{\gamma}(\Ball^n)}=\kappa^{\min(\gamma,\frac12)}\|f\|_{\Co^{\gamma}(\Ball^n)},
    \end{align*}
    completing the proof.
\end{proof}
}

\begin{proof}[Proof of Proposition \ref{Prop::Key::Scaling}]
First we construct 1-forms $\rho^1,\dots,\rho^n\in\Co^\beta(\mu_0\Ball^n;T^*\R^n)$ such that for $i=1,\dots,n$, 
\begin{enumerate}[parsep=-0.3ex,label=(\alph*)]
    \item\label{Item::Key::Scaling::RhoFormEqual} $\rho^i\big|_0=0$ and $d\rho^i\big|_{\frac{\mu_0}2\Ball^n}=d\theta^i\big|_{\frac{\mu_0}2\Ball^n}$.
    \item\label{Item::Key::Scaling::RhoFormEst} There is a $C_0=C_0(n,\alpha,\beta,\mu_0)>0$ that does not depend on $\theta^i$, such that 
    \begin{equation}\label{Eqn::Key::Scaling::RhoFormEst}
        \|\rho^i\|_{\Co^\beta(\mu_0\Ball^n;T^*\R^n)}\le C_0(\|\theta^i\|_{\Co^\alpha(\mu_0\Ball^n;T^*\R^n)}+\|d\theta^i\|_{\Co^{\beta-1}(\mu_0\Ball^n;\wedge^2T^*\R^n)}).
    \end{equation}
\end{enumerate}

Take a $\chi_0\in C_c^\infty(\mu_0\Ball^n)$ such that $\chi_0\big|_{\frac{\mu_0}2\Ball^n}\equiv1$. Define
\begin{equation}
    \tilde\rho^i:=\Green\ast\codiff d(\chi_0\theta^i)=\Green\ast\codiff (\chi_0\cdot d\theta^i+d\chi_0\wedge \theta^i),\quad\rho^i:=\tilde\rho^i-(\tilde\rho^i\big|_0),\quad i=1,\dots,n.
\end{equation}
Recall $\codiff$ is the codifferential from Notation \ref{Note::FuncRevis::Codiff}, and $\Green$ is the fundamental solution of Laplacian as in \eqref{Eqn::FuncRevis::GreensFunction}. 
The convolution is defined in $\R^n$ using Lemma \ref{Lemma::FuncRevis::GreensOp}, since the support $\supp \codiff d(\chi_0\theta^i)\subseteq\supp\chi_0\Subset\mu_0\Ball^n$ is compact.

Clearly $\rho^i\big|_0=0$. Similar to the proof of Lemma \ref{Lemma::FuncRevis::DecomposeForms}, since $\chi_0\big|_{\frac{\mu_0}2\Ball^n}\equiv1$, we have $$d\theta^i\big|_{\frac{\mu_0}2\Ball^n}=d(\chi_0\theta^i)\big|_{\frac{\mu_0}2\Ball^n}=(\codiff d+d\codiff)(\Green\ast d(\chi_0\theta^i))\big|_{\frac{\mu_0}2\Ball^n}=(\Green\ast d\codiff d(\chi_0\theta^i))\big|_{\frac{\mu_0}2\Ball^n}=d\tilde\rho^i\big|_{\frac{\mu_0}2\Ball^n}=d\rho^i\big|_{\frac{\mu_0}2\Ball^n}.$$
So condition \ref{Item::Key::Scaling::RhoFormEqual} is satisfied.

By Lemma \ref{Lemma::FuncRevis::NewtonianBoundedness} we have, for every $\mu>0$,
\begin{equation}\label{Eqn::Key::ScalingProof3}
    \|\Green\ast\codiff\omega\|_{\Co^{\beta}(\mu\Ball^n;\wedge^2T^*\R^n)}\lesssim_{\beta,\mu}\|\omega\|_{\Co^{\beta-1}(\mu\Ball^n;T^*\R^n)},\quad\forall \omega\in\Co^{\beta-1}_c(\mu\Ball^n;T^*\R^n).
\end{equation}

Take $\omega=\codiff (\chi_0\cdot d\theta^i+d\chi_0\wedge \theta^i)$, $\mu=\mu_0$ in \eqref{Eqn::Key::ScalingProof3} and by Lemma \ref{Lemma::FuncSpace::Product}, we have
\begin{equation}\label{Eqn::Key::Scaling::RhoFormEstCompute}
    \begin{aligned}
    &\|\rho^i\|_{\Co^\beta(\mu_0\Ball^n;T^*\R^n)}\le \|\tilde\rho^i\|_{\Co^\beta(\mu_0\Ball^n;T^*\R^n)}+|\tilde\rho^i(0)|\le 2\|\tilde\rho^i\|_{\Co^\beta(\mu_0\Ball^n;T^*\R^n)}
    \\
    &\lesssim_{\beta,\mu_0}\|\codiff(\chi_0\cdot d\theta^i+d\chi_0\wedge \theta^i)\|_{\Co^{\beta-2}(\mu_0\Ball^n;\wedge^2T^*\R^n)}\lesssim_{\beta,\mu_0}\|\chi_0\|_{\Co^{\alpha+1}}\|d\theta^i\|_{\Co^{\beta-1}}+\|d\chi_0\|_{\Co^{\alpha}}\|\theta^i\|_{\Co^{\beta-1}}
    \\
    &\lesssim_{\alpha,\beta,\mu_0,\chi_0}\|d\theta^i\|_{\Co^{\beta-1}(\mu_0\Ball^n;\wedge^2T^*\R^n)}+\|\theta^i\|_{\Co^\alpha(\mu_0\Ball^n;T^*\R^n)}.
\end{aligned}
\end{equation}

\eqref{Eqn::Key::Scaling::RhoFormEstCompute} gives us the $C_0$ for condition \ref{Item::Key::Scaling::RhoFormEst}. This complete the proof of \ref{Item::Key::Scaling::RhoFormEqual} and \ref{Item::Key::Scaling::RhoFormEst} and we get $\rho^1,\dots,\rho^n$ as desired.

\medskip
Fix $\chi_1\in C_c^\infty(\frac12\Ball^n)$ such that $\chi_1\big|_{\frac13\Ball^n}\equiv1$. For $\kappa>0$, let $\phi_\kappa(x):=\kappa\cdot x$, so $\phi_\kappa$ maps $\frac12\Ball^n$ into $\frac{\mu_0}2\Ball^n$ when $\kappa\in(0,\mu_0]$. For $\kappa\in(0,\mu_0]$, we define 1-forms $\lambda_\kappa^1,\dots,\lambda_\kappa^n$ and $\tau_\kappa^1,\dots,\tau_\kappa^n$ by
\begin{equation}\label{Eqn::Key::Scaling::DefLambdaTau}
    \lambda_\kappa^i:=dx^i+\tfrac1\kappa\chi_1\cdot\phi_\kappa^*(\theta^i-dx^i),\quad\tau_\kappa^i:=\tfrac1\kappa\chi_1\cdot(\phi_\kappa^*\rho^i)+\tfrac1\kappa\Green\ast\codiff\left(d\chi_1\wedge \phi_\kappa^*(\theta^i-dx^i)\right),\quad i=1,\dots,n.
\end{equation}

Since $\theta^i\in\Co^\alpha$ and $\rho^i\in\Co^\beta$, we have $\lambda_\kappa^i\in\Co^\alpha(\Ball^n;T^*\R^n)$, $\tau_\kappa^i\in\Co^\beta(\Ball^n;T^*\R^n)$ (by \eqref{Eqn::Key::ScalingProof3}) and $\supp(\lambda_\kappa^i-dx^i)\subseteq\supp\chi_1\Subset\frac12\Ball^n$. And since $\chi_1\big|_{\frac13\Ball^n}\equiv1$ and $\phi_\kappa^*dx=\kappa dx$,
\begin{equation}\label{Eqn::Key::Scaling::FormEqualProof}
    \lambda_\kappa^i\big|_{\frac13\Ball^n}=dx^i+\tfrac1\kappa\cdot\phi_\kappa^*(\theta^i-dx^i)\big|_{\frac13\Ball^n}=\tfrac1\kappa(\phi_\kappa^*\theta^i)\big|_{\frac13\Ball^n}.
\end{equation}

We write $\theta^i$ and $\rho^i$, $i=1,\dots,n$ as
\begin{equation}\label{Eqn::Key::Scaling::ThetaRhoAB}
    \theta^i=dx^i+\sum_{j=1}^na^i_j(x)dx^j,\quad\rho^i=\sum_{j=1}^nb^i_j(x)dx^j,\quad\text{where }a^i_j\in\Co^\alpha(\mu_0\Ball^n),\ b^i_j\in\Co^\beta(\mu_0\Ball^n).
\end{equation}
By assumption $\theta^i\big|_0=dx^i\big|_0$ and $\rho^i\big|_0=0$ for $i=1,\dots,n$, so $a^i_j(0)=b^i_j(0)=0$ for all $1\le i,j\le n$. And we have
\begin{equation}\label{Eqn::Key::ScalingProofLambdaCoordinate}
    \lambda_\kappa^i=dx^i+\sum_{j=1}^n\chi_1(x)a^i_j(\kappa x)dx^j,\quad\tau_\kappa^i=\sum_{j=1}^n\Big(\chi_1(x) b^i_j(\kappa x)dx^j+\Green\ast\codiff\big(a^i_j(\kappa x)d\chi_1\wedge dx^j\big)\Big),\quad i=1,\dots,n.
\end{equation}

Since $\phi_\kappa(\frac12\Ball^n)\subseteq\frac{\mu_0}2\Ball^n$ and $\supp\chi_1\Subset\frac12\Ball^n$, by condition \ref{Item::Key::Scaling::RhoFormEqual}  we have $\chi_1\cdot\phi_\kappa^*d\rho^i=\chi_1\cdot\phi_\kappa^*d\theta^i$ for $i=1,\dots,n$. It follows that  $d\lambda_\kappa^i=d\tau_\kappa^i$ for $i=1,\dots,n$; indeed,
\begin{equation}\label{Eqn::Key::ScalingTmp2}
    d\lambda_\kappa^i-d\tau_\kappa^i
    =\tfrac1\kappa d\chi_1\wedge \phi_\kappa^*(\theta^i-dx^i)
    +\tfrac1\kappa\chi_1\cdot\phi_\kappa^*d\theta^i
    -\tfrac1\kappa d\chi_1\wedge\phi_\kappa^*(\theta^i-dx^i)
    -\tfrac1\kappa\chi_1\cdot\phi_\kappa^*d\rho^i
    =\tfrac1\kappa \chi_1\cdot\phi_\kappa^*(d\theta^i-d\rho^i)=0.
\end{equation}

Applying Lemmas \ref{Lemma::FuncSpace::Product} and \ref{Lemma::Key::ScalingLemma} to $a^i_j$  we have
\begin{equation}\label{Eqn::Key::Scaling::BddA}
    \|\chi_1\cdot a^i_j(\kappa\cdot)\|_{\Co^\alpha(\Ball^n)}\lesssim_\alpha\|\chi_1\|_{\Co^\alpha}\|a^i_j(\kappa\cdot)\|_{\Co^\alpha(\Ball^n)}\lesssim_{\alpha,\mu_0}\kappa^{\min(\alpha,\frac12)}\|\chi_1\|_{\Co^\alpha}\|a^i_j\|_{\Co^\alpha(\mu_0\Ball^n)},\quad\forall\kappa\in(0,\mu_0].
\end{equation}

Using \eqref{Eqn::Key::ScalingProofLambdaCoordinate} and \eqref{Eqn::Key::Scaling::BddA}, we deduce that

\begin{equation}\label{Eqn::Key::ScalingProof2.1}
\sum_{i=1}^n\|\lambda_\kappa^i-dx^i\|_{\Co^\alpha}\le\sum_{i,j=1}^n\|\chi_1(x)a^i_j(\kappa x)\|_{\Co^\alpha}\lesssim_{\alpha,\mu_0}\kappa^{\min(\alpha,\frac12)}\|\chi_1\|_{\Co^\alpha}\sum_{i,j=1}^n\| a^i_j\|_{\Co^\alpha(\mu_0\Ball^n)}.
\end{equation}

By \eqref{Eqn::Key::ScalingTmp2} we have $d\lambda_\kappa^i=d\tau_\kappa^i$, and therefore 
\begin{equation}\label{Eqn::Key::ScalingTmpDlambdaDkappa}
    \|d\lambda_\kappa^i\|_{\Co^{\beta-1}(\Ball^n;\wedge^2T^*\Ball^n)}=\|d\tau_\kappa^i\|_{\Co^{\beta-1}(\Ball^n;\wedge^2T^*\Ball^n)}\lesssim_\beta\|\tau_\kappa^i\|_{\Co^{\beta}(\Ball^n;T^*\Ball^n)},\quad i=1,\dots,n.
\end{equation}

Applying Lemma \ref{Lemma::Key::ScalingLemma} to $ a^i_j$ and $ b^i_j$ we get that for $0<\kappa<\mu_0$,
\begin{gather}\label{Eqn::Key::ScalingTmp2.1}
    \|\chi_1\cdot b^i_j(\kappa x)\|_{\Co^\beta(\Ball^n)}\lesssim_\beta\|\chi_1\|_{\Co^\beta}\|b^i_j(\kappa x)\|_{\Co^\beta(\Ball^n)}\lesssim_\beta\kappa^{\min(\beta,\frac12)}\|\chi_1\|_{\Co^\beta}\|b^i_j\|_{\Co^\beta}\le \kappa^{\min(\alpha,\frac12)}\|\chi_1\|_{\Co^\beta}\|b^i_j\|_{\Co^\beta(\mu_0\Ball^n)}.
    \\\label{Eqn::Key::ScalingTmp2.2}
    \|a^i_j(\kappa x)d\chi_1\|_{\Co^\alpha(\Ball^n;T^*\R^n)}\lesssim_\alpha\|a^i_j(\kappa x)\|_{\Co^\alpha(\Ball^n)}\|\chi_1\|_{\Co^{\alpha+1}}\lesssim_\alpha\kappa^{\min(\alpha,\frac12)}\|\chi_1\|_{\Co^{\alpha+1}}\|a^i_j\|_{\Co^\alpha(\mu_0\Ball^n)}.
\end{gather}

Letting $\omega=a_j^i(\kappa x)d\chi_1\wedge dx^j\in\Co^{\alpha}(\Ball^n;T^*\Ball^n)\subseteq\Co^{\beta-1}(\Ball^n;T^*\Ball^n)$ and $\mu=1$ in \eqref{Eqn::Key::ScalingProof3}, we see that  
\begin{equation}
\label{Eqn::Key::ScalingProof2.2}
\begin{aligned}
    &\sum_{i=1}^n\|d\lambda_\kappa^i\|_{\Co^{\beta-1}(\Ball^n;\wedge^2T^*\Ball^n)}=\sum_{i=1}^n\|d\tau_\kappa^i\|_{\Co^{\beta-1}(\Ball^n;\wedge^2T^*\Ball^n)}\lesssim_\beta\sum_{i=1}^n\|\tau_\kappa^i\|_{\Co^\beta(\Ball^n;T^*\Ball^n)}&\text{by \eqref{Eqn::Key::ScalingTmpDlambdaDkappa}}
    \\
    &\le\sum_{i,j=1}^n\big(\|\chi_1(x)b_j^i(\kappa x)dx^j\|_{\Co^\beta}+\big\|\Green\ast\codiff\big(a^i_j(\kappa x)d\chi_1\wedge dx^j\big)\big\|_{\Co^\beta}\big)&\text{by \eqref{Eqn::Key::ScalingProofLambdaCoordinate}}
    \\
    &\lesssim_\beta\sum_{i,j=1}^n\big(\|\chi_1(x)b_j^i(\kappa x)\|_{\Co^\beta(\Ball^n)}+\|a^i_j(\kappa x)d\chi_1\wedge dx^j\big\|_{\Co^{\beta-1}(\Ball^n;\wedge^2T^*\Ball^n)}\big)&\text{by \eqref{Eqn::Key::ScalingProof3}}
    \\
    &\lesssim_{\alpha,\beta}\|\chi_1\|_{\Co^{\alpha+1}}\sum_{i,j=1}^n\big(\|b_j^i(\kappa x)\|_{\Co^\beta(\Ball^n)}+\|a^i_j(\kappa x)d\chi_1\big\|_{\Co^{\alpha}(\Ball^n;T^*\Ball^n)}\big)&\text{since }\alpha\ge\beta-1
    \\
    &\lesssim_{\alpha,\beta}\kappa^{\min(\alpha,\frac12)}\|\chi_1\|_{\Co^{\alpha+1}}\sum_{i,j=1}^n\Big(\| a^i_j\|_{\Co^\alpha(\mu_0\Ball^n)}+\|b^i_j\|_{\Co^\beta(\mu_0\Ball^n)}\Big)&\text{by \eqref{Eqn::Key::ScalingTmp2.1} and \eqref{Eqn::Key::ScalingTmp2.2}}.
\end{aligned}\end{equation}

Note that $\chi_1$ is a fixed cut-off function whose $\Co^\alpha$ and $\Co^{\alpha+1}$-norms depend only on $n,\alpha$. So combining \eqref{Eqn::Key::ScalingProof2.1} and \eqref{Eqn::Key::ScalingProof2.2} we have
\begin{equation}\label{Eqn::Key::Scaling::FinalBdd1}
    \sum_{i=1}^n(\|\lambda_\kappa^i-dx^i\|_{\Co^\alpha}+\|d\lambda_\kappa^i\|_{\Co^{\beta-1}})\lesssim_{\alpha,\beta,\mu_0}\kappa^{\min(\alpha,\frac12)}\sum_{i,j=1}^n\big(\|a^i_j\|_{\Co^\alpha(\mu_0\Ball^n;T^*\R^n)}+\|b^i_j\|_{\Co^{\beta}(\mu_0\Ball^n;\wedge^2T^*\R^n)}\big).
\end{equation}

By \eqref{Eqn::Key::Scaling::ThetaRhoAB} we have $\sum_{i,j=1}^n\|a^i_j\|_{\Co^\alpha}\lesssim 1+\sum_{i=1}^n\|\theta^i\|_{\Co^\alpha}$ and $\sum_{i,j=1}^n\|b^i_j\|_{\Co^\beta}\lesssim\sum_{i=1}^n\|\rho^i\|_{\Co^\beta}$. And combining \eqref{Eqn::Key::Scaling::FinalBdd1} with \eqref{Eqn::Key::Scaling::RhoFormEst}, we can find a $C_1=C_1(n,\alpha,\beta,\mu_0)>0$ that does not depend on the other quantities, such that
\begin{equation}\label{Eqn::Key::Scaling::FinalBdd2}
    \sum_{i=1}^n\|\lambda_\kappa^i-dx^i\|_{\Co^\alpha}+\|d\lambda_\kappa^i\|_{\Co^{\beta-1}}\le C_1\cdot\kappa^{\min(\alpha,\frac12)}\sum_{i=1}^n\Big(1+\|\theta^i\|_{\Co^\alpha(\mu_0\Ball^n;T^*\R^n)}+\|d\theta^i\|_{\Co^{\beta-1}(\mu_0\Ball^n;\wedge^2T^*\R^n)}\Big),\quad\forall \kappa\in(0,\mu_0].
\end{equation}

Now applying assumption \eqref{Eqn::Key::Scaling::AssumptionOnTheta} to \eqref{Eqn::Key::Scaling::FinalBdd2} we have $$\sum_{i=1}^n\|\lambda_\kappa^i-dx^i\|_{\Co^\alpha}+\|d\lambda_\kappa^i\|_{\Co^{\beta-1}}\le \kappa^{\min(\alpha,\frac12)}C_1\cdot (M+n).$$ 

Since $\|(\langle\lambda_\kappa^i-dx^i,\Coorvec{x^j}\rangle)_{n\times n}\|_{C^0(\Ball^n;\Mbb^{n\times n})}\lesssim_\alpha\sum_{i=1}^n\|\lambda_\kappa^i-dx^i\|_{\Co^\alpha(\Ball^n)}$, we can find a $\tilde c'=\tilde c'(n,\alpha)>0$ such that 
$$\sum_{i=1}^n\|\lambda_\kappa^i-dx^i\|_{\Co^\alpha}\le \tilde c'\quad\text{implies}\quad\textstyle\|(\langle\lambda_\kappa^i-dx^i,\Coorvec{x^j}\rangle)_{n\times n}\|_{C^0(\Ball^n;\Mbb^{n\times n})}\le\frac12.$$ In particular $\mleft(\langle\lambda_\kappa^i,\Coorvec{x^j}\rangle\mright)_{n\times n}=I+\mleft(\langle\lambda_\kappa^i-dx^i,\Coorvec{x^j}\rangle\mright)_{n\times n}$ is invertible at every point in $\Ball^n$, which means $(\lambda_\kappa^1,\dots,\lambda_\kappa^n)$  span the tangent space at every point in $\Ball^n$.

We take $\kappa_0=\kappa_0(n,\alpha,\beta,\mu_0,M,\tilde c)>0$ such that
$$0<\kappa_0<\mu_0\quad\text{and}\quad\kappa_0^{\min(\alpha,\frac12)}C_1\cdot (M+n)\le \min(\tilde c,\tilde c').$$ 

Take $\lambda^i=\lambda_{\kappa_0}^i$ for $i=1,\dots,n$. We have $\sum_{i=1}^n\|\lambda^i-dx^i\|_{\Co^\alpha}+\|d\lambda^i\|_{\Co^{\beta-1}}\le\min(\tilde c,\tilde c')$. {By our assumption on $\tilde c'$,} $\lambda^1,\dots,\lambda^n$ span the tangent space at every point in $\Ball^n$. Note that by \eqref{Eqn::Key::ScalingProofLambdaCoordinate} $\supp(\lambda_{\kappa_0}^i-dx^i)\subseteq\supp\chi_1\Subset\frac12\Ball^n$. This shows conclusion \ref{Item::Key::Scaling::Est} is satisfied.

By \eqref{Eqn::Key::Scaling::FormEqualProof} we get conclusion \ref{Item::Key::Scaling::FormEqual}, finishing the proof.
\end{proof}


We can now prove a special case of the Theorem \ref{Thm::TheMainResult}:
\begin{cor}\label{Cor::Key::BasicCaseforMainThm}
Let $\alpha>0$ and let $\beta\in[\alpha,\alpha+1]$. Let $\lambda^1,\dots,\lambda^n$ be $\Co^\alpha$ 1-forms on a $\Co^{\alpha+1}$-manifold $\Manifold$ of dimension $n$, which span the cotangent space at every point. Then the following are equivalent:

\begin{enumerate}[(1)]
    \item\label{Item::Key::BasicCaseforMainThm::a} For every $p\in\Manifold$  there exist a neighborhood $U\subseteq\Manifold$ of $p$ and a $\Co^{\alpha+1}_\loc$-diffeomorphism $\Phi:\Ball^n\xrightarrow{\sim}U\subseteq \Manifold$ 
    such that $\Phi(0)=p$ and $\Phi^*\lambda^1,\dots,\Phi^*\lambda^n\in\Co^\beta(\Ball^n;T^*\Ball^n)$.
    \item\label{Item::Key::BasicCaseforMainThm::b} $d\lambda^1,\dots,d\lambda^n$ have regularity  $\Co^{\beta-1}_\loc$ (see Definition \ref{Defn::FuncRn::dRegularityLow}).
\end{enumerate}
\end{cor}

\begin{rmk}\label{Rmk::Key::SpecialCase}
This is the special case of Theorem \ref{Thm::TheMainResult} when $\beta\in[\alpha,\alpha+1]$ and $q=n$: given $\Co^\alpha$-vector fields $X_1,\dots,X_n$ that span the $n$-dimensional tangent space at every point, we take 1-forms $\lambda^1,\dots,\lambda^n$ to be the corresponding dual basis that span the cotangent space at every point. 
\end{rmk}

\begin{proof}
By passing to a local coordinate system, we can assume $\Manifold$ to be an open subset of $\R^n$, since by Proposition \ref{Prop::FuncRevis::QuantdOfLowRegularityForm} both conditions \ref{Item::Key::BasicCaseforMainThm::a} and \ref{Item::Key::BasicCaseforMainThm::b} are invariant under $\Co^{\alpha+1}_\loc$-diffeomorphisms.

\ref{Item::Key::BasicCaseforMainThm::a}$\Rightarrow$\ref{Item::Key::BasicCaseforMainThm::b}: For such $\Phi$, we have $d(\Phi^*\lambda^i)\in\Co^{\beta-1}_\loc(\Ball^n;T^*\Ball^n)$ for $i=1,\dots,n$. So $\Phi^{-1}:U\subseteq\Manifold\to\R^n$ is the desired coordinate chart that shows $d\lambda^1,\dots,d\lambda^n$ fulfill the 
conditions for $d\lambda^1,\dots,d\lambda^n$ to have regularity in $\Co^{\beta-1}_\loc$
(see Definition \ref{Defn::FuncRn::dRegularityLow}).

\ref{Item::Key::BasicCaseforMainThm::b}$\Rightarrow$\ref{Item::Key::BasicCaseforMainThm::a}: Let $p\in\Manifold$. 

By passing to local coordinate system and applying an invertible linear transformation we can find a $\mu_0>0$ and a $\Co^{\alpha+1}$-coordinate chart $F_0:U_0\subseteq\Manifold\xrightarrow{\sim} \mu_0\Ball^n_x$ such that
\begin{itemize}[parsep=-0.3ex]
    \item $F_0(p)=0$.
    \item $((F_0)_*\lambda^i)\big|_0=dx^i\big|_0$ for $i=1,\dots,n$.
\end{itemize}

Take $c>0$ be the small constant in Theorem \ref{Thm::Keythm}. By Proposition \ref{Prop::Key::Scaling} with $\tilde c=c$, we can find a $\kappa_0\in(0,\mu_0]$ and 1-forms $\tilde \lambda^1,\dots,\tilde\lambda^n\in\Co^\alpha(\Ball^n;T^*\Ball^n)$ such that for scaling map $\phi_{\kappa_0}:\Ball^n\to\mu_0\Ball^n$, $\phi_{\kappa_0}(x)=\kappa_0x$, we have
\begin{enumerate}[parsep=-0.3ex,label=(\alph*)]
    \item $\tilde\lambda^1,\dots,\tilde\lambda^n$ span the cotangent space at every point.
    \item $\supp(\tilde\lambda^i-dx^i)\subsetneq\frac12\Ball^n$ for $i=1,\dots,n$.
    \item\label{Item::Key::BasicCaseforMainThm::Tmp} $(\tilde\lambda^1,\dots,\tilde \lambda^n)\big|_{\frac13\Ball^n}=\frac1{\kappa_0}\cdot((F_0^{-1}\circ\phi_{\kappa_0})^*\lambda^1,\dots,(F_0^{-1}\circ\phi_{\kappa_0})^*\lambda^n)\big|_{\frac13\Ball^n}$.
    \item $\sum_{i=1}^n\|\tilde\lambda^i-dx^i\|_{\Co^\alpha}+\|d\tilde\lambda^i\|_{\Co^{\beta-1}}<c$.
\end{enumerate}
We set $F_1:=\phi_{\kappa_0}^{-1}$. 

Applying Theorem \ref{Thm::Keythm} to $\tilde\lambda^1,\dots,\tilde\lambda^n$ we obtain a $\Co^{\alpha+1}$-chart $F_2:\Ball^n\xrightarrow{\sim}\Ball^n$ such that $F_2(\frac13\Ball^n)\supseteq B^n(F_2(0),\frac16)$, $\|F_2-\id\|_{\Co^{\alpha+1}}<c$ and $(F_2)_*\tilde\lambda^1,\dots,(F_2)_*\tilde\lambda^n\in\Co^\beta(\Ball^n;T^*\Ball^n)$.

By \ref{Item::Key::BasicCaseforMainThm::Tmp}, $\left((F_2\circ F_1\circ F_0)_*\lambda^i\right)\big|_{F_2(\frac13\Ball^n)}={\tfrac1{\kappa_0}}\cdot(F_2)_*\tilde\lambda^i\big|_{F_2(\frac13\Ball^n)}\in\Co^\beta$ for $i=1,\dots,n$. We can take an
affine linear transformation
$F_3:\R^n\to\R^n$ such that $F_3(F_2(0))=0$ and $F_3(B^n(F_2(0),\frac16))\supseteq\Ball^n$. 


{Now we have $F_0(U_0)\supseteq\mu_0\Ball^n$, $F_1(\mu_0\Ball^n)\supseteq\Ball^n$, $F_2(\frac13\Ball^n)\supseteq B^n(F_2(0),\frac16)$ and $F_3(B^n(F_2(0),\frac16)\supseteq\Ball^n$.} Take $\Phi:=(F_3\circ F_2\circ F_1\circ F_0)^{-1}:\Ball^n\to\Manifold$. Since $F_0,F_1,F_2,F_3$ are all $\Co^{\alpha+1}$-diffeomorphism onto their images, we know $\Phi:\Ball^n\xrightarrow{\sim}\Phi(\Ball^n)\subseteq\Manifold$ is a $\ZygSymb^{\alpha+1}$ diffeomorphism. Moreover, we have $\Phi(0)=p$ because $F_0(p)=0$, $F_1(0)=0$, $F_3(F_2(0))=0$. Thus, $\Phi$ is the diffeomorphism we desire with $U:=\Phi(\Ball^n)$, completing the proof.
\end{proof}


    
\section{Function Spaces Along Vector Fields, Revisited}\label{Section::FuncVF}



Let $\alpha>0$, $\beta>1-\alpha$, and let $\Phi:\Nanifold\xrightarrow{\sim}\Manifold$ be a $\Co^{\alpha+1}$-diffeomorphism between two $\Co^{\alpha+1}$-manifolds $\Nanifold$ and $\Manifold$. Let $X\in\Co^\alpha_\loc(\Manifold;T\Manifold)$, $f\in\Co^\beta_\loc(\Manifold)$, $Y\in\Co^\beta_\loc(\Manifold;T\Manifold)$ and $\theta\in\Co^\beta_\loc\mleft(\Manifold;\mywedge^kT^*\Manifold\mright)$. We have the following
\begin{itemize}
    \item $\Phi^*(Xf)=(\Phi^*X)(\Phi^*f)$ on $\Nanifold$.
    \item $\Phi^*[X,Y]=[\Phi^*X,\Phi^*Y]$ and $\Phi^*(\Lie{X}\theta)=\Lie{\Phi^*X}\Phi^*\theta$ on $\Nanifold$, provided that $\alpha>\frac12$.
\end{itemize}
In fact Lemma \ref{Lemma::FuncRn::PullbackComm} establishes the above facts on open subsets in $\R^n$. And since these results are local, they hold in manifolds as well.

\begin{rmk}\label{Rmk::FuncVF::ZygVFUnderDiffeo}

Let $\alpha,\gamma>0$, $\beta\in(-\min(\alpha,\gamma),\infty)$, and let $\Manifold$ and $\Nanifold$ be two $\Co^{\max(\alpha,\gamma)+1}$-manifolds. Assume $\Phi:\Nanifold\to\Manifold$ is a $\Co^{\gamma+1}$-diffeomorphism.
Let $X_1,\dots,X_q$ be $\Co^{\alpha}_\loc$-vector fields on $\Manifold$ that span the tangent space at every point. We write $\Phi^*X$ for the list of pullback vector fields $(\Phi^*X_1,\dots,\Phi^*X_q)$ defined on $\Nanifold$. The following properties follow easily from Definitions \ref{Defn::FuncVF::Funregularity}, \ref{Defn::FuncVF::VFregularity}, \ref{Defn::FuncVf::formregularity} and \ref{Defn::FuncVf::RegularityofForms}:
\begin{enumerate}[parsep=-0.3ex,label=(\roman*)]
    \item\label{Item::FuncVF::ZygVFUnderDiffeo::Fun} If $f\in\Co^\beta_{X,\loc}(\Manifold)$, then $\Phi^*f\in\Co^\beta_{\Phi^*X,\loc}(\Nanifold)$.
    \item\label{Item::FuncVF::ZygVFUnderDiffeo::VF} If $Y\in\Co^\beta_{X,\loc}(\Manifold;T\Manifold)$, then $\Phi^*Y\in\Co^\beta_{\Phi^*X,\loc}(\Nanifold;T\Nanifold)$, provided that $\alpha,\gamma>\frac12$.
    \item\label{Item::FuncVF::ZygVFUnderDiffeo::Forms} Let $1\le k\le n$. If $\theta\in\Co^\beta_{X,\loc}\mleft(\Manifold;\mywedge^kT^*\Manifold\mright)$, then $\Phi^*\theta\in\Co^\beta_{\Phi^*X,\loc}\mleft(\Nanifold;\mywedge^kT^*\Nanifold\mright)$, provided that $\alpha,\gamma>\frac12$.
    \item\label{Item::FuncVF::ZygVFUnderDiffeo::dForms} Let $1\le k\le n-1$ and let {$\theta\in\Co^{(-\alpha)^+}_\loc\mleft(\Manifold;\mywedge^kT^*\Manifold\mright)$ (see Convention \ref{Conv::C^+Space})}. If $d\theta$ has regularity in $\Co^{\beta-1}_{X,\loc}(\Manifold)$, then $d\Phi^*\theta$ has regularity in $\Co^{\beta-1}_{\Phi^*X,\loc}(\Nanifold)$.
\end{enumerate}
\end{rmk}

When the vector fields
$X_1,\ldots, X_q$ are sufficiently smooth, our $\Co^\beta_{X,\loc}$-spaces coincide with the standard  $\Co^\beta_{\loc}$-spaces, as  the next result shows.
\begin{prop}\label{Prop::FuncVF::ZygVF=Zyg}
Let $\alpha>0$ and let $X_1,\dots,X_q$ be $\Co^\alpha_\loc$-vector fields on a $\Co^{\alpha+1}$-manifold $\Manifold$ that span the tangent space at every point. Then
\begin{enumerate}[parsep=-0.3ex,label=(\roman*)]
    \item\label{Item::FuncVF::ZygVF=Zyg::Fun} $\Co^\beta_{X,\loc}(\Manifold)=\Co^\beta_\loc(\Manifold)$ for all $\beta\in(-\alpha,\alpha+1]$.
    \item\label{Item::FuncVF::ZygVF=Zyg::VF} When $\alpha>\frac12$, $\Co^\beta_{X,\loc}(\Manifold;T\Manifold)=\Co^\beta_\loc(\Manifold;T\Manifold)$ for all $\beta\in(-\alpha,\alpha]$.
    \item\label{Item::FuncVF::ZygVF=Zyg::Forms} When $\alpha>\frac12$, for $1\le k\le n$, $\Co^\beta_{X,\loc}\mleft(\Manifold;\mywedge^kT^*\Manifold\mright)=\Co^\beta_\loc\mleft(\Manifold;\mywedge^kT^*\Manifold\mright)$ for all $\beta\in(-\alpha,\alpha]$.
    \item\label{Item::FuncVF::ZygVF=Zyg::dFormHasReg} Assume $\alpha\ge1$. Let $\beta\in(1-\alpha,1+\alpha]$ and $1\le k\le n$. Let $\theta\in\Co^{(-\alpha)^+}_\loc\mleft(\Manifold;\mywedge^kT^*\Manifold\mright)$ be a $k$-form, then $d\theta$ has regularity $\Co^{\beta-1}_{X,\loc}(\Manifold)$ if and only if $d\theta\in\Co^{\beta-1}_{\loc}\mleft(\Manifold;\mywedge^{k+1}T^*\Manifold\mright)$.
\end{enumerate}
\end{prop}

To prove Proposition \ref{Prop::FuncVF::ZygVF=Zyg} we need a simple result concerning H\"older-Zygmund spaces:
\begin{lemma}\label{Lemma::FuncRevisVF::GradCharofZyg}
Let $U\subseteq\R^n$ be an open subset, let $\beta\in\R$ and $f\in\Dist(U)$. Then $f\in\Co^\beta_\loc(U)$ if and only if $f,\partial_{x^1}f,\dots,\partial_{x^n}f\in\Co^{\beta-1}_\loc(U)$.
\end{lemma}

\begin{proof}
When $\beta>1$, the result follows from Remark \ref{Rmk::FuncSpace::CharofZyg} \ref{Item::FuncSpace::CharofZyg::>1}.  Thus, 
we consider only the case $\beta\leq 1$.

Let $V\Subset U$ be an arbitrary precompact open subset of $U$. We fix  $\chi\in C_c^\infty(U)$ such that $\chi\big|_V\equiv1$ so $\chi f$ is defined in $\R^n$. 
Since we have $\partial_{x^j}f\big|_V=\partial_{x^j}(\chi f)\big|_V$, $j=1,\dots,n$, it is enough to prove that $\tilde f:=\chi f\in\Co^\beta(\R^n)$ if and only if $\tilde f,\partial_{x^1}\tilde f,\dots,\partial_{x^n}\tilde f\in\Co^{\beta-1}(\R^n)$.

Clearly $\tilde f\in\Co^\beta(\R^n)$ implies $\partial_{x^1}\tilde f,\dots,\partial_{x^n}\tilde f\in\Co^{\beta-1}(\R^n)$. 

Conversely, suppose $\tilde f,\partial_{x^1}\tilde f,\dots,\partial_{x^n}\tilde f\in\Co^{\beta-1}(\R^n)$ hold, we have $(I+\Lap)\tilde f=\tilde f-\sum_{j=1}^n\partial_{x^j}(\partial_{x^j}\tilde f)\in\Co^{\beta-2}(\R^n)$. Since $I+\Lap:\Co^{\beta}(\R^n)\to\Co^{\beta-2}(\R^n)$ is a isomorphism of Banach spaces for every $\beta\in\R$ (see \cite[Theorem 2.3.8]{TriebelTheoryOfFunctionSpacesI}), we have $\tilde f=(I+\Lap)^{-1}(I+\Lap)\tilde f\in\Co^\beta(\R^n)$, completing the proof.
\end{proof}

\begin{proof}[Proof of Proposition \ref{Prop::FuncVF::ZygVF=Zyg} \ref{Item::FuncVF::ZygVF=Zyg::Fun}, \ref{Item::FuncVF::ZygVF=Zyg::VF}, and \ref{Item::FuncVF::ZygVF=Zyg::Forms}]

By passing to a local coordinate chart it suffices to prove the results on an open subset $\Manifold=U\subseteq\R^n$ endowed with the standard coordinate system $(x^1,\dots,x^n)$. In this coordinate system we write
\begin{equation}\label{Eqn::FuncRevisVF::ZygVF=ZygPF::CoeffforX}
    X_i=\sum_{j=1}^na_i^j\Coorvec{x^j},\quad\text{with }a_i^j\in\Co^\alpha_\loc(U).
\end{equation}
The assumption that $X_1,\dots,X_q$ span the tangent space at every point shows that the matrix function $(a_i^j)_{n\times q}\in\Co^\alpha_\loc(U;\Mbb^{n\times q})$ has full rank $n$ at every point. So we can find a matrix $(b_j^i)_{q\times n}\in\Co^\alpha_\loc(U;\Mbb^{q\times n})$ such that 
\begin{equation}\label{Eqn::FuncRevisVF::ZygVF=ZygPF::CoeffforInvMat}
    \sum_{i=1}^qb_j^ia_i^k=\delta_j^k=\begin{cases}1,&j=k,\\0,&j\neq k.\end{cases}
\end{equation}

\noindent\ref{Item::FuncVF::ZygVF=Zyg::Fun}: By Definition \ref{Defn::FuncVF::Funregularity}, $\Co^\beta_{X,\loc}(U)=\Co^\beta_\loc(U)$ holds for $\beta\in(-\alpha,1]$. {For $\beta\in(0,\alpha+1]$ we proceed by induction on $r=\lceil\beta\rceil$}. The base case $r=1$, {which is $\beta\in(0,1]$},  follows from the definition. {Let $r\geq 2$} and suppose we {have the case $\lceil\beta\rceil=r-1$,} i.e. $\Co^\beta_{X,\loc}(U)=\Co^\beta_\loc(U)$ for all $\beta\in(r-1,r]$, we wish to prove it for $\beta\in(r,\min(r+1,\alpha+1)]$. 

Let $f\in\Co^\beta_\loc(U)$, so $f\in C^1_\loc\cap\Co^{\beta-1}_\loc(U)$ and by Proposition \ref{Prop::FuncRn::LieOnManiofold} \ref{Item::FuncRn::VectActionCont}, $X_1f,\dots,X_qf\in\Co^{\beta-1}_\loc(U)$. By the inductive hypothesis $f\in C^1_\loc\cap\Co^{\beta-1}_{X,\loc}(U)$ and $X_1f,\dots,X_qf\in\Co^{\beta-1}_{X,\loc}(U)$, so $f\in\Co^\beta_{X,\loc}(U)$.

Conversely, suppose $f\in\Co^\beta_{X,\loc}(U)$, so $f\in C^1_\loc\cap\Co^{\beta-1}_{X,\loc}(U)$ and $X_1f,\dots,X_qf\in\Co^{\beta-1}_{X,\loc}(U)$. By the inductive hypothesis  $f,X_1f,\dots,X_qf\in\Co^{\beta-1}_{\loc}(U)$. Using \eqref{Eqn::FuncRevisVF::ZygVF=ZygPF::CoeffforInvMat} and Lemma \ref{Lemma::FuncSpace::Product} we have $\partial_jf=\sum_{i=1}^q\sum_{k=1}^nb_j^ia_i^k\partial_kf=\sum_{i=1}^qb_j^iX_if\in\Co^{\beta-1}_\loc(U)$ for $j=1,\dots,n$. By Lemma \ref{Lemma::FuncRevisVF::GradCharofZyg} we obtain $f\in\Co^\beta_\loc(U)$.

\medskip
\noindent\ref{Item::FuncVF::ZygVF=Zyg::VF}: Here $\alpha>\frac12$ and Definition \ref{Defn::FuncVF::VFregularity} imply that $\Co^\beta_{X,\loc}(U;TU)=\Co^\beta_\loc(U;TU)$ holds for $\beta\in(-\alpha,\frac12]$.

{For $\beta\in(-\frac12,\alpha]$ }we prove $\Co^\beta_{X,\loc}(U;TU)=\Co^\beta_\loc(U;TU)$ 
by induction on $r=\lceil\beta+\frac12\rceil$. The base case $r=1$. {which is  $\beta\in(-\frac12,\frac12]$,} was established above. {Let $r\geq 2$} and suppose we {have the case $\lceil\beta+\frac12\rceil=r-1$}, i.e. $\Co^\beta_{X,\loc}(U;TU)=\Co^\beta_\loc(U;TU)$ for all $\beta\in(r-\frac32,r-\frac12]$, we wish to prove it for $\beta\in(r-\frac12,\min(r+\frac12,\alpha)]$. 

Suppose $Y\in\Co^\beta_\loc(U;TU)$, so $Y\in\Co^\frac12_\loc\cap\Co^{\beta-1}_\loc(U;TU)$ and by Proposition \ref{Prop::FuncRn::LieOnManiofold} \ref{Item::FuncRn::CommutatorCont}, $[X_1,Y],\dots,[X_q,Y]\in\Co^{\beta-1}_\loc(U;TU)$.  By the inductive hypothesis $Y\in \Co^\frac12_\loc\cap\Co^{\beta-1}_{X,\loc}(U;TU)$ and $[X_1,Y],\dots,[X_q,Y]\in\Co^{\beta-1}_{X,\loc}(U;TU)$, which is the definition of $Y\in\Co^\beta_{X,\loc}(U;TU)$ (see Definition \ref{Defn::FuncVF::VFregularity}).

Conversely, suppose $Y\in\Co^\beta_{X,\loc}(U;TU)$, so $Y\in\Co^\frac12_\loc\cap\Co^{\beta-1}_{X,\loc}(U;TU)$ and $[X_1,Y],\dots,[X_q,Y]\in\Co^{\beta-1}_{X,\loc}(U;TU)$. By the inductive hypothesis $Y\in\Co^{\max(\frac12,\beta-1)}_\loc(U;TU)$ and $[X_1,Y],\dots,[X_q,Y]\in\Co^{\beta-1}_{\loc}(U;TU)$.

Write $Y=\sum_{j=1}^n\rho^j\Coorvec{x^j}$ where $\rho^j\in\Co^{\min(\frac12,\beta-1)}_\loc(U)$. By Lemma \ref{Lemma::FuncSpace::Product} we know the equation below is defined.
\begin{equation}\label{Eqn::FuncRevisVF::ZygVF=ZygPF::LieforVF}
    [X_i,Y]=\sum_{j,k=1}^n\mleft(a_i^j\frac{\partial\rho^k}{\partial x^j}-\rho^j\frac{\partial a_i^k}{\partial x^j}\mright)\Coorvec{x^k}=\sum_{k=1}^n(X_i\rho^k)\Coorvec{x^k}+L_i(Y),\quad 1\le i\le q.
\end{equation}
Here $L_1,\dots,L_q$ are multiplication operators (0-th order differential operators) with $\Co^{\alpha-1}_\loc$-coefficients.

By Lemma \ref{Lemma::FuncSpace::Product}, since $\max(\frac12,\beta-1)>1-\alpha$, we have $L_i(Y)\in\Co^{\beta-1}_\loc(U;TU)$, so $\rho^k\in\Co^{\beta-1}_\loc(U)$ and $X_i\rho^k\in\Co^{\beta-1}_\loc(U)$ for $1\le i\le q$, $1\le k\le n$. {Using \eqref{Eqn::FuncRevisVF::ZygVF=ZygPF::CoeffforInvMat} and Lemma \ref{Lemma::FuncSpace::Product} we have $\partial_j\rho^k=\sum_{i=1}^qb_j^i\cdot(X_i\rho^k)\in\Co^{\beta-1}_\loc(U)$, for $j=1,\dots,n$. By Lemma \ref{Lemma::FuncRevisVF::GradCharofZyg} we obtain $\rho^k\in\Co^\beta_\loc(U)$  for all $1\le k\le n$, which means $Y\in\Co^\beta_\loc(U;TU)$. This finishes the induction argument and hence the proof of \ref{Item::FuncVF::ZygVF=Zyg::VF}.}


\medskip The proof of \ref{Item::FuncVF::ZygVF=Zyg::Forms} is the same as \ref{Item::FuncVF::ZygVF=Zyg::VF}. Indeed, similar to \eqref{Eqn::FuncRevisVF::ZygVF=ZygPF::LieforVF}, we can write
$$\Lie{X_i}\theta=\sum_{1\le j_1<\dots<j_k\le n}^n(X_i\theta_{j_1\dots j_k})dx^{j_1}\wedge\dots\wedge dx^{j_k}+L_i'(\theta),\quad1\le i\le q.$$
where $L_1',\dots,L_n'$ are product operators on $k$-forms with $\Co^{\alpha-1}_\loc$-coefficients. Using the same method {as for \ref{Item::FuncVF::ZygVF=Zyg::VF}} we can prove that $\theta_{j_1\dots j_k}\in\Co^\beta_{X,\loc}(U)$ {if and only if $\theta_{j_1\dots j_k},X_1\theta_{j_1\dots j_k},\dots,X_n\theta_{j_1\dots j_k}\in\Co^{\beta-1}_\loc(U)$} if and only if $\theta_{j_1\dots j_k}\in\Co^\beta_\loc(U)$. We omit the details.
\end{proof}

\begin{cor}\label{Cor::FuncVF::ZygVF=ZygCor} Let $\alpha\ge1$ and let $X_1,\dots,X_q$ be $\Co^\alpha_\loc$-vector fields on a $\Co^{\alpha+1}$-manifold $\Manifold$ that span the tangent space at every point. Let $\beta>1-\alpha$ and $1\le k\le n$.
\begin{enumerate}[parsep=-0.3ex,label=(\roman*)]
    \item\label{Item::FuncVF::ZygVF=Zyg::dForms} Let $\theta\in\Co^{(-\alpha)^+}_\loc\mleft(\Manifold;\mywedge^kT^*\Manifold\mright)$ be a $k$-form. Then $d\theta$ has regularity $\Co^{\beta-1}_{X,\loc}(\Manifold)$ if and only if $d\theta\in\Co^{\beta-1}_{X,\loc}\mleft(\Manifold;\mywedge^{k+1}T^*\Manifold\mright)$. 
    \item\label{Item::FuncVF::ZygVF=Zyg::Diff} Let $\theta\in\Co^{\beta}_{X,\loc}\mleft(\Manifold;\mywedge^kT^*\Manifold\mright)$, then $d\theta\in\Co^{\beta-1}_{X,\loc}\mleft(\Manifold;\mywedge^{k+1}T^*\Manifold\mright)$.
\end{enumerate}
\end{cor}
\begin{proof}{ By passing to a local coordinate chart it suffices to prove the results on an open subset $\Manifold=U\subseteq\R^n$.} 

\noindent\ref{Item::FuncVF::ZygVF=Zyg::dForms}: First we prove the case $\beta\in(1-\alpha,1]$. By Definition, \ref{Defn::FuncVf::RegularityofForms}, $d\theta$ having regularity $\Co^{\beta-1}_{X,\loc}(U)$ is equivalent as $d\theta$ having regularity $\Co^{\beta-1}_\loc(U)$ when $\beta\in(-\alpha,1]$. By Lemma \ref{Lemma::FuncRn::dwelldefined}, $d\theta$ having regularity $\Co^{\beta-1}_\loc(U)$ is equivalent as $d\theta\in\Co^{\beta-1}_\loc\mleft(U;\mywedge^{k+1}T^*U\mright)$ for $\beta\in(1-\alpha,\alpha+1]$. By Proposition \ref{Prop::FuncVF::ZygVF=Zyg} \ref{Item::FuncVF::ZygVF=Zyg::Forms}, $\Co^{\beta-1}_\loc\mleft(U;\mywedge^{k+1}T^*U\mright)=\Co^{\beta-1}_{X,\loc}\mleft(U;\mywedge^{k+1}T^*U\mright)$. So $d\theta$ has regularity $\Co^{\beta-1}_\loc(U)$ if and only if $d\theta\in\Co^{\beta-1}_\loc\mleft(U;\mywedge^{k+1}T^*U\mright)$. 

{For $\beta\in(0,\infty)$} we proceed by induction on $r=\lceil\beta\rceil$. The base case $r=1$, {which is the case $\beta\in(0,1]\subseteq(1-\alpha,1]$}, was established above. {Let $r\geq 2$ and} suppose we { have the case $\lceil\beta\rceil=r-1$, i.e.} \ref{Item::FuncVF::ZygVF=Zyg::dForms} holds for all $\beta\in(r-1,r]$ and we wish to prove it for $\beta\in(r,r+1]$. 

Suppose $d\theta\in\Co^{\beta-1}_{X,\loc}\mleft(U;\mywedge^{k+1}T^*U\mright)$. By Definition \ref{Defn::FuncVf::formregularity}, $d\theta\in\Co^{0^+}_\loc\cap\Co^{\beta-1}_{X,\loc}$ and $d\IntProd{X_i}d\theta=\Lie{X_i}d\theta\in\Co^{\beta-2}_{X,\loc}\mleft(U;\mywedge^{k+1}T^*U\mright)$ for all $i=1,\dots,q$. By the inductive hypothesis $d\theta,d\IntProd{X_1}d\theta,\dots,d\IntProd{X_q}d\theta$ have regularity $\Co^{\beta-2}_{X,\loc}(U)$, which is
the definition of 
$d\theta$ having regularity $\Co^{\beta-1}_{X,\loc}(U)$ by Definition \ref{Defn::FuncVf::formregularity}.

Conversely, suppose $d\theta$ has regularity $\Co^{\beta-1}_{X,\loc}(U)$, i.e. $d\theta\in\Co^{0^+}_\loc\mleft(U;\mywedge^{k+1}T^*U\mright)$ and $d\theta,d\IntProd{X_1}d\theta,\dots,d\IntProd{X_q}d\theta$ have regularity $\Co^{\beta-2}_{X,\loc}(U)$. Applying the inductive hypothesis to $\theta,\IntProd{X_1}d\theta,\dots,\IntProd{X_q}d\theta$ we have $d\theta,d\IntProd{X_1}d\theta,\dots,d\IntProd{X_q}d\theta\in\Co^{\beta-2}_{X,\loc}\mleft(U;\mywedge^{k+1}T^*U\mright)$. Note that $d\theta\in\Co^{0^+}_\loc$ implies $d\theta\in\Co^\eps_\loc$ for some $\eps>0$, and by assumption $X_i\in\Co^\alpha_\loc\subseteq\Co^1_\loc$, so by Proposition \ref{Prop::FuncRn::LieOnManiofold} \ref{Item::FuncRn::LieFormsCont}, $\Lie{X_i}d\theta$ is defined and $\Lie{X_i}d\theta=d\IntProd{X_i}d\theta\in\Co^{\beta-2}_{X,\loc}\mleft(U;\mywedge^{k+1}T^*U\mright)$. By Definition \ref{Defn::FuncVf::formregularity} we get $d\theta\in\Co^{\beta-1}_{X,\loc}\mleft(U;\mywedge^{k+1}T^*U\mright)$, finishing the induction argument.

\medskip\noindent\ref{Item::FuncVF::ZygVF=Zyg::Diff}: For the case $\beta\in(1-\alpha,1]$, by  Proposition \ref{Prop::FuncVF::ZygVF=Zyg} \ref{Item::FuncVF::ZygVF=Zyg::Forms}, since $\alpha>1$, we have $\Co^{\beta}_{X,\loc}\mleft(U;\mywedge^{k}T^*U\mright)=\Co^{\beta}_{\loc}(U;\mywedge^{k}T^*U)$ and $\Co^{\beta-1}_{X,\loc}\mleft(U;\mywedge^{k+1}T^*U\mright)=\Co^{\beta-1}_{\loc}\mleft(U;\mywedge^{k+1}T^*U\mright)$. Therefore $\theta\in\Co^\beta_{X,\loc}$ implies $d\theta\in\Co^{\beta-1}_{X,\loc}$.

{For $\beta\in(0,\infty)$} we proceed the induction on $r=\lceil\beta\rceil$. The base case $r=1$, {i.e. $\beta\in(0,1]$,} was established above. 
{Let $r\ge2$ and} suppose we { have the case $\lceil\beta\rceil=r-1$, i.e.} $\beta\in(r-1,r]$, we wish to prove the case $\beta\in(r,r+1]$. 

Let $\theta\in\Co^{\beta}_{X,\loc}\mleft(U;\mywedge^kT^*U\mright)$, so by assumption $\theta,\Lie{X_1}\theta,\dots,\Lie{X_q}\theta\in\Co^{\beta-1}_{X,\loc}\mleft(U;\mywedge^kT^*U\mright)$. By the inductive hypothesis $d\theta,d\Lie{X_1}\theta,\dots,d\Lie{X_q}\theta\in\Co^{\beta-2}_{X,\loc}\mleft(U;\mywedge^{k+1}T^*U\mright)$. Since $\beta>1$ and $d\theta\in\Co^{\beta-1}_{X,\loc}\subsetneq\Co^{0^+}_\loc$, by Proposition \ref{Prop::FuncRn::LieOnManiofold} \ref{Item::FuncRn::LieFormsCont}, $\Lie{X_i}d\theta$ is defined, and therefore $\Lie{X_i}d\theta=d\Lie{X_i}\theta\in\Co^{\beta-2}_{X,\loc}$. So by Proposition \ref{Prop::FuncVF::ZygVF=Zyg} \ref{Item::FuncVF::ZygVF=Zyg::Forms} if $\beta\le\frac32$, or by Definition \ref{Defn::FuncVf::formregularity} if $\beta\ge\frac32$, we get $d\theta\in\Co^{\beta-1}_{X,\loc}\mleft(U;\mywedge^{k+1}T^*U\mright)$.
\end{proof}

\begin{proof}[Proof of Proposition \ref{Prop::FuncVF::ZygVF=Zyg} \ref{Item::FuncVF::ZygVF=Zyg::dFormHasReg}]{ By passing to a local coordinate chart it suffices to prove the result on an open subset $\Manifold=U\subseteq\R^n$.} 

By Corollary \ref{Cor::FuncVF::ZygVF=ZygCor} \ref{Item::FuncVF::ZygVF=Zyg::dForms}, $d\theta$ has regularity $\Co^{\beta-1}_{X,\loc}(U)$ if and only if $d\theta\in\Co^\beta_{X,\loc}\mleft(U;\mywedge^{k+1}T^*U\mright)$. By Proposition \ref{Prop::FuncVF::ZygVF=Zyg}  \ref{Item::FuncVF::ZygVF=Zyg::Forms}, since $\beta-1\in(-\alpha,\alpha]$, $\Co^{\beta-1}_{X,\loc}\mleft(U;\mywedge^{k+1}T^*U\mright)=\Co^{\beta-1}_{\loc}\mleft(U;\mywedge^{k+1}T^*U\mright)$. So $d\theta$ has regularity $\Co^{\beta-1}_{X,\loc}(U)$ if and only if $d\theta\in\Co^{\beta-1}_{\loc}\mleft(U;\mywedge^{k+1}T^*U\mright)$.
\end{proof}
\begin{rmk}
Proposition \ref{Prop::FuncVF::ZygVF=Zyg} \ref{Item::FuncVF::ZygVF=Zyg::dFormHasReg} and Corollary \ref{Cor::FuncVF::ZygVF=ZygCor} \ref{Item::FuncVF::ZygVF=Zyg::dForms} and \ref{Item::FuncVF::ZygVF=Zyg::Diff} also hold for general $\alpha>0$. However, the proof is more complicated and is not used in the proof of our main results. Because of this,
we do not include these more general results.
\end{rmk}

\section{The Proof of the Main Theorem}

In this section, we prove Theorem \ref{Thm::TheMainResult}. 

\begin{lemma}\label{Lemma::ProofThm::ALemma}
Let $\alpha,\beta,\gamma>\frac12$. Let $X_1,\dots,X_q$ be $\Co^\alpha_\loc$-vector fields on a manifold $\Manifold$ that span the tangent space at every point, and let $\Phi:\Nanifold\to\Manifold$ be a $\Co^{\gamma+1}$-diffeomorphism.

Let $f\in\Co^\beta_{X,\loc}(\Manifold)$, $Y\in\Co^\beta_{X,\loc}(\Manifold;T\Manifold)$ and $\theta\in\Co^\beta_{X,\loc}(\Manifold;T^*\Manifold)$, then $f\theta\in\Co^\beta_{X,\loc}(\Manifold;T^*\Manifold)$, $\langle \theta,Y\rangle\in \Co^\beta_{X,\loc}(\Manifold)$, and $\langle\Phi^*\theta,\Phi^*Y\rangle\in\Co^\beta_{\Phi^*X,\loc}(\Nanifold)$.
\end{lemma}
\begin{proof}

First we prove $\langle \theta,Y\rangle\in \Co^\beta_{X,\loc}(\Manifold)$. The argument to show $f\theta\in\Co^\beta_{X,\loc}(\Manifold)$ is similar and we omit the details.

Consider first $\beta\in(\frac12,\frac32]$. By assumption (see Definitions \ref{Defn::FuncVF::VFregularity} and \ref{Defn::FuncVf::formregularity}) we see that $Y,\theta\in\Co^\frac12_\loc$ and $\Lie{X_1}Y,\Lie{X_1}\theta,\dots,\Lie{X_q}Y,\Lie{X_q}\theta\in\Co^{\beta-1}_\loc$. So by Proposition \ref{Prop::FuncRn::LieOnManiofold} \ref{Item::FuncRn::IntProdCont}, $\langle \theta,Y\rangle=\IntProd{Y}\theta\in\Co^\frac12_\loc(\Manifold)$ satisfies
\begin{equation*}
    X_i\langle \theta,Y\rangle=\langle \theta,\Lie{X_i}Y\rangle+\langle \Lie{X_i}\theta,Y\rangle\in\Co^{\beta-1}_\loc(\Manifold),\quad\text{for all }i=1,\dots,q.
\end{equation*}
By Proposition \ref{Prop::FuncVF::ZygVF=Zyg} \ref{Item::FuncVF::ZygVF=Zyg::Fun}, since $\beta\le\alpha+1$, we get $\langle \theta,Y\rangle\in\Co^\beta_\loc(\Manifold)=\Co^\beta_{X,\loc}(\Manifold)$.

{For $\beta\in(\frac12,\infty)$, we proceed by induction on $r=\lceil \beta+\frac12\rceil$.} The base case, $r=2$, is established above. Let $r\ge3$ and suppose {the result is true for $\lceil\beta+\frac12\rceil=r-1$, i.e. for} $\beta\in(r-\frac32,r-\frac12]$, we wish to prove it for $\beta\in(r-\frac12,r+\frac12]$. 

By assumption $Y,\theta\in\Co^{\beta-1}_{X,\loc}$ and $Y,\theta,\Lie{X_1}Y,\Lie{X_1}\theta,\dots,\Lie{X_q}Y,\Lie{X_q}\theta\in\Co^{\beta-1}_{X,\loc}$. By the inductive hypothesis $\langle \theta,Y\rangle,\langle\theta,\Lie{X_i}Y\rangle,\langle\Lie{X_i}\theta,Y\rangle\in\Co^{\beta-1}_{X,\loc}(\Manifold)$ for all $i=1,\dots, q$. So $X_i\langle \theta,Y\rangle=\langle\theta,\Lie{X_i}Y\rangle+\langle\Lie{X_i}\theta,Y\rangle\in\Co^{\beta-1}_{X,\loc}(\Manifold)$ for all $i=1,\dots,q$ as well. By Definition \ref{Defn::FuncVF::Funregularity} we get $\langle \theta,Y\rangle\in\Co^\beta_{X,\loc}(\Manifold)$. This completes the induction argument.

Finally, we show $\langle\Phi^*\theta,\Phi^*Y\rangle\in\Co^\beta_{\Phi^*X,\loc}(\Nanifold)$. Note that $\langle \theta,Y\rangle=\IntProd{Y}\theta$. So by Lemma \ref{Lemma::FuncRn::PullbackComm} \ref{Item::FuncRn::IntProdComm} and Remark \ref{Rmk::FuncVF::ZygVFUnderDiffeo} \ref{Item::FuncVF::ZygVFUnderDiffeo::Fun} we get $\langle\Phi^*\theta,\Phi^*Y\rangle=\IntProd{\Phi^*Y}\Phi^*\theta=\Phi^*(\IntProd{Y}\theta)=\Phi^*\langle \theta,Y\rangle$ and their common value is in $\Co^\beta_{\Phi^*X,\loc}(\Nanifold)$.
\end{proof}

\begin{lemma}\label{Lemma::ProofThm::BLemma}
Let $\alpha>\frac12$ and $\beta\in[\alpha,\alpha+1]$. Assume that $\lambda^1,\dots,\lambda^n$ form a $\Co^\alpha$-local basis for the cotangent bundle of a manifold $\Manifold$. Let $X_1,\dots,X_n$ be its dual basis. 

Suppose $\lambda^1,\dots,\lambda^n\in\Co^\beta_{X,\loc}(\Manifold;T^*\Manifold)$ then $d\lambda^1,\dots,d\lambda^n\in\Co^{\beta-1}_{\loc}\mleft(\Manifold;\mywedge^2T^*\Manifold\mright)$.
\end{lemma}

\begin{proof}The inverse of a $\Co^\alpha_\loc$-matrix is a $\Co^\alpha_\loc$-matrix, so as the dual basis of $(\lambda^1,\dots,\lambda^n)$, we know $X_1,\dots,X_n$ are $\Co^\alpha_\loc$-vector fields on $\Manifold$.

Write $c_{ij}^k:=\langle \Lie{X_i}\lambda^k,X_j\rangle$ for $1\le i,j,k\le n$. Since $X_1,\dots,X_n$ and $\lambda^1,\dots,\lambda^n$ are dual bases, we have
\begin{equation}\label{Eqn::ProofThm::BLemma::Tmp}
    d\lambda^k(X_i,X_j)=X_i\langle\lambda^k,X_j\rangle-\langle\lambda^k,[X_i,X_j]\rangle-X_j\langle\lambda^k,X_i\rangle=\langle\Lie{X_i}\lambda^k,X_j\rangle-0=c_{ij}^k\ \Rightarrow\ d\lambda^k=\sum_{1\le i<j\le n}c_{ij}^k\lambda^i\wedge\lambda^j.
\end{equation}
Note that the products in \eqref{Eqn::ProofThm::BLemma::Tmp} are all defined due to Lemma \ref{Lemma::FuncSpace::Product}.


By Definition \ref{Defn::FuncVf::formregularity}, $\Lie{X_i}\lambda^k\in\Co^{\beta-1}_{X,\loc}(\Manifold;T^*\Manifold)$. By Proposition \ref{Prop::FuncVF::ZygVF=Zyg} \ref{Item::FuncVF::ZygVF=Zyg::Forms}, {since $\beta-1\le\alpha$, we have} $\Co^{\beta-1}_{X,\loc}(\Manifold;T^*\Manifold)=\Co^{\beta-1}_\loc(\Manifold;T^*\Manifold)$, so $\Lie{X_i}\lambda^k\in\Co^{\beta-1}_{\loc}(\Manifold;T^*\Manifold)$.
Since $X_1,\dots,X_n\in\Co^\alpha_\loc$, by Proposition \ref{Prop::FuncRn::LieOnManiofold} \ref{Item::FuncRn::IntProdCont}, $\langle \Lie{X_i}\lambda^k,X_j\rangle\in\Co^{\beta-1}_{\loc}(\Manifold) $; i.e. $c_{ij}^k\in\Co^{\beta-1}_{\loc}(\Manifold)$. By Proposition \ref{Prop::FuncRn::LieOnManiofold} \ref{Item::FuncRn::WedgeProdCont}, {since $\beta-1+\alpha>0$, we know} $\sum_{1\le i<j\le n}c_{ij}^k\lambda^i\wedge\lambda^j\in\Co^{\beta-1}_\loc\mleft(\Manifold;\mywedge^2T^*\Manifold\mright)$, so by \eqref{Eqn::ProofThm::BLemma::Tmp}, we conclude that $d\lambda^k\in\Co^{\beta-1}_\loc\mleft(\Manifold;\mywedge^2T^*\Manifold\mright)$ for $1\le k\le n$.
\end{proof}
\begin{rmk}
Lemma \ref{Lemma::ProofThm::BLemma} is similar to Corollary \ref{Cor::FuncVF::ZygVF=ZygCor} \ref{Item::FuncVF::ZygVF=Zyg::Diff}. However, it does not follow from Corollary \ref{Cor::FuncVF::ZygVF=ZygCor} \ref{Item::FuncVF::ZygVF=Zyg::Diff} as we require $\alpha\ge1$ in Corollary \ref{Cor::FuncVF::ZygVF=ZygCor} \ref{Item::FuncVF::ZygVF=Zyg::Diff}.
\end{rmk}

\begin{proof}[Proof of Theorem \ref{Thm::TheMainResult}]
{The case $\beta\le\alpha$, for each condition is trivial.
Indeed, fix $p\in\Manifold$, take any $\Co^{\alpha+1}$-coordinate chart $\Phi^{-1}:U\subseteq\Manifold\xrightarrow{\sim}\Ball^n$ near $p$.  We have $\Phi^*X_1,\dots,\Phi^*X_q\in\Co^\alpha_\loc(\Ball^n;T\Ball^n)\subseteq\Co^\beta_\loc(\Ball^n;T\Ball^n)$ and $\Phi^*\lambda^1,\dots,\Phi^*\lambda^n\in\Co^\alpha_\loc(\Ball^n;T^*\Ball^n)$. By shrinking the domain and  scaling, we can make replace $\Co^\beta_\loc$ by $\Co^\beta$.  In other words, we have $\Phi^*X_1,\dots,\Phi^*X_q\in\Co^\beta(\Ball^n;T\Ball^n)$ and $\Phi^*\lambda^1,\dots,\Phi^*\lambda^n\in\Co^\alpha(\Ball^n;T^*\Ball^n)$. Therefore, \ref{Item::TheMainResult::ExistChart} and  \ref{Item::TheMainResult::dForm} are automatically satisfied for $\alpha>0$, and the same for \ref{Item::TheMainResult::1Form} and \ref{Item::TheMainResult::VF} when $\alpha>\frac12$.
For the remainder of the proof, we assume $\beta>\alpha$.


\medskip}

We will show \ref{Item::TheMainResult::ExistChart} $\Leftrightarrow$ \ref{Item::TheMainResult::dForm} for arbitrary $\alpha>0$, and  \ref{Item::TheMainResult::ExistChart} $\Rightarrow$ \ref{Item::TheMainResult::VF} $\Rightarrow$ \ref{Item::TheMainResult::1Form} $\Rightarrow$ \ref{Item::TheMainResult::dForm} when $\alpha>\frac12$.

We first prove that \textbf{\ref{Item::TheMainResult::ExistChart}$\Rightarrow$\ref{Item::TheMainResult::dForm}} for arbitrary $\alpha>0$.

Suppose \ref{Item::TheMainResult::ExistChart} holds, and thus there exists a neighborhood $U\subseteq\Manifold$ of $p$ and a diffeomorphism $\Phi:\Ball^n\xrightarrow{\sim}U\subseteq\Manifold$ as in \ref{Item::TheMainResult::ExistChart}. Since by assumption $\Phi^*X_1,\dots,\Phi^*X_n\in\Co^\beta(\Ball^n;T\Ball^n)$ and the inverse of $\Co^\beta_\loc$-matrix is still $\Co^\beta_\loc$-matrix, we see that the dual basis $\Phi^*\lambda^1,\dots,\Phi^*\lambda^n\in\Co^\beta_\loc(\Ball^n;T^*\Ball^n)$. So $d\Phi^*\lambda^i\in\Co^{\beta-1}_\loc$ and $\langle\Phi^*\lambda^i,\Phi^*X_j\rangle\in\Co^\beta_\loc$ for $1\le i\le n$, $1\le j\le n$.

By Proposition \ref{Prop::FuncVF::ZygVF=Zyg} \ref{Item::FuncVF::ZygVF=Zyg::Fun}, since $\beta<\beta+1$, we have $\Co^\beta_\loc(\Ball^n)=\Co^\beta_{\Phi^*X,\loc}(\Ball^n)$. Thus, $\langle\Phi^*\lambda^i,\Phi^*X_j\rangle\in\Co^\beta_{\Phi^*X,\loc}(\Ball^n)$ for $1\le i\le n$, $1\le j\le q$. 

By Definition \ref{Defn::FuncRn::dthetaisZyg}, $d\Phi^*\lambda^i\in\Co^{\beta-1}_\loc(\Ball^n;T^*\Ball^n)$ is the same as $d\Phi^*\lambda^i$ having regularity $\Co^{\beta-1}_\loc(\Ball^n)$. By Proposition \ref{Prop::FuncVF::ZygVF=Zyg} \ref{Item::FuncVF::ZygVF=Zyg::dFormHasReg}, if $\beta\ge1$ (where we use $\alpha=\beta$ and $\Phi^*X\in\Co^\beta_\loc$ in the proposition) or by Definition \ref{Defn::FuncVf::RegularityofForms} if $\beta\le1$ (where $\beta-1\le0$), $d\Phi^*\lambda^i$ having regularity $\Co^{\beta-1}_\loc(\Ball^n)$ is the same as  $d\Phi^*\lambda^i$ having regularity $\Co^{\beta-1}_{\Phi^*X,\loc}(\Ball^n)$.

Now $d\Phi^*\lambda^i$ has regularity $\Co^{\beta-1}_{\Phi^*X,\loc}(\Ball^n)$ and $\langle\Phi^*\lambda^i,\Phi^*X_j\rangle\in\Co^\beta_{\Phi^*X,\loc}(\Ball^n)$ for $1\le i\le n$, $1\le j\le q$. Note that $\Phi_*=(\Phi^{-1})^*$. By Remark \ref{Rmk::FuncVF::ZygVFUnderDiffeo} \ref{Item::FuncVF::ZygVFUnderDiffeo::dForms}, $d\lambda^i=d\Phi_*\Phi^*\lambda^i$ has regularity $\Co^{\beta-1}_{\Phi_*\Phi^*X,\loc}(U)=\Co^{\beta-1}_{X,\loc}(U)$ for $1\le i\le n$. And by Lemma \ref{Lemma::ProofThm::ALemma}, $\langle\lambda^i,X_j\rangle=\langle\Phi_*\Phi^*\lambda^i,\Phi_*\Phi^*X_j\rangle\in\Co^\beta_{\Phi_*\Phi^*X,\loc}(U)=\Co^\beta_{X,\loc}(U)$, finishing the direction \ref{Item::TheMainResult::ExistChart} $\Rightarrow$ \ref{Item::TheMainResult::dForm}.

Next we prove \textbf{\ref{Item::TheMainResult::ExistChart}$\Rightarrow$\ref{Item::TheMainResult::VF}} by assuming $\alpha>\frac12$. 

By assumption of \ref{Item::TheMainResult::ExistChart} $\Phi^*X_1,\dots,\Phi^*X_q\in\Co^\beta(\Ball^n;T\Ball^n)$. By Proposition \ref{Prop::FuncVF::ZygVF=Zyg} \ref{Item::FuncVF::ZygVF=Zyg::VF}, $\Co^\beta_\loc(\Ball^n;T\Ball^n)=\Co^\beta_{\Phi^*X,\loc}(\Ball^n;T\Ball^n)$, so $\Phi^*X_1,\dots,\Phi^*X_q\in\Co^\beta_{\Phi^*X,\loc}$. By Remark \ref{Rmk::FuncVF::ZygVFUnderDiffeo} \ref{Item::FuncVF::ZygVFUnderDiffeo::VF}, where we use $\Phi_*=(\Phi^{-1})^*$, we have $X_j=\Phi_*\Phi^*X_j\in\Co^\beta_{\Phi_*\Phi^*X}(U;TU)=\Co^\beta_{X}(U;TU)$ for $1\le j\le q$. Therefore, we get \ref{Item::TheMainResult::ExistChart} $\Rightarrow$ \ref{Item::TheMainResult::VF}.


\bigskip
We now  prove \textbf{\ref{Item::TheMainResult::dForm}$\Rightarrow$\ref{Item::TheMainResult::ExistChart}}
for all $\alpha>0$.  Fix $p\in\Manifold$. {We proceed by induction on $r=\lceil\beta\rceil$.}


{We start with the base case $r=1$, i.e. $\beta\in(0,1]$. And as mentioned in the beginning of the proof} we  assume $\beta>\alpha$
and therefore $\alpha<1$. 
By Definition \ref{Defn::FuncVf::RegularityofForms},  \ref{Item::TheMainResult::dForm} is equivalent to $d\lambda^1,\dots,d\lambda^n$ having regularity $\Co^{\beta-1}_\loc(\Manifold)$. By Corollary \ref{Cor::Key::BasicCaseforMainThm} there is a $\Co^{\alpha+1}$-diffeomorphism $\Phi:\Ball^n\xrightarrow{\sim}U\subseteq\Manifold$ such that $\Phi(0)=p$ and $\Phi^*X_1,\dots,\Phi^*X_n\in\Co^\beta(\Ball^n;T\Ball^n)$. 

We let $X^0$ denote the sub-collection of $X$ given by $$X^0=(X_1,\dots,X_n).$$
So $\Co^\beta_{X,\loc}(U)\subseteq \Co^\beta_{X^0,\loc}(U)$.
By assumption $\langle \lambda^i,X_j\rangle\in\Co^\beta_{X,\loc}(U)\subseteq\Co^\beta_{X^0,\loc}(U)$ for all $1\le i\le n$ and $1\le j\le q$. Applying Lemma \ref{Lemma::ProofThm::ALemma} to $\langle \lambda^i,X_j\rangle$ with $X=(X_1,\dots,X_q)$ in that lemma replaced by $X^0=(X_1,\dots,X_n)$, we get $\langle\Phi^* \lambda^i,\Phi^*X_j\rangle\in\Co^\beta_{\Phi^*X^0,\loc}(\Ball^n)$. By Definition \ref{Defn::FuncVF::Funregularity}, we have $\Co^\beta_{\Phi^*X^0,\loc}(\Ball^n)=\Co^\beta_\loc(\Ball^n)$ since $\beta\le 1$. Therefore $\langle\Phi^* \lambda^i,\Phi^*X_j\rangle\in\Co^\beta_{\loc}(\Ball^n)$ and thus
\begin{equation}\label{Eqn::ProofThm::PhiX}
    \Phi^*X_j=\sum_{i=1}^n\langle\Phi^*\lambda^i,\Phi^*X_j\rangle \Phi^*X_i\in\Co^\beta_\loc(\Ball^n;T\Ball^n),\quad n+1\le j\le q.
\end{equation}

{Now let $r\ge1$ and suppose \ref{Item::TheMainResult::dForm} $\Rightarrow$ \ref{Item::TheMainResult::ExistChart} is true for the case $\lceil\beta\rceil\le r$, i.e. for} $\beta\in(0,r]$, we wish to prove the case $\beta\in(r,r+1]$. 

By the inductive hypothesis, there is a neighborhood $U_0\subseteq\Manifold$ of $p$ and a $\Co^{r+1}$-diffeomorphism $\Phi_0:\Ball^n\xrightarrow{\sim}U_0\subseteq\Manifold$ such that $\Phi_0(0)=p$ and $\Phi_0^*X_1,\dots,\Phi_0^*X_q\in\Co^r(\Ball^n;T\Ball^n)$. Since the inverse of $\Co^r_\loc$-matrix is still a $\Co^r_\loc$-matrix and $\Phi_0^*X_1,\dots,\Phi_0^*X_n\in\Co^r_\loc$, for the dual basis we have $\Phi_0^*\lambda^1,\dots,\Phi_0^*\lambda^n\in\Co^r_\loc(\Ball^n;T^*\Ball^n)$. 

By assumption, $d\lambda^i$ has regularity $\Co^{\beta-1}_{X,\loc}(U_0)$, so by Remark \ref{Rmk::FuncVF::ZygVFUnderDiffeo} \ref{Item::FuncVF::ZygVFUnderDiffeo::dForms}, $d\Phi_0^*\lambda^i$ has regularity $\Co^{\beta-1}_{\Phi_0^*X,\loc}(\Ball^n)$ for $i=1,\dots,n$. 
By Proposition \ref{Prop::FuncVF::ZygVF=Zyg} \ref{Item::FuncVF::ZygVF=Zyg::dFormHasReg}, since $\Phi_0^*X\in\Co^r_\loc$ and  $\beta\le r+1$, we know that $d\Phi_0^*\lambda^1,\dots,d\Phi_0^*\lambda^n$ have regularity $\Co^{\beta-1}_\loc(\Ball^n)$.
By Definition \ref{Defn::FuncRn::dthetaisZyg}, this is to say $d\Phi_0^*\lambda^1,\dots,d\Phi_0^*\lambda^n\in\Co^{\beta-1}_\loc\mleft(\Ball^n;\mywedge^2T^*\Ball^n\mright)$.

{Applying }Corollary \ref{Cor::Key::BasicCaseforMainThm} \ref{Item::Key::BasicCaseforMainThm::b} $\Rightarrow$ \ref{Item::Key::BasicCaseforMainThm::a}
to $d\Phi_0^*\lambda^1,\dots,d\Phi_0^*\lambda^n\in\Co^{\beta-1}_\loc\mleft(\Ball^n;\mywedge^2T^*\Ball^n\mright)$, since by assumption $\beta-1\in(r-1,r]$, we can find a map $\Phi_1:\Ball^n\to\Ball^n$ that is $\Co^{r+1}$-diffeomorphism onto its image, such that $\Phi_1(0)=0$ and $\Phi_1^*\Phi_0^*\lambda^1,\dots,\Phi_1^*\Phi_0^*\lambda^n\in\Co^\beta(\Ball^n;T^*\Ball^n)$. Since the inverse of $\Co^\beta_\loc$-matrix is still a $\Co^\beta_\loc$-matrix, see that  the dual basis $\Phi_1^*\Phi_0^*X_1,\dots,\Phi_1^*\Phi_0^*X_n\in\Co^\beta_\loc(\Ball^n;T\Ball^n)$. 

Take $\Phi=\Phi_0\circ\Phi_1$ and $U=\Phi(\Ball^n)\subseteq\Manifold$. So $\Phi(0)=\Phi_0(\Phi_1(0))=\Phi_0(0)=p$ and $\Phi^*X_1,\dots,\Phi^*X_n\in\Co^\beta_\loc(\Ball^n;T\Ball^n)$. We are going to prove $\Phi^*X_{n+1},\dots,\Phi^*X_q\in\Co^\beta_\loc(\Ball^n;T\Ball^n)$.

We still use $X^0=(X_1,\dots,X_n)$ as the sub-collection of $X=(X_1,\dots,X_q)$. By assumption \ref{Item::TheMainResult::dForm}, $\langle \lambda^i,X_j\rangle\in\Co^\beta_{X,\loc}(U)\subseteq \Co^\beta_{X^0,\loc}(U)$, so applying Lemma \ref{Lemma::ProofThm::ALemma} to $\Phi$ and $X^0$, we have $\langle\Phi^*\lambda^i,\Phi^*X_j\rangle\in\Co^\beta_{\Phi^*X^0,\loc}(\Ball^n)$ for all $1\le i\le n$, $1\le j\le q$. By Proposition \ref{Prop::FuncVF::ZygVF=Zyg} \ref{Item::FuncVF::ZygVF=Zyg::Fun}, since $\Phi^*X^0\in\Co^\beta_\loc$, we have $\langle\Phi^*\lambda^i,\Phi^*X_j\rangle\in\Co^\beta_{\loc}(\Ball^n)$ for all $1\le i\le n$, $1\le j\le q$. Therefore using \eqref{Eqn::ProofThm::PhiX} we get $\Phi^*X_{n+1},\dots,\Phi^*X_q\in\Co^\beta_\loc(\Ball^n;T\Ball^n)$.

We conclude $\Phi^*X_1,\dots,\Phi^*X_q\in\Co^\beta_\loc(\Ball^n;T\Ball^n)$. Replacing $\Phi(t)$ by $\Phi(\frac12t)$ for $t\in\Ball^n$, we can replace $\Co^\beta_\loc$ by $\Co^\beta$, finishing the induction argument and hence the proof of \textbf{\ref{Item::TheMainResult::dForm} $\Rightarrow$ \ref{Item::TheMainResult::ExistChart}}.

\bigskip
Now assume $\alpha>\frac12$, we will show \ref{Item::TheMainResult::VF} $\Rightarrow$ \ref{Item::TheMainResult::1Form} $\Rightarrow$ \ref{Item::TheMainResult::dForm}.  {As mentioned in the beginning of the proof} we assume $\beta>\alpha$. In particular, we assume $\beta>\frac12$.

In the following proof, we fix a neighborhood $U$ of $p\in\Manifold$ where $X_1,\dots,X_n$ form a $\Co^\alpha$-local basis on $U$.

\medskip\noindent
\textbf{\ref{Item::TheMainResult::VF}$\Rightarrow$\ref{Item::TheMainResult::1Form}}: By assumption $X_1,\dots,X_n\in\Co^\beta_{X,\loc}(U;TU)$, by Definition \ref{Defn::FuncVF::VFregularity} since $\beta>\frac12$, we have $X_1,\dots,X_n\in\Co^\frac12_{\loc}(U;TU)$. The inverse of $\Co^\frac12_\loc$-matrix is still a $\Co^\frac12_\loc$-matrix, so for the dual basis of $(X_1,\dots,X_n)$ we have $\lambda^1,\dots,\lambda^n\in\Co^\frac12_\loc(U;T^*U)$. 
To prove \ref{Item::TheMainResult::1Form}, by Definition \ref{Defn::FuncVf::formregularity}, it remains to show $\lambda^i,\Lie{X_k}\lambda^i\in\Co^{\beta-1}_{X,\loc}(U;T^*U)$ for $1\le i\le n$ and $1\le j\le q$.

For $1\le i,j\le n$, $\langle \lambda^i,X_j\rangle=\delta_j^i$ is a constant function, so $0=\Lie{X_k}\langle\lambda^i,X_j\rangle=\langle\Lie{X_k}\lambda^i,X_j\rangle+\langle\lambda^i,\Lie{X_k}X_j\rangle$. Therefore 
\begin{equation}\label{Eqn::ProofThm::Lie}
    \Lie{X_k}\lambda^i=\sum_{k=1}^n\langle\Lie{X_k}\lambda^i,X_j\rangle\lambda^j=-\sum_{k=1}^n\langle\lambda^i,\Lie{X_k}X_j\rangle\lambda^j,\quad\text{for }1\le i\le n,1\le k\le q.
\end{equation}
We prove \ref{Item::TheMainResult::VF}$\Rightarrow$\ref{Item::TheMainResult::1Form} by induction on $r=\lceil\beta+\frac12\rceil$, we work on the range $\beta\in(\frac12,\infty)$.

We start with the base case $r=2$ which is the case $\beta\in(\frac12,\frac32]$. By assumption \ref{Item::TheMainResult::VF} and Definition \ref{Defn::FuncVF::VFregularity}, since $\beta-1\le\frac12$, we have $\Lie{X_k}X_j\in\Co^{\beta-1}_\loc(U;TU)$ for $1\le j,k\le q$. Since $\lambda^1,\dots,\lambda^n\in\Co^\frac12_\loc(U;T^*U)$, by Proposition \ref{Prop::FuncRn::LieOnManiofold} \ref{Item::FuncRn::IntProdCont} and \ref{Item::FuncRn::WedgeProdCont} we know $\langle\lambda^i,\Lie{X_k}X_j\rangle\in\Co^{\beta-1}_\loc$ and $\langle\lambda^i,\Lie{X_k}X_j\rangle\lambda^j\in\Co^{\beta-1}_\loc$ for $1\le i,j\le n$, $1\le k\le q$. Using \eqref{Eqn::ProofThm::Lie} we get $\Lie{X_k}\lambda^i\in\Co^{\beta-1}_\loc(U;T^*U)$ for $1\le i\le n$, $1\le k\le q$. 

We have $\lambda^i\in\Co^\frac12_\loc(U;T^*U)$ and $\Lie{X_k}\lambda^i\in\Co^{\beta-1}_\loc(U;T^*U)$  for $1\le i\le n$, $1\le k\le q$, so by Definition \ref{Defn::FuncVf::formregularity}, $\lambda^1,\dots,\lambda^n\in\Co^\beta_{X,\loc}(U;T^*U)$. Note that by assumption \ref{Item::TheMainResult::VF}, $X_1,\dots,X_q\in\Co^\beta_{X,\loc}(U;TU)$, so by Lemma \ref{Lemma::ProofThm::ALemma}, $\langle\lambda^i,X_j\rangle\in\Co^\beta_{X,\loc}(U)$ for $1\le i\le n,1\le j\le q$.

Now let $r\ge3$ and suppose \ref{Item::TheMainResult::VF} $\Rightarrow$ \ref{Item::TheMainResult::1Form}  holds {for the case $\lceil\beta+\frac12\rceil=r-1$ i.e. for} $\beta\in(r-\frac32,r-\frac12]$, we wish to prove it for $\beta\in(r-\frac12,r+\frac12]$. 

By Definition \ref{Defn::FuncVF::VFregularity} $X_1,\dots,X_n\in\Co^\beta_{X,\loc}(U;TU)$ implies $\Lie{X_k}X_j\in\Co^{\beta-1}_{X,\loc}(U;TU)$ for $1\le j,k\le q$. By the inductive hypothesis we have $\lambda^1,\dots,\lambda^n\in\Co^{\beta-1}_{X,\loc}(U;T^*U)$. So by Lemma \ref{Lemma::ProofThm::ALemma} we get $\langle\lambda^i,\Lie{X_k}X_j\rangle\in\Co^{\beta-1}_{X,\loc}(U)$ and $\langle\lambda^i,\Lie{X_k}X_j\rangle\lambda^j\in\Co^{\beta-1}_{X,\loc}(U;T^*U)$. So by \eqref{Eqn::ProofThm::Lie} we get $\Lie{X_k}\lambda^i\in\Co^{\beta-1}_{X,\loc}(U;T^*U)$ for $1\le i\le n$, $1\le k\le q$. 

We have $\lambda^i,\Lie{X_k}\lambda^i\in\Co^{\beta-1}_{X,\loc}(U;T^*U)$ for $1\le i\le n$, $1\le k\le q$. By Definition \ref{Defn::FuncVf::formregularity} we get $\lambda^1,\dots,\lambda^n\in\Co^\beta_{X,\loc}(U;T^*U)$. Since $X_1,\dots,X_q\in\Co^\beta_\loc(U;TU)$, by Lemma \ref{Lemma::ProofThm::ALemma} again we get $\langle \lambda^i,X_j\rangle\in\Co^\beta_{X,\loc}(U)$ for $1\le i\le n$, $1\le j\le q$, finishing the induction argument and hence the proof of  \ref{Item::TheMainResult::VF} $\Rightarrow$ \ref{Item::TheMainResult::1Form}.

\bigskip
\noindent \textbf{\ref{Item::TheMainResult::1Form}$\Rightarrow$\ref{Item::TheMainResult::dForm}}: The result $\langle\lambda^i,X_j\rangle\in\Co^\beta_{X,\loc}(U)$ follows from the assumption. We need is to show $d\lambda^i$ has regularity $\Co^{\beta-1}_{X,\loc}(U)$, for $1\le i\le n$.

First we assume $\alpha\ge1$. By assumption \ref{Item::TheMainResult::1Form}, $\lambda^1,\dots,\lambda^n\in\Co^\beta_{X,\loc}(U;T^*U)$, so by Corollary \ref{Cor::FuncVF::ZygVF=ZygCor} \ref{Item::FuncVF::ZygVF=Zyg::Diff} we get $d\lambda^1,\dots,d\lambda^n\in\Co_{X,\loc}^{\beta-1}\mleft(U;\mywedge^2T^*U\mright)$. 

We next consider $\alpha\in(\frac12,1]$. First we assume $\beta\le 1$. Note that we only need to consider $\beta\in(\alpha,1]$ otherwise it is trivial.

We use $X^0=(X_1,\dots,X_n)$ as the sub-collection of $X=(X_1,\dots,X_q)$.
By assumption of the theorem, $X_1,\dots,X_n\in\Co^\alpha_\loc(U;TU)$. Since the inverse of $\Co^\alpha_\loc$-matrix is a $\Co^\alpha_\loc$-matrix, for the dual basis of $(X_1,\dots,X_n)$ we have $\lambda^1,\dots,\lambda^n\in\Co^\alpha_\loc(U;T^*U)$.

Note that $\Co^\beta_{X,\loc}(U;T^*U)\subseteq\Co^\beta_{X^0,\loc}(U;T^*U)$. Since $(X_1,\dots,X_n)$ and $(\lambda^1,\dots,\lambda^n)$ are dual bases, applying Lemma \ref{Lemma::ProofThm::BLemma} and using
that $\lambda^1,\dots,\lambda^n\in\Co^\beta_{X^0,\loc}(U;T^*U)$, we obtain  $d\lambda^1,\dots,d\lambda^n\in\Co^{\beta-1}_\loc\mleft(U;\mywedge^2T^*U\mright)$. By Definition \ref{Defn::FuncRn::dthetaisZyg} this is the same as $d\lambda^1,\dots,d\lambda^n$ having regularity $\Co^{\beta-1}_\loc(U)$. By Definition \ref{Defn::FuncVf::RegularityofForms}, since $\beta-1\le0$, this is the equivalent to $d\lambda^1,\dots,d\lambda^n$ having regularity $\Co^{\beta-1}_{X,\loc}(U)$.

When $\beta>1$, by the established case $\beta=1$ from above we know that $d\lambda^1,\dots,d\lambda^n$ have regularity $\Co^0_{X,\loc}(U)$. By assumption \ref{Item::TheMainResult::1Form}, 
$\langle\lambda^i,X_j\rangle\in\Co^\beta_{X,\loc}(U)\subseteq \Co^1_{X,\loc}(U)$, 
for $1\le i\le n$, $1\le j\le q$. Therefore, by the already proved implication \ref{Item::TheMainResult::dForm} $\Rightarrow$ \ref{Item::TheMainResult::ExistChart} we can find a $\Co^2$-atlas on $U$ such that $X_1\big|_U,\dots,X_q\big|_U$ are $\Co^1$ on this atlas. That is to say we can assume $\alpha=1$ in this case. 
Since we have already established the case $\alpha\ge 1$, we see that $d\lambda^1,\dots,d\lambda^n$ have regularity $\Co^\beta_{X,\loc}(U)$, completing the proof.
\end{proof}




\section{Harmonic Coordinates and Canonical Coordinates}\label{Section::DeTurck}
Given a non-smooth Riemannian metric $g$ on a manifold, $\Manifold$,
DeTurck and Kazdan showed that $g$ has optimal regularity
in harmonic coordinates \cite[Lemma 1.2]{DeTurckKazdan} (in
the Zygmund-H\"older sense),
but may not have optimal regularity in geodesic normal
coordinates \cite[Example 2.3]{DeTurckKazdan} (in fact, the regularity
of $g$ in geodesic normal coordinates may be two derivatives worse than
the regularity in harmonic coordinates).  


In this section, we present analogous results for vector fields.
 Let $X_1,\dots,X_n$ be $C^1_\loc$-vector fields on a $C^2$-manifold $\Manifold$ that form a local basis for the tangent space at every point. In Section \ref{Section::KeyThm} we defined a Riemannian metric $g=\sum_{i=1}^n\lambda^i\cdot\lambda^i$ where $(\lambda^1,\dots,\lambda^n)$ is the dual basis of $(X_1,\dots,X_n)$ (see Remark \ref{Rmk::Key::RmkRiemMetric}). 
 With respect to this metric, $X_1,\dots,X_n$ form an orthonormal basis at every point. Since $X_1,\dots,X_n\in C^1$, 
 we can talk about the metric Laplacian $\Lap_g$ with respect to $g$.

\begin{prop}\label{Prop::RmkCoor::HarmOptimal}
In harmonic coordinates with respect to $g$, $X_1,\dots,X_n$ have optimal regularity. 

More precisely, let $X_1,\dots,X_n$ and $g$ be as above, and let $\beta>1$. Suppose there is a $\Co^{\beta+1}$-atlas $\As$ which is compatible with the $C^2$-atlas of $\Manifold$, such that $X_1,\dots,X_n$ are $\Co^\beta$ on $\As$, let $\psi:U\subseteq\Manifold\to V\subseteq\R^n$ be a harmonic coordinate chart\footnote{Such a harmonic coordinate chart $\psi$ always exists locally when $\beta>1$, see also \cite[Lemma 1.2]{DeTurckKazdan}.}, then $\psi_*X_1,\dots,\psi_*X_n\in\Co^\beta_\loc(V;\R^n)$.
\end{prop}
\begin{proof}
It suffices to show that $\psi$ is a $\Co_{\loc}^{\beta+1}$-map with respect to $(\Manifold,\As)$. Once this is done, applying Lemma \ref{Lemma::FuncRevis::PushForwardFuncSpaces} \ref{Item::FuncRevis::PullBackVect} on $\Phi=\psi^{-1}$ and using that $X_1,\dots,X_n$ are $\Co^\beta_\loc$ with respect to $(\Manifold,\As)$, we get that $\psi_*X_1,\dots,\psi_*X_n\in\Co^\beta_\loc$.

Since the statement $\psi\in\Co_{\loc}^{\beta+1}(U;V)$ is local, 
we may without loss of generality, shrink $U$.
By doing so, the hypotheses of the proposition
imply that there is a $\Co^{\beta+1}$-coordinate chart $x=(x^1,\dots,x^n):U\subseteq\Manifold\to\R^n$ on $U$ (respect to $\As$). In this coordinate chart we can write $\Lap_g=-\frac1{\sqrt{\det g}}\sum_{i,j=1}^n\Coorvec{x^i}(\sqrt{\det g}g^{ij}\Coorvec{x^j})$ where $g^{ij}$ and $\sqrt{\det g}$ are as in \eqref{Eqn::Key::Riemannianmetric2}.

By assumption $\Lap_g \psi^k=0$ for $k=1,\dots,n$. Note that on $(x^1,\dots,x^n)$, $\Lap_g$ is a second order divergent form elliptic operator with $\Co^\beta$-coefficients. By a classical elliptic estimates (for example, \cite[Proposition 4.1]{TaylorPDE3}) we have that $\psi^k$ are all $\Co^{\beta+1}_\loc$ with respect to $\As$, completing the proof.
\end{proof}
\begin{rmk}


In fact the coordinate chart we construct in Proposition \ref{Prop::Key::ExistPDE} (also see \eqref{Eqn::Key::PDEforR}) is closely related to harmonic coordinates.
\end{rmk}
\begin{rmk}
While Proposition \ref{Prop::RmkCoor::HarmOptimal} shows that $X_1,\dots,X_n$ have optimal regularity with respect to harmonic coordinates, this fact along does not give a practical test for what the optimal regularity is. Theorem \ref{Thm::TheMainResult}, on the other hand, provides such a test.
\end{rmk}

\begin{rmk}
Proposition \ref{Prop::RmkCoor::HarmOptimal} shows that harmonic coordinates induces a $\Co^{\beta+1}$-atlas with respect to which $X_1,\dots,X_n$ are $\Co^\beta_\loc$. It is possible that the harmonic coordinates induces some $\Co^{\gamma+1}$-atlas for some $\gamma>\beta$ while $X_1,\dots,X_n$ are only $\Co^\beta$ with respect to this atlas; see Example \ref{Ex::CanonicalMighNotBeOptimal}.
\end{rmk}

\begin{example}\label{Ex::CanonicalMighNotBeOptimal}
Endow $\R^2$ with standard coordinates $(x,y)$, and let $\theta\in C^1(\R^2)$ be a function which is not smooth. Set $X:=\cos(\theta(x,y))\Coorvec x+\sin(\theta(x,y))\Coorvec y$ and $Y:=-\sin(\theta(x,y))\Coorvec x+\cos(\theta(x,y))\Coorvec x$. The corresponding metric is $g=dx^2+dy^2$ since $X,Y$ form an orthonormal basis with respect to the standard Euclidean metric, thus $\Lap_g=\Lap=-\partial_x^2-\partial_y^2$.

Therefore the singleton $\{(x,y):\R^2\to\R^2\}$ is an atlas of harmonic coordinates for $\R^2$, and since harmonic functions are real-analytic, we know the collection of harmonic coordinates with respect to $\Lap_g$ defines an real-analytic structure for $\R^2$ (which coincides with the standard real-analytic structure). Even though the differential structure induced by the harmonic coordinates is real analytic, $X$ and $Y$ cannot be smooth under any coordinate system (since they are not smooth with respect to these
harmonic coordinates).
\end{example}





As mentioned before, 
DeTurck and Kazdan showed that a Riemannian metric may not have optimal
regularity with respect to
geodesic normal
coordinates \cite[Example 2.3]{DeTurckKazdan}.
A natural analog of geodesic normal coordiantes for vector fields
are canonical coordinates (of the first kind).
Next, we show that vector fields may not have optimal regularity
with respect to these canonical coordinates.


Given $C^1$-vector fields $X_1,\dots,X_n$ on $\Manifold$ that form a basis on the tangent space at every point, the canonical coordinates at $p\in\Manifold$ is the map $\Phi_p(t^1,\dots,t^n):=e^{t^1X_1+\dots+t^nX_n}p$ defined via solving the ordinary differential equation, provided that it is solvable:
\begin{equation}\label{Eqn::RmkCoor::DefnCan}
    e^{t\cdot X}(p)=E(1),\quad{\text{where }}E:[0,1]\to\Manifold,\quad\frac d{dr} E(r)=r\mleft(t^1X_1(E(r))+\dots+t^nX_n(E(r))\mright),\quad E(0)=p.
\end{equation}
When $X_1,\dots,X_n\in\Co^\alpha$ for some $\alpha>1$, classical regularity theorems for ODEs show that $\Phi_p$ is at least $\Co^\alpha$. 
Therefore, $\Phi_p^{*}X_1,\ldots, \Phi_p^{*}X_n$ are at least $\Co^{\alpha-1}$; which is one derivative less than the original
regularity of $X_1,\ldots, X_n$.  The next result
shows that this loss of one derivative is sometimes inevitable.


\begin{lemma}\label{Lemma::CanonicalCoords}
Endow $\R^2$ with standard coordinate system $(x,y)$. Let $\alpha>1$ and let $X:=\partial_x$ and $Y:=xf(y)\partial_x+\partial_y$ where $f(y):=\alpha \max(0,y)^{\alpha-1}$.

Then we can find a  new $\Co^{\alpha+1}$ atlas $\As$ on $\R^2$ which is compatible with the standard $\Co^\alpha$-structure on $\R^2$, such that $X,Y$ are $\Co^\alpha_\loc$ with respect to this the new atlas, but for the canonical coordinates $\Phi(t,s):=e^{tX+sY}(0)$ we have $\Phi^*Y\notin\Co^{\alpha-1+\eps}$ near $(0,0)$, in particular the collection $(\Phi^*X,\Phi^*Y)$ is not $\Co^{\alpha-1+\eps}$.
\end{lemma}


Note that $X$ and $Y$ form a local basis of the tangent space at every point, and $f\in\Co^{\alpha-1}_\loc(\R)$.
\begin{proof}

First we show the existence of the new atlas $\As$ with respect to which $X$ and $Y$ are $\Co^\alpha$. In particular $X,Y$ are $C^1$ in $\As$, so \eqref{Eqn::RmkCoor::DefnCan} is uniquely solvable and hence $\Phi$ is well-defined.

Note that $[X,Y]=f(y)\partial_x\in\Co^{\alpha-1}(\R^2;\R^2)$. Specifically, the dual basis of $X,Y$ are 1-forms $\lambda=dx-xf(y)dy,\eta=dy$ which satisfy that $d\lambda=f(y)dx\wedge dy$ and $d\eta=0$ are both $\Co^{\alpha-1}$ 2-forms, so the condition \ref{Item::TheMainResult::dForm} in Theorem \ref{Thm::TheMainResult} is satisfied. By Theorem \ref{Thm::TheMainResult} we can find a $\Co^{\alpha+1}$-atlas $\As$ on $\R^2$ such that $X,Y$ are both $\Co^\alpha_\loc$ on $\As$.

Since $X_1,\ldots, X_n$
are $\Co^{\alpha}\subsetneq C^1$ with respect to $\As$, we see that $\Phi(t,s)$ is well-defined near $(t,s)=(0,0)$. We can compute $\Phi$ in terms of $f$. 

Clearly $\Phi(t,s)=(*,s)$ since $e^{s\partial_y}(x,y)=(x,y+s)$. We can write $\Phi(t,s)=(\phi(t,s),s)$. Define $\Phi(t,s;r)$ for $r\in\R$ as the solution to the ODE $\Coorvec r\Phi(t,s;r)=tX(\Phi(t,s;r))+sY(\Phi(t,s;r))$, $\Phi(t,s;0)=0$. So $\Phi(t,s)=\Phi(t,s;1)$ and we have $\Phi(t,s;r)=(\phi(t,s;r),rs)$ where
\begin{equation*}
    \Coorvec r\phi(t,s;r)=t+sf(rs)\phi(t,s;r),\quad\phi(t,s;0)=0,\quad t,s,r\in\R.
\end{equation*}
Solving this ODE we have
\begin{equation*}
    \phi(t,s;r)=e^{\int_0^r sf(\rho s)d\rho}\int_0^re^{-\int_0^\rho sf(\mu s)d\mu}td\rho,\quad\phi(t,s)=\phi(t,s;1)=t\frac{e^{-\int_0^sf(\rho)d\rho}}s\int_0^se^{\int_0^{\rho }f(\mu)d\mu}d\rho.
\end{equation*}

Now plug in \[f(y)=\begin{cases}\alpha y^{\alpha-1},&y\ge0,\\0&y\le0.\end{cases}\] We have
\begin{equation*}
    \phi(t,s)=t \frac {e^{-s^\alpha}}s\int_0^se^{\rho^\alpha} d\rho\quad\text{when } s>0;\quad \phi(t,s)=t\quad\text{when } s\le0.
\end{equation*}
Thus, $\phi(t,s)=tg(s)$ where 
\begin{equation}\label{Eqn::RmkCoor::g}
    g(s)=\begin{cases}\frac {e^{-s^\alpha}}s\int_0^se^{\rho^\alpha} d\rho&s>0\\1&s\le0.\end{cases}
\end{equation}

We are going to show $\Phi\in\Co^\alpha_\loc(\R^2;\R^2)$ and $\Phi\notin\Co^{\alpha+\eps}$ near $(0,0)$. To see this, it suffices to show $g\in\Co^\alpha_\loc(\R)$ and $g\notin\Co_{\loc}^{\alpha+\eps}$ near $0$, for every $\eps>0$.

By Taylor's expansion on the exponential function we have, when $s>0$,
\begin{align*}
    g(s)&=e^{-s^\alpha}\frac1s\int_0^se^{\rho^\alpha}d\rho=\sum_{j=0}^\infty\frac{(-1)^js^{j\alpha}}{j!}\frac1s\int_0^s\sum_{k=0}^\infty\frac{\rho^{k\alpha}}{k!}d\rho=\sum_{j,k=0}^\infty\frac{(-1)^j}{j!k!}\frac{s^{j\alpha}s^{k\alpha}}{k\alpha+1}
    \\&=\sum_{l=0}^\infty\Big(\sum_{k=0}^l\frac{(-1)^{l-k}}{k\alpha+1}\Big)s^{l\alpha}=1-\big(1-\tfrac1{\alpha+1}\big)s^\alpha+O(s^{2\alpha}).
\end{align*}
In other words
\begin{equation*}
    g(s)=\sum_{l=0}^\infty\Big(\sum_{k=0}^l\frac{(-1)^{l-k}}{k\alpha+1}\Big)\max(s,0)^{l\alpha}=1-\big(1-\tfrac1{\alpha+1}\big)\max(s,0)^\alpha+\sum_{l=2}^\infty\Big(\sum_{k=0}^l\frac{(-1)^{l-k}}{k\alpha+1}\Big)\max(s,0)^{l\alpha},\quad\forall s\in\R.
\end{equation*}

Note that for $\beta>0$ the function $\max(s,0)^\beta=\begin{cases}s^\beta&s\ge0\\0&s\le0\end{cases}$ is $\Co^\beta_\loc$ but not $\Co^{\beta+\eps}$ near $0$.
Indeed when $0<\beta<1$ we know that $\max(s,0)^\beta\in C^{0,\beta}_\loc=\Co^\beta_\loc$ and is not $C^{0,\beta+\eps}=\Co^{\beta+\eps}$ near $0$ for any $0<\eps<1-\beta$, when $\beta=1$ we know $\max(s,0)\in C^{0,1}_\loc\subsetneq\Co^1_\loc$ and is not $C^1$ near 0 so is not $\Co^{1+\eps}$ for any $\eps>0$. For $\beta>1$ by passing to its derivatives we see that $\max(s,0)^\beta\in\Co^\beta_\loc$ and is not $\Co^{\beta+\eps}$ near $0$.

So the remainder $\sum_{l=2}^\infty\big(\sum_{k=0}^l\frac{(-1)^{l-k}}{k\alpha+1}\big)\max(s,0)^{l\alpha}$ is $\Co^{2\alpha}\subset \Co^{\alpha+\eps}_\loc$ for all $0<\eps\le\alpha$, while the main term $1-\big(1-\tfrac1{\alpha+1}\big)\max(s,0)^\alpha$ is $\Co^\alpha_\loc$ but not $\Co^{\alpha+\eps}$ near 0. Therefore we conclude that $g\in\Co^\alpha(\R)$, but $g\notin\Co^{\alpha+\eps}$ near $s=0$.


Now we know $\Phi\in\Co^{\alpha}_\loc(\R^2;\R^2)$ but $\Phi\notin\Co^{\alpha+\eps}_\loc$ near $(t,s)=(0,0)$. Consider the inverse function of $\Phi$, 
and set $(u(x,y),v(x,y)):=\Phi^{-1}(x,y)$; so that $\Phi^*Y=(Yu,Yv)\circ\Phi$. We have $v(x,y)=y$ and $u(x,y)=\frac1{g(y)}x$. Note that $g(s)>0$ for every $s$, so $y\mapsto\frac1{g(y)}$ is not $\Co^{\alpha+\eps}$ near $y=0$, for any $\eps>0$. Therefore $Yv=x\cdot\partial_y\frac1{g(y)}$ is not $\Co^{\alpha+\eps-1}$ near $(x,y)=(0,0)$. By composing with $\Phi$ which is a $\Co^\alpha_\loc\subset\Co^{\alpha+\eps-1}_\loc$-diffeomorphism (for $0<\eps<1$), we see that $\Phi^*Y$ is not $\Co^{\alpha+\eps-1}$ near $(t,s)=(0,0)$, for every $\eps>0$. 
\end{proof}
\begin{rmk}
As a differentiable map $\Phi:(\R^2_{t,s},\mathrm{std})\to(\R^2,\As)$ between two $\Co^{\alpha+1}$-manifolds, we see that $\Phi$ is not $\Co^{\alpha+\eps}$ near $(0,0)$ for any $\eps>0$. Otherwise since $X$ and $Y$ are $\Co^\alpha_\loc$ on $\As$, by Lemma \ref{Lemma::FuncRevis::PushForwardFuncSpaces} \ref{Item::FuncRevis::PullBackVect} we have $\Phi^*Y\in\Co^{\min(\alpha,\alpha-1+\eps)}=\Co^{\alpha-1+\min(\eps,1)}$ near $(0,0)$, contradicting to Lemma \ref{Lemma::CanonicalCoords}.
\end{rmk}


\section{The  Quantitative Result}\label{Section::Quant}
Let $\Manifold$ be an $n$-dimensional manifold and let $\alpha>\frac12$. If we are given $\Co^\alpha_\loc$-vector fields $X_1,\dots,X_q$ that span the tangent space at every point, we can write $[X_i,X_j]=\sum_{k=1}^qc_{ij}^kX_k$ for some $\Co^{\alpha-1}_\loc$-functions $c_{ij}^k$. In Theorem \ref{Thm::TheMainResult} we show that, for $s_0>\alpha-1$ and near each point $p$ the following are equivalent
\begin{itemize}
    \item There exists a $\Co^{\alpha+1}$-parameterization $\Phi$ near $p$ such that $\Phi^*X_1,\dots,\Phi^*X_q$ are $\Co^{s_0+1}_\loc$.
    \item We may choose $c_{ij}^k$ with  $c_{ij}^k\in\Co^{s_0}_{X,\loc}$ near $p$.
\end{itemize}
By contrast the range of $s_0$ in \cite{StovallStreetII} is $s_0>1$.

If one traces through the proof of Theorem \ref{Thm::TheMainResult}, the size of the neighborhood of $p$ and the $\Co^{s_0+1}$-norms of $\Phi^*X_1,\dots,\Phi^*X_q$ depend on the $\Co^\alpha$-norms of $X_1,\dots,X_q$ under some fixed initial coordinate system near $p$, and on a lower
bound for \eqref{Eqn::Intro::QualVsQuant::LowerBoundDet} at $x=p$ (in some fixed initial coordinate
system).
However, in \cite{StovallStreetII}, when $X_1,\dots,X_q\in C^1$ and $s_0>1$, a similar coordinate system $\Phi$ was constructed, but where all of the  estimates depend only on the diffeomorphic invariant quantities  like the norms $\|c_{ij}^k\|_{\Co^{s_0}_X}$ (see \cite[Section 5.1]{StovallStreetI}).

Using the methods of this paper, we can extend the main results of
\cite{StovallStreetII} and \cite{StreetSubHermitian} (namely,
\cite[Theorem 2.14]{StovallStreetII} and \cite[Theorem 4.5]{StreetSubHermitian}) from $s_0>1$ to $s_0>0$.

\begin{thm}\label{Thm::Quant}
 \cite[Theorem 2.14]{StovallStreetII} and \cite[Theorem 4.5]{StreetSubHermitian} are still true with $s_0>1$ replaced by $s_0>0$, and leaving rest of the assumptions and statements unchanged.  
\end{thm}

In these papers,
the assumption $s_0>1$ is used in the following  places:
\begin{itemize}
    \item In \cite[Theorem 4.7]{StovallStreetI}, which is used to prove
    \cite[Proposition 4.1]{StovallStreetII}, 
    ``1-admissible constants'' are used in order to obtain the results (d), (e), and (f) of \cite[Theorem 4.7]{StovallStreetI}. The proof of \cite[Theorem 4.7 (d), (e), (f)]{StovallStreetI} is done in \cite[Proposition 9.22]{StovallStreetI}. 
    The 1-admissible constants are allowed to depend on quantities like $\| c_{ij}^k\|_{C^1_X}$.
    
    We are going to show that in \cite[Proposition 9.22]{StovallStreetI}, if we only need the conclusion that $\Phi$ is a $\Co^{s_0+1}$-diffeomorphism, then the assumption ``1-admissible constants'' can be replaced by ``$\{s_0\}$-admissible constants,''
    for a fixed $s_0>0$ (which is possibly $\le1$), where $\{s_0\}$-admissible constants are defined 
    in \cite[Definition 2.13]{StovallStreetII}. See Lemma \ref{Lemma::Quant::InjLemma}.
    
    \item In \cite[Proposition 6.8]{StovallStreetII}, the assumption $s_0>1$ is used in order to set up some well-defined elliptic PDEs. 
    
    In this paper we use different elliptic PDEs that are defined when $0<s_0\le1$, as illustrated in Section \ref{Section::KeyThm::Outline}. See Proposition \ref{Prop::Quant} for the precise statement to the modification of \cite[Proposition 6.8]{StovallStreetII}.
    
    \item In \cite[Theorem 2.14]{StovallStreetII}, a map $\Phi$
    is constructed, which depends on $s_0$, such that $\Phi^*X_1,\dots,\Phi^*X_q\in\Co^{s_0+1}$.
    Moreover, this map satisfies for $s\geq s_0$, that if
    $c_{ij}^k\in\Co^s_{X,\loc}$, then $\Phi^*X_1,\dots,\Phi^*X_n\in \Co^{s+1}$, with appropriate bounds on their $\Co^{s+1}$ norms.
    In \cite[Theorem 2.14]{StovallStreetII}, these estimates required
    $s\geq s_0>1$; but by using the regularity theory of elliptic
    PDEs we will be able to extend this to $s\geq s_0>0$.
    
    
    
    \item 
    Once we have established \cite[Theorem 2.14 (a)-(j)]{StovallStreetII}
    for $s_0>0$, the proof from \cite{StovallStreetII}
    of \cite[Theorem 2.14 (k) and (l)]{StovallStreetII} also establishes
    these results for $s_0>0$.
    
    \item The only place that $s_0>1$ is used in \cite[Theorem 4.5]{StreetSubHermitian} is when it refers to \cite[Theorem 2.14]{StovallStreetII}. Once \cite[Theorem 2.14]{StovallStreetII} is established for $s_0>0$, the same is true of \cite[Theorem 4.5]{StreetSubHermitian}.
    
\end{itemize}





\cite[Theorem 2.14]{StovallStreetII}
begins with 
 $X_1,\dots,X_q$ which are $C^1$-vector fields on $\Manifold$ such that $[X_i,X_j]=\sum_{k=1}^qc_{ij}^kX_k$ for some $c_{ij}^k\in C^0_\loc(\Manifold)$. By passing to an immersed submanifold using \cite[Proposition 3.1]{StovallStreetI} we may assume that $X_1,\dots,X_q$ span the tangent space at every point. 

Fix a point $p\in\Manifold$.  We choose $J_0=(j^0_1,\dots,j^0_n)\in\{1,\dots,q\}^n$ such that $X_{j^0_1}(p)\wedge\dots X_{j^0_n}(p)\neq0$ and  
\begin{equation}\label{Eqn::Quant::DensityBound}
    \max_{1\le j_1<\dots<j_n\le q}\mleft|\frac{X_{j_1}(p)\wedge\dots\wedge X_{j_n}(p)}{X_{j^0_1}(p)\wedge\dots\wedge X_{j^0_n}(p)}\mright|\le\zeta^{-1}.
\end{equation}
Here $\zeta>0$ is a constant which all of our estimates may depend on; one
can always pick $j_1^0,\ldots, j_n^0$ so that the left hand side of \eqref{Eqn::Quant::DensityBound}
equals $1$, though it is convenient in some applications to allow for $\zeta<1$.



In \cite[Definition 4.1]{StovallStreetI}  ``0-admissible constants''  are defined to be constants that depend only on diffeomorphic invariant quantities like  $\sum_{i,j,k=1}^q\|c_{ij}^k\|_{C^0(B_{X_{J_0}}(p,\xi))}$ where $\xi>0$ is a small, given constant
on which our estimates may depend,
and $X_{J_0}=(X_{j_1^0},\ldots, X_{j_n^0})$.  See \cite[Definition 4.1]{StovallStreetI} for the
precise definition.


Fix $s_0>0$.  For $s\geq s_0$ we define $\{s\}$-admissible constants as in 
\cite[Definition 2.13]{StovallStreetII} except we only require $s_0>0$ rather than $s_0>1$. 
These are constants which depend only on diffeomorphic invariant quantities like
$\sum_{i,j,k=1}^q\|c_{ij}^k\|_{\Co^s_{X_{J_0}}(C^0(B_{X_{J_0}}(p,\xi)))}$.
See \cite[Definition 2.13]{StovallStreetII} for the precise definition.


Recall that $X_1,\ldots, X_q$ span the tangent space to $\Manifold$ at every point.
Moreover, by reordering $X_1,\ldots, X_q$ so that $j_1^0=1,\ldots, j_n^0=n$,
we may assume that $X_1(p),\ldots, X_n(p)$ form a basis for $T_p\Manifold$
and \eqref{Eqn::Quant::DensityBound} holds with $X_1(p)\wedge \cdots \wedge X_n(p)$
in the denominator.


We begin by considering the canonical coordinates
$\Phi_0(x_1,\dots,x_n):=e^{x_1X_1+\dots+x_nX_n}(p)$. In the following lemma we prove an analog of \cite[Proposition 4.1]{StovallStreetII} when $s_0>0$ (as opposed to $s_0>1$). 
In this case, we only show that $\Phi_0$ is locally a $C^1$-diffeomorphism
rather than globally a $C^2$-diffeomorphism.  In particular, we only
show that $\Phi_0$ is locally injective.

\begin{lemma}\label{Lemma::Quant::Phi0}
There is an {$\{s_0\}$-admissible} constant\footnote{In \cite{StovallStreetI} and \cite{StovallStreetII}
a constant similar to $\mu_0$ was called $\eta_0$.} $\mu_0>0$, such that $\Phi_0(x):=e^{x\cdot X}(p)$ is defined for $ x\in B^n(\mu_0)$ and $\Phi_0:B^n(\mu_0)\to \Manifold$ 
is a locally $C^1$-diffeomorphism, so that we can pullback $X_1,\dots,X_q$ to $B^n(\mu_0)$. Moreover, by writing $Y_j=\Phi_0^*X_j$ for $j=1,\dots,q$ and $[Y_1,\dots,Y_n]^\top=(I+A)\nabla$, we have
\begin{enumerate}[parsep=-0.3ex,label=(\roman*)]
    \item\label{Item::Quant::Phi0::A} $A(0)=0$, $\sup\limits_{x\in B^n(\mu_0)}|A(x)|_{\Mbb^{n\times n}}\le\frac12$ and $\|A\|_{\Co^s(B^n(\mu_0);\Mbb^{n\times n})}\lesssim_{\{s\}}1$ for $s\ge s_0$.
    \item\label{Item::Quant::Phi0::LinDep} There exist $b_k^l\in \Co^{s_0+1}(B^n(\mu_0))$, $1\le l\le n<k\le q$, such that $Y_k=\sum_{l=1}^nb_k^lY_l$ for $n+1\le k\le q$. Moreover $\displaystyle\sum_{k=n+1}^q\sum_{l=1}^n\|b_k^l\|_{\Co^{s+1}(B^n(\mu_0))}\lesssim_{\{s\}}1$.
    \item\label{Item::Quant::Phi0::TrueCijk} There exist $\tilde c_{ij}^k\in\Co^{s_0}(B^n(\mu_0))$ for $1\le i,j\le q$ and $1\le k\le n$, such that $\Phi_0^*[X_i,X_j]=\sum_{k=1}^n\tilde c_{ij}^k\cdot Y_k$ on $B^n(\mu_0)$ for $1\le i,j\le q$. Moreover $\displaystyle\sum_{i,j=1}^q\sum_{k=1}^n\|\tilde c_{ij}^k\|_{\Co^s(B^n(\mu_0))}\lesssim_{\{s\}}1$.
\end{enumerate}
\end{lemma}


\begin{rmk}
\begin{enumerate}[parsep=-0.3ex,label=(\alph*)]
    \item {The constant $\mu_0$ does not depend on $\|c_{ij}^k\|_{\Co^{s_0}}$. It is almost a 0-admissible constant (see \cite[Definition 4.1]{StovallStreetI}) except it also depends on the quantity $\eta>0$ in \cite[Section 3.2]{StovallStreetI}.}
    \item When $\Phi_0$ is locally $C^1$-diffeomorphism, for any (continuous) vector field $L$ on $\Manifold$ there is a unique vector field $\tilde L$ on $B^n(\mu_0)$ such that $d\Phi_0(x)\tilde L\big|_x=L\big|_x$. So we define the pullback vector field as $\tilde L:=\Phi_0^*L$.
    \item We cannot say $[Y_i,Y_j]=\sum_{k=1}^n\tilde c_{ij}^kY_k$ yet, since $Y_1,\dots,Y_q$ are $\Co^{s_0}$ and we may not be able to talk about commutators of $\Co^{s_0}$-vector fields when $0<s_0\le\frac12$. Nevertheless $[Y_i,Y_j]$ can be thought of as $\Phi_0^*[X_i,X_j]$.
    \item When $0<s<1$, by standard results from  ODEs we only know that $\Phi_0$ is a $\Co^{1+s}$-map and we do not expect $\Phi_0$ to be $C^2$. 
    Unfortunately, the proof of injectivity for $\Phi_0$ in \cite[Proposition 9.15]{StovallStreetI}
    requires $\Phi_0^*X_1,\dots,\Phi_0^*X_n$ to be $C^1$
    Nevertheless we will show that $\Phi_0$ is injective when restricted to a ball centered at 0 with a smaller radius, and this smaller radius is a $\{s_0\}$-admissible constant. See Lemma \ref{Lemma::Quant::InjLemma} and Remark \ref{Rmk::Quant::Phi0isInj}.
    
    \item
    Lemma \ref{Lemma::Quant::Phi0} ``loses one derivative''
    in the sense that it implies $\| Y_j\|_{\Co^{s}}\lesssim_{\{s\}} 1$, but our main result
    gives $\|Y_j\|_{\Co^{s+1}}\lesssim_{\{s\}} 1$.
    Similar to the proof in \cite{StovallStreetII}, we will recover this lost derivative
    by composing with another map $\Phi_1$ in Proposition \ref{Prop::Quant} (see also \cite[Proposition 6.3]{StovallStreetII}).

\end{enumerate}
\end{rmk}

\begin{proof}[Proof of Lemma \ref{Lemma::Quant::Phi0}]
    Let $\tilde\eta>0$ be a number such that $\Phi_0$ is defined on $B^n(\tilde\eta)$ and $\Phi_0$ cannot be defined on $B^n(\tilde\eta')$ for any $\tilde\eta'>\tilde\eta$. 
    Note that $\tilde\eta$ is bounded below by a positive {$\{s_0\}$-admissible constant (also see \cite[Definitions 3.7 and 3.10]{StovallStreetI})}. 

    We first prove the result in the special case $q=n$. In this case,
    $X_1,\ldots, X_n$ form a basis of the tangent space to $\Manifold$ at every point.  Thus, $[X_i,X_j]=\sum_{k=1}^nc_{ij}^kX_k$ where $(c_{ij}^k)_{i,j,k=1}^n$ are uniquely determined by $X_1,\dots,X_n$.
    
    By \cite[Lemma 9.6]{StovallStreetI} we know there is a unique\footnote{In \cite[Section 9.3.1]{StovallStreetI}, what we call $\tilde A$ is called $A$, and what we call
    $A$ is called $\widehat A$.} 
    $\tilde A\in C^0(B^n(0,\tilde\eta);\Mbb^{n\times n})$ such that 
\begin{equation}\label{Eqn::Quant::InjLemma::ODEforA}
\begin{gathered}
    \begin{cases}
    \Coorvec r(r\tilde A(r\theta))=-\tilde A(r\theta)^2-C(r\theta)\tilde A(r\theta)-C(r\theta),&\text{for }|r|<\tilde\eta\text{ and }\theta\in\mathbb S^{n-1},\\\tilde A(0)=0,
    \end{cases}
    \\
    \text{ where }C(x)_i^j:=\sum_{k=1}^nx_k\cdot c_{ik}^j(\Phi_0(x)),\quad x\in B^n(0,\tilde\eta),\quad 1\le i,j\le n.
\end{gathered}
\end{equation}
    
By \cite[Proposition 9.4]{StovallStreetI} there is a $0$-admissible constant $D>0$ (in fact, depending only on $n$ and upper bounds for $\sum_{i,j,k=1}^n\|c_{ij}^k\|_{C^0(\Manifold)}$), such that (see \eqref{Eqn::FuncRevis::MatrixNorms} in Convention \ref{Conv::FuncRevis::MatrixNorms})
\begin{equation}\label{Eqn::Quant::InjLemma::BddofA}
    |\tilde A(x)|_{\Mbb^{n\times n}}\le D|x|,\quad x\in B^n(0,\tilde\eta).
\end{equation}

Take $\mu_0:=\frac1{2D}$, so $\|\tilde A\|_{C^0(B^n(\mu_0);\Mbb^{n\times n})}\le\frac12$. Therefore $I+A(x)$ is invertible matrix at every point $x\in B^n(\mu_0)$, which means $\Phi_0$ has non-degenerate tangent map at every point. By the Inverse Function Theorem, $\Phi_0$ is a locally $C^1$-diffeomorphism.

Define the matrix $A$ by $Y=:(I+A)\nabla$; where we are treating
$Y$ as the column vector of vector fields $Y=[Y_1,\dots,Y_n]^\top$,
and $\nabla$ is thought of as a column vector.
It follows from \cite[Proposition 9.18]{StovallStreetI}, that $A=\tilde A$ in $B^n(\mu_0)$ and $\|\tilde A\|_{\Co^s(B^n(\mu_0);\Mbb^{n\times n})}\lesssim_{\{s\}}1$. Since $Y=(I+A)\nabla$, we have $\sum_{i=1}^n\|Y_i\|_{\Co^s(B^n(\mu_0);\R^n)}\lesssim_{\{s\}}1$, {finishing the proof of  \ref{Item::Quant::Phi0::A} when $n=q$}.

Since $\|A\|_{C^0(B^n(\mu_0);\Mbb^{n\times n})}\le\frac12$, we have $\|(I+A)^{-1}\|_{C^0(B^n(\mu_0);\Mbb^{n\times n})}\le\sum_{k=0}^\infty(\frac12)^k=\frac32<\infty$, so by \cite[Lemma 5.7]{StovallStreetII} {(applied to the cofactor representation of $(I+A)^{-1}$)} with $\|A\|_{\Co^s(B^n(\mu_0);\Mbb^{n\times n})}\lesssim_{\{s\}}1$, we get
\begin{equation}\label{Eqn::Quant::I+AisAdmiss}
    \|(I+A)\|_{\Co^s(B^n(\mu_0))}+\|(I+A)^{-1}\|_{\Co^s(B^n(\mu_0))}\lesssim_{\{s\}}1.
\end{equation}

Take $\tilde c_{ij}^k:=c_{ij}^k\circ\Phi_0$ for $1\le i,j,k\le n$. We separate the proof of $\sum_{i,j,k=1}^n\|\tilde c_{ij}^k\|_{\Co^s(B^n(\mu_0))}\lesssim_{\{s\}}1$ into the case $s>1$ and the case $0<s\le1$.

{
When $s>1$, we have $Y_1,\dots,Y_n\in \Co^{s}\subset C^1$. By \cite[Lemma 9.24]{StovallStreetI} we have $\sum_{i,j,k=1}^n\|\tilde c_{ij}^k\|_{\Co^s(B^n(\mu_0))}\lesssim_{\{s\}}1$.

When $0<s\le1$, we may not have $Y_1,\dots,Y_n\in C^1$. In order to use previous results such as \cite[Proposition 8.6]{StovallStreetI}, we need $\Phi_0$ to be (qualitatively) $C^2$; though we do not require any estimates on any $C^2$ norm of $\Phi_0$. To get around the fact that $\Phi_0$ is not $C^2$, we introduce another atlas on $B^n(\mu_0)$ with respect to which $\Phi_0$ is $C^2$}. 
Indeed, we say $f:B^n(\mu_0)\to\R$ is $C^2_\loc$ with respect to the atlas $\As$, if for every open subset $U\subseteq B^n(\mu_0)$ such that $\Phi_0:U\to \Phi_0(U)$ is bijective, we have $f\circ\Phi_0^{-1}:\Phi_0(U)\subseteq\Manifold\to\R$ is $C^2_\loc$.

Since $Y_1,\dots,Y_n$ are $\Co^s$ and span the tangent space at every point in $B^n(\mu_0)$, we know $\As$ is compatible with the standard $\Co^{s+1}$-structure on $B^n(\mu_0)$. In particular both $(B^n(\mu_0),\As)$ 
agrees with the standard $C^{1,s/2}$-structure on $B^n(\mu_0)$.

{Now $\Phi_0:(B^n(\mu_0),\As)\to \Manifold$ is $C^2$. By \cite[Proposition 8.6]{StovallStreetI}} we have \begin{equation}\label{Eqn::Quant::Phi0:TildeC1}
    \|\tilde c_{ij}^k\|_{\Co^s_Y(B^n(\mu_0),\As)}=\|c_{ij}^k\circ\Phi_0\|_{\Co^s_Y(B^n(\mu_0),\As)}\le\|c_{ij}^k\|_{\Co^s_X(\Manifold)}\quad 1\le i,j,k\le n,\quad 0<s\le1.
\end{equation}

On the other hand by \cite[Section 2.2]{StovallStreetI}, the definition of $\|\cdot\|_{\Co^s_Y}$ involves only the $C^{1,s/2}$-structure of the manifold. {Indeed for $0<s\le1$, on the domain $U=B^n(\mu_0)$,
{\begin{gather*}
    \|f\|_{\Co^s_Y}=\|f\|_{C^0}+\sup\limits_{x,y\in U;x\neq y}\frac{|f(x)-f(y)|}{\dist_Y(x,y)^\frac s2}+\sup\limits_{\gamma\in\mathcal P_{Y,\frac s2}(h)}h^{-s}|f(\gamma(2h))-2f(\gamma(h))+f(\gamma(0))|,\quad\text{where}
    \\
    \mathcal P_{Y,\frac s2}(h)=\Big\{\gamma\in C^{1,\frac s2}([0,2h];U):\dot\gamma(t)=\sum_{j=1}^nd_j(t)Y_j(\gamma(t)),\ d_1,\dots, d_n\in C^{0,\frac s2}[0,2h],\ \sum_{j=1}^n\|d_j\|_{C^{0,\frac s2}[0,2h]}^2\le1\Big\},
    \\\dist_Y(x,y)=\inf\Big\{T>0:\exists \gamma\in C^{0,1}([0,T];U),\gamma(0)=x,\gamma(T)=y,\dot\gamma(t)=\sum_{j=1}^nd_j(t)Y_j(\gamma(t)),\ \sum_{j=1}^n\|d_j\|_{L^\infty[0,T]}^2\le1\Big\}.
\end{gather*}
Here $\dist_Y$ only depends on the Lipschitz structure of $B^n(\mu_0)$ and $\mathcal P_{Y,\frac s2}(h)$ only depends on the $C^{1,\frac s2}$-structure of $B^n(\mu_0)$.}
}

Since $\As$ is compatible with the standard $C^{1,\frac s2}$-structure for $B^n(\mu_0)$, we have 
\begin{equation}\label{Eqn::Quant::Phi0:TildeC2}
    \|\tilde c_{ij}^k\|_{\Co^s_Y(B^n(\mu_0),\As)}=\|\tilde c_{ij}^k\|_{\Co^s_Y(B^n(\mu_0))},\quad 1\le i,j,k\le n,\quad0<s\le1.
\end{equation}

To show $\|\tilde c_{ij}^k\|_{\Co^s(B^n(\mu_0))}\lesssim_{\{s\}}1$, it remains to show that $\|\tilde c_{ij}^k\|_{\Co^s(B^n(\mu_0))}\lesssim_{\{s\}}\|\tilde c_{ij}^k\|_{\Co^s_Y(B^n(\mu_0))}$.

Note that $\Co^s(B^n(\mu_0))\subset C^{0,s/2}(B^n(\mu_0))$ with $\|f\|_{C^{0,s/2}(B^n(\mu_0))}\lesssim_{n,s,\mu_0}\|f\|_{\Co^s(B^n(\mu_0))}$ for all $f$. 
Using \eqref{Eqn::Quant::I+AisAdmiss} and \cite[Proposition 8.12]{StovallStreetI} we get $\|\tilde c_{ij}^k\|_{\Co^s(B^n(\mu_0))}\lesssim_{\{s-1\}}\|\tilde c_{ij}^k\|_{\Co^s_Y(B^n(\mu_0))}$, in particular $\|\tilde c_{ij}^k\|_{\Co^s(B^n(\mu_0))}\lesssim_{\{s\}}\|\tilde c_{ij}^k\|_{\Co^s_Y(B^n(\mu_0))}$.
{Here the implicit $\{s-1\}$-admissible constant (that appears when we say ``$\lesssim_{\{s-1\}}$'') is given in \cite[Definition 8.10]{StovallStreetI}, which depends on the upper bound of $C^{0,\frac s2}$-norms of $(\tilde c_{ij}^k)_{i,j,k=1}^n$.}

Combining this with \eqref{Eqn::Quant::Phi0:TildeC1} and \eqref{Eqn::Quant::Phi0:TildeC2} we get $\sum_{i,j,k=1}^n\|\tilde c_{ij}^k\|_{\Co^s(B^n(\mu_0))}\lesssim_{\{s\}}1$ for $0<s\le1$, finishing the proof of \ref{Item::Quant::Phi0::TrueCijk} when $q=n$.

The general case $q>n$ can be reduced to the case $q=n$ as in \cite[Section 9.3.2]{StovallStreetI}: we get \ref{Item::Quant::Phi0::LinDep} from \cite[Lemma 9.33]{StovallStreetI} and \ref{Item::Quant::Phi0::TrueCijk} from \cite[Lemma 9.34]{StovallStreetI}.
\end{proof}

Lemma \ref{Lemma::Quant::Phi0} yields a local $C^1$-diffeomorphism $\Phi_0:B^n(\mu_0)\to\Manifold$ and $\tilde c_{ij}^k\in C^0(B^n(\mu_0))$ for $1\le i,j,k\le n$ such that $\Phi_0^*[X_i,X_j]=\sum_{k=1}^n\tilde c_{ij}^kY_k$ for $1\le i,j\le n$ and $\sum_{i=1}^q\|\Phi_0^*X_i\|_{\Co^s(B^n(\mu_0);\R^n)}+\sum_{i,j,k=1}^n\|\tilde c_{ij}^k\|_{\Co^s(B^n(\mu_0))}\lesssim_{\{s\}}1$ for $s\ge s_0$. 
Thus, we have reduced the problem to studying vector fields on $B^n(\mu_0)$,
which have estimates in terms of classical function spaces, instead of the
abstract function spaces $\Co^s_X$.

By Lemma \ref{Lemma::Quant::Phi0} \ref{Item::Quant::Phi0::LinDep}, $\Phi_0^*X_{n+1},\dots,\Phi_0^*X_q$ are now linear combinations of $\Phi_0^*X_1,\dots,\Phi_0^*X_n$ whose {coefficients have $\Co^{s_0+1}$-norms} bounded by a $\{s_0\}$-admissible constant. {So we can assume $q=n$}, just as  in the beginning of \cite[Section 6]{StovallStreetII}. 

When $s_0>1$, the proof of \cite[Theorem 2.14]{StovallStreetII} is as follow.
In \cite[Proposition 6.3]{StovallStreetII} it is shown
that there is a map $\Phi_2:\Ball^n\to B^n(\mu_0)$ which is $\Co^{s_0+1}$-diffeomorphism onto its image such that $\Phi_2^*Y_1,\dots,\Phi_2^*Y_n$ are $\Co^{s_0+1}$ and have $\Co^{s_0+1}$-norms bounded by a $\{s_0\}$-admissible constant. Meanwhile if $Y_1,\dots,Y_n$ and $(\tilde c_{ij}^k)_{i,j,k=1}^n$ are all $\Co^s$ for some $s>s_0$, then $\Phi_2$ is automatically $\Co^{s+1}$ and the $\Co^{s+1}$-norms of the coefficients of $\Phi_2^*Y_1,\dots,\Phi_2^*Y_n$ are automatically bounded by  a $\{s\}$-admissible constant.  This completes the proof when $s_0>1$.  Our goal
is to generalize this argument to $s_0>0$.

Based on the techniques from Section \ref{Section::KeyThm}, we can prove an analog of \cite[Proposition 6.3]{StovallStreetII} in the setting of $s_0>0$. We formulate the statement in
the proposition below.

\bigskip
Let $X_1,\dots,X_n$ be $C^1$-vector fields on a $C^2$-manifold $\Manifold$ that form a  basis of the tangent space at every point. Near a fixed point $p\in\Manifold$ we define $\Phi_0(x):=e^{x\cdot X}(p)$. Write $[X_i,X_j]=\sum_{k=1}^nc_{ij}^kX_k$ where $c_{ij}^k\in C^0_\loc(\Manifold)$ are uniquely determined by $X_1,\dots,X_n$. 

On the subset of the domain of $\Phi_0$ where $\nabla\Phi_0$ is non-degenerate (so $\Phi_0$ is locally $C^1$-diffeomorphism on this set), we denote $Y_i:=\Phi_0^*X_i$ and $\tilde c_{ij}^k:=\Phi_0^*c_{ij}^k$ for $1\le i,j,k\le n$. And we write $Y=[Y_1,\dots,Y_n]^\top$ as $Y=(I+A(x))\Coorvec x$ where $A$ is a $\Mbb^{n\times n}$-valued function defined on the domain of $\Phi_0$.
\begin{prop}\label{Prop::Quant} 

Let $s_0,\mu_0>0$, $s\ge s_0$ and $M_0,M_1>0$. There are constants $\widehat K=\widehat K(n,s_0,\mu_0,M_0)>0$, $K_0=K_0(n,s_0,\mu_0,M_0)>0$, and $K_1=K_1(n,s_0,s,\mu_0,M_0,M_1)>0$ that satisfy the following:

Let $X_1,\dots,X_n$ and $\Phi_0(x)=e^{x\cdot X}(p)$ be as above. Suppose we have  the following:
\begin{itemize}[parsep=-0.3ex]
    \item $\Phi_0:B^n(0,\mu_0)\to\Manifold$ is defined and is a locally $C^1$-diffeomorphism onto its image (so that $Y,A,\tilde c_{ij}^k$ are defined on $B^n(0,\mu_0)$).
    \item $\sup\limits_{|x|<\mu_0}|A(x)|\le\frac12$ and 
    \begin{equation}\label{Eqn::Quant::AssumptionM0}
        \|A\|_{\Co^{s_0}(B^n(0,\mu_0);\Mbb^{n\times n})}+\sum_{i,j,k=1}^n\|\tilde c_{ij}^k\|_{\Co^{s_0}(B^n(0,\mu_0))}<M_0.
    \end{equation}
\end{itemize}
Then
\begin{enumerate}[parsep=-0.3ex,label=(\roman*)]
    \item\label{Item::Quant::Phi1Exist} There is a map $\Phi_1:\Ball^n\to B^n(0,\mu_0)$ such that
\begin{itemize}[parsep=-0.3ex]
    \item $\Phi_1(0)=0$ and $\Phi_1$ is $\Co^{s_0+1}$-diffeomorphism onto its image.
    \item $\Phi_1^*Y=[\Phi_1^*Y_1,\dots,\Phi_1^*Y_n]^\top$ is a collection of $\Co^{s_0+1}$-vector fields on $\Ball^n$ that can be written as
    \begin{equation}\label{Eqn::Quant::Khat}
        \Phi_1^*Y=\widehat K (I+\widehat A)\nabla,\quad\text{where }\widehat A(0)=0, \ \|\widehat A\|_{C^0(\Ball^n;\Mbb^{n\times n})}\le\tfrac{1}{2}.
    \end{equation}
    {Moreover, we have the estimate
    \begin{equation}\label{Eqn::Quant::K0}
        \|\Phi_1\|_{\Co^{s_0+1}(\Ball^n;\R^n)}+\|\widehat A\|_{\Co^{s_0+1}(\Ball^n;\Mbb^{n\times n})}\le K_0.
    \end{equation}}
\end{itemize}
\item\label{Item::Quant::Phi1Reg} Suppose additionally $A$ and $\tilde c_{ij}^k$ are all $\Co^s$ with
\begin{equation}\label{Eqn::Quant::AssumptionReg}
        \|A\|_{\Co^{s}(B^n(0,\mu_0);\Mbb^{n\times n})}+\sum_{i,j,k=1}^n\|\tilde c_{ij}^k\|_{\Co^{s}(B^n(0,\mu_0))}<M_1.
    \end{equation}
    
    Then $\Phi_1:\Ball^n\to B^n(0,\mu_0)$ is a $\Co^{s+1}$-map and $\widehat A\in\Co^{s+1}(\Ball^n;\Mbb^{n\times n})$. Moreover 
    \begin{equation}\label{Eqn::Quant::K1}
        \|\Phi_1\|_{\Co^{s+1}(\Ball^n;\R^n)}\le K_1,\qquad \|\widehat A\|_{\Co^{s+1}(\Ball^n;\Mbb^{n\times n})}\le K_1.
    \end{equation}
\end{enumerate}

\end{prop}
\begin{rmk}
\begin{enumerate}[parsep=-0.3ex,label=(\alph*)]
    \item For the proof of Theorem \ref{Thm::Quant} we will apply Proposition \ref{Prop::Quant} with $\Phi_0$
    and $\mu_0$
    as in Lemma \ref{Lemma::Quant::Phi0}.  In this application,
    $\widehat K$ and $K_0$ are $\{s_0\}$-admissible constants and $K_1$ is an $\{s\}$-admissible constant.
    In particular, in this application, we have $\|\Phi_1\|_{\Co^{s+1}}\lesssim_{\{s\}}1$ and $\|\widehat A\|_{\Co^{s+1}}\lesssim_{\{s\}}1$.
    \item A map similar to $\Phi_1$ in Proposition \ref{Prop::Quant} is called $\Phi_2$ in \cite[Proposition 6.3]{StovallStreetII}. In the proof of \cite[Proposition 6.3]{StovallStreetII}, $\Phi_2$ is decomposed as a $\Co^{s_0+1}$-diffeomorphism and a scaling map. In our setting the scaling is already done
    in Proposition \ref{Prop::Key::Scaling}. Also see Lemma \ref{Lemma::Quant::ScalingReg}.
\end{enumerate}
\end{rmk}

We need some preliminary results to prove Proposition \ref{Prop::Quant}.

Suppose we have Proposition \ref{Prop::Quant} \ref{Item::Quant::Phi1Exist}, that is, we construct a $\Phi_1$ such that $\|\Phi_1\|_{\Co^{s_0+1}}+\sum_{i=1}^n\|\Phi_1^*Y_i\|_{\Co^{s_0+1}}\lesssim_{\{s_0\}}1$. Note that the $\Phi_1$ does not depend on the index $s$. In order to prove Proposition \ref{Prop::Quant} \ref{Item::Quant::Phi1Reg}, i.e. $\|\Phi_1\|_{\Co^{s+1}}+\sum_{i=1}^n\|\Phi_1^*Y_i\|_{\Co^{s+1}}\lesssim_{\{s\}}1$ for every $s>s_0$ (see also for \cite[Theorem 2.14 (j)]{StovallStreetII}), we need to give regularity estimate for Theorem \ref{Thm::Keythm}. 

Instead of vector fields, we proceed by using 1-forms.

Recall in Section \ref{Section::KeyThm::Outline} we start with 1-forms $\lambda^i=dx^i+\sum_{j=1}^na_j^idx^j$, $i=1,\dots,n$, defined on $\Ball^n\subset\R^n$ such that $A=(a_j^i)_{n\times n}$ is supported in $\frac12\Ball^n$. We let $F=\id+R:\Ball^n_x\to\Ball^n_y$
be the map in Proposition \ref{Prop::Key::ExistPDE}, which is a $C^1$-diffeomorphism and solves \eqref{Eqn::Key::PDEforR} with $R\big|_{\partial\Ball^n}=0$. We write the pushforward 1-forms $\eta^i=F_*\lambda^i$, $i=1,\dots,n$ on $\Ball^n_y$ as $\eta^i=dy^i+\sum_{j=1}^nb_j^idy^j$. By Lemma \ref{Lemma::Key::AtoB} we know $B=(b_j^i)_{n\times n}:\Ball^n\to\Mbb^{n\times n}$ solves \eqref{Eqn::Key::PDEforB}, which can be rewritten as \eqref{Eqn::Key::RegularityPDE::EqnasCalR}.

\begin{prop}\label{Prop::Quant::Schauder}
Fix $s_0>0$.  There is a $c'=c'(n,s_0)>0$, such that in additional to the results in Theorem \ref{Thm::Keythm} with $\alpha=s_0$, $\beta=s_0+1$ and $c(n,s_0,s_0+1)=c'(n,s_0)$, we have the following:

\begin{enumerate}[parsep=-0.3ex,label=(\roman*)]
    \item\label{Item::Quant::SchauderSupofB} 
    For the collection of 1-forms $[F_*\lambda^1,\dots,F_*\lambda^n]^\top=(I+B)dy$ we have $\|B\|_{C^0(\Ball^n;\Mbb^{n\times n})}\le\frac14$.
\end{enumerate}

And for any $s\ge s_0$, $M'>0$, there is a $K'=K'(n,s,s_0,M')>0$ such that if in addition to the assumptions of Theorem \ref{Thm::Keythm}, we have $\lambda^1,\dots,\lambda^n,d\lambda^1,\dots,d\lambda^n\in\Co^s$ with 
$$\sum_{i=1}^n\|\lambda^i\|_{\Co^s(\Ball^n;T^*\Ball^n)}+\|d\lambda^i\|_{\Co^s(\Ball^n;\wedge^2T^*\Ball^n)}<M',$$ then:
\begin{enumerate}[parsep=-0.3ex,label=(\roman*)]\setcounter{enumi}{1}
    \item\label{Item::Quant::QuantExt} The map $F:\Ball^n_x\to\Ball^n_y$ in Theorem \ref{Thm::Keythm} is $\Co^{s+1}$ and its inverse $\Phi$ satisfies $\|\Phi\|_{\Co^{s+1}(\Ball^n;\R^n)}<K'$.
    \item\label{Item::Quant::QuantReg} The 1-forms $\eta^1,\dots,\eta^n$ are all $\Co^{s+1}$ and the coefficient matrix $B:\Ball^n\to\Mbb^{n\times n}$ satisfies $\|B\|_{\Co^{s+1}(\frac34\Ball^n;\Mbb^{n\times n})}<K'$.
\end{enumerate}
\end{prop}
Informally Proposition \ref{Prop::Quant::Schauder} \ref{Item::Quant::QuantReg} is saying $\sum_{i=1}^n\|\eta^i\|_{\Co^{s+1}}\lesssim_{\{s\}}1$.

Note that we require $\|\lambda\|_{\Co^{s_0}}$ and $\|d\lambda\|_{\Co^{s_0}}$ to be small (bounded by the constant $c$ from Theorem \ref{Thm::Keythm}) in the assumption of
Proposition \ref{Prop::Quant::Schauder}. 
However,  by taking $M'$ large in Proposition \ref{Prop::Quant::Schauder}, we can allow $\|\lambda\|_{\Co^s}$ and $\|d\lambda\|_{\Co^{s}}$ to be large for $s>s_0$.

\begin{proof}
For \ref{Item::Quant::SchauderSupofB}, let $c=c(n,s_0,s_0+1)>0$ be the original constant in Theorem \ref{Thm::Keythm}. By \eqref{Eqn::Keythm::Conclusion} we see that for any $c'\in(0,c]$, $\sum_{i=1}^n\|\lambda^i-dx^i\|_{\Co^{s_0}}+\|d\lambda^i\|_{\Co^{s_0}}<c'$ implies $\sum_{i=1}^n\|F_*\lambda^i-dy^i\|_{\Co^{s_0+1}(\Ball^n;T^*\Ball^n)}\le\frac{c'}c$.

Recall the notation $F_*\lambda=(I+B)dy$. Since $\|B\|_{C^0}\lesssim_{s_0}\|B\|_{\Co^{s_0+1}}\approx_{s_0}\sum_{i=1}^n\|F_*\lambda^i-dy^i\|_{\Co^{s_0+1}}$, by choosing $c'\in(0,c]$ small enough we can ensure $\|B\|_{C^0(\Ball^n;\Mbb^{n\times n})}<\frac14$.

To prove \ref{Item::Quant::QuantExt} and \ref{Item::Quant::QuantReg}, we choose $c'(n,s_0)\in (0,c]$ 
as follows.

    Let $c_1=c_1(n,s_0,s_0+1)>0$, $c_2=c_2(n,s_0,s_0+1)>0$, and $c_3=c_3(n,s_0,s_0+1)>0$ be the constants in Proposition \ref{Prop::Key::ExistPDE}, Lemma \ref{Lemma::Key::AtoB}, and Proposition \ref{Prop::Key::RegularityPDE}, respectively.
    
    In the proof of Theorem \ref{Thm::Keythm} we see that if $\sum_{i=1}^n\|\lambda^i\|_{\Co^{s_0}}+\|d\lambda^i\|_{\Co^{s_0}}<\frac1{2n^2}\min(c_1,c_2c_3)$ then $B\in\Co^{s_0+1}(\Ball^n;\Mbb^{n\times n})\subset C^1(\Ball^n;\Mbb^{n\times n})$. We are going to find a smaller constant $c_3'(n,s_0)\in(0,c_3)$ and then take $c'\le\frac1{2n^2}\min(c_1,c_2c_3')$.
    

    Recall {$\Rc_i^k(B)=\sum_{j=1}^n(\sqrt{\det h}h^{ij}-\delta^{ij})b_j^k$, $1\le i,k\le n$  in \eqref{Eqn::Key::RegularityPDE::EqnasCalR} (see \eqref{Eqn::Key::Riemannianmetric2} for $h^{ij}$ and $\sqrt{\det h}$) are rational functions which are finite near the origin such that $\Rc_i^k(B)=O(|B|^2)$ near $B=0\in\Mbb^{n\times n}$ (see \eqref{Eqn::Key::Riemannianmetric1}, \eqref{Eqn::Key::Riemannianmetric2}, and Lemma \ref{Lemma::Key::TaylorExpansionofRiemMetric})}. So when $B\in C^1$, we can write $\partial_{y^i}\Rc_i^k(B)=\sum_{j,l=0}^n\widetilde\Rc_{il}^{kj}(B)\cdot\partial_{y^i}b_j^l$ where \begin{equation*}
        \widetilde\Rc_{il}^{kl}(v):=\frac{\partial\Rc_i^k}{\partial v_j^l}(v),\quad\text{for }1\le i,j,k,l\le n,\quad \text{defined for }v\in \Mbb^{n\times n}\text{ closed to }0.
    \end{equation*}
    So $\widetilde\Rc_{il}^{kl}(0)=0$ since $|\Rc_i^k(v)|\lesssim|v|^2$ for small $v$. The equation \eqref{Eqn::Key::RegularityPDE::EqnasCalR} can be rewritten as 
    \begin{equation}\label{Eqn::Quant::NewEqnforB1}
        \sum_{i=1}^n\frac{\partial }{\partial y^i}b_i^k+\sum_{i,j,l=0}^n\widetilde\Rc_{il}^{jk}(B)\frac{\partial }{\partial y^i}b_j^l=0\quad\text{in }\Ball^n_y,\quad k=1,\dots,n.
    \end{equation}
    
    
    By the conclusions of Proposition \ref{Prop::Key::RegularityPDE} we have $B\in\Co^{s_0+1}\subset C^1$. And by the assumptions of Proposition \ref{Prop::Key::RegularityPDE} combined
    with \eqref{Eqn::Quant::NewEqnforB1}, we know $u=B$ is a $C^1$-solution to the following system of equations in $u=(u_i^j)_{n\times n}$:
    \begin{equation}\label{Eqn::Quant::NewEqnforB2}
        \begin{aligned}
            &\frac{\partial u_k^i}{\partial y^j}-\frac{\partial u_j^i}{\partial y^k}=d\eta^i\Big(\Coorvec{y^j},\Coorvec{y^k}\Big),&i=1,\dots,n,\quad1\le j<k\le n.
            \\
            &\sum_{i=1}^n\frac{\partial }{\partial y^i}u_i^k+\sum_{i,j,l=0}^n\widetilde\Rc_{il}^{jk}(B)\frac{\partial }{\partial y^i}u_j^l=0,& k=1,\dots,n.
        \end{aligned}
    \end{equation}
    \eqref{Eqn::Quant::NewEqnforB2} is of the form $$\Ec u+\Lc_Bu=g_B,$$
    where $\Ec:C^\infty(\Ball^n;\R^{n^2})\to C^\infty(\Ball^n;\R^{n\frac{n^2-n+2}2})$ is a first order constant linear differential operator that does not depend on $B$, and $\Lc_B$ is a first order linear differential operator with coefficients comes from $\widetilde\Rc_{il}^{jk}(B)$. Here, $g_B$ is the vector-valued function which is the right hand side of \eqref{Eqn::Quant::NewEqnforB2}; i.e., $g_B=(d\eta^1,\dots,d\eta^n,0)$.
    
    If we write $u^i=\sum_{j=1}^nu^i_jdy^j$, $i=1,\dots,n$, then we see that $\Ec (u^i_j)=(du^i,\codiff_{\R^n}u^i)_{i=1}^n$. So $\Ec^*\Ec=d\codiff+\codiff d=\Lap$ is elliptic, which implies that $\Ec$ is an  elliptic operator.
    
    By classical elliptic theory (see \cite[Proposition A.1]{StovallStreetII}) there is a $\gamma=\gamma(\Ec)>0$ such that if $\sum_{i,j,k,l=1}^n\|\widetilde\Rc_{il}^{jk}(B)\|_{L^\infty(\Ball^n)}<\gamma$ then $\widetilde\Rc_{il}^{jk}(B),g_B\in\Co^s$ implies $u\in\Co^{s+1}(\Ball^n;\R^{n^2})$ with $\|u\|_{\Co^{s+1}}\lesssim_{n,s_0,s,\|\eta^i\|_{\Co^s},\|d\eta^i\|_{\Co^s}}\|u\|_{\Co^{s_0+1}}+\|g_B\|_{\Co^{s_0}}$. 
    
    Thus, for any $\tilde\sigma>0$ there is $C'_0=C'_0(n,s_0,s,c_3,\gamma,\tilde\sigma)>0$ that does not depend on $B$, such that
    \begin{equation}\label{Eqn::Quant::Schauder::Tmp0}
    \begin{gathered}
        \text{ if } \sum_{i,j,k,l=1}^n\|\widetilde\Rc_{il}^{jk}(B)\|_{L^\infty(\Ball^n)}<\gamma\quad\text{and}\quad \sum_{i=1}^n\|\eta^i-dy^i\|_{\Co^s(\Ball^n;T^*\Ball^n)}+\|d\eta^i\|_{\Co^s(\Ball^n;\wedge^2T^*\Ball^n)}<\tilde\sigma,
        \\\text{ then }\|B\|_{\Co^{s+1}(\frac34\Ball^n;\Mbb^{n\times n})}\le C'_0(n,s_0,s,c_3,\gamma,\tilde\sigma).
    \end{gathered}
    \end{equation}
    
    When $\|B\|_{L^\infty}$ suitably small we have $$\sum_{i,j,k,l=1}^n\|\widetilde\Rc_{il}^{jk}(B)\|_{L^\infty}\lesssim_{\Rc}\|B\|_{L^\infty}\lesssim \sum_{k=1}^n\|\eta^k-dy^k\|_{L^\infty}\lesssim_{s_0}\sum_{k=1}^n\|\eta^k-dy^k\|_{\Co^{s_0}}.$$
    So we can take a $c_3'\in(0,c_3)$ (which still only depends on $n,s_0$) such that
    \begin{equation}
        \sum_{k=1}^n\|\eta^k-dy^k\|_{\Co^{s_0}}+\|d\eta^k\|_{\Co^{s_0}}<c_3'\quad\Rightarrow\quad\sum_{i,j,k,l=1}^n\|\widetilde\Rc_{il}^{jk}(B)\|_{L^\infty(\Ball^n)}<\gamma.
    \end{equation}
    
    Using the same proof as Theorem \ref{Thm::Keythm} in Section \ref{Section::ProofKey}, take $c'=\frac1{2n^2}\min(c_1,c_2c_3')$ and we see that $\sum_{k=1}^n\|\lambda^k-dx^k\|_{\Co^{s_0}}+\|d\lambda^k\|_{\Co^{s_0}}<c'$ implies $\sum_{k=1}^n\|\eta^k-dy^k\|_{\Co^{s_0}}+\|d\eta^k\|_{\Co^{s_0}}<c_3'$ and $\|F-\id\|_{\Co^{s_0+1}(\Ball^n;\R^n)}+\sum_{i=1}^n\|\eta^i-dy^i\|_{\Co^{s_0+1}(\Ball^n;\R^n)}\le1$.
    
    We then prove \ref{Item::Quant::QuantExt} and \ref{Item::Quant::QuantReg} using this constant $c'$.
    
    \medskip
    \ref{Item::Quant::QuantExt}: Recall by assumption $\supp A\subsetneq\frac12\Ball^n$, so $\Lap R\big|_{\Ball^n\backslash\frac12\Ball^n}=0$ and thus similar to \eqref{Eqn::Key::ExistPDE::TmpBoundEstR} we have \begin{equation}\label{Eqn::Quant::Schauder::Tmp1}
        \|R\|_{\Co^{s+1}(\Ball^n\backslash\frac34\Ball^n;\R^n)}\lesssim_{s,s_0}\|R\|_{\Co^{s_0+1}(\Ball^n;\R^n)}\lesssim_{s_0}\|A\|_{\Co^{s_0}}.
    \end{equation}
    
    On the other hand by classical interior Schauder estimates we know that
    \begin{equation}\label{Eqn::Quant::Schauder::Tmp2}
        \|R\|_{\Co^{s+1}(\frac45\Ball^n)}\le C\sum_{i,j,k=1}^n\|\sqrt{\det g}g^{ij}a_i^k\|_{\Co^s(\Ball^n)},
    \end{equation}
    where $C$ is a constant that only depends on $n,s$, the upper bounds of $\|I+A\|_{C^0(\Ball^n;\Mbb^{n\times n})},\|(I+A)^{-1}\|_{C^0(\Ball^n;\Mbb^{n\times n})}$ and $\|\sqrt{\det g}g^{ij}\|_{\Co^s(\Ball^n)}$.  For the precise form of the interior Schauder's estimate we use, see, for example, \cite[Corollary 2.28]{XavierPDE} for $s>1,s\notin\Z$ and \cite[Theorem 8.32]{GilbargTrudinger} for $0<s<1$. The proof for $s\in\Z_+$ is similar to these. 
    
    In Proposition \ref{Prop::Key::ExistPDE} we chose $c_1$ small so that $\|A\|_{C^0}<\frac12$; thus $\|I+A\|_{C^0(\Ball^n;\Mbb^{n\times n})}$ and $\|(I+A)^{-1}\|_{C^0(\Ball^n;\Mbb^{n\times n})}$ are already uniformly bounded. And since $\sqrt{\det g}g^{ij}a_i^k$ and $\sqrt{\det g}g^{ij}$ are all polynomials of the components of $A$, by Lemma \ref{Lemma::FuncSpace::Product} their $\Co^s$-norms are bounded by a constant depending only on the upper bound of $\|A\|_{\Co^s}$. 
    Therefore combining \eqref{Eqn::Quant::Schauder::Tmp1} and \eqref{Eqn::Quant::Schauder::Tmp2}, since $F=\id+R$, we have
    \begin{equation}\label{Eqn::Quant::Schauder::Tmp3}
        \|F\|_{\Co^{s+1}(\Ball^n)}\le C_1'(n,s,s_0,c_1,M'),
    \end{equation}
    for some $C_1'>0$ that only depends on $n,s,s_0$ and an upper bound for $\|A\|_{\Co^s}$. Since $\|A\|_{\Co^s}\le M'$, $C_1'$ does not
    depend on $\|A\|_{\Co^s}$, just on $M'$.
    
    By  Proposition \ref{Prop::Key::ExistPDE} \ref{Item::Key::ExistPDE::2}, we have $\Phi=F^{-1}:\Ball^n\xrightarrow{\sim}\Ball^n$ and $\|\Phi\|_{\Co^{s_0+1}(\Ball^n;\R^n)}<c_1^{-1}$, where $c_1=c_1(n,s_0,s_0+1)$ is the constant from Proposition \ref{Prop::Key::ExistPDE}. So $\inf\limits_{x\in\Ball^n}|\det(\nabla F)(x)|$ is bounded  below by a constant depending only on $n,s_0,c_1$. Applying \cite[Lemmas 5.9 and 5.8]{StovallStreetII} with $\|F\|_{\Co^{s+1}(\Ball^n)}\le C_1'$ and $\|A\|_{\Co^s(\Ball^n;\Mbb^{n\times n})}+\sum_{i=1}^n\|d\lambda^i\|_{\Co^s}<M'$ we get $\Phi\in\Co^{s+1}(\Ball^n;\R^n)$ with
    \begin{equation}\label{Eqn::Quant::Schauder::Tmp4}
        \|\Phi\|_{\Co^{s+1}(\Ball^n;\R^n)}+\|A\circ\Phi\|_{\Co^s(\Ball^n;\Mbb^{n\times n})}+\sum_{i,j,k=1}^n\mleft\|\mleft(d\lambda^i\big(\tfrac\partial{\partial x^j},\tfrac\partial{\partial x^k}\big)\mright)\circ\Phi\mright\|_{\Co^s(\Ball^n)}\le C_2',
    \end{equation}
    for some $C_2'>0$ that only depends on $n,s,s_0,C_1'$ and the upper bound of $\|A\|_{\Co^s}$. Since $\|A\|_{\Co^s}\le M'$ and $C_1'$ depends only
    on $n,s,s_0,c_1$ and $M'$,  the same is true of $C_2'$, i.e.
    $C_2'=C_2'(n,s,s_0,c_1,M')$. Taking $\zeta$ such that $\zeta^{-1}\ge C_2'$ we complete the proof of \ref{Item::Quant::QuantExt}.
    
    \medskip
    \ref{Item::Quant::QuantReg}: By a direct computation (also see \eqref{Eqn::Key::AtoB::Proof1}) we have 
    \begin{equation}\label{Eqn::Quant::Schauder::Tmp5}
        B=\nabla\Phi-I+(A\circ\Phi)\nabla\Phi,\quad d\eta^i=\Phi^*d\lambda^i=\sum_{1\le j<k\le n}\mleft(\mleft(d\lambda^i\big(\tfrac\partial{\partial x^j},\tfrac\partial{\partial x^k}\big)\mright)\circ\Phi\mright)\cdot d\phi^j\wedge d\phi^k.
    \end{equation}
    Applying Lemma \ref{Lemma::FuncSpace::Product} to \eqref{Eqn::Quant::Schauder::Tmp5} and using  \eqref{Eqn::Quant::Schauder::Tmp4} we can find a $C'_3=C'_3(n,s,s_0,c_1,M')$ such that,
    \begin{equation}\label{Eqn::Quant::Schauder::Tmp6}
        \|B\|_{\Co^s(\Ball^n;\Mbb^{n\times n})}+\sum_{i=1}^n\|d\eta^i\|_{\Co^s(\Ball^n;\wedge^2T^*\Ball^n)}<C'_3(n,s,s_0,c_1,M').
    \end{equation}
    
    Applying \eqref{Eqn::Quant::Schauder::Tmp0} with $\tilde\sigma=C'_3$, where we recall that $\gamma$ depends only on $n$, we see that $C'_0(n,s_0,s,c_3',\gamma,C'_3)$ is a constant depending only on $n,s_0,s,M'$.
    
    Take $K'=C_2'(n,s,s_0,c_1,M')+C'_0(n,s,s_0,c_3',\gamma,C'_3)$, since $c_1$ and $c_3'$ are constants that only depend on $n,s_0$ we know $K'=K'(n,s,s_0,M')$ depends only on $n,s,s_0,M'$, which completes the proof of \ref{Item::Quant::QuantReg}.
\end{proof}

The proof of Proposition \ref{Prop::Key::Scaling} gives a similar regularity estimate:

\begin{lemma}\label{Lemma::Quant::ScalingReg}
Let $\alpha,\beta\in[\alpha,\alpha+1],\mu_0,\tilde c,M>0$ be as in Proposition \ref{Prop::Key::Scaling}. Let $s\ge\alpha$, $\tilde M>1$, there is a $\tilde K=\tilde K(n,\alpha,\beta,s,\mu_0,\tilde c,M,\tilde M)>0$ that satisfies the following:

Let $\theta^1,\dots,\theta^n\in\Co^\alpha(\mu_0\Ball^n;T^*\R^n)$ be as in the assumptions of Proposition \ref{Prop::Key::Scaling}. Suppose in addition to these assumptions  we have $\theta^1,\dots,\theta^n\in\Co^s$, $d\theta^1,\dots,d\theta^n\in\Co^s$ with estimate 
\begin{equation}\label{Eqn::Quant::ScalingReg::AssumptionBdd}
    \sum_{i=1}^n\|\theta^i\|_{\Co^s(\mu_0\Ball^n;T^*\R^n)}+\|d\theta^i\|_{\Co^s(\mu_0\Ball^n;\wedge^2T^*\R^n)}<\tilde M.
\end{equation}
Then $\lambda^1,\dots,\lambda^n$ constructed in Proposition \ref{Prop::Key::Scaling} satisfy $\lambda^1,\dots,\lambda^n\in\Co^s$, $d\lambda^1,\dots,d\lambda^n\in\Co^s$ with estimate 
\begin{equation*}
    \sum_{i=1}^n\|\lambda^i\|_{\Co^s(\Ball^n;T^*\R^n)}+\|d\lambda^i\|_{\Co^s(\Ball^n;\wedge^2T^*\R^n)}<\tilde K.
\end{equation*}
\end{lemma}
\begin{proof}
In the proof of Proposition \ref{Prop::Key::Scaling}, we construct $\lambda^1,\dots,\lambda^n$ as follows: For $i=1,\dots,n$,
\begin{gather}
    \tilde\rho^i:=\Green\ast\codiff (\chi_0\cdot d\theta^i+d\chi_0\wedge \theta^i),\quad\rho^i:=\tilde\rho^i-(\tilde\rho^i\big|_0),
    \\
    \lambda^i:=dx^i+\tfrac1{\kappa_0}\chi_1\cdot\phi_{\kappa_0}^*(\theta^i-dx^i),\quad\tau^i:=\tfrac1{\kappa_0}\chi_1\cdot(\phi_{\kappa_0}^*\rho^i)+\tfrac1{\kappa_0}\Green\ast\codiff\left(d\chi_1\wedge \phi_{\kappa_0}^*(\theta^i-dx^i)\right).
\end{gather}
Here $\chi_0\in C_c^\infty(\mu_0\Ball^n)$ satisfies $\chi_0\big|_{\frac{\mu_0}2\Ball^n}\equiv1$; $\chi_1\in C_c^\infty(\frac12\Ball^n)$ satisfies $\chi_1\big|_{\frac13\Ball^n}\equiv1$; $\Green$ is the Newtonian potential; $\codiff$ is the codifferential operator; $\kappa_0=\kappa_0(n,\alpha,\beta,\mu_0,\tilde c,M)>0$ is the scaling constant and $\phi_{\kappa_0}(x)=\kappa_0x$. 

By assumption $\theta^i,d\theta^i\in\Co^s$ with bound \eqref{Eqn::Quant::ScalingReg::AssumptionBdd}. Since $\phi_{\kappa_0}$ is a scaling map depending only on $\kappa_0$, we have
\begin{equation}\label{Eqn::Quant::ScalingReg:RegEst1}
    \|\lambda^i\|_{\Co^s(\Ball^n;T^*\R^n)}\lesssim_{\kappa_0,\mu_0,\chi_1}1+\|\theta^i-dx^i\|_{\Co^s(\mu_0\Ball^n;T^*\R^n)}\lesssim_{\kappa_0,\chi_1}\tilde M\quad i=1,\dots,n.
\end{equation}
Applying Lemma \ref{Lemma::FuncRevis::NewtonianBoundedness} (see also \eqref{Eqn::Key::ScalingTmp2}), with the same argument as \eqref{Eqn::Key::Scaling::RhoFormEstCompute}, we have for $i=1,\dots,n$, $$\|\rho^i\|_{\Co^{s+1}(\mu_0\Ball^n;T^*\R^n)}\lesssim_{s,\mu_0,\chi_0}\tilde M.$$
Also, by a direct estimate,
$$\|\tau^i\|_{\Co^{s+1}(\mu_0\Ball^n;T^*\R^n)}\lesssim_{\kappa_0,s,\mu_0,\chi_1}\|\rho^i\|_{\Co^{s+1}}+\|\theta^i-dx^i\|_{\Co^s}\lesssim_{s,\kappa_0,\mu_0,\chi_0,\chi_1}\tilde M.$$
In  \eqref{Eqn::Key::ScalingTmp2} it is shown that $d\lambda^i=d\tau^i$ and therefore,
\begin{equation}\label{Eqn::Quant::ScalingReg:RegEst2}
    \|d\lambda^i\|_{\Co^{s+1}(\mu_0\Ball^n;\wedge^2T^*\R^n)}=\|d\tau^i\|_{\Co^{s+1}(\mu_0\Ball^n;\wedge^2T^*\R^n)}\lesssim_s\|\tau^i\|_{\Co^{s+1}(\mu_0\Ball^n;T^*\R^n)}\lesssim_{s,\kappa_0,\mu_0,\chi_0,\chi_1}\tilde M.
\end{equation}

Combining \eqref{Eqn::Quant::ScalingReg:RegEst1} and \eqref{Eqn::Quant::ScalingReg:RegEst2}, since $\chi_0$ and $\chi_1$ are fixed cut-off functions and $\kappa_0=\kappa_0(n,\alpha,\beta,\mu_0,\tilde c,M)$, we get $\tilde K=\tilde K(n,s,\mu_0,\kappa_0,\tilde M)=\tilde K(n,\alpha,\beta,s,\mu_0,\tilde c,M,\tilde M)$ as desired.
\end{proof}

We can now prove Proposition \ref{Prop::Quant} by applying Proposition \ref{Prop::Quant::Schauder} and Lemma \ref{Lemma::Quant::ScalingReg}.
\begin{proof}[Proof of Proposition \ref{Prop::Quant}]
Let $\theta^1,\dots,\theta^n$ be the dual basis to $Y_1,\dots,Y_n$ on $B^n(0,\mu_0)$. Write $\theta=[\theta^1,\dots,\theta^n]^\top$ as $\theta=(I+B)dx$ where $B=(I+A)^{-1}-I$ and $dx=[dx^1,\dots,dx^n]^\top$. 

Clearly $B(0)=0$ because $\nabla\Phi_0(0)=I$. So
\begin{equation}\label{Eqn::Quant::Det(I+A)}
    \|(I+A)^{-1}\|_{C^0(B^n(0,\mu_0);\Mbb^{n\times n})}\le\sum_{j=0}^\infty\|A\|_{C^0(B^n(0,\mu_0);\Mbb^{n\times n})}^j\le\sum_{j=0}^\infty2^{-j}\le2,\text{ implying }\inf\limits_{|x|<\mu_0}|\det(I+A(x))|\ge 2^{-n}.
\end{equation} By assumption
$\|I+A\|_{\Co^{s_0}(B^n(0,\mu_0);\Mbb^{n\times n})}\le\|I\|_{\Co^{s_0}(B^n(0,\mu_0);\Mbb^{n\times n})}+\|A\|_{\Co^{s_0}(B^n(0,\mu_0);\Mbb^{n\times n})}\lesssim_{n,s_0,\mu_0,M_0}1$. Applying
\cite[Lemma 5.7]{StovallStreetII} along with \eqref{Eqn::Quant::Det(I+A)}, we have $\|I+B\|_{\Co^{s_0}(B^n(0,\mu_0);\Mbb^{n\times n})}\lesssim_{n,s_0,\mu_0,M_0}1$, which means
\begin{equation}\label{Eqn::Quant::BddofTheta}
    \exists \widehat M_1=\widehat M_1(n,s_0,\mu_0,M_0)>0,\quad\text{such that }\sum_{i=1}^n\|\theta^i\|_{\Co^{s_0}(B^n(0,\mu_0);T^*\R^n)}\le \widehat M_1.
\end{equation}

By \eqref{Eqn::ProofThm::BLemma::Tmp}, $[X_i,X_j]=\sum_{k=1}^nc_{ij}^kX_k$ implies that 
\begin{equation}\label{Eqn::Quant::Tmp0}
    d((\Phi_0)_*\theta^k)=\sum_{1\le i<j\le n}c_{ij}^k((\Phi_0)_*\theta^i)\wedge((\Phi_0)_*\theta^j), \quad\text{so}\quad d\theta^k=\sum_{1\le i<j\le n}\tilde c_{ij}^k\theta^i\wedge\theta^j,\quad k=1,\dots,n.
\end{equation}
Note that we cannot say $[Y_i,Y_j]=\sum_{k=1}^n\tilde c_{ij}^kY_k$ since we cannot define $[Y_i,Y_j]$ when $s\in(0,\frac12]$, while $d\theta^k$ and $\tilde c_{ij}^k\theta^i\wedge\theta^j$ in \eqref{Eqn::Quant::Tmp0} are defined due to Proposition \ref{Prop::FuncRn::LieOnManiofold} \ref{Item::FuncRn::WedgeProdCont} with the equality holding in the sense of distributions.

So by Lemma \ref{Lemma::FuncSpace::Product} we have 
\begin{equation}\label{Eqn::Quant::BddDTheta}
    \sum_{k=1}^n\|d\theta^k\|_{\Co^{s_0}(B^n(0,\mu_0);T^*\R^n)}\lesssim_{s_0,\mu_0}\sum_{i,j,k=1}^n\|\tilde c_{ij}^k\|_{\Co^{s_0}}\|\theta^i\|_{\Co^{s_0}}\|\theta^j\|_{\Co^{s_0}}\lesssim_{n,s_0,\mu_0}M_0^2\widehat M_1^2\lesssim_{n,s_0,\mu_0,M_0}1.
\end{equation}
In other words,
\begin{equation}\label{Eqn::Quant::BddofDTheta}
    \exists \widehat M_2=\widehat M_2(n,s_0,\mu_0,M_0)>0,\quad\text{such that }\sum_{i=1}^n\|d\theta^i\|_{\Co^{s_0}(B^n(0,\mu_0);T^*\R^n)}\le \widehat M_2.
\end{equation}

Applying {Lemma \ref{Lemma::Quant::ScalingReg}}
with $\alpha=s_0$, $\beta=s_0+1$, $M=\widehat M_1+\widehat M_2$, $\mu_0=\mu_0$ and $\tilde c=c'$, where $\widehat M_1$ is in \eqref{Eqn::Quant::BddofTheta}, $\widehat M_2$ is in \eqref{Eqn::Quant::BddofDTheta} and $c'=c'(n,s_0)$ is the constant in Proposition \ref{Prop::Quant::Schauder}, we can find $\kappa_0=\kappa_0(n,s_0,\mu_0,M_0)\in(0,\mu_0]$ and 1-forms $\lambda^1,\dots,\lambda^n\in\Co^{s_0}(\Ball^n;T^*\Ball^n)$ that satisfy the assumptions
of Proposition \ref{Prop::Quant::Schauder} with constant $c'$, that is
\begin{enumerate}[parsep=-0.3ex,label=(\alph*)]
    \item $\lambda^1,\dots,\lambda^n$ span the tangent space at every point in $\Ball^n$.
    \item $\supp(\lambda^i-dx^i)\Subset\frac12\Ball^n$ for $i=1,\dots,n$.
    \item\label{Item::Quant::ScalingCondition} $\lambda^i\big|_{\frac13\Ball^n}=\frac1{\kappa_0}\cdot(\phi_{\kappa_0}^*\theta^i)\big|_{\frac13\Ball^n}$ for $i=1,\dots,n$. Here $\phi_{\kappa_0}:\Ball^n\to B^n(0,\mu_0)$, $\phi_{\kappa_0}(x)=\kappa_0\cdot x$.
    \item\label{Item::Quant::SizeOfLambdaD} $\sum_{i=1}^n(\|\lambda^i-dx^i\|_{\Co^{s_0}}+\|d\lambda^i\|_{\Co^{s_0}})\le c'$
\end{enumerate}

By Proposition \ref{Prop::Quant::Schauder} with this $c'$ (see Theorem \ref{Thm::Keythm}, with $\alpha=s_0$ and $\beta=s_0+1$), we can find a map $F:\Ball^n\xrightarrow{\sim}\Ball^n$, such that $F(\frac13\Ball^n)\supseteq B^n(F(0),\frac16)$ and by endowing the codomain of $F$ with standard coordinate system $y=(y^1,\dots,y^n)$,
\begin{equation}\label{Eqn::Quant::BddofFDelta}
    \|F-\id\|_{\Co^{s_0+1}(\Ball^n;\R^n)}+\|F_*\lambda-dy\|_{\Co^{s_0+1}(\Ball^n;\Mbb^{n\times n})}\le c'^{-1}\sum_{i=1}^n\left(\|\lambda^i-dy^i\|_{\Co^{s_0}(\Ball^n;\R^n)}+\|d\lambda^i\|_{\Co^{s_0}(\Ball^n;\wedge^2T^*\Ball^n)}\right).
\end{equation}

Write $F_*\lambda=:(I+\widehat B)dy$. Note that by condition \ref{Item::Quant::SizeOfLambdaD}, the right hand side of \eqref{Eqn::Quant::BddofFDelta} is bounded by $1$ and therefore $\|F_*\lambda-dy\|_{\Co^{s_0+1}(\Ball^n;\Mbb^{n\times n})}\le1$, and
there is a $C_1=C_1(n,s_0)>0$ such that
\begin{equation}\label{Eqn::Quant::BddofHatB}
\|I+\widehat B\|_{\Co^{s_0+1}(\Ball^n;\Mbb^{n\times n})}\le C_1(n,s_0). 
\end{equation}
And by Proposition \ref{Prop::Quant::Schauder} \ref{Item::Quant::SchauderSupofB} we have $\|\widehat B\|_{C^0(\Ball^n;\Mbb^{n\times n})}<\frac14$. So 
\begin{equation}\label{Eqn::Quant::(I+B)^-1}
   \textstyle |(I+\widehat B(F(0)))^{-1}|_{\Mbb^{n\times n}}\le\|(I+\widehat B)^{-1}\|_{C^0(\Ball^n;\Mbb^{n\times n})}\le\sum_{k=0}^\infty(1/4)^k=\frac43.
\end{equation}

{Define an affine linear map 
\begin{equation}\label{Eqn::Quant::DefPsi}
    \psi(t):=(I+\widehat B(F(0)))^{-1}\cdot\tfrac t9+F(0),\quad t\in\Ball^n.
\end{equation} Note that by \eqref{Eqn::Quant::(I+B)^-1} we have $\psi(\Ball^n)\subseteq\frac43\cdot\frac19\Ball^n+F(0)\subset B^n(F(0),\frac16)$.

Define $\Phi_1:\Ball^n\to B^n(0,\mu_0)$ by
\begin{equation}\label{Eqn::Quant::DefPhi1}
    \Phi_1(t):=\phi_{\kappa_0}\circ F^{-1}\circ\psi(t)=\kappa_0\cdot F^{-1}\mleft((I+\widehat B(F(0)))^{-1}\cdot\tfrac t9+F(0)\mright),\quad t\in\Ball^n.
\end{equation}
Here $\Phi_1$ is well-defined because $\psi(\Ball^n)\subset B^n(F(0),\frac16)$, $B^n(F(0),\frac16)\subseteq F(\frac13\Ball^n)$ and $\phi_{\kappa_0}(\frac13\Ball^n)\subset\mu_0\Ball^n$. Clearly $\Phi_1(0)=\kappa_0\cdot F^{-1}(F(0))=0$.}

By condition \ref{Item::Quant::ScalingCondition}, $\lambda^i\big|_{\frac13\Ball^n}=\frac1{\kappa_0}\cdot(\phi_{\kappa_0}^*\theta^i)\big|_{\frac13\Ball^n}$, and the fact that $F^{-1}\circ\psi(\Ball^n)\subset\frac13\Ball^n$, we have
\begin{equation}\label{Eqn::Quant::Phi1*Theta}
\begin{aligned}
    (\Phi_1^*\theta)(t)&=\kappa_0\psi^*(F_*\lambda)=\kappa_0\psi^*((I+\widehat B(y))dy)=\kappa_0\mleft(I+\widehat B(\psi(t))\mright)d\psi(t)
    \\
    &=\tfrac{\kappa_0}9\mleft(I+\widehat B(\psi(t))\mright)\cdot \mleft(I+\widehat B(F(0))\mright)^{-1}dt.
\end{aligned}
\end{equation}

Since $\Phi_1^*Y_1,\dots,\Phi_1^*Y_n$ and $\Phi_1^*\theta^1,\dots,\Phi_1^*\theta^n$ are dual basis to each other, we can write
\begin{equation}\label{Eqn::Quant::EqnofPhi*Y}
    \Phi_1^*Y=:\tfrac 9{\kappa_0}\cdot(I+\widehat A)\tfrac\partial{\partial t},\quad\text{where }\widehat A(t)=(I+\widehat B(F(0)))\cdot(I+\widehat B(\psi(t)))^{-1}-I,\quad t\in\Ball^n.
\end{equation}

Taking $t=0$ in \eqref{Eqn::Quant::EqnofPhi*Y}, since $\psi(0)=F(0)$, we get $\widehat A(0)=0$.

Let $\widehat K:=\frac9{\kappa_0}$, since $\kappa_0=\kappa_0(n,s_0,\mu_0,M_0)$, we have $\widehat K=\widehat K(n,s_0,\mu_0,M_0)$ is as desired for \ref{Item::Quant::Phi1Exist}.

Since $\|\widehat B\|_{C^0}<\frac14$ and using the  power series $\widehat A(t)=(I+\widehat B(F(0)))\cdot\sum_{j=1}^\infty\widehat B(\psi(t))^j$, we have
\begin{equation}\label{Eqn::Quant::SupofHatA}
    \|\widehat A\|_{C^0(\Ball^n;\Mbb^{n\times n})}\le |(I+\widehat B(F(0)))|_{\Mbb^{n\times n}}\sum_{j=1}^\infty\|\widehat B\|_{C^0}^j<\tfrac54\sum_{j=1}^\infty(\tfrac14)^j=\tfrac53<\frac12,\quad \|I+\widehat A\|_{C^0(\Ball^n;\Mbb^{n\times n})}<\frac32.
\end{equation}

{This finishes the proof of \eqref{Eqn::Quant::Khat}.

To prove \eqref{Eqn::Quant::K0} we need to find the constant $K_0$.}

Applying \cite[Lemma 5.7]{StovallStreetII} to \eqref{Eqn::Quant::SupofHatA}, \eqref{Eqn::Quant::EqnofPhi*Y}, and \eqref{Eqn::Quant::BddofHatB}, we see that 
there is a $C_2=C_2(n,s_0,C_1)=C_2(n,s_0,\mu_0,M_0)>0$ such that 
\begin{equation}\label{Eqn::Quant::BddHatACs}
    \|I+\widehat A\|_{\Co^{s_0+1}(\Ball^n;\Mbb^{n\times n})}\le C_2(n,s_0,\mu_0,M_0).
\end{equation}

Since $F$ is constructed in Proposition \ref{Prop::Key::ExistPDE} (see Remark \ref{Rmk::Key::ExistPDE:FisNeeded}), by Proposition \ref{Prop::Key::ExistPDE} \ref{Item::Key::ExistPDE::2} we have $\|F^{-1}\|_{\Co^{s_0+1}}\le c_1^{-1}$ where $c_1=c_1(n,s_0,s_0+1)$ is the constant in Proposition \ref{Prop::Key::ExistPDE}. Since by \eqref{Eqn::Quant::DefPhi1}, $\Phi_1$ is an affine transform of $F^{-1}$, we have $\|\Phi_1\|_{\Co^{s_0+1}(\Ball^n;\R^n)}\lesssim_{n,s_0,\kappa_0}\|F^{-1}\|_{\Co^{s_0+1}(\Ball^n;\R^n)}$. So we can find a $C_3=C_3(n,s_0,\mu_0)>0$ such that $\|\Phi_1\|_{\Co^{s_0+1}(\Ball^n;\R^n)}\le C_3$.

Taking $K_0=\max(C_2+\|I\|_{\Co^{s_0+1}(\Ball^n;\Mbb^{n\times n})},C_3)$, we get \eqref{Eqn::Quant::K0} which completes the proof of \ref{Item::Quant::Phi1Exist}.

\bigskip
We now focus on the proof of  \ref{Item::Quant::Phi1Reg}, 
where we assume that additionally we have \eqref{Eqn::Quant::AssumptionReg}. By \cite[Lemma 5.7]{StovallStreetII} along with \eqref{Eqn::Quant::Det(I+A)}, we have $\|(I+A)^{-1}\|_{\Co^s}\lesssim_{n,s,\mu_0,M_0,M_1}1$, i.e. $\|\theta^i\|_{\Co^s}\lesssim_{n,s,\mu_0}1$. By \eqref{Eqn::Quant::Tmp0} with the same argument as \eqref{Eqn::Quant::BddDTheta}, we get $\|d\theta^i\|_{\Co^s}\lesssim_{n,s,\mu_0,M_0,M_1}1$. In other words, where exists a $\tilde M=\tilde M(n,s,\mu_0,M_0,M_1)>0$ such that
\begin{equation*}
    \sum_{i=1}^n\big(\|\theta^i\|_{\Co^s(B^n(0,\mu_0),T^*\R^n)}+\|d\theta^i\|_{\Co^s(B^n(0,\mu_0),\wedge^2T^*\R^n)}\big)\le \tilde M.
\end{equation*}

By Lemma \ref{Lemma::Quant::ScalingReg}, the 1-forms $\lambda^1,\dots,\lambda^n\in\Co^{s_0}(\Ball^n;T^*\Ball^n)$ constructed above are all $\Co^s$ and satisfy $d\lambda^1,\dots,d\lambda^n\in\Co^{s}$ with estimate
\begin{equation}\label{Eqn::Quant::BddLambdaReg}
    \sum_{i=1}^n\big(\|\lambda^i\|_{\Co^s(\Ball,T^*\R^n)}+\|d\lambda^i\|_{\Co^s(\Ball^n,\wedge^2T^*\R^n)}\big)\le \tilde K(n,s,s_0,\mu_0,M_0,M_1),
\end{equation}
where $\tilde K=\tilde K(n,s,s_0,\mu_0,M_0,M_1)>0$ is the constant obtained in Lemma \ref{Lemma::Quant::ScalingReg}.

By Proposition \ref{Prop::Quant::Schauder} with assumption \eqref{Eqn::Quant::BddLambdaReg} (i.e. $M'=\tilde K$ in its assumption), we have $F\in\Co^{s+1}(\Ball^n;\R^n)$, $F_*\lambda^1,\dots,F_*\lambda^n\in\Co^{s+1}$, and moreover there is a $C_4=C_4(n,s_0,s,\mu_0,M_0,M_1)>0$ (which is the $K'$ in the conclusion of Proposition \ref{Prop::Quant::Schauder}) that does not depend on $F$ and $\lambda^1,\dots,\lambda^n$, such that 
\begin{equation}\label{Eqn::Quant::RegF1Lambda}
    \|F^{-1}\|_{\Co^{s+1}(\Ball^n;\R^n)}+\|F_*\lambda\|_{\Co^{s+1}(\frac34\Ball^n;\Mbb^{n\times n})}\le C_4(n,s_0,s,\mu_0,M_0,M_1).
\end{equation}

Since we have $B^n(F(0),\frac16)\subseteq F(\frac13\Ball^n)\cap\frac34\Ball^n$ from Theorem \ref{Thm::Keythm}, combining \eqref{Eqn::Quant::RegF1Lambda} and \eqref{Eqn::Quant::Phi1*Theta} we can find a $C_5=C_5(n,s_0,s,\mu_0,M_0,M_1)>0$ such that
\begin{equation}\label{Eqn::Quant::RegPhi*Theta}
    \|\Phi_1\|_{\Co^{s+1}(\Ball^n;\R^n)}+\|\Phi_1^*\theta\|_{\Co^{s+1}(\Ball^n;\Mbb^{n\times n})}\le C_5(n,s_0,s,\mu_0,M_0,M_1).
\end{equation}

Applying \cite[Lemma 5.7]{StovallStreetII} on \eqref{Eqn::Quant::RegPhi*Theta}, \eqref{Eqn::Quant::EqnofPhi*Y} and \eqref{Eqn::Quant::SupofHatA} we see that $\|I+\widehat A\|_{\Co^{s+1}}\lesssim_{n,s_0,s,\mu_0,M_0,M_1}1$. So there is a $C_6=C_6(n,s_0,s,\mu_0,M_0,M_1)$ such that
\begin{equation}\label{Eqn::Quant::BddHatAReg}
    \|\widehat A\|_{\Co^{s+1}(\Ball^n;\Mbb^{n\times n})}\le C_6(n,s_0,s,\mu_0,M_0,M_1).
\end{equation}

Take $K_1=\max(C_5,C_6)$, we get \eqref{Eqn::Quant::K1} which completes the proof of \ref{Item::Quant::Phi1Reg}.
\end{proof}

Taking $\Phi=\Phi_1\circ\Phi_0$ nearly completes the proof of Theorem \ref{Thm::Quant} (see also \cite[Theorem 2.14]{StovallStreetII}) except we have not established
the injectivity of $\Phi$, since we have only shown $\Phi_0$ is a local $C^1$-diffeomorphism rather than a global $C^1$-diffeomorphism onto its image. This problem can be resolved through the next result:

\begin{lemma}\label{Lemma::Quant::InjLemma}
Let $s_0,\mu_0,M_0>0$ be as in Proposition \ref{Prop::Quant}. Then there is a $\mu_1=\mu_1(n,s_0,\mu_0,M_0)\in(0,1]$ depending only on $n,s_0,\mu_0,M_0$ and satisfying the following:

If $C^1$-vector fields $X_1,\dots,X_n$ on $\Manifold$ that satisfy the assumptions of Proposition \ref{Prop::Quant} with addition that\footnote{Also see the quantity $\eta>0$ in \cite[Section 3.2]{StovallStreetI}.}:
\begin{itemize}
    \item Let $U:=\Phi_0\circ\Phi_1(\Ball^n)\subseteq\Manifold$. For any point $q\in U$ and $\mu\in(0,\mu_0]$, if the exponential $t\mapsto e^{t\cdot X}(q)$ is defined for $t\in B^n(0,\mu)$, then $e^{t\cdot X}(q)\neq q$ holds for $t\in B^n(0,\mu)\backslash\{0\}$.
\end{itemize}
Then $\Phi_0\circ\Phi_1$ is injective in $B^n(0,\mu_1)$. Moreover $\Phi_0\circ\Phi_1\big|_{B^n(0,\mu_1)}:B^n(0,\mu_1)\to\Manifold$ is $C^2$-diffeomorphism onto its image.
\end{lemma}
\begin{rmk}\label{Rmk::Quant::Phi0isInj}
By Proposition \ref{Prop::Quant::Schauder}, $\Phi_1:\Ball^n\to B^n(\mu_0)$ is a $\Co^{s_0+1}$-diffeomorphism onto its image and satisfies $\Phi_1(0)=0$. 
By \eqref{Eqn::Quant::Khat}, we have $\Phi_1^*(I+A)\nabla=\widehat K(I+\widehat A)\nabla$ where $\|A\|_{C^0},\|\widehat A\|_{C^0}\le\frac12$, so   
\begin{equation}\label{Eqn::Quant::InjLemma::PfTmp}
    \|(\nabla\Phi_1)^{-1}\|_{C^0(\Ball^n;\Mbb^{n\times n})}=\|\Phi_1^*\nabla\|_{C^0(\Ball^n;\Mbb^{n\times n})}\le \widehat K\|I+\widehat A\|_{C^0(\Ball^n;\Mbb^{n\times n})}\|(I+A)^{-1}\|_{C^0(\mu_0\Ball^n;\Mbb^{n\times n})}\le 3\widehat K.
\end{equation}
Here $\widehat K$ is an $\{s_0\}$-admissible constant. 
\eqref{Eqn::Quant::InjLemma::PfTmp} implies $\Phi_1(B^n(0,r))\supseteq B^n(0,3\widehat Kr)$ for all $r\in(0,1]$, 
therefore Lemma \ref{Lemma::Quant::InjLemma} tells us that $\Phi_0\big|_{B^n((3\widehat K)^{-1}\mu_1)}:B^n\left(0,(3\widehat K)^{-1}\mu_1\right)\to\Manifold$ is injective.
\end{rmk}
\begin{proof}

By Proposition \ref{Prop::Quant}, we have $Y=\Phi_0^*X=(I+A)\Coorvec x$ and $\Phi_1^*Y=\widehat K(I+\widehat A)\Coorvec t$ are such that $\|A\|_{C^0(B^n(\mu_0);\Mbb^{n\times n})},\|\widehat A\|_{C^0(\Ball^n;\Mbb^{n\times n})}\le\frac12$ and $\widehat K\lesssim_{n,s_0,\mu_0,M_0}1$ ($\widehat K$ is an $\{s_0\}$-admissible constant).


Clearly $\|(I+\widehat A)^{-1}\|_{C^0(\Ball^n;\Mbb^{n\times n})}\le2$, so
\begin{equation*}
    \inf\limits_{t\in\Ball^n}\det(\Phi_1^*Y)(t)\ge\widehat K\|(I+\widehat A)^{-1}\|_{C^0(\Ball^n;\Mbb^{n\times n})}^{-n}\ge2^{-n}\widehat K\gtrsim_{n,s_0,\mu_0,M_0}1.
\end{equation*}

On the other hand by \eqref{Eqn::Quant::K0} we have $$\|\Phi_1^*Y\|_{C^1(\Ball^n;\Mbb^{n\times n})}\lesssim_{n,s_0}\|\Phi_1^*Y\|_{\Co^{s_0+1}(\Ball^n;\Mbb^{n\times n})}=\widehat K\|I+\widehat A\|_{\Co^{s_0+1}(\Ball^n;\Mbb^{n\times n})}\le\widehat K(1+K_0)\lesssim_{n,s_0,\mu_0,M_0}1.$$ 
{And by \eqref{Eqn::Quant::AssumptionM0}}
we have $$\sum_{i,j,k=1}^n\|(\Phi_0\circ\Phi_1)^*c_{ij}^k\|_{C^0(\Ball^n)}\le\sum_{i,j,k=1}^n\|c_{ij}^k\|_{C^0(\Manifold)}\le M_0\lesssim_{M_0}1.$$ Therefore applying \cite[Proposition 9.15]{StovallStreetI} to the map $\Phi_0\circ\Phi_1(t)=e^{t\cdot\Phi_1^*Y}(0)$ we can find a $\mu_1>0$ that depends only on $n,s_0,\mu_0$  such that $\Phi_0\circ\Phi_1$ is injective on $B^n(\mu_1)$.

Since $(X_1,\dots,X_n)$ and $(\Phi_0\circ\Phi_1)^*(X_1,\dots, X_n)$ are both $C^1$ and span their respective tangent spaces at every point, we know $\Phi_0\circ\Phi_1|_{B^n(0,\mu_1)}$ is a  $C^2$-map with non-degenerate tangent map at every point in the domain. Since we
have also shown it is injective, we conclude it is a $C^2$-diffeomorphism onto its image.
\end{proof}

By combining Lemma \ref{Lemma::Quant::InjLemma}, Propositions \ref{Prop::Quant}, and \ref{Prop::Quant::Schauder}, we can prove Theorem \ref{Thm::Quant}:
\begin{proof}[Proof of Theorem \ref{Thm::Quant}]
    As mentioned before, once we establish \cite[Theorem 2.14]{StovallStreetII}
    for $s,s_0>0$, the same follows for \cite[Theorem 4.5]{StreetSubHermitian}.
    Thus, we prove only \cite[Theorem 2.14]{StovallStreetII} for $s,s_0>0$.
    
    Fix $x_0\in\Manifold$. Since the results \cite[Theorem 2.14 (a), (b), and (c)]{StovallStreetII} do not depend on $s_0$ and $s$, we do not need to change their proof.

    Recall, we have reordered $X_1,\ldots, X_q$ such that \eqref{Eqn::Quant::DensityBound}
    holds with $j_1^0=1,\ldots, j_n^0=n$.
    Set $X_{J_0}:=(X_1,\dots,X_n)$. 
    Let $\Phi_0(t):=e^{t_1X_1+\dots+t_nX_n}(x_0)$. By Lemma \ref{Lemma::Quant::Phi0} we can find a $0$-admissible constant $\mu_0$ such that $\Phi_0:B^n(\mu_0)\to\Manifold$ is a local $C^1$-diffeomorphism. And moreover by writing $Y_i:=\Phi_0^*X_i$ for $i=1,\dots,q$ and $Y_{J_0}=:(I+A)\nabla$, we have $\|A\|_{C^0(B^n(\mu_0);\Mbb^{n\times n})}\le\frac12$, $\|A\|_{\Co^s(B^n(\mu_0);\Mbb^{n\times n})}\lesssim_{\{s\}}1$. And we can find $(b_k^l)_{1\le l\le n<k\le q}$ and $(\tilde c_{ij}^k)_{i,j,k=1}^n$ such that $Y_k=\sum_{l=1}^nb_k^lY_l$, $\Phi_0^*[X_i,X_j]=\sum_{l=1}^n\tilde c_{ij}^lY_l$ for $1\le i,j\le n<k\le q$ with $\sum_{l=1}^n\sum_{k=n+1}^q\|b_k^l\|_{\Co^{s+1}(B^n(\mu_0))}+\sum_{i,j,l=1}^n\|\tilde c_{ij}^l\|_{\Co^s(B^n(\mu_0))}\lesssim_{\{s\}}1$.
    
    Let $\Phi_1$ be the map given in Proposition \ref{Prop::Quant}, and let $\mu_1>0$ be the constant (which is $\{s_0\}$-admissible) from Lemma \ref{Lemma::Quant::InjLemma}. We define $\Phi:B^n(1)\to\Manifold$ by
    \begin{equation}\label{Eqn::Quant::FinalProof:DefPhi}
        \Phi(t):=\Phi_0\circ\Phi_1(\mu_1\cdot t).
    \end{equation}
    
    By Lemma \ref{Lemma::Quant::InjLemma}, $\Phi$ is a $C^2$-diffeomorphism onto its image and \cite[Theorem 2.14 (d), (e) and (g)]{StovallStreetII} follow.
    
    {By \eqref{Eqn::Quant::Khat} and \eqref{Eqn::Quant::FinalProof:DefPhi}, (where we use $X=X_{J_0}$ and $Y=Y_{J_0}$ in Proposition \ref{Prop::Quant}), we have $(\Phi_0\circ\Phi_1)^*X_{J_0}=\Phi_1^*Y_{J_0}=\widehat K(1+\widehat A)\Coorvec t$, so}
    \begin{equation}\label{Eqn::Quant::FinalProof::PullbackX}
        \Phi^*X_{J_0}(t)=\tfrac{\widehat K}{\mu_1}(I+\widehat A(\mu_1\cdot t))\nabla,\quad t\in B^n(1).
    \end{equation}
    Note that $\tfrac{\widehat K}{\mu_1}$ is bounded by a $\{s_0\}$-admissible constant and by Proposition \ref{Prop::Quant} \ref{Item::Quant::Phi1Exist} we have $\widehat A(0)=0$ and $\|\widehat A(\mu_1\cdot )\|_{C^0(\Ball^n;\Mbb^{n\times n})}\le\frac12$ and \cite[Theorem 2.14 (h) and (i)]{StovallStreetII} follow. In particular we have
    \begin{equation*}
        \tfrac{\widehat K}{2\mu_1}\dist_{\Phi^*X_{J_0}}(t_1,t_2)\le\dist_\nabla(t_1,t_2)\le \tfrac{3\widehat K}{2\mu_1}\dist_{\Phi^*X_{J_0}}(t_1,t_2),\quad\forall t_1,t_2\in\Ball^n.
    \end{equation*}
    
    Since $|t_1-t_2|=\dist_\nabla(t_1,t_2)$, taking pushforward of $\Phi$ we get $B_{X_{J_0}}(p,\tfrac{2\mu_1}{3\widehat K})\subseteq\Phi(\Ball^n)\subseteq B_{X_{J_0}}(p,\tfrac{2\mu_1}{\widehat K})$.
    
    {By Lemma \ref{Lemma::Quant::Phi0} \ref{Item::Quant::Phi0::LinDep}, we have $Y_k=\sum_{l=1}^nb_k^lY_l$ for $n+1\le k\le q$ such that $\|b_k^l\|_{C^0(B^n(\mu_0))}\lesssim_{s_0}\|b_k^l\|_{\Co^{s_0+1}(B^n(\mu_0))}\lesssim_{\{s_0\}}1$. So on the set $\Phi_1(B^n(\mu_0))$ (note that $\Phi_0\big|_{\Phi_1(B^n(\mu_0))}$ is injective), there is a $\{s_0\}$-admissible constant $C_1>0$ such that 
    \begin{equation*}
        \dist_{Y_{J_0}}(x_1,x_2)\le C_1\dist_{Y}(x_1,x_2),\quad\forall x_1,x_2\in \Phi_1(B^n(\mu_0)).
    \end{equation*}
    Taking pushforward of $\Phi_0$ we get $B_X(p,\tfrac{2\mu_1}{3C_1\widehat K})\subseteq B_{X_{J_0}}(p,\tfrac{2\mu_1}{3\widehat K})$. This proves \cite[Theorem 2.14 (f)]{StovallStreetII}.} 
    
    Combining \eqref{Eqn::Quant::FinalProof::PullbackX} and \eqref{Eqn::Quant::BddHatAReg}, since $\tfrac{\widehat K}{\mu_1}$ is $\{s_0\}$-admissible which is $\{s\}$-admissible, we get \cite[Theorem 2.14 (j)]{StovallStreetII}.
    
    Finally the proof of \cite[Theorem 2.14 (k) and (l)]{StovallStreetII} is the same as in \cite[Section 7]{StovallStreetII}.
\end{proof}




\bibliographystyle{amsalpha}

\bibliography{regvfs}

\center{\it{University of Wisconsin-Madison, Department of Mathematics, 480 Lincoln Dr., Madison, WI, 53706}}

\center{\it{street@math.wisc.edu, lyao26@wisc.edu}}

\center{MSC 2020:  57R25 (Primary), 35R05 and 53C17 (Secondary)}

\end{document}